\documentclass[11pt,english]{extreport}
\PassOptionsToPackage{natbib=true}{biblatex}
\usepackage[T1]{fontenc}
\usepackage[latin9]{inputenc}
\usepackage{babel}
\usepackage{array}
\usepackage{verbatim}
\usepackage{textcomp}
\usepackage{bbding}
\usepackage{url}
\usepackage{multirow}
\usepackage{amsmath}
\usepackage{amsthm}
\usepackage{amssymb}
\usepackage{graphicx}
\usepackage{tablefootnote}
\usepackage{setspace}
\usepackage[unicode=true,pdfusetitle,
 bookmarks=true,bookmarksnumbered=false,bookmarksopen=false,
 breaklinks=false,pdfborder={0 0 1},backref=false,colorlinks=false]
 {hyperref}

\makeatletter

\let\pr@chap=\pr@cha
\providecommand{\tabularnewline}{\\}

\usepackage{geometry}
\usepackage{babel} 
\usepackage{url} 
\usepackage{enumerate}
\usepackage{xcolor}
\usepackage[framemethod=tikz]{mdframed}
\usepackage{capt-of}
\usepackage{import}
\usepackage{csquotes}

\usepackage[T1]{fontenc}
\usepackage{lmodern}
\usepackage{MnSymbol}

\usepackage[hypcap=false,tableposition=top]{caption}
\usepackage{multicol}
\usepackage[section]{algorithm}
\usepackage[noend]{algpseudocode}
\usepackage{varwidth}

\newcommand{\hsp}{\hspace{20pt}}
\newcommand{\framewidth}{1.5pt}
\definecolor{green}{rgb}{0,0.5,0.0} 

\mdfsetup{
    roundcorner=5pt
    nobreak=false,
    needspace=100pt,
	backgroundcolor=gray!4,
    skipabove=1em,
	splittopskip=1.5em,
	linewidth=\framewidth,
	align=center,
    firstextra={\draw[dashed,very thin,xshift=1pt] (O) -- (P|-O);},
    secondextra={\draw[dashed,very thin,xshift=-1pt] (O|-P) -- (P);}
}
\hypersetup{
    colorlinks,
    citecolor=black,
    filecolor=black,
    linkcolor=black,
    urlcolor=black
}

\newenvironment{smalgorithmic}
{
    \begin{algorithmic}[1]
	\singlespacing
	\vspace{-3em}
	\setlength{\abovedisplayskip}{0.25em}
	\setlength{\belowdisplayskip}{0.25em}
	\bgroup
	\everymath{\displaystyle}
}
{
	\egroup
	\end{algorithmic}
}

\newcommand{\algpart}[1]{
    \hspace*{-1em} \fbox{\textsc{Part~#1}} \enskip
}
\newcommand{\mdalgcaption}[2]{
	\captionof{algorithm}{#1}\label{#2}
	\rule{1\linewidth}{\framewidth}
}
\newcommand{\StateEq}[1]{
    \vspace*{0.1em}
    \State{#1}
    \vspace*{0.1em}
}
\newcommand{\StateStep}[1]{
	\vspace*{0.5em}\State{#1}
}

\algrenewcommand\algorithmicrequire{\textbf{Require:}}
\algnewcommand\algorithmicinput{\textbf{Input:}}
\algnewcommand\Input{\item[\algorithmicinput]}
\algnewcommand\algorithmicoutput{\textbf{Output:}}
\algnewcommand\Output{\item[\algorithmicoutput]}
\algnewcommand\algorithmicinitialize{\textbf{Initialize:}}
\algnewcommand\Initialize{\item[\algorithmicinitialize]}

\algdef{SE}[DOWHILE]{Do}{doWhile}{\algorithmicdo}[1]{\algorithmicwhile\ #1}%

\newcounter{algorithmicH}
\let\oldalgorithmic\algorithmic
\renewcommand{\algorithmic}{%
  \stepcounter{algorithmicH}
  \oldalgorithmic}
\renewcommand{\theHALG@line}{ALG@line.\thealgorithmicH.\arabic{ALG@line}}
\makeatother

\usepackage{etoolbox}
\usepackage{titlesec}

\usepackage{color}
\definecolor{gray75}{gray}{0.75}
\titleformat{\chapter}[display]{\centering\large\bfseries}{\vspace*{-2em}CHAPTER \thechapter}{-0.5em}{\MakeUppercase}

\titlespacing*{\chapter}{0pt}{*0}{2em}

\makeatletter
\patchcmd{\chapter}{\if@openright\cleardoublepage\else\clearpage\fi}{}{}{}

\titleformat{\section}{\large\bfseries}{\thesection\hsp}{0em}{}

\titleformat{\subsection}{\normalsize\bfseries}{\thesubsection\hsp}{0em}{}

\usepackage{enumitem}
\newcounter{assumption}
\setenumerate[1]{label=(\Alph{assumption}\arabic*)}

\usepackage{newfloat}
\DeclareFloatingEnvironment[name=Problem,within=section]{problem}
\newcommand{\mdprbcaption}[2]{
    \captionof{problem}{#1}\label{#2}
    \rule{1\linewidth}{\framewidth}
}

\usepackage[capitalise,noabbrev,nameinlink]{cleveref}
\let\prettyref\cref
\crefname{problem}{Problem}{Problems}
\crefname{prop}{Proposition}{Propositions}

\usepackage{float}
\floatstyle{plaintop}
\restylefloat{table}

\usepackage[titles]{tocloft}
\usepackage{etoolbox}

\makeatletter

\g@addto@macro\appendix{%
  \addtocontents{toc}{%
    \protect%
  }%
}

\makeatother
\global\long\def\cK{{\cal K}}%
\global\long\def\rn{\mathbb{R}^{n}}%
\global\long\def\R{\mathbb{R}}%
\global\long\def\r{\mathbb{R}}%
\global\long\def\n{\mathbb{N}}%
\global\long\def\c{\mathbb{C}}%
\global\long\def\pt{\mathbb{\partial}}%
\global\long\def\lam{\lambda}%
\global\long\def\argmin{\operatorname*{argmin}}%
\global\long\def\Argmin{\operatorname*{Argmin}}%
\global\long\def\argmax{\operatorname*{argmax}}%
\global\long\def\dom{\operatorname*{dom}}%
\global\long\def\ri{\operatorname*{ri}}%
\global\long\def\inner#1#2{\langle#1,#2\rangle}%
\global\long\def\dg{\operatorname*{dg}}%
\global\long\def\trc{\operatorname*{tr}}%
\global\long\def\Dg{\operatorname*{Dg}}%
\global\long\def\cConv{\overline{{\rm Conv}}\ }%
\global\long\def\prox{\operatorname{prox}}%
\global\long\def\aff{\operatorname{aff}}%
\global\long\def\ri{\operatorname{ri}}%
\global\long\def\cl{\operatorname{cl}}%
\global\long\def\intr{\operatorname{int}}%
\global\long\def\dist{\operatorname{dist}}%

\global\long\def\acgX#1{z_{#1}}%
\global\long\def\acgU#1{r_{#1}}%
\global\long\def\acgMatX#1{Z_{#1}}%
\global\long\def\icgY#1{z_{#1}}%
\global\long\def\icgMatY#1{\uppercase{z}_{#1}}%
\global\long\def\aicgY#1{z_{#1}^{a}}%
\global\long\def\aicgYMin#1{z_{#1}}%
\global\long\def\aicgXTilde#1{\tilde{y}_{#1}}%
\global\long\def\aicgXhat#1{\widehat{y}_{#1}}%
\global\long\def\aicgX#1{y_{#1}}%
\global\long\def\pTick{+}%
\global\long\def\gammaBFn{\gamma}%
\global\long\def\gammaBTFn{\widetilde{\gamma}}%
\global\long\def\gammaFn#1{\gammaBFn\left(#1\right)}%
\global\long\def\gammaTFn#1{\gammaBTFn\left(#1\right)}%
\global\long\def\Y{z^{a}}%
\global\long\def\YMin{z}%
\global\long\def\Xt{\tilde{y}}%
\global\long\def\X{y}%
\global\long\def\Xh{\widehat{\X}}%
\global\long\def\YM{z}%
\global\long\def\XtM{\Xt}%
\global\long\def\XM{\X}%
\global\long\def\aM{a}%
\global\long\def\AM{A}%
\global\long\def\YP{\Y_{\pTick}}%
\global\long\def\YMinP{\YMin_{\pTick}}%
\global\long\def\XhP{\Xh_{\pTick}}%
\global\long\def\XP{\X_{\pTick}}%
\global\long\def\AP{A_{\pTick}}%
\global\long\def\thetaM{\theta}%
\global\long\def\thetaP{\theta_{\pTick}}%

\theoremstyle{plain}
\newtheorem{thm}{\protect\theoremname}[section]
\theoremstyle{definition}
\newtheorem{defn}[thm]{\protect\definitionname}
\theoremstyle{plain}
\newtheorem{prop}[thm]{\protect\propositionname}
\theoremstyle{plain}
\newtheorem{lem}[thm]{\protect\lemmaname}
\theoremstyle{plain}
\newtheorem{cor}[thm]{\protect\corollaryname}

\makeatother

\usepackage[style=numeric,giveninits=true,sorting=nyt,sortcites,backend=bibtex]{biblatex}
\providecommand{\corollaryname}{Corollary}
\providecommand{\definitionname}{Definition}
\providecommand{\lemmaname}{Lemma}
\providecommand{\propositionname}{Proposition}
\providecommand{\theoremname}{Theorem}

\begin{document}
\pagenumbering{roman}
\singlespacing\begin{center}
\textbf{\large{}Accelerated Inexact First-Order Methods for Solving
Nonconvex Composite Optimization Problems}\\
\vspace{9em}
A Dissertation\\
Presented to\\
The Academic Faculty\\
\vspace{2em}
By\\
\vspace{2em}
Weiwei Kong\\
\vspace{3em}
In Partial Fulfillment\\
of the Requirements for the Degree\\
Doctor of Philosophy in Operations Research\\
\vspace{5em}
\par\end{center}

\begin{center}
H. Milton Stewart School of Industrial and Systems Engineering\\
Georgia Institute of Technology\\
May 2021\\
\vfill{}
Copyright © 2021 by Weiwei Kong%
\par\end{center}
\thispagestyle{empty}

\clearpage

\newpage{}

\singlespacing\begin{center}
\textbf{\large{}Accelerated Inexact First-Order Methods for Solving
Nonconvex Composite Optimization Problems}\\
\textbf{\large{}\vfill{}
}{\large\par}
\par\end{center}

\begin{center}
\begin{minipage}[t]{0.5\columnwidth}%
Approved by:\vspace*{3em}\\
Dr. Renato D.C. Monteiro (Advisor)\\
H. Milton Stewart School of Industrial \\
and Systems Engineering\\
\emph{Georgia Institute of Technology}\vspace*{1em}\\
Dr. Arkadi Nemirovski\\
H. Milton Stewart School of Industrial \\
and Systems Engineering\\
\emph{Georgia Institute of Technology}\vspace*{1em}\\
Dr. Guanghui Lan\\
H. Milton Stewart School of Industrial\\
and Systems Engineering\\
\emph{Georgia Institute of Technology}\vspace*{1em}%
\end{minipage}\hspace*{\fill}%
\begin{minipage}[t]{0.5\columnwidth}%
\vspace*{3.5em}Dr. Santanu S. Dey\\
H. Milton Stewart School of Industrial\\
and Systems Engineering\\
\emph{Georgia Institute of Technology}\vspace*{1em}\\
Dr. Edmond Chow\\
School of Computational Science \\
and Engineering\\
\emph{Georgia Institute of Technology}\vspace*{2em}\\
\end{minipage}\vspace*{1em}
\par\end{center}

\begin{flushright}
Date Approved: April 2, 2021
\par\end{flushright}
\thispagestyle{empty}

\newpage{}

\doublespacing\begin{center}
\vspace*{\fill}
\par\end{center}

\begin{center}
\emph{\large{}To my parents, my sister, and }{\large\par}
\par\end{center}

\begin{center}
\emph{\large{}the many friends and mentors that I have met along the
way.}{\large\par}
\par\end{center}

\begin{center}
\vspace*{\fill}
\par\end{center}
\thispagestyle{empty}\newpage{}

\doublespacing
\phantomsection\begin{center}
\MakeUppercase{\textbf{\textsc{\large{}Acknowledgments}}}\vspace*{2em}
\par\end{center}

First and foremost, I would like to extend my deepest gratitude towards
my advisor, Renato D.C. Monteiro, who has provided me immense support
throughout my Ph.D. journey. Without his invaluable insight, unrelenting
guidance, and steadfast patience, this thesis would not be where it
is today. I am also grateful to my collaborator Jefferson G. Melo
for his countless conversations about research, life, and career advancement.
Special thanks should go to Arkadi Nemirovski, whose numerous suggestions
and discussions have immensely influenced this thesis's direction.
I would also like to thank all my committee members, including my
advisor, Renato D.C. Monteiro, Arkadi Nemirovski, Guanghui \textquotedbl George\textquotedbl{}
Lan, Edmond Chow, and Santanu S. Dey. 

Next, I would like to thank faculty members Craig Tovey and the late
Shabbir Ahmed for their extensive support during my semesters as a
graduate teaching assistant. Professor Tovey provided invaluable feedback
on my problem formulations, and I have had many meaningful conversations
with Professor Ahmed about career advice and exciting problems in
polyhedral theory. I am also grateful for the support that I have
received from the ISyE department as a whole. Particular thanks go
out to Alan Erera, who has provided valuable guidance in my early
years, and Amanda Ford, who has been especially helpful with my studies'
administrative side. 

The material in this thesis would not have been possible without the
generous financial support of the ISyE department, the Natural Sciences
and Engineering Research Council of Canada (NSERC), the Institute
for Data Engineering and Science (IDEaS), the Transdisciplinary Research
Institute for Advancing Data Science (TRIAD), and the Johnson family.
Specific grants include the Alexander Graham Bell Postgraduate Scholarship
(\#\textsc{PGSD3-516700-2018}), the IDEaS-TRIAD Research Scholarship,
and the Thomas Johnson Fellowship.

I am incredibly grateful to Bill Cook, whose advanced course in optimization
at the University of Waterloo was a driving factor for my decision
to begin graduate studies at Georgia Tech. I am also highly appreciative
of his helpful career and research advice. From my summer internships
at Google Research, I would like to thank my mentors Aranyak Mehta,
D. Sivakumar, Nicolas Mayoraz, Walid Krishne, and my fellow collaborators
Christopher Liaw, Steffen Rendle, and Li Zhang. Their insightful discussions
have helped shape several of my research and career decisions. 

I would now like to extend my gratitude to the many friends and colleagues
I have made at Tech. A few of the fellow graduate studies that I would
like to thank are Tyler Perini, Cyrus Rich, Andrew ElHabr, Adrian
Rivera, Mohamed El Tonbari, Ramon Auad, Ian Herszterg, Edward Yuhang
He, Reem Khir, Georgios Kotsalis, Digvijay Boob, Sarah Wiegreffe,
and Zhehui Chen. Additionally, I would like to thank my fellow lab
member, Jiaming Liang, and my office roommate, Alexander Stroh, for
the many memorable conversations. A few of the members of the Georgia
Tech Hapkido Club that I would like to thank are Graham Saunders,
Jason Ngor Shing Yi, Andrew Schulz, Erik Anderson, Christian Girala,
and my instructors Olivia Lodise, Mike Mackenzie, Hung Le, Joel Dunham,
Christi Nakajima, Melissa Johnson, Matthieu Bloch, and Grandmaster
Nils Onsager. 

Additionally, I would like to thank several of my friends from Canada
and abroad, including Jeffrey Negrea, Snow Murdock, Jamie Murdoch,
Lawson Fulton, Robert Zimmerman, Johnew Zhang, Jess Zhang, Dmytro
Korol, Ishan Patel, and Shashanth Shetty. Thank you all for the valuable
friendships over the years.

Finally, to my parents Xiaofen \textquotedbl Cindy\textquotedbl{}
Chen and Deying Kong, as well as my sister Willa Kong, I am incredibly
grateful for your endless love and patience, continual encouragement,
and unconditional trust. This journey would not have been possible
without you.

\addcontentsline{toc}{chapter}{Acknowledgments}

\newpage{}

\doublespacing{

\hypersetup{linkcolor=black}

\renewcommand{\contentsname}{\vspace*{-2em} \hspace*{\fill}\textbf{\textsc{\large Table of Contents}} \hspace*{\fill}}\tableofcontents{}\newpage{}

\renewcommand{\listtablename}{\vspace*{-2em} \hspace*{\fill}\textbf{\scshape{\large List of Tables}} \hspace*{\fill}}

\listoftables

\addcontentsline{toc}{chapter}{List of Tables}

\newpage{}

\renewcommand{\listfigurename}{\vspace*{-2em} \hspace*{\fill}\textbf{\textsc{\large List of Figures}} \hspace*{\fill}}

\listoffigures

\addcontentsline{toc}{chapter}{List of Figures}

\newpage{}

\renewcommand{\listalgorithmname}{\vspace*{-2em} \hspace*{\fill}\textbf{\textsc{\large List of Algorithms}} \hspace*{\fill}\vspace*{1em}}

\listofalgorithms

\addcontentsline{toc}{chapter}{List of Algorithms}\newpage{}

\phantomsection

\begin{center}
\MakeUppercase{\textbf{\textsc{\large{}List of Symbols and Abbreviations}}}\vspace*{2em}
\par\end{center}

\begin{flushleft}
\begin{tabular}{ccl}
NCO &  & nonconvex composite optimization\tabularnewline
CNCO &  & constrained nonconvex composite optimization\tabularnewline
MCO &  & min-max composite optimization\tabularnewline
SNCO &  & spectral nonconvex composite optimization\tabularnewline
${\cal C}(Z)$ &  & continuously differentiable functions on $Z$\tabularnewline
${\cal C}_{L}(Z)$ &  & functions in ${\cal C}(Z)$ that are $L$-smooth\tabularnewline
$\overline{{\rm Conv}}\ (Z)$ &  & proper, closed, convex functions with domain $Z$\tabularnewline
${\cal F}_{\mu}(Z)$ &  & functions in $\overline{{\rm Conv}}\ (Z)$ that are $\mu$-strongly
convex\tabularnewline
${\cal F}_{\mu,L}(Z)$ &  & functions in ${\cal F}_{\mu}(Z)$ that are $L$-smooth\tabularnewline
${\cal C}_{m,M}(Z)$ &  & functions in ${\cal C}(Z)$ that have curvature pair $(m,M)$\tabularnewline
CG &  & composite gradient\tabularnewline
ACG &  & accelerated composite gradient\tabularnewline
GIPP &  & general inexact proximal point\tabularnewline
GD &  & general descent\tabularnewline
CR / SR / PR &  & composite / specialized / proximal refinement\tabularnewline
AIPP &  & accelerated inexact proximal point\tabularnewline
AIP.QP &  & accelerated inexact proximal quadratic penalty\tabularnewline
AIP.AL &  & accelerated inexact proximal augmented Lagrangian\tabularnewline
AICG &  & accelerated inexact composite gradient\tabularnewline
D.AICG &  & doubly-accelerated composite gradient\tabularnewline
\end{tabular}
\par\end{flushleft}

\addcontentsline{toc}{chapter}{List of Symbols and Abbreviations}

}

\newpage{}

\doublespacing
\phantomsection\begin{center}
\MakeUppercase{\textbf{\textsc{\large{}Summary}}}\vspace*{2em}
\par\end{center}

This thesis focuses on developing and analyzing accelerated and inexact
first-order methods for solving or finding stationary points of various
nonconvex composite optimization (NCO) problems. Our main tools mainly
come from variational and convex analysis, and our key results are
in the form of iteration complexity bounds and how these bounds compare
to other ones in the literature.

Our first study problem is the classic unconstrained NCO problem studied
by Mine and Fukushima (1981), and we develop an accelerated inexact
proximal point method for finding approximate stationary points of
it. By analyzing the method's variational properties, we establish
an iteration complexity bound that is optimal in the number of first-order
oracle evaluations. As an additional result, we show that our accelerated
method and the classic composite/proximal gradient method are instances
of a general inexact proximal point framework under different stepsizes
and levels of inexactness.

Following our developments for the unconstrained setting, we move
to study two instances of a function-constrained NCO problem. The
first instance comprises a set of linear set constraints, and we develop
a quadratic penalty method for finding approximate stationary points
of it. We then establish an iteration complexity bound that is several
orders of magnitude better than the previous state-of-the-art bound.
As part of the analysis, we show that one can start the method from
any point where the objective function is finite (and not necessarily
from a near feasible point) and that no regularity conditions are
needed to obtain convergence. The second instance consists of a set
of nonlinear cone constraints, and we develop a proximal inexact augmented
Lagrangian method for finding approximate stationary points of it.
We then establish a competitive iteration complexity bound under an
easily verifiable Slater-like condition. As part of the analysis,
we show that the Lagrange multipliers generated by the method are
bounded, without needing to dampen the (dual) multiplier update, and,
like in the penalty method, the initial point can be any point where
the objective function is finite.

Before moving on to other problems, we discuss some efficient implementation
strategies of the above methods. In particular, we present some efficient
line search subroutines, an adaptive stepsize selection scheme, an
efficient warm-start strategy, and a discussion about how to relax
some algorithms' convexity assumptions. We also present a large number
of real-world applications and numerical experiments that highlight
our methods' performance against other modern solvers.

Our second-to-last study problem is a class of nonconvex-concave min-max
NCO problems, and we develop an accelerated smoothing method for finding
two kinds of approximate stationary points of it. Using prior results
from our study of the unconstrained NCO problem, we establish iteration
bounds that substantially improve on similar ones in the literature.
Additionally, we give a brief discussion about how to generalize our
smoothing method to solve linearly constrained min-max NCO problems.
We then end with some numerical experiments in the unconstrained setting
to validate the efficacy of our approach.

Our final study problems are a popular class of spectral NCO problems
in which the inputs are general $m$-by-$n$ real-valued matrices.
As part of the study, we develop two inexact composite gradient methods
--- one based on the classic composite/proximal gradient method and
another based on an accelerated variant of it --- to find approximate
stationary points. Extending some techniques for analyzing accelerated
methods, we show that the accelerated variant obtains a competitive
convergence rate in the nonconvex setting and an accelerated convergence
rate in the convex setting. A vital conclusion of the study is that
we show the methods perform nearly all of their iterations over the
vector space $\mathbb{R}^{\min\{m,n\}}$ rather than the matrix space
$\mathbb{R}^{m\times n}$. We then end with some numerical experiments
to show the effectiveness of the previous conclusion.

\addcontentsline{toc}{chapter}{Summary}\newpage{}

\doublespacing\pagenumbering{arabic}

\chapter{Introduction}

\label{chap:intro}
\begin{quote}
\begin{flushright}
\emph{If everything seems under control, you're just not going fast
enough.}\\
\emph{-Mario Andretti}
\par\end{flushright}

\end{quote}
Efficient optimization algorithms play a ubiquitous role in both the
theory and application of machine learning and scientific computing.
From web search engines to facial recognition software, their presence
is found in many indispensable systems of modern society.

In this thesis, we contribute to a class of popular continuous optimization
algorithms called \emph{first-order methods}, consisting of iterative
optimization algorithms that exploit information about the function
value and subgradient(s) of the objective function. Since Cauchy's
study on the gradient descent method \citep{Cauchy1847} in 1847,
these methods have found extensive use in smooth convex minimization
(fast gradient methods \citep{Nesterov1983,Nemirovski1982,Nemirovski1983}),
nonsmooth convex minimization (subgradient descent \citep{Shor1967,Polyak1967},
mirror descent \citep{Nemirovski1979,Nemirovski1983,Beck2003}, and
bundle methods \citep{Ben-Tal2005,Lemarechal1981,Kiwiel1995}), and
convex-concave saddle-point problems (smoothing methods \citep{Nesterov2005}
and mirror prox \citep{Nemirovski1983,Nemirovski1981,Nemirovski2005}).
Recently, first-order methods have gained a renewed interest due to
their ability to obtain cheap (nearly) dimension-free\footnote{In contrast, interior point methods are known to grow nonlinearly
with respect to the dimension, or equivalently, the number of decision
variables.} guarantees for large-scale problems in a broad spectrum of disparate
fields.

Our focus problems are variants of the following classic smooth nonconvex
(additive) composite optimization (NCO) problem, first studied in
\citep{Mine1981} by Mine and Fukushima:
\begin{equation}
\min_{x\in\rn}\left\{ \phi(x)=f(x)+h(x)\right\} \tag{\ensuremath{{\cal NCO}}},\label{prb:nco_intro}
\end{equation}
where $h:\rn\mapsto(-\infty,\infty]$ is a closed, proper, convex,
but not necessarily differentiable, function and $f:\rn\mapsto(-\infty,\infty]$
is a function that is continuously differentiable on an open set containing
the domain of $h$, but not necessarily convex. Problems such as \ref{prb:nco_intro}
frequently appear in areas such as recommender systems \citep{Jiang2020,Fang2017},
signal processing \citep{Candes2013,Combettes2011}, sparse regularization
\citep{Yao2017,Wen2020,Gu2014}, and compressed sensing \citep{Aybat2009,Aybat2012}. 

In the forty years following Mine and Fukushima's work, there has
been an immense amount of literature devoted to creating efficient
methods for finding approximate stationary points\footnote{In general, finding even \emph{approximate} minimizers of \ref{prb:nco_intro}
is intractable \citep{Nemirovski1983,Murty1987,Nesterov2013}. } of \ref{prb:nco_intro} and its variants. Recent developments, in
particular, have focused on generalizing Nesterov's seminal work on
\emph{accelerated} gradient methods for smooth convex optimization
\citep{Nesterov1983} to the nonconvex setting of \ref{prb:nco_intro}
under a structural weak convexity assumption \citep{Paquette2017,Ghadimi2016,Ghadimi2019,Carmon2018,Li2015},
i.e. where we assume that $f+m\|\cdot\|^{2}/2$ is convex for sufficiently
large enough $m>0$. 

Our goal in this thesis is to continue this work and present several
accelerated \emph{nonconvex} first-order methods that explicitly take
advantage of structural weak convexity in a meaningful way. The main
theme that pervades most of our studies is that of variational inclusions,
e.g. $0\in\pt^{*}\phi(x)=\nabla f(x)+\pt h(x)$ where $\pt^{*}\phi$
(resp. $\pt h$) is the Clarke\footnote{See \prettyref{def:Clarke_subdiff}.}
(resp. regular\footnote{See \prettyref{def:subdiff}.}) subdifferential
of $\phi$ (resp. $h$). By studying the inexact and exact variational
properties of several accelerated methods in the convex setting, we
construct accelerated methods with similar properties in the nonconvex
setting. The efficacy of this approach is validated through competitive
iteration complexity bounds, promising numerical experiments, and
its utility in established optimization frameworks, e.g. penalty and
augmented Lagrangian frameworks.

\section{Contributions of the Thesis}

This section carefully describes the organization and key contributions
of this thesis. It it divided into three subsections. The first one
is dedicated to optimization algorithms for smooth NCO problems, the
second one to efficient implementation strategies, and the last one
to optimization algorithms for NCO problems with additional structure.

Throughout this section, we let $\pt^{*}\phi(x)$ denote the Clarke
subdifferential of $\phi$ at $x$ and ${\rm dist}(x,C)$ denote the
distance between a point $x$ and a set $C$.

\subsection{Smooth NCO Problems}

In the next three chapters of this thesis, we propose a substantial
number of iterative first-order optimization methods for finding approximate
stationary points of \ref{prb:nco_intro} in the unconstrained and
function-constrained setting. Under the assumption that $f$ is weakly
convex and its gradient $\nabla f$ is Lipschitz continuous, each
method comes with an iteration complexity bound and a comparison with
similar methods in the literature. Below, we briefly summarize the
contributions of these methods.\\

\noindent \textbf{Complexity Optimal Proximal Point Method for Unconstrained
\ref{prb:nco_intro} Problems}. In \prettyref{chap:unconstr_nco},
we develop a general inexact proximal point framework for finding
approximate stationary points of \ref{prb:nco_intro}. More specifically,
this framework is designed to find a $\rho$-approximate stationary
point $\bar{x}\in\rn$ satisfying
\begin{equation}
{\rm dist}(0,\pt^{*}\phi(\bar{x}))\leq\rho.\label{eq:approx_nco_statn}
\end{equation}
Using a special inexactness criterion and several variational properties
of an accelerated gradient method, we present a specific instance
of the framework that is (iteration) complexity optimal in terms of
the smoothness parameters of $f$ and the tolerance $\rho$. It is
worth mentioning that this instance does not require the domain of
$h$ to be bounded and only requires $\phi_{*}$ in \ref{prb:nco_intro}
to be finite. Furthermore, the inexactness criterion does not depend
on the tolerance $\rho$ but rather on a special proximal residual.\\

\noindent \textbf{Quadratic Penalty} \textbf{Method} \textbf{for Linearly-Constrained
\ref{prb:nco_intro} Problems}. In the first section of \prettyref{chap:cnco},
we develop a quadratic penalty method for finding approximate stationary
points of linearly set-constrained\footnote{\noindent The constraint is of the form $Az\in S$ for some linear
operator $A$ and closed convex set $S$.} instances of \ref{prb:nco_intro}. More specifically, this method
is designed to find a $(\rho,\eta)$-approximate stationary point
$(\bar{x},\bar{p})$ satisfying
\begin{equation}
{\rm dist}(0,\pt^{*}\phi(\bar{x})+A^{*}\bar{p})\leq\rho,\quad{\rm dist}(A\bar{x},S)\leq\eta.\label{eq:approx_lc_nco_statn}
\end{equation}

\noindent Using our developments in \prettyref{chap:unconstr_nco}
and some additional properties about penalty functions, we show that
the method obtains an ${\cal O}(\rho^{-2}\eta^{-1})$ iteration complexity
bound, which substantially improves upon the previously known bound
of ${\cal O}(\rho^{-6})$ that was obtained by a multiblock ADMM-type
method \citep{Jiang2019} for the case of $\rho=\eta$. The main novelty
of the proposed method is that the initial starting point $z_{0}$
only needs to be in the domain of $h$, i.e. $h(z_{0})<\infty$, and
not necessarily feasible with respect to the linear set constraint,
i.e. $Az_{0}\in S$. It is also worth mentioning that the method does
not require any regularity condition on its linear constraints and
that the inexactness criterion does not depend on the tolerance pair
$(\rho,\eta)$.

\noindent 

\noindent \textbf{Proximal Augmented Lagrangian Method} \textbf{for
Nonlinearly-Constrained \ref{prb:nco_intro} Problems}. In the second
section of \prettyref{chap:cnco}, we develop an inexact proximal
augmented Lagrangian method for finding approximate stationary points
of nonlinearly cone-constrained instances of \ref{prb:nco_intro}
in which: (i) the function $h$ is Lipschitz continuous and its domain
is bounded; and (ii) the function $g$ forming the cone constraint
$g(x)\preceq_{{\cal K}}0$ is ${\cal K}$-convex. More specifically,
this method is designed to find a $(\rho,\eta)$-approximate stationary
point $(\bar{x},\bar{p})$ satisfying
\[
{\rm dist}(0,\pt^{*}\phi(\bar{x})+\nabla g(x)\bar{p})\leq\rho,\quad\dist(g(\bar{x}),{\cal F}(\bar{p}))\leq\eta,\quad\bar{p}\succeq_{{\cal K}^{+}}0,
\]
where ${\cal K}^{+}$ is the dual cone of ${\cal K}$ and the set
${\cal F}(\bar{p})$ is given by
\[
{\cal F}(\bar{p}):=\left\{ g(x):\left\langle g(x),\bar{p}\right\rangle \leq0,g(x)\preceq_{{\cal K}}0,h(x)<\infty\right\} .
\]

\noindent Using a special inexactness criterion and several recent
developments from convex analysis, we show that the method obtains
an ${\cal O}([\eta^{-1/2}\rho^{-2}+\rho^{-3}]\log[\rho^{-1}+\eta^{-1}]))$
iteration complexity bound under a weak Slater-like condition. The
contribution of the method is twofold. First, the method proposes
a novel way of generating the penalty parameters $c_{k}$ based on
the change in the augmented Lagrangian between consecutive iterations
rather than based on the feasibility of a particular iterate\footnote{\noindent Other methods in the literature \citep{Birgin2020,Grapiglia2019,Xie2019}
usually consider increasing $c_{k}$ whenever $\|\max\{0,g(x_{k})\}\|$
has not sufficiently decreased between iterations}. Second, it is shown that the multipliers $\{p_{k}\}_{k\geq1}$ generated
by the classic (dual) multiplier update are bounded without requiring
any normalization\footnote{\noindent Other methods in the literature \citep{Birgin2020,Xie2019}
usually add a step that projects the multipliers $\{p_{k}\}_{k\geq1}$
into a bounded Euclidean box after the classic multiplier update is
computed .}.

\subsection{Efficient Implementation Strategies}

Following the above developments, we dedicate \prettyref{chap:practical}
to efficient implementation strategies. Additionally, we present iteration
complexity bounds for variants of the methods in \prettyref{chap:unconstr_nco}
and \prettyref{sec:qp_aipp} that use some of these strategies and
give several numerical experiments. Below, we highlight some of the
most effective strategies.\\

\noindent \textbf{Adaptive Stepsize Selection}. We propose several
different approaches of choosing several key ``stepsize'' parameters
based on a finite set of key inequalities. These approaches are designed
to adapt to the local geometry of the objective function and improve
the convergence rate of the convex \emph{and} nonconvex methods that
use them.\\

\noindent \textbf{Relaxation of Convexity}. Several of the methods
for the smooth NCO problems rely on the ``stepsize'' parameters
to be within a particular range of values in order to ensure some,
not necessarily verifiable, convexity conditions hold. We propose
a way to relax some of these conditions to a verifiable set of finite
inequalities to allow the ``stepsize'' parameters to be arbitrarily
large or small.\\

\noindent \textbf{Warm-Start Strategy}. For methods that operate by
finding approximate stationary points of a sequence of optimization
subproblems, we propose a warm-start strategy for initializing the
starting point of each subproblem. More specifically, we propose a
strategy where the current subproblem uses a point obtained from the
last iterate of the previous subproblem. We then show that a (convexity)
relaxed quadratic penalty method obtains an ${\cal O}(\eta^{-2})$
factor improvement in its iteration complexity bound (for finding
$(\rho,\eta)$-stationary points as in \eqref{prb:lc_approx_sp_mco})
when a warm-start strategy is used in place of a cold-start strategy.

\subsection{NCO Problems with Additional Structure}

Following the developments in prior chapters, the last two chapters
of this thesis consider variants of \ref{prb:nco_intro} with additional
structure and give several numerical experiments. Below, we summarize
the contributions of these methods.\\

\noindent \textbf{Smoothing Method}s. In \prettyref{chap:min_max},
we first develop a smoothing method for finding approximate stationary
points of nonconvex-concave min-max instances of \ref{prb:nco_intro}.
More specifically, when $f$ is a max function of the form $f(x)=\max_{y}\Phi(x,y)$,
the method is designed to obtain stationary points of two kinds: (i)
a $\delta$-approximate directional stationary point $x$ satisfying
\[
\exists\bar{x}\text{ s.t. }\inf_{\|d\|\leq1}\phi'(\bar{x};d)\leq\delta,\quad\|x-\bar{x}\|\leq\delta,
\]
where $\phi'(x;d)$ is the directional derivative of $\phi$ at $x$
for the direction $d$, and (ii) a $(\rho_{x},\rho_{y})$-approximate
primal-dual stationary point $(\bar{x},\bar{y})$ satisfying 
\[
{\rm dist}(0,\pt^{*}\psi_{\bar{y}}(\bar{x}))\leq\rho_{x},\quad{\rm dist}(0,\pt^{*}\psi_{\bar{x}}(\bar{y}))\leq\rho_{y}
\]
where $\psi_{\bar{x}}(\cdot):=-\Phi(\bar{x},\cdot)$ and $\psi_{\bar{y}}(\cdot):=\Phi(\cdot,\bar{y})+h(\cdot)$.
Using several results from convex analysis and the efficient method
in \prettyref{chap:unconstr_nco}, we show that the smoothing method
obtains ${\cal O}(\delta^{-3})$ and ${\cal O}(\rho_{x}^{-2}\rho_{y}^{-1/2})$
iteration complexity bounds for obtaining $\delta$-approximate directional
stationary points and $(\rho_{x},\rho_{y})$-approximate primal-dual
stationary points, respectively. Following these developments, we
propose a quadratic penalty smoothing method for solving linearly-constrained
instances of the min-max problem and establish an iteration complexity
bound for finding an approximate primal-dual stationary point of the
constrained problem. The main contributions are significantly improved
complexity bounds (see \prettyref{tab:comp1,tab:comp2}) and a
new complexity bound for the constrained case. It is worth mentioning
that the methods do not assume that the domain of $h$ is bounded.\\

\noindent \textbf{Spectral Optimization Methods}. In \prettyref{chap:spectral},
we develop two inexact spectral composite optimization methods, one
accelerated and one unaccelerated, for finding $\rho$-approximate
stationary points of \ref{prb:nco_intro} as in \eqref{eq:approx_nco_statn}
in which $\phi$ admits an additional spectral decomposition. More
specifically, for a given input point $X\in\r^{m\times n}$, we consider
the instances where the composite term $h$ is a function of the singular
values of $X$ and the smooth term $f$ can be decomposed as $f=f_{1}+f_{2}$
where $f_{2}$ is also a function of the singular values of $X$.
Using a special inexactness criterion and several variational properties
of an accelerated gradient method, we show that both methods obtain
an ${\cal O}(\rho^{-2})$ iteration complexity bound and that the
accelerated method obtains an ${\cal O}(\rho^{-2/3})$ complexity
bound when $\phi$ is convex. A key contribution is that the methods
mainly iterate over a space of singular values rather than the larger
space of input matrices.

\newpage{}

\chapter{Background}

\label{chap:background}

This chapter presents the basic concepts, well-known results, and
notational conventions that are used throughout the thesis. \textbf{Aside
from the notation in \prettyref{subsec:fn_classes}}, the materials
in this chapter are well-established, and hence, may be skipped upon
first reading.

\subsection*{Organization}

This chapter contains two sections. The first one presents theoretical
background material while the second one presents algorithmic background
material. 

\section{Theoretical Background}

This section presents material that is relevant to the theoretical
developments of the thesis.

\subsection{Basics}

This subsection states basic definitions, conventions, and notation.\vspace{0.5em}

\textbf{Sets}. We denote $\r$, $\mathbb{Z}$, $\mathbb{N}$, and
$\mathbb{C}$ to be the set of real numbers, integers, natural numbers,
and complex numbers, respectively. The sets $\r_{+}$ and $\r_{++}$
denote the nonnegative a positive numbers, respectively. For sets
$A,B$, we denote their Cartesian product as $A\times B=\{(a,b):a\in A,b\in B\}$
and their Minkowski sum as $A+B=\{a+b:a\in A,b\in B\}$. For ease
of notation, we denote $\{a\}+B\equiv a+B$ and $\lam A=\{\lam a:a\in A\}$
for any $a\in A$ and $\lam\in\mathbb{C}$. For $n\in\n$, we define
$A^{n}=\overbrace{A\times...\times A}^{n\text{ times}}$. The empty
set is denoted by $\emptyset$. For $a,b\in\rn$ we denote the line
interval between $a$ and $b$ as $[a,b]=\{ta+(1-t)b:0\leq t\leq1\}$.
We also denote $[a,b)=[a,b]\backslash\{b\}$, $(a,b]=[a,b]\backslash\{a\}$,
and $(a,b)=[a,b]\backslash\{a,b\}$. The set $\{x_{i}\}_{i=1}^{k}$
consists of the elements $x_{1},...,x_{k}$. The set $\{x_{i}\}_{i\geq1}$
consists of the elements $x_{i}$ for every $i\in\n$.

\vspace{0.5em}

\textbf{Functions}. Let $X$, $Y$, and $Z$ be arbitrary sets. We
denote $f:X\mapsto Y$ and $F:X\rightrightarrows Y$ to be single-valued
and set-valued functions from $X$ to $Y$, respectively. For any
set $S$, we denote $f(S)=\{f(s):s\in S\}$. For functions $f:X\mapsto Y$
and $g:Y\mapsto Z$, we denote $g\circ f(x)=g(f(x))$ for every $x\in X$.
\vspace{0.5em}

\textbf{Basic Operators}. Let $x\in\r$, $f:X\mapsto\r$ be an arbitrary
function, and $S$ be an arbitrary set. We denote $\lceil x\rceil$
(resp. $\lfloor x\rfloor$) to be the smallest (resp. largest) element
in $\mathbb{Z}$ that is greater (resp. less) than or equal to $x$.
We denote $\sup_{x\in S}f(x)$ (resp. $\inf_{x\in S}f(x)$) as the
smallest (resp. largest) element $B$ in $\r$ that satisfies $f(s)\leq B$
(resp. $f(s)\geq B$) for every $s\in S$. The function ${\rm sgn}(x)$
takes value +1 if $x\geq0$ and -1 otherwise. As a convention, we
take $a/0=+\infty$ and $-a/0=-\infty$ for every $a>0$.

\vspace{0.5em}

\textbf{Computational Complexity}. For functions $f,g:\r_{++}\mapsto\n$,
we use the following asymptotic notation:
\begin{itemize}
\item $f(x)={\cal O}(g(x))$ if there exists $(C,\underline{x})\in\r_{++}^{2}$
such that for every $x\geq\underline{x}$ it holds that $f(x)\leq Cg(x)$.
\item $f(x)=\Omega(g(x))$ if there exists $(C,\underline{x})\in\r_{++}^{2}$
such that for every $x\geq\underline{x}$ it holds that $f(x)\geq Cg(x)$.
\item $f(x)=\Theta(g(x))$ if $f(x)={\cal O}(g(x))$ and $f(x)=\Omega(g(x))$.
\item $f(x)=o(g(x))$ if for every $C>0$ there exists $\underline{x}>0$
such for every $x\geq\underline{x}$ it holds that $f(x)\leq Cg(x)$.
\end{itemize}

\subsection{Analysis}

This subsection reviews relevant materials from analysis. 

We first start with some basic definitions and notation. 
\begin{defn}
For a vector space ${\cal X}$, an \textbf{inner product} $\left\langle \cdot,\cdot\right\rangle :{\cal X}\times{\cal X}\mapsto\r$
is a mapping that satisfies, for every $x,y,z\in{\cal X}$ and $\alpha,\beta\in\r$,
the relations: 
\begin{itemize}
\item[(i)] $\left\langle x,y\right\rangle =\left\langle y,x\right\rangle $
(symmetry);
\item[(ii)] $\left\langle \alpha x+\beta y,z\right\rangle =\alpha\left\langle x,z\right\rangle +\beta\left\langle y,z\right\rangle $
(linearity);
\item[(iii)] $\left\langle x,x\right\rangle >0$ if $x\neq0$ (non-degeneracy). 
\end{itemize}
A vector space equipped with an inner product is said to be a \textbf{inner
product space}.
\end{defn}

\begin{defn}
The \textbf{induced norm} of an inner product space ${\cal X}$, denoted
by $\|\cdot\|$, is given by $\|x\|^{2}=\left\langle x,x\right\rangle $
for every $x\in{\cal X}$. It is well-known that every inner product
satisfies the \textbf{Cauchy-Schwarz} inequality $\left\langle x,y\right\rangle \leq\|x\|\cdot\|y\|$
and the triangle inequality $\|x+y\|\leq\|x\|+\|y\|$ for every $x,y\in{\cal X}$.
\end{defn}

For the rest of this subsection, we let ${\cal X}$, ${\cal Y}$,
and ${\cal Z}$ be inner product spaces with a common inner product
$\left\langle \cdot,\cdot\right\rangle $. Moreover, we denote $\|\cdot\|$
to be their induced norm. 
\begin{defn}
For a point $z\in{\cal Z}$ and parameter $r>0$, the \textbf{open}
\textbf{ball ${\cal B}_{r}(z)$} and \textbf{closed ball }$\overline{{\cal B}}_{r}(z)$
of radius $r$ at $z$ is defined by
\begin{align*}
{\cal B}_{r}(z) & :=\left\{ z'\in{\cal Z}:\|z'-z\|<r\right\} ,\\
\overline{{\cal B}}_{r}(z) & :=\left\{ z'\in{\cal Z}:\|z'-z\|\leq r\right\} .
\end{align*}
A set $Z\subseteq{\cal Z}$ is said to be \textbf{open }if for every
$z\in Z$ there exists $\varepsilon>0$ such that ${\cal B}_{\varepsilon}(z)\subseteq Z$.
A set $\tilde{Z}\subseteq{\cal Z}$ is said to be \textbf{closed}
if the set ${\cal Z}\backslash\tilde{Z}$ is open. Finally, a set
$Z\subseteq{\cal Z}$ is said to be bounded if there exists $r\in\r_{++}$
such that $Z\subseteq{\cal B}_{r}(0)$.
\end{defn}

\begin{defn}
A set $C\subseteq{\cal Z}$ is said to be \textbf{compact} if for
any collection of open sets ${\cal D}=\{D_{i}\}_{i\in{\cal I}}$,
for some index set ${\cal I}$, satisfying $C\subseteq\bigcup_{i\in{\cal I}}D_{i}$
there exists a finite subcollection $\tilde{{\cal D}}=\{\tilde{D}_{i}\}_{i=1}^{k}\subseteq{\cal D}$
such that $C\subseteq\bigcup_{i=1}^{k}D_{i}$. If ${\cal Z}=\rn$,
it is well-known that a set $C\subseteq\rn$ is compact if and only
if it is closed and bounded.
\end{defn}

\begin{defn}
For a sequence $\{z_{n}\}_{n\geq1}\subseteq{\cal Z}$, we say that
$z_{n}$ converges to $z$, or equivalently $\lim_{i\to\infty}z_{n}=z\in{\cal Z}$,
if for every $\varepsilon>0$, there exists $\underline{n}\in\n$
such that for every $k\geq\underline{n}$ we have $\|z-z_{k}\|\leq\varepsilon$. 
\end{defn}

The next result is a well-known result about bounded sequences.
\begin{thm}
(Bolzano-Weierstrass) Every bounded sequence in a finite dimensional
inner product space has a convergent subsequence.
\end{thm}

We now present definitions and results about some special classes
functions.
\begin{defn}
A function $\phi:{\cal X}\mapsto{\cal Y}$ is said to be \textbf{continuous}
on a set $X\subseteq{\cal X}$ if for every $x\in X$ and $\varepsilon>0$
there exists $\delta>0$ such that for every $x'\in X$ satisfying
$\|x-x'\|\leq\delta$ we have that $\|\phi(x)-\phi(x')\|\leq\varepsilon$.
It is well-known that if $\{x_{i}\}_{i\geq1}\subseteq X$ is such
that $\lim_{i\to\infty}x_{i}=x\in X$ and $\phi$ is continuous on
$X$, then $\lim_{i\to\infty}\phi(x_{i})=\phi(\lim_{i\to\infty}x_{i})=\phi(x)$.
\end{defn}

\begin{defn}
A function $\phi:{\cal X}\mapsto{\cal Y}$ is said to be \textbf{$L$-Lipschitz
continuous} on a set $X\subseteq{\cal X}$ if 
\[
\|\phi(x)-\phi(x')\|\leq L\|x-x'\|\quad\forall x,x'\in X.
\]
\end{defn}

\begin{defn}
For a closed convex set $Z\subseteq{\cal Z}$, the (single-valued)
\textbf{projection mapping $\Pi_{Z}$ }at a point $z$ is defined
by 
\[
\Pi_{Z}(z)=\argmin_{u\in Z}\frac{1}{2}\|u-z\|^{2}.
\]
The distance function $\dist(\cdot,Z)$ at a point $z$ is defined
by
\[
\dist(z,Z)=\|z-\Pi_{Z}(z)\|.
\]
\end{defn}

\begin{defn}
Let $f:{\cal X}\mapsto{\cal Y}$ be a function that is well-defined
in an open ball around a point $x\in{\cal X}$. The function $f$
is said to be \textbf{(Fréchet) differentiable }at\textbf{ $x$} if
there exists a linear function $Df_{x}:{\cal X}\mapsto{\cal Y}$,
called the \textbf{derivative }of $f$ at\textbf{ $x$}, that approximates
the change $f(x+\Delta x)-f(x)$ up to a residual, called the \textbf{first-order
Taylor residual}, that is $o(\Delta x)$. More specifically, the function
$f$ is differentiable at $x$ if and only if 
\[
\|f(x+\Delta x)-f(x)-Df_{x}(\Delta x)\|=o(\Delta x)
\]
for every $\Delta x$ such that $f(x+\Delta x)$ is well-defined.
\end{defn}

\begin{defn}
A differentiable function $f:{\cal X}\mapsto{\cal Y}$ is said to
be \textbf{continuously differentiable} at $x$ if the function $x\mapsto Df_{x}(\Delta x)$
is continuous for every $\Delta x\in{\cal X}$.
\end{defn}

\begin{defn}
Let $f:{\cal X}\mapsto{\cal Y}$ be differentiable at a point $x\in{\cal Z}$.
The \textbf{gradient} of $f$ at $x$ is the unique matrix $\nabla f(x)$
that satisfies
\[
\nabla f(x)^{T}u=Df_{x}(u)
\]
for every $u\in{\cal X}$ in a neighborhood of $x$. The \textbf{derivative
matrix} of $f$ at $x$ is the transpose of $\nabla f(x)$ and is
denoted by $f'(x)=\nabla f(x)^{T}$.
\end{defn}

\begin{defn}
The \textbf{linear approximation }of a differentiable function $f:{\cal X}\mapsto{\cal Y}$
at a point $x_{0}\in{\cal X}$ is defined as 
\[
\ell_{f}(x;x_{0}):=f(x_{0})+\left\langle \nabla f(x_{0}),x-x_{0}\right\rangle \quad\forall x\in{\cal X}.
\]
\end{defn}

The next three results present some fundamental properties involving
derivatives and gradients and can be found, for example, in \citep{Coleman2012,Williamson2004}.
\begin{thm}
(Chain rule) Let $f:{\cal X}\mapsto{\cal Y}$ be differentiable at
$x\in{\cal X}$ and let $g:{\cal Y}\mapsto{\cal Z}$ be differentiable
at $y=f(x)\in{\cal Y}$. Then, $g\circ f$ is differentiable at $x$
and 
\[
D(g\circ f)_{x}=Dg_{y}\circ Df_{x}.
\]
\end{thm}

\begin{thm}
(Mean Value Theorem) For any differentiable function $f:{\cal X}\mapsto{\cal Y}$
and $x_{0},x_{1}\in{\cal X}$, there exists $t\in[0,1]$ such that
\[
f(x_{1})=f(x_{0})+\nabla f(x_{t})^{T}(x_{1}-x_{0}),
\]
where $x_{t}=tx_{0}+(1-t)x_{1}$.
\end{thm}

\begin{thm}
(Gradient Theorem) Let $x_{0},x_{1}\in{\cal X}$ and $r:[0,1]\mapsto{\cal X}$
be such that $r(0)=x_{0}$ and $r(1)=x_{1}$. For any continuously
differentiable function $\phi:{\cal X}\mapsto\r$, we have 
\[
\phi(x_{1})-\phi(x_{0})=\int_{0}^{1}\nabla\phi(r(t))\cdot r'(t)\ dt.
\]
\end{thm}

The below material deals with the convolution of two functions. 
\begin{defn}
The \textbf{convolution} of functions $f,g:{\cal X}\mapsto\r$ is
\[
(f*g)(x):=\int_{-\infty}^{\infty}f(u)g(x-u)\ du\quad\forall x\in{\cal X}
\]
\end{defn}

The following result can be found, for example, in \citep[Chapter 6]{Bracewell2000}.
\begin{prop}
Let $f,g:{\cal X}\mapsto{\cal Y}$ be continuously differentiable
functions. Then, it holds that
\[
D(f*g)_{x}=Df_{x}*g=f*Dg_{x}\quad\forall x\in{\cal X}.
\]
\end{prop}

\subsection{Linear Algebra}

This subsection reviews notation and relevant materials from linear
algebra. 

We first start with some basic notation and definitions.

For every $(n,m)\in\n^{2}$, we denote $\mathbb{F}^{n\times m}$ to
be the set of matrices with $n$ rows and $m$ columns with entries
from $\mathbb{F}\in\{\r,\mathbb{Z},\mathbb{N},\mathbb{C}\}$. The
entry in the $i^{{\rm th}}$ row and $j^{{\rm th}}$ column of $A$
is denoted by $[A]_{ij}$ or $A_{ij}$. 
\begin{defn}
For matrices $A\in\r^{n\times p}$ and $B\in\r^{p\times m}$, the
\textbf{matrix product} $AB\in\r^{n\times m}$ is given by the relation
$[AB]_{ij}=\sum_{k=1}^{p}[A]_{ik}[B]_{kj}$.
\end{defn}

\begin{defn}
The \textbf{conjugate transpose (or adjoint)} of a matrix $A\in\mathbb{C}^{m\times n}$,
denoted by $A^{*}$, is given by the relation $A_{ij}^{*}=\overline{A_{ij}}$.
The \textbf{transpose} of a matrix, denoted by $A^{T}$, is given
by the relation $A_{ij}^{T}=A_{ij}$. It is well-known that 
\[
\left\langle Ax,y\right\rangle =\left\langle x,A^{*}y\right\rangle \quad\forall(x,y)\in\c^{m}\times\c^{n}.
\]
\end{defn}

If $a_{i}\in\r^{n}$ for $i\in\{1,...,k\}$, then we denote $(a_{1},...,a_{k})$
to be the matrix whose $i^{{\rm th}}$ column is $a_{i}$. If $A$
is a linear operator, then we denote $Az\equiv A(z)$.
\begin{defn}
A matrix $A\in\r^{m\times n}$ is \textbf{symmetric} if $A^{*}=A$.
\end{defn}

\begin{defn}
A matrix $A\in\r^{n\times n}$ is \textbf{positive (semi-)definite},
or equivalently $A>(\geq)0$, if $A$ is symmetric and for every $x\in\rn\backslash\{0\}$
we have $x^{T}Ax>(\geq)0$. The \textbf{set of positive (semi-)definite
matrices} in $\r^{n\times n}$ is denoted by $\mathbb{S}_{++}^{n}$
($\mathbb{S}_{+}^{n}$).
\end{defn}

\begin{defn}
The \textbf{trace} of a matrix $A\in\r^{n\times n}$ is given by $\trc(A)=\sum_{i=1}^{n}A_{ij}$.
It is well-known that $\trc(AB)=\trc(BA)$ for any matrices $A,B\in\r^{n\times n}$.
\end{defn}

\begin{defn}
The \textbf{identity matrix} of size $n$, denoted by $I_{n}$, is
given by $(I_{n})_{ij}=1$ if $i=j$ and $0$ if $i\neq j$.
\end{defn}

\begin{defn}
A matrix $A\in\r^{n\times n}$ is said to be \textbf{invertible} (or
\textbf{non-singular}) if there exists a matrix $A^{-1}$, called
the \textbf{inverse} of $A$, that satisfies $A^{-1}A=AA^{-1}=I_{n}$. 
\end{defn}

\begin{defn}
A matrix $A\in\r^{n\times n}$ is said to be \textbf{orthogonal} if
$A^{T}=A^{-1}$.
\end{defn}

\begin{defn}
The \textbf{determinant} of a matrix $A\in\r^{n\times n}$, denoted
by $\det(A)$, is $[A]_{11}$ if $n=1$ and is given recursively by
\[
\det(A)=\sum_{j=1}^{n}(-1)^{i+j}A{}_{ij}\det(M_{ij})=\sum_{i=1}^{n}(-1)^{i+j}A{}_{ij}\det(M_{ij})
\]
for $n\geq2$, where $M_{ij}\in\r^{(n-1)\times(n-1)}$ is the \textbf{minor}
that results from removing the $i^{{\rm th}}$ row and $j^{{\rm th}}$
column from $A$. It is well-known that $\det(AB)=\det(A)\det(B)$
and $\det(A)=\det(A^{T})$ for any matrices $A,B\in\r^{n\times n}$.
\end{defn}

\begin{defn}
The \textbf{eigenvalues} of a matrix $A\in\r^{n\times n}$ are the
roots of the characteristic polynomial $\det(A-\lam I_{n})$ as a
univariate function in $\lam$. An \textbf{eigenvector} $v\in\r^{n\times n}$
corresponding to some eigenvalue $\lam$ is any vector satisfying
$Av=\lam v$. We denote $\lam_{k}(A)$ to be the \textbf{$k^{{\rm th}}$
largest eigenvalue} of $A\in\r^{n\times n}$. Moreover, we use the
shorthand $\lam_{\min}(A)=\lam_{n}(A)$ and $\lam_{\max}(A)=\lam_{1}(A)$.
\end{defn}

\begin{defn}
The \textbf{singular value decomposition} (SVD) of a matrix $A\in\r^{m\times n}$
is a factorization of the form $A=P\Sigma Q^{*}$ where $P\in\r^{m\times m}$
and $Q\in\r^{n\times n}$ are orthogonal matrices and $\Sigma\in\r^{m\times n}$
is a rectangular diagonal matrix with nonnegative entries on the diagonal.
The diagonal entries $\{\Sigma_{ii}\}_{i\geq1}$ are known as the
\textbf{singular values }of $A$. 
\end{defn}

The following is a well-known (see, for example, \citep[Corollary 4.3.15]{Horn2012})
result about eigenvalues of matrix sums.
\begin{thm}
(Weyl's Inequality) Let $A,B\in\r^{n\times n}$ be symmetric matrices
and let $\lam_{k}(M)$ denote the $k^{{\rm th}}$ largest eigenvalue
of a matrix $M$. Then, it holds that
\[
\lam_{k}(A)+\lam_{n}(B)\leq\lam_{k}(A+B)\leq\lam_{k}(A)+\lam_{1}(B)
\]
for every $1\leq i\leq n$.
\end{thm}

\subsection{Convex and Variational Analysis}

This subsection presents relevant material from convex and variational
analysis. 

We first state some key definitions.
\begin{defn}
The interior of a set $Z\subseteq{\cal Z}$ is defined as
\[
\intr Z:=\left\{ z\in{\cal Z}:\exists\delta>0\text{ such that }{\cal B}_{\delta}(z)\subseteq Z\right\} .
\]
\end{defn}

\begin{defn}
For a convex set $Z\subseteq{\cal Z}$, the \textbf{affine hull} $\aff Z$
and \textbf{relative interior} $\ri Z$ of $Z$ are defined by 
\begin{align*}
\aff Z & :=\left\{ \gamma\in{\cal Z}:\gamma=\sum_{i=1}^{k}\alpha_{i}z_{i},\ z_{i}\in Z,\ \sum_{i=1}^{k}\alpha_{i}=1\text{ for }i\leq k,\ k=1,2,...\right\} ,\\
\ri Z & :=\left\{ \gamma\in\aff Z:\exists\delta>0\text{ such that }\aff Z\cap{\cal B}_{\delta}(\gamma)\subseteq Z\right\} .
\end{align*}
Another interpretation of $\aff Z$ is that it is the smallest affine
manifold containing $Z$. Under this interpretation, a point $z$
is in $\ri Z$ if it is in the interior of $Z$ relative to the topology
given by $\aff Z$.
\end{defn}

\begin{defn}
The (effective)\textbf{ domain} of a function $f:{\cal Z}\mapsto(-\infty,\infty]$
is the set 
\[
\dom f:=\left\{ z\in{\cal Z}:f(z)\in\r\right\} 
\]
and $f$ is said to be \textbf{proper }if $\dom f\neq\emptyset$.
\end{defn}

\begin{defn}
A proper function $f:{\cal Z}\mapsto(-\infty,\infty]$ is said to
be \textbf{convex} if 
\[
f(\alpha z+[1-\alpha]z')\leq\alpha f(z)+(1-\alpha)f(z')\quad\forall z,z'\in{\cal Z},\quad\forall\alpha\in(0,1).
\]
It is well-known that if $f$ is convex and differentiable, then $f(\cdot)-\ell_{f}(\cdot;z_{0})\geq0$
for any $z_{0}\in\dom f$.
\end{defn}

\begin{defn}
A proper function $f:{\cal Z}\mapsto(-\infty,\infty]$ is said to
be \textbf{$\mu$-strongly convex} if the function $f-\mu\|\cdot\|^{2}$
is convex and \textbf{$m$-weakly convex }if the function $f+m\|\cdot\|^{2}$
is convex. 

It is well-known that if $f$ is $\mu$-strongly convex and differentiable,
then $f(\cdot)-\ell_{f}(\cdot;z_{0})\geq\mu\|\cdot-z_{0}\|^{2}/2$
for every $z_{0}\in\dom f$. It is also well-known that if $f$ is
$m$-weakly convex and differentiable, then $f(\cdot)-\ell_{f}(\cdot;z_{0})\geq-m\|\cdot-z_{0}\|^{2}/2$
for every $z_{0}\in\dom f$.
\end{defn}

\begin{defn}
A proper convex function $f:{\cal Z}\mapsto[-\infty,\infty)$ is said
to be \textbf{closed }or\textbf{ lower semicontinuous} if 
\[
\liminf_{z\to z_{0}}f(z)\geq f(z_{0})\quad\forall z_{0}\in{\cal Z}.
\]
\end{defn}

\begin{defn}
\label{def:subdiff}For a proper convex function $f:{\cal Z}\mapsto[-\infty,\infty)$
and a point $z\in\dom f$, the\textbf{ $\varepsilon$-subdifferential}
of $f$ at $z$ is defined by 
\[
\pt_{\varepsilon}f(z)=\left\{ v\in{\cal Z}:f(z')\geq f(z)+\left\langle v,z'-z\right\rangle \ \forall z'\in{\cal Z}\right\} ,
\]
and the (regular) \textbf{subdifferential} of $f$ at $z$ is $\pt_{0}f(z)$
and is commonly denoted by $\pt f(z)$. It is well-known that $z_{*}\in\argmin_{z'\in{\cal Z}}f(z')$
if and only if $0\in\pt f(z_{*})$.
\end{defn}

\begin{defn}
\label{def:Clarke_subdiff}For a proper function $f:{\cal Z}\mapsto[-\infty,\infty)$,
the Clarke subdifferential of $f$ at a point $z\in\dom f$ is the
set
\[
\pt^{*}\phi(x):=\left\{ v:\left\langle v,\cdot\right\rangle \leq d\phi(x;\cdot)\right\} 
\]
where $d\phi(x;u):=\limsup_{t\downarrow0,y\to x}[\phi(y+tu)-\phi(y)]/t$.
\end{defn}

\begin{defn}
For a closed convex set $Z\subseteq{\cal Z}$ and a point $z\in{\cal Z}$,
the \textbf{indicator function} $\delta_{Z}$ and the normal cone
$N_{Z}$ at a point $z\in{\cal Z}$ are given by 
\begin{align*}
\delta_{Z}(z) & :=\begin{cases}
0, & z\in Z,\\
\infty, & {\rm otherwise},
\end{cases}\\
N_{Z}(z) & :=\left\{ v\in{\cal Z}:\left\langle v,z'-z\right\rangle \leq0\ \forall z'\in Z\right\} .
\end{align*}
\end{defn}

\begin{defn}
For a proper, lower semicontinuous function $f:{\cal Z}\mapsto[-\infty,\infty)$,
a parameter $\lam>0$, and a point $z\in{\cal Z}$, the \textbf{Moreau
envelope} $e_{\lam}f$ and the \textbf{proximal mapping} $\prox_{\lam}f$
of $f$ at $z$ are defined by
\begin{align*}
e_{\lam}f(z) & :=\inf_{z'\in{\cal Z}}\left\{ f(z)+\frac{1}{2\lam}\|z'-z\|^{2}\right\} \leq f(z)\\
\prox_{\lam}f(z) & :=\Argmin_{z'\in{\cal Z}}\left\{ f(z)+\frac{1}{2\lam}\|z'-z\|^{2}\right\} .
\end{align*}
The function $f$ is said to be \textbf{prox-bounded} if there exists
a threshold $\lam>0$ such that $e_{\lam}f(z_{0})>-\infty$ for some
$z_{0}\in{\cal Z}$.
\end{defn}

\begin{defn}
For an extended real-valued function $f:{\cal Z}\mapsto[-\infty,\infty]$,
the function $f^{*}:{\cal Z}^{*}\mapsto[-\infty,\infty]$ given by
\[
f^{*}(u):=\max_{z\in{\cal Z}}\left\{ \left\langle u,z\right\rangle -f(z)\right\} \quad\forall u\in{\cal Z}^{*}
\]
is called the \textbf{conjugate function }of $f$.
\end{defn}

\begin{defn}
For $K\subseteq{\cal Z}$, the\textbf{ dual cone $K^{+}$} and \textbf{polar
cone $K^{-}$} are given by
\begin{align*}
K^{+} & :=\left\{ z\in{\cal Z}:\left\langle z,z'\right\rangle \geq0\ \forall z'\in K\right\} ,\\
K^{-} & :=\left\{ z\in{\cal Z}:\left\langle z,z'\right\rangle \leq0\ \forall z'\in K\right\} =-K^{+}.
\end{align*}
\end{defn}

We now state some basic properties about the above objects. 

The first result, whose proof can be found in \citep[Theorem 2.26]{Rockafellar2009},
describes the continuity of the prox related objects.
\begin{prop}
For a proper, lower semicontinuous, convex function $f:{\cal Z}\mapsto[-\infty,\infty)$
and parameter $\lam>0$, the following properties hold:
\begin{itemize}
\item[(a)] the proximal mapping $\prox_{\lam}f$ is single-valued and continuous;
\item[(b)] the ($\lam$-Moreau) envelope function $e_{\lam}f$ is convex, continuously
differentiable, and its gradient is given by
\[
\nabla e_{\lam}f(z)=\frac{1}{\lam}\left[z-\prox_{\lam f}(z)\right]\quad\forall z\in{\cal Z}.
\]
\end{itemize}
\end{prop}

The following proposition, whose proof can be found in \citep[Example 3.5]{Beck2017}
and \citep[Theorem 6.24]{Beck2017}, presents some properties about
indicator functions.
\begin{prop}
For any closed convex set $Z\subseteq{\cal Z}$ and point $z\in{\cal Z}$,
the following properties hold:
\begin{itemize}
\item[(a)] $\pt\delta_{Z}(z)=N_{Z}(z)$;
\item[(b)] for any $\lam>0$, we have $\prox_{\lam}\delta_{Z}(z)=\Pi_{Z}(z)$.
\end{itemize}
\end{prop}

The next result, whose proof can be found in \citep[Proposition XI.1.3.1]{Hiriart-Urruty1993},
presents some basic calculus rules for the approximate subdifferential.
\begin{prop}
For a proper convex function $f:{\cal Z}\mapsto[-\infty,\infty)$,
$\varepsilon>0$, and point $z\in{\cal Z}$, the following properties
hold:
\begin{itemize}
\item[(a)] for any $\alpha>0$ and $r\in{\cal Z}$, we have $\pt_{\varepsilon}(\alpha f+r)(z)=\alpha\pt_{\varepsilon/\alpha}f(z)$;
\item[(b)] for any $\alpha\neq0$, we have $\pt_{\varepsilon}f(\alpha z)=\alpha\pt_{\varepsilon}f(z)$;
\item[(c)] for any $s\in{\cal Z}$, we have $\pt_{\varepsilon}(f+\left\langle s,\cdot\right\rangle )(z)=\pt_{\varepsilon}f(z)+\{s\}$.
\end{itemize}
\end{prop}

The below result, whose proof can be found in \citep[Theorem XI.3.1.1]{Hiriart-Urruty1993},
presents a characterization of the approximate subdifferential on
sums of functions.
\begin{prop}
For proper convex functions $f_{1},f_{2}:{\cal Z}\mapsto(-\infty,\infty]$,
parameter $\varepsilon>0$, and $z\in{\cal Z}$, it holds that 
\[
\pt_{\varepsilon}(f_{1}+f_{2})(z)\supseteq\bigcup_{\substack{\varepsilon_{1}+\varepsilon_{2}\leq\varepsilon,\\
\varepsilon_{1},\varepsilon_{2}\geq0
}
}\left\{ \pt_{\varepsilon_{1}}f_{1}(z)+\pt_{\varepsilon_{2}}f_{2}(z)\right\} .
\]
Moreover, if $\ri\dom f_{1}\cap\ri\dom f_{2}\neq\emptyset$, then
the above relation holds at equality.
\end{prop}

The following transportation formula can be found in \citep[Proposition XI.4.2.2]{Hiriart-Urruty1993}.
\begin{prop}
\label{prop:transportation}(Transportation Formula) For a function
$\psi\in\cConv({\cal Z})$, points $z,\bar{z}\in\dom\psi$, and subgradient
$s\in\pt\psi(z)$, it holds that $s\in\pt_{\varepsilon}\psi(\bar{z})$
where $\varepsilon=f(\bar{z})-f(z)-\left\langle s,\bar{z}-z\right\rangle \geq0$.
\end{prop}

The next result, whose proof can be found in \citep[Theorem 6.45]{Beck2017},
presents a well-known decomposition .
\begin{prop}
(Extended Moreau Decomposition) Let $f:{\cal Z}\mapsto(-\infty,\infty]$
be proper, closed, and convex. Then, for any $z\in{\cal Z}$ and $\lam>0$,
it holds that 
\[
\prox_{\lam}f(z)+\lam\prox_{\lam^{-1}}f^{*}(z/\lam)=z.
\]
\end{prop}

\subsection{\label{subsec:fn_classes}Function Classes}

This sub-subsection defines some important function classes and their
properties.

We first define the key function classes considered in this thesis.
\begin{defn}
Let ${\cal C}(Z)$ denote the set of continuously differentiable functions
from $Z\subseteq{\cal Z}$ to $\r$.
\end{defn}

\begin{mdframed}
	\textbf{Important Note}: To be concise, we adopt the convention that if $Z$ is a closed set and $f\in{\cal C}(Z)$, then it is implicitly assumed that $f$ is finite on some open set $\Omega$ containing $Z$.
\end{mdframed}

\begin{defn}
Let ${\cal C}_{L}(Z)$ denote the set of functions in ${\cal C}(Z)$
whose gradient is $L$-Lipschitz continuous on $Z$. Such functions
are typically called $L$\textbf{-smooth}.
\end{defn}

\begin{defn}
Let ${\cal C}_{m,M}(Z)$ denote the set of functions in ${\cal C}(Z)$
that satisfy
\begin{equation}
-\frac{m}{2}\|z-z'\|^{2}\leq f(z)-\ell_{f}(z;z')\leq\frac{M}{2}\|z-z'\|^{2}\quad\forall z,z'\in Z.\label{eq:bgd_curv_def}
\end{equation}
A function $f\in{\cal C}(Z)$ is said to have a\textbf{ curvature
pair} $(m,M)$ if it is in ${\cal C}_{m,M}(Z)$.
\end{defn}

\begin{defn}
Let $\cConv({\cal {\cal Z}})$ be the set of proper, lower semicontinuous,
convex functions from ${\cal Z}$ to $(-\infty,\infty]$. For a convex
set $Z\subseteq{\cal Z}$, let $\cConv(Z)$ be the set of functions
in that $\cConv({\cal {\cal Z}})$ are real-valued on $Z$ and take
value $+\infty$ outside of $Z$.
\end{defn}

\begin{defn}
Let ${\cal F}_{\mu}(Z)$ denote the set of functions in $\cConv(Z)$
that are $\mu$-strongly convex. Let ${\cal F}_{\mu,L}(Z)$ denote
the set of functions in ${\cal F}_{\mu}(Z)$ that are also $L$-smooth.
\end{defn}

The next set of results present different characterizations of the
above classes. The first results is a straightforward consequence
of \citep[Proposition 6.1.3]{Ben-tal2020a}.
\begin{prop}
If $f:Z\mapsto\r$ is twice differentiable with $\lam_{\min}(\nabla^{2}f(z))=-m$
and $\lam_{\max}(\nabla^{2}f(z))=M$ for every $z\in Z$, then $f\in{\cal C}_{m,M}(Z)$. 
\end{prop}

The below result\footnote{Special thanks to Arkadi Nemirovski for helping with this proof.}
relates ${\cal C}_{m,M}(Z)$ with ${\cal C}_{L}(Z)$.
\begin{prop}
Let $f:Z\mapsto\mathbb{R}$ be a continuously differentiable function
for some $Z\subseteq{\cal Z}$. Then $f\in{\cal C}_{L}(Z)$ if and
only if $f\in{\cal C}_{L,L}(Z)$.
\end{prop}

\begin{proof}
Let $x,y\in Z$ be arbitrary. Suppose $f\in{\cal C}_{L}(Z)$ and define
$\boldsymbol{r}(t)=x+t(y-x)$ for every $t\in[0,1]$. Using the Gradient
Theorem, it holds that
\begin{align*}
f(y)-f(x) & =\int_{t=0}^{t=1}\nabla f(\boldsymbol{r}(t))\cdot d\boldsymbol{r}(t)=\int_{0}^{1}\left\langle \nabla f(x+t(y-x)),y-x\right\rangle dt\\
 & =\left\langle \nabla f(x),y-x\right\rangle +\int_{0}^{1}\left\langle \nabla f(x+t(y-x))-\nabla f(x),y-x\right\rangle dt.
\end{align*}
Using the Cauchy-Schwarz inequality, the above relation, and Lipschitz
continuity of $\nabla f$, we now conclude that
\begin{align*}
\left|f(y)-\ell_{f}(y;x)\right| & \leq\int_{0}^{1}\left|\left\langle \nabla f(x+t(y-x))-\nabla f(x),y-x\right\rangle \right|dt\\
 & \leq\int_{0}^{1}\|\nabla f(x+t(y-x))-\nabla f(x)\|\cdot\|y-x\|dt\\
 & \leq\int_{0}tL\|y-x\|^{2}dt=\frac{L}{2}\|y-x\|^{2}
\end{align*}
and hence $f\in{\cal C}_{L,L}(Z)$. 

Conversely, suppose $f\in{\cal C}_{L,L}(Z)$ and let $\{\delta_{n}\}_{n\geq1}$
be a sequence of smooth, real-valued, (mollifier) functions over ${\cal Z}$
where, for every $n\geq1$, we have: (i) $\delta_{n}\geq0$; (ii)
$\int_{Z}\delta_{n}(t)\ dt=1$; and (iii) $\delta_{n}(t)=0$ for $t$
satisfying $\|t\|\geq1/n$. Moreover, for every $n\geq1$, define
$g_{n}=\delta_{n}*f$ and denote $d=x-y$. It now follows that 
\begin{align*}
\left|g_{n}(y)-\ell_{g_{n}}(y;x)\right|= & \left|\delta_{n}*\left[f(y)-f(x)\right]+\left\langle \delta_{n}*\nabla f(x),d\right\rangle \right|\\
= & \left|\int_{{\cal Z}}\delta_{n}(\tau)\left[f(y-\tau)-f(x-\tau)\right]\ d\tau+\left\langle \int_{{\cal Z}}\delta_{n}(\tau)\nabla f(x-\tau)\ d\tau,d\right\rangle \right|\\
= & \left|\int_{{\cal Z}}\delta_{n}(\tau)\left[f(y-\tau)-f(x-\tau)+\left\langle \nabla f(x-\tau),d\right\rangle \right]\ d\tau\right|\\
\leq & \int_{{\cal Z}}\delta_{n}(\tau)\left|f(y-\tau)-f(x-\tau)+\left\langle \nabla f(x-\tau),d\right\rangle \right|\ d\tau\\
\leq & \frac{L}{2}\int_{{\cal Z}}\delta_{n}(\tau)\|d\|^{2}\ d\tau=\frac{L}{2}\|d\|^{2},
\end{align*}
and hence that $g_{n}\in{\cal C}_{L,L}(Z)$ as well. Using the smoothness
of $\delta_{n}$ (and hence $g_{n}$), Taylor's Theorem, and the previous
result, it holds that there exists $\xi\in[x,y]$ such that 
\begin{align*}
\frac{L}{2} & \geq\left|\frac{g_{n}(y)-\ell_{g_{n}}(y;x)}{\|d\|^{2}}\right|=\left|\frac{\left\langle d,\nabla^{2}g_{n}(\xi)d\right\rangle }{2\|d\|^{2}}+\frac{o(\|d\|^{2})}{\|d\|^{2}}\right|.
\end{align*}
Taking $y\to x$ in the above inequality, we thus conclude that $\|\nabla^{2}g_{n}(z)\|\leq L$
for every $z\in Z$, and hence, it holds that $g_{n}\in{\cal C}_{L}(Z)$.
Since $\nabla g_{n}\to\nabla f$ uniformly, it follows that $f\in{\cal C}_{L}(Z)$
as well.
\end{proof}

\section{Algorithmic Background}

This section presents some fundamental algorithms that will be relevant
in the algorithmic developments of the thesis.

Throughout this section, we let $Z\subseteq{\cal Z}$ be a nonempty
convex set. Moreover, for all the algorithms in the thesis, we use
the notation ``$\gets$'' for scalar or vector variable assignment
and ``$\Lleftarrow$'' for function assignment.

\subsection{Composite Gradient (CG) Method}

\label{subsec:cgm}

The composite gradient (CG) method (also known as the proximal gradient
method) is a popular optimization algorithm \citep{Boyd2014} for
solving and/or finding stationary points of the problem 
\begin{equation}
\min_{z\in{\cal Z}}\left\{ \psi(z):=\psi_{s}(z)+\psi_{n}(z)\right\} \tag{\ensuremath{{\cal CO}}}\label{prb:eq:co}
\end{equation}
where $\psi_{n}\in\cConv(Z)$ and $\psi_{s}\in{\cal C}(Z)$. More
specifically, it is an iterative method that, at its $k^{{\rm th}}$
iteration, performs the following update: given $z_{k-1}\in Z$ and
$\lam_{k}>0$, compute
\begin{align*}
z_{k} & =\prox_{\lam_{k}\psi_{n}}(z_{k-1}-\lam_{k}\nabla\psi_{s}(z_{k-1})).
\end{align*}
When $\psi_{n}=\delta_{C}$ for some closed convex set $C$, it is
straightforward to see that the CG method (CGM) reduces to the classical
projected gradient method for the problem $\min_{z\in C}\psi_{s}(z)$.
For ease of future reference and discussion, we give a description
in \prettyref{alg:cgm} which includes an important set of auxiliary
iterates $\{v_{k}\}_{k\geq1}$.

\begin{mdframed}
\mdalgcaption{CG Method}{alg:cgm}
\begin{smalgorithmic}
	\Require{$\psi_n \in \cConv(Z), \enskip \psi_s \in {\cal C}(Z), \enskip z_0 \in Z, \enskip \{\lam_k\}_{k\geq 1}  \subseteq \r_{++}$;}
	\vspace*{.5em}
	\Procedure{CG}{$\psi_s, \psi_n, z_0, \{\lam_k\}$}
	\For{$k=1,...$}
		\StateEq{$z_k \gets \argmin_{u\in{\cal Z}}\left\{\lam_{k} \left[ \ell_{\psi_s}(u;z_{k-1}) + \psi_n(u) \right] + \frac{1}{2}\|u-z_{k-1}\|^2\right\}$} \label{ln:cgm_step}
		\StateEq{$v_k \gets \frac{1}{\lam_{k}}(z_{k-1}-z_k) + \nabla \psi_s(z_k) - \nabla \psi_s(z_{k-1})$}
	\EndFor
	\EndProcedure
\end{smalgorithmic}
\end{mdframed}

The proposition below, whose proof can be found in \prettyref{app:prox_vartn_props},
presents some basic properties about the CGM.
\begin{prop}
\label{prop:cgm_basic_vartn}Let $\{(z_{k},v_{k})\}_{k\geq1}$ be
generated by the CGM for some $\{\lam_{k}\}_{k\geq1}$. Then, the
following statements hold for every $k\geq1$:
\begin{itemize}
\item[(a)] $v_{k}\in\nabla\psi_{s}(z_{k})+\pt\psi_{n}(z_{k})$;
\item[(b)] if there exists $L_{k}\in(0,2/\lam_{k})$ such that
\begin{equation}
\psi_{s}(z_{k})-\ell_{\psi_{s}}(z_{k};z_{k-1})\leq\frac{L_{k}}{2}\|z_{k}-z_{k-1}\|^{2},\label{eq:pgm_descent}
\end{equation}
then it holds that
\begin{equation}
\psi(z_{k})<\psi(z_{k})+\left(\frac{1}{\lam_{k}}-\frac{L_{k}}{2}\right)\|z_{k-1}-z_{k}\|^{2}\leq\psi(z_{k-1});\label{eq:descent_pgm_props}
\end{equation}
\item[(c)] if there exists scalars $\{L_{i}\}_{i=1}^{k}\subseteq\r_{++}$ such
that 
\begin{equation}
\|\nabla\psi_{s}(z_{i-1})-\nabla\psi_{s}(z_{i})\|\leq L_{i}\|z_{i-1}-z_{i}\|,\quad L_{i}<\frac{2}{\lam_{i}},\label{eq:pgm_lipschitz}
\end{equation}
for every $i\leq k$, then it holds that
\begin{equation}
\begin{gathered}\min_{i\leq k}\|v_{i}\|^{2}\leq\frac{4\left[\psi(z_{0})-\psi(z_{k})\right]}{\sum_{i=1}^{k}\xi_{i}\lam_{i}},\end{gathered}
\label{eq:pgm_props}
\end{equation}
where $\xi_{i}:=(2-\lam_{i}L_{i})/(1+[\lam_{i}L_{i}]^{2})>0$ for
every $i\leq k$.
\end{itemize}
\end{prop}

The next proposition, whose proof can also be found in \prettyref{app:prox_vartn_props},
presents additional variational properties about a general iteration
in the CGM.
\begin{prop}
\label{prop:cgm_ext_vartn}Given $(\lam,z^{-})\in\r_{+}\times{\cal Z}$,
define
\begin{align*}
z & :=\argmin_{u\in{\cal Z}}\left\{ \lam\left[\ell_{\psi_{s}}(u;z^{-})+\psi_{n}(u)\right]+\frac{1}{2}\|u-z^{-}\|^{2}\right\} ,\\
q & :=\frac{1}{\lam}(z^{-}-z),\quad v:=q+\nabla\psi_{s}(z)-\nabla\psi_{s}(z^{-}),\\
\varepsilon & :=\psi_{n}(z^{-})-\psi_{n}(z)+\inner{q-\nabla\psi_{s}(z^{-})}{z-z^{-}}.
\end{align*}
Then, the following statements hold:
\begin{itemize}
\item[(a)] $q\in\nabla\psi_{s}(z^{-})+\partial_{\varepsilon}\psi_{n}(z^{-})$
and $\varepsilon\geq0$;
\item[(b)] it holds that
\begin{equation}
(q,\varepsilon)=\argmin_{(r,\delta)\in{\cal Z}\times\r_{+}}\left\{ \lam\|r\|^{2}+2\delta:r\in\nabla\psi_{s}(z^{-})+\partial_{\delta}\psi_{n}(z^{-})\right\} ;\label{eq:minimal_cgs}
\end{equation}
\item[(c)] if there exists $L>0$ satisfying
\[
\psi_{s}(z)-\ell_{\psi_{s}}(z;z^{-})\leq\frac{L}{2}\|z-z^{-}\|^{2},
\]
 then it holds that
\begin{align*}
\lam\|q\|^{2}+2\varepsilon & \leq2\left[\psi(z^{-})-\psi(z)\right]+\left(L-\frac{1}{\lam}\right)\|z^{-}-z\|^{2}.
\end{align*}
\end{itemize}
\end{prop}

\subsection{Accelerated Composite Gradient (ACG) Method}

Accelerated composite gradient (ACG) methods are extensions to the
CGM in \prettyref{subsec:cgm} in which additional computations are
performed to improve the rate at which a near optimal solution (or
stationary point) is obtained. 

The ACG variant that we consider in this thesis is based on the accelerated
method in \citep{Monteiro2016}. More specifically, this ACG method
(ACGM) assumes that $\psi_{s}\in{\cal F}_{\mu}(Z)$ for some $\mu\geq0$
and, at its $k^{{\rm th}}$ iteration, performs the following update:
given $(y_{k-1},x_{k-1})\in Z^{2}$, $A_{k-1}\geq0$, and $\lam_{k}>0$,
compute
\begin{align*}
\tau_{k-1} & =\lam_{k}(1+\mu A_{k-1}),\\
a_{k-1} & =\frac{\tau_{k-1}+\sqrt{\tau_{k-1}^{2}+4\tau_{k-1}A_{k-1}}}{2},\quad A_{k}=A_{k-1}+a_{k-1},\\
\tilde{x}_{k-1} & =\frac{A_{k-1}}{A_{k}}y_{k-1}+\frac{a_{k-1}}{A_{k}}x_{k-1},\\
q_{k} & \equiv\ell_{\psi_{s}}(\cdot;\tilde{x}_{k-1})+\psi_{n}(\cdot)+\frac{\mu}{2}\|\cdot-\tilde{x}_{k-1}\|^{2},\\
y_{k} & =\argmin_{y\in{\cal Z}}\left\{ \lam_{k}q_{k}(y)+\frac{1}{2}\|y-\tilde{x}_{k-1}\|^{2}\right\} ,\\
x_{k} & =x_{k-1}+\frac{a_{k-1}}{1+\mu A_{k}}\left[\frac{1}{\lam_{k}}(y_{k}-\tilde{x}_{k-1})+\mu(y_{k}-x_{k-1})\right].
\end{align*}
Using the definition of the proximal operator $\prox_{f}(\cdot)$,
it is straightforward to see that the updates for $y_{k}$ and $x_{k}$
can be written as
\begin{align*}
\alpha_{k} & =\frac{\lam_{k}}{1+\lam_{k}\mu},\quad\beta_{k}=\frac{a_{k-1}}{1+\mu A_{k}},\\
y_{k} & =\prox_{\alpha_{k}\psi_{n}}\left(\tilde{x}_{k-1}-\alpha_{k}\nabla\psi_{s}(\tilde{x}_{k-1})\right),\\
\gamma_{k-1} & \equiv q_{k}(y_{k})+\frac{1}{\lam_{k}}\left\langle \tilde{x}_{k-1}-y_{k},\cdot-y_{k}\right\rangle +\frac{\mu}{2}\|\cdot-y_{k}\|^{2},\\
x_{k} & =x_{k-1}-\beta_{k}\nabla\gamma_{k-1}(x_{k-1}).
\end{align*}
where it can be shown (see \prettyref{app:acgm_vartn_props}) that
$\gamma_{k-1}\leq q_{k}\leq\psi$ for every $k\geq1$. For ease of
future reference and discussion, we give a precise description of
this ACG method in \prettyref{alg:acgm}, which includes an important
set of auxiliary iterates $\{(r_{k},\tilde{r}_{k},\eta_{k},\tilde{\eta}_{k},L_{k})\}_{k\geq1}$.

\begin{mdframed}
\mdalgcaption{ACG Method}{alg:acgm}
\begin{smalgorithmic}
	\Require{$\mu \geq 0, \enskip \psi_n \in \cConv(Z), \enskip \psi_s \in {\cal F}_{\mu}(Z), \enskip y_0 \in Z, \enskip \{\lam_k\}_{k\geq 1} \subseteq \r_{++}$;}
	\Initialize{$\psi \gets \psi_s + \psi_n, \enskip A_0 \gets 0, \enskip \Gamma_0 \Lleftarrow 0, \enskip x_0 \gets y_0;$}
	\vspace*{.5em}
	\Procedure{ACG}{$\psi_s, \psi_n, \mu, y_0, \mu, \{\lam_k\}$}
	\For{$k=1,...$}
		\StateStep{\algpart{1}\textbf{Compute} the supporting quantities:}
		\StateEq{$\tau_{k-1} \gets \lam_{k}(1+\mu A_{k-1})$}
		\StateEq{$a_{k-1} \gets \frac{\tau_{k-1} + \sqrt{\tau_{k-1}^2 + 4\tau_{k-1} A_{k-1}}}{2}$}
		\StateEq{$A_k \gets A_{k-1} + a_{k-1}$}
		\StateEq{$\tilde{x}_{k-1} \gets \frac{A_{k-1}}{A_k}y_{k-1}+\frac{a_{k-1}}{A_k}x_{k-1}$} 		
		\StateEq{$q_k \Lleftarrow \ell_{\psi_s}(\cdot; \tilde{x}_{k-1}) + \psi_n(\cdot) + \frac{\mu}{2}\|\cdot-\tilde{x}_{k-1}\|^{2}$}
		\StateStep{\algpart{2}\textbf{Perform} the accelerated prox steps:}
		\StateEq{$y_k \gets \argmin_{y\in{\cal Z}}\left\{\lam_{k} q_k(y)+\frac{1}{2}\|y-\tilde{x}_{k-1}\|^2\right\}$}
		\StateEq{$x_k\gets x_{k-1}+\frac{a_{k-1}}{1+A_k\mu}\left[\frac{1}{\lam_{k}}(y_k-\tilde{x}_{k-1})+\mu(y_k-x_{k-1})\right]$}
		\StateStep{\algpart{3}\textbf{Compute} the auxiliary quantities:}
		\StateEq{$\gamma_{k-1} \Lleftarrow q_k(y_k) + \frac{1}{\lam_{k}}\left\langle\tilde{x}_{k-1}-y_k,\cdot-y_k\right\rangle +\frac{\mu}{2}\|\cdot-y_k\|^2$}
		\StateEq{$\Gamma_k \Lleftarrow \frac{A_{k-1}}{A_k}\Gamma_{k-1}+\frac{a_{k-1}}{A_k}\gamma_{k-1}$}
		\StateEq{$r_k \gets \frac{x_0 - x_k}{A_k}$}
		\StateEq{$\eta_k \gets \psi(y_k) - \Gamma_k(x_k) - \inner{r_k}{y_k - x_k}$}
		\StateEq{$\tilde{r}_k \gets r_k + \mu(y_k - x_k)$}
		\StateEq{$\tilde{\eta}_k \gets r_k + \frac{\mu}{2}\|y_k - x_k\|^2$}
	\EndFor
	\EndProcedure
\end{smalgorithmic}
\end{mdframed}

The next results, whose proofs are given in \prettyref{app:acgm_vartn_props},
present some key properties about the ACGM and its generated iterates. 
\begin{prop}
\label{prop:acgm_vartn}Let $\{(y_{k},r_{k},\eta_{k})\}_{k\geq1}$
be generated by the ACGM for some $\{\lam_{k}\}_{k\geq1}$. Then the
following statements hold for every $k\geq1$:
\begin{itemize}
\item[(a)] it holds that $\eta_{k}\geq0$ and
\begin{equation}
r_{k}\in\pt_{\eta_{k}}\psi(y_{k});\label{eq:acgm_incl}
\end{equation}
\item[(b)] if there exists $L_{k}>0$ such that 
\begin{equation}
\psi_{s}(y_{k})-\ell_{\psi_{s}}(y_{k};\tilde{x}_{k-1})\leq\frac{L_{k}}{2}\|y_{k}-\tilde{x}_{k-1}\|^{2},\quad L_{k}-\mu\leq\frac{1}{\lam_{k}},\label{eq:acgm_descent}
\end{equation}
then it holds that
\begin{equation}
\|A_{k}r_{k}+y_{k}-y_{0}\|^{2}+2A_{k}\eta_{k}\le\|y_{k}-y_{0}\|^{2};\label{eq:acg_invar}
\end{equation}
\item[(c)] it holds that
\[
A_{k}\geq\max\left\{ \frac{1}{4}\left(\sum_{i=1}^{k-1}\sqrt{\lam_{k-1}}\right)^{2},\lam_{1}\prod_{i=2}^{k}\left(1+\sqrt{\frac{\lam_{i-1}\mu}{2}}\right)^{2}\right\} ;
\]
\item[(d)] for every minimizer $y^{*}$ of the problem $\min_{y\in\dom\psi}\psi(y)$,
it holds that 
\[
\psi(y_{k})-\psi(y^{*})\leq\frac{1}{2A_{k}}\|y^{*}-y_{0}\|^{2}\quad\forall k\geq1.
\]
\end{itemize}
\end{prop}

\begin{prop}
\label{prop:acgm_vartn_ext}Let $\{(y_{k},\tilde{r}_{k},\tilde{\eta}_{k})\}_{k\geq1}$
be generated by the ACGM for some $\{\lam_{k}\}_{k\geq1}$. Then the
following statements hold for every $k\geq1$:
\begin{itemize}
\item[(a)] it holds that $\eta_{k}\geq0$ and
\begin{equation}
\tilde{r}_{k}\in\pt_{\tilde{\eta}_{k}}\left(\psi-\frac{\mu}{2}\|\cdot-y_{k}\|^{2}\right)(y_{k});\label{eq:acgm_incl_ext}
\end{equation}
\item[(b)] if there exists $L_{k}>0$ such that
\begin{equation}
\psi_{s}(y_{k})-\ell_{\psi_{s}}(y_{k};\tilde{x}_{k-1})\leq\frac{L_{k}}{2}\|y_{k}-\tilde{x}_{k-1}\|^{2},\quad L_{k}-\mu\leq\frac{1}{\lam_{k}},\label{eq:acgm_descent_ext}
\end{equation}
then it holds that
\begin{equation}
\left(\frac{1}{1+\mu A_{k}}\right)\|A_{k}\tilde{r}_{k}+y_{k}-y_{0}\|^{2}+2A_{k}\tilde{\eta}_{k}\le\|y_{k}-y_{0}\|^{2};\label{eq:acg_invar_ext}
\end{equation}
\end{itemize}
\end{prop}

\subsection{Proximal Point Method}

The proximal point (PP) method is a classic optimization algorithm
\citep{Rockafellar1976} for minimizing a function $\psi\in\cConv({\cal Z})$.
More specifically, it is an iterative method that, at its $k^{{\rm th}}$
iteration, performs the following update: given $z_{k-1}\in\dom\psi$
and $\lam_{k}>0$, perform
\begin{equation}
z_{k}=\prox_{\lam_{k}\psi}(z_{k-1}).\label{eq:ppm_update}
\end{equation}
It is well-known (see, for example, \citep[Theorem 27.1]{Bauschke2010})
that if $\sum_{k=1}^{\infty}\lam_{k}=\infty$, then $\psi(z_{k})$
converges to $\inf_{z\in{\cal Z}}\phi(z)$. Moreover, if there exists
$z^{*}$ satisfying $\phi(z^{*})=\inf_{z\in{\cal Z}}\phi(z)$, then
$z_{k}$ converges to the set of minimizers of $\phi$. 

The following proposition, whose proof can be found in \prettyref{app:prox_vartn_props},
presents some basic properties (cf. \prettyref{prop:cgm_basic_vartn})
about the PP method (PPM).
\begin{prop}
\label{prop:ppm_vartn}Let $\{z_{k}\}_{k\geq1}$ be generated by the
PPM for some $\{\lam_{k}\}_{k\geq1}$ and define $v_{k}:=(z_{k-1}-z_{k})/\lam_{k}$
for every $k\geq1$. Then, the following statements hold for every
$k\geq1$:
\begin{itemize}
\item[(a)] $v_{k}\in\pt\psi(z_{k})$;
\item[(b)] it holds that
\[
\psi(z_{k})<\psi(z_{k})+\frac{1}{\lam_{k}}\|z_{k}-z_{k-1}\|^{2}\leq\psi(z_{k-1});
\]
\item[(c)] it holds that 
\begin{equation}
\min_{i\leq k}\|v_{i}\|^{2}\leq\frac{\psi(z_{0})-\psi(z_{k})}{\sum_{i=1}^{k}\lam_{i}}.\label{eq:ppm_props}
\end{equation}
\end{itemize}
\end{prop}

Throughout this thesis, we make reference to the \textbf{inexact}
proximal point method which is a variant of the PPM in which the update
\eqref{eq:ppm_update} is computed inexactly, i.e. $z_{k}$ approximates
the solution of the problem in \eqref{eq:ppm_update} according to
some inexactness criterion.

One interesting instance of the proximal point method is when $\psi(x)=(1/2)\left\langle x,Ax\right\rangle -\left\langle b,x\right\rangle $
where $A\in\mathbb{S}_{++}^{n}$. Clearly, the optimal solution of
$\min_{x\in\rn}\psi(x)$ is the unique solution of the linear system
of equations $Ax=b$. The proximal update in the case of $\lam_{k}=\lam\in\r_{++}$
for every $k\geq1$ is 
\[
x_{k+1}=x_{k}+(A+\lam^{-1}I_{n})^{-1}(b-Ax_{k}),
\]
which is a well-known algorithm called iterative refinement \citep{Boyd2014}.
The above update is particularly useful when $A$ is ill-conditioned
and/or the computation of $A^{-1}b$ is not stable. 

\newpage{}

\chapter{Unconstrained Composite Optimization}

\label{chap:unconstr_nco}

Our main goal in this chapter is to describe and establish the iteration
complexity of an accelerated \textbf{inexact} proximal point (AIPP)
method for finding approximate stationary points of the classic NCO
problem 
\begin{equation}
\phi_{*}=\min_{z\in{\cal Z}}\left[\phi(z):=f(z)+h(z)\right],\tag{\ensuremath{{\cal NCO}}}\label{prb:eq:nco}
\end{equation}
where ${\cal Z}$ is a finite dimensional inner product space, $h\in\cConv(Z)$
for some nonempty convex set $Z\subseteq{\cal Z}$, and $f\in{\cal C}_{m,M}(Z)$
for some $(m,M)\in\r_{++}^{2}$. 

The AIPP method (AIPPM) of this chapter uses an ACGM, specifically
\prettyref{alg:acgm}, to perform the following proximal point-type
update to generate its $k^{{\rm th}}$ iterate: given $z_{k-1}$ and
$\lam$, compute
\[
z_{k}\approx\min_{z\in{\cal Z}}\left\{ f(z)+h(z)+\frac{1}{2\lam}\|z-z_{k-1}\|^{2}\right\} 
\]
according to some \textbf{relative} inexactness criterion. Throughout
our presentation, it is assumed that efficient oracles for evaluating
the quantities $f(z)$, $\nabla f(z)$, and $h(z)$ and for obtaining
exact solutions of the subproblem 
\[
\min_{z\in{\cal Z}}\left\{ \lam h(z)+\frac{1}{2}\|z-z_{0}\|^{2}\right\} ,
\]
for any $z_{0}\in{\cal Z}$ and $\lam>0$, are available. Moreover,
we define an \textbf{oracle call} to be a collection of the above
oracles of size ${\cal O}(1)$ where each of them appears at least
once.

For a given tolerance $\hat{\rho}>0$ and a suitable choice of $\lam$,
the main result of this chapter shows that the AIPPM, started from
any point $z_{0}\in Z$ obtains a pair $(\hat{z},\hat{v})$ satisfying
the approximate stationarity condition
\begin{equation}
\hat{v}\in\nabla f(\hat{z})+\partial h(\hat{z}),\quad\|\hat{v}\|\le\hat{\rho},\label{eq:rho_approx_nco}
\end{equation}
in
\begin{equation}
\mathcal{O}\left(\sqrt{\frac{M}{m}+1}\ \left[\frac{m\cdot\min\left\{ \phi(z_{0})-\phi_{*},md_{0}^{2}\right\} }{\hat{\rho}^{2}}+\log_{1}^{+}\left(\frac{M}{m}\right)\right]\right)\label{eq:bound_nco_intro}
\end{equation}
oracle calls, where $d_{0}=\min_{z\in{\cal Z}}\{\|z_{0}-z_{*}\|:\phi(z_{*})=\phi_{*}\}$
and $\log_{1}^{+}(\cdot)=\max\{1,\log(\cdot)\}$. It is worth mentioning
that this result is obtained under the mild assumption that $\phi_{*}$
is finite and neither assumes neither that $Z$ is bounded nor that
\ref{prb:eq:nco} has an optimal solution. Near the end of the chapter,
we compare the above complexity against ones obtained by other NCO
methods.

It is also shown in \prettyref{subsec:compl_bds} that the complexity
bound in \eqref{eq:bound_nco_intro} is optimal in the sense that
it is within the same order of magnitude of a recent established complexity
lower bound for finding pairs $(\hat{z},\hat{v})$ satisfying \eqref{eq:rho_approx_nco}
using linear-span first-order methods.

The content of this chapter is based on paper \citep{Kong2019} (joint
work with Jefferson G. Melo and Renato D.C. Monteiro) and several
passages may be taken verbatim from it.

\subsection*{Related Works}

The developments in \citep{Kong2019} appear to be the first ones
to consider an accelerated proximal method for obtaining approximate
stationary points as in \eqref{eq:rho_approx_nco} for general $h$
and nonconvex $f$. Previous developments, which we list below, have
only considered the special case of $h=0$.

Under the assumption that $\dom\phi$ is bounded, paper \citep{Ghadimi2016}
presents an ACG method applied directly to \ref{prb:eq:nco} which
obtains a pair $(\hat{z},\hat{v})$ satisfying \eqref{eq:rho_approx_nco}
in 
\begin{equation}
{\cal O}\left(\frac{MmD_{z}^{2}}{\hat{\rho}^{2}}+\left[\frac{Md_{0}}{\hat{\rho}}\right]^{2/3}\right)\label{eq:compl-lan}
\end{equation}
where $D_{z}$ denotes the diameter of $\dom\phi$. Motivated by the
developments in \citep{Ghadimi2016}, other papers, such as \citep{Carmon2018,Ghadimi2019,Drusvyatskiy2019,Paquette2017,Li2015,Liang2018,Liang2019,Liang2019a},
have proposed ACG-like methods under different assumptions on the
functions $f$ and $h$. For example, paper \citep{Carmon2018} establishes
a complexity which is ${\cal O}(\sqrt{M}\log M)$ in terms of its
dependence on $M$, but is ${\cal O}(\hat{\rho}^{-2}\log\hat{\rho}^{-1})$
in terms of its dependence on $\hat{\rho}$. It should be noted that
the second complexity bound in \eqref{eq:bound_nco_intro} in terms
of $d_{0}$ is new in the context of problem \ref{prb:eq:nco} and
follows as a special case of a more general bound, namely \eqref{cor:spec_aipp_compl},
which actually unifies both bounds in \eqref{eq:bound_nco_intro}.
Moreover, in contrast to the analysis of \citep{Ghadimi2016}, the
analysis in this chapter does not assume that $D_{z}$ in \eqref{eq:compl-lan}
is finite.

Inexact proximal point methods and HPE variants of the ones studied
in \citep{Monteiro2010,Solodov2001} for solving convex-concave saddle
point problems and monotone variational inequalities --- which inexactly
solve a sequence of proximal subproblems by means of an ACG variant
--- were previously proposed by \citep{He2015,He2016,Monteiro2012,Ouyang2015,Kolossoski2017}.
The behavior of an accelerated gradient method near saddle points
is studied in \citep{ONeill2019}.

Complexity lower bounds in terms of $\max\{m,M\}$ for finding stationary
points as in \eqref{eq:rho_approx_nco} using first-order methods
were recently established in \citep{Carmon2020,Carmon2021}. A follow-up
work \citep{Zhou2019} establishes tighter bounds in terms of $m$
and $M$ for the smaller class of linear-span first-order methods.

\subsection*{Organization}

This chapter contains three sections. The first one gives some preliminary
references and discusses our notion of a stationary point given in
\eqref{eq:rho_approx_nco}. The second one presents a general inexact
proximal point framework which will be important in our analysis of
the AIPPM. The third one presents the AIPPM and its iteration complexity.
The last one gives a conclusion and some closing comments.

\section{Preliminaries}

This section enumerates the assumptions on problem \ref{prb:eq:nco},
states the main problem of interest, and discusses the notion of an
approximate stationary point given in \eqref{eq:rho_approx_nco}.

It is assumed that $(f,h,\phi)$ in \ref{prb:eq:nco} satisfy:

\stepcounter{assumption}
\begin{enumerate}
\item \label{asmp:nco1}$h\in\cConv(Z)$ for some nonempty convex set $Z\subseteq{\cal Z}$;
\item \label{asmp:nco2}$f\in{\cal C}_{m,M}(Z)$ for some $(m,M)\in\r_{++}^{2}$;
\item \label{asmp:nco3}$\phi_{*}>-\infty$.
\end{enumerate}
We now make a few remarks about the above assumptions. First, assumption
\ref{asmp:nco1} implies that the effective domain of $h$ is $Z$.
Second, if $\nabla f$ is $M$-Lipschitz continuous, then assumption
\ref{asmp:nco2} holds with $m=M$. Third, it is well-known that a
necessary condition for $z^{*}\in Z$ to be a local minimum of \ref{prb:eq:nco}
is that $z^{*}$ be a stationary point of $f+h$, i.e. $0\in\nabla f(z^{*})+\partial h(z^{*})$.

In view of the above assumptions and remarks, we are interested in
solving the problem given in \prettyref{prb:approx_nco}.

\begin{mdframed}
\mdprbcaption{Find an approximate stationary point of ${\cal NCO}$}{prb:approx_nco}
Given $\hat{\rho} > 0$, find a pair $(\hat{z},\hat{v}) \in Z \times {\cal Z}$ satisfying condition \eqref{eq:rho_approx_nco}.
\end{mdframed}

The next proposition, which follows from \prettyref{lem:compl_approx2},
gives another well-known (see, for example, \citep{Nesterov2013})
interpretation of our notion of an approximate stationary point.
\begin{prop}
Given $\hat{z}\in Z$, there exists $\hat{v}\in{\cal Z}$ such that
$(\hat{z},\hat{v})$ satisfies \eqref{eq:rho_approx_nco} if and only
if $\inf_{\|d\|\leq1}\phi'(\hat{z};d)\geq-\hat{\rho}$.
\end{prop}

\section{General Inexact Proximal Point (GIPP) Framework}

\label{sec:gippf}

This section presents and discusses general inexact proximal point
framework which will be important in our analysis of the AIPPM. It
contains three subsections. The first one presents some important
properties of the framework. The second one presents a procedure to
turn iterates generated by the framework into iterates that \emph{nearly
}solve \prettyref{prb:approx_nco}. The last one gives some instances
of the framework.

We begin by first stating the framework in \prettyref{alg:gippf}.

\begin{mdframed}
\mdalgcaption{GIPP Framework}{alg:gippf}
\begin{smalgorithmic}
	\Require{$h \in \cConv(Z), \enskip f \in {\cal C}(Z), \enskip z_0 \in Z, \enskip \sigma \in (0,1), \enskip \{\lam_k\}_{k\geq 1} \subseteq \r_{++}$;}
	\vspace*{.5em}
	\Procedure{GIPP}{$f, h, z_0, \sigma, \{\lam_k\}_{k\geq 1}$}
		\For{$k=1,...$}
			\StateEq{Find $(z_k, \tilde{v}_k, \tilde{\varepsilon}_k) \in \dom h \times {\cal Z} \times \r_{+}$ satisfying:
			\begin{gather}
\tilde{v}_{k}\in\partial_{\tilde{\varepsilon}_{k}}\left(\lambda_{k}\phi+\frac{1}{2}\|\cdot-z_{k-1}\|^{2}\right)(z_{k}),\label{eq:GIPPF} \\ 
\|\tilde{v}_{k}\|^{2}+2\tilde{\varepsilon}_{k}\leq\sigma\|z_{k-1}-z_{k}+\tilde{v}_{k}\|^{2};\label{eq:err_crit_GIPP} 
			\end{gather}}
	\EndFor
	\EndProcedure
\end{smalgorithmic}
\end{mdframed}

Observe that the GIPP framework (GIPPF) is not a well-specified algorithm
but rather a conceptual framework consisting of (possibly many) specific
instances. In particular, it does not specify how the quadruple $(z_{k},\tilde{v}_{k},\tilde{\varepsilon}_{k})$
is computed or even if it exists. Later in this chapter, we will discuss
two specific instances of the above GIPPF for solving \ref{prb:eq:nco},
namely, the CGM (see \prettyref{alg:cgm}) and an accelerated proximal
point method presented in \prettyref{sec:aipp}. In both of these
instances, the sequences $\{\tilde{z}_{k}\}_{k\geq1}$ and $\{\tilde{\varepsilon}_{k}\}_{k\geq1}$
are non-trivial (see \prettyref{prop:gradient method} and \prettyref{lem:AIPPmethod}(c)). 

\subsection{Key Properties of the Framework}

This subsection presents some key properties of the GIPPF.

Let $\{(z_{k},\tilde{v}_{k},\tilde{\varepsilon}_{k})\}_{k\geq1}$
be the sequence generated by an instance of the GIPPF for some $\{\lam_{k}\}_{k\geq1}$,
and consider the sequence $\{(v_{k},\varepsilon_{k})\}_{k\geq1}$
defined as 
\begin{equation}
(v_{k},\varepsilon_{k}):=\frac{1}{\lambda_{k}}(\tilde{v}_{k},\tilde{\varepsilon}_{k})\quad\forall k\geq1.\label{def:rvepsk}
\end{equation}
Without necessarily assuming that the error condition \eqref{eq:err_crit_GIPP}
holds, the following technical but straightforward result derives
bounds on $\tilde{\varepsilon}_{k}$ and $\|\tilde{v}_{k}+z_{k-1}-z_{k}\|/\lam_{k}$
in terms of the quantities
\begin{equation}
\delta_{k}=\delta_{k}(\sigma):=\frac{1}{\lambda_{k}}\max\left\{ 0,\|\tilde{v}_{k}\|^{2}+2\tilde{\varepsilon}_{k}-\sigma\|z_{k-1}-z_{k}+\tilde{v}_{k}\|^{2}\right\} ,\quad\Lambda_{k}:=\sum_{i=1}^{k}\lambda_{i}\label{def:rk_Lambak_deltak}
\end{equation}
where $\sigma\in[0,1)$ is a given parameter. Note that if \eqref{eq:err_crit_GIPP}
is assumed, then $\delta_{k}=0$.
\begin{lem}
\label{lem:gipp_main_bd}Assume that the sequence $\{(\lambda_{k},z_{k},\tilde{v}_{k},\tilde{\varepsilon}_{k})\}$
satisfies \eqref{eq:GIPPF} and let $\sigma\in(0,1)$ be given. Then,
for every $k\geq1$, there holds 
\begin{equation}
\frac{1}{\sigma\lambda_{k}}\left(\|\tilde{v}_{k}\|^{2}+2\tilde{\varepsilon}_{k}-\lambda_{k}\delta_{k}\right)\leq\frac{1}{\lambda_{k}}\|z_{k-1}-z_{k}+\tilde{v}_{k}\|^{2}\leq\frac{2[\phi(z_{k-1})-\phi(z_{k})]+\delta_{k}}{1-\sigma}\label{eq:pre_complex}
\end{equation}
where $\delta_{k}$ is as in \eqref{def:rk_Lambak_deltak}.
\end{lem}

\begin{proof}
First note that the inclusion in \eqref{eq:GIPPF} is equivalent to
\[
\lambda_{i}\phi(z)+\frac{1}{2}\|z-z_{i-1}\|^{2}\ge\lambda_{i}\phi(z_{i})+\frac{1}{2}\|z_{i}-z_{i-1}\|^{2}+\left\langle \tilde{v}_{i},z-z_{i}\right\rangle -\tilde{\varepsilon}_{i}\qquad\forall z\in\Re^{n}.
\]
Setting $z=z_{i-1}$ in the above inequality and using the definition
of $\delta_{i}$ given in \eqref{def:rk_Lambak_deltak}, we obtain
\begin{align*}
 & \lambda_{i}(\phi(z_{i-1})-\phi(z_{i}))\ge\frac{1}{2}\left(\|z_{i-1}-z_{i}\|^{2}+2\left\langle \tilde{v}_{i},z_{i-1}-z_{i}\right\rangle -2\tilde{\varepsilon}_{i}\right)\\
 & =\frac{1}{2}\left[\|z_{i-1}-z_{i}+\tilde{v}_{i}\|^{2}-\|\tilde{v}_{i}\|^{2}-2\tilde{\varepsilon}_{i}\right]\geq\frac{1}{2}\left[(1-\sigma)\|z_{i-1}-z_{i}+\tilde{v}_{i}\|^{2}-\lambda_{i}\delta_{i}\right]
\end{align*}
and hence the proof of the second inequality in \eqref{eq:pre_complex}
follows after simple rearrangements. The first inequality in \eqref{eq:pre_complex}
follows immediately from \eqref{def:rk_Lambak_deltak}. 
\end{proof}
The next result shows characterizes the approximate optimality of
$z_{k}$ in terms of $\lam_{k}$, $z_{k-1}$, and $\sigma$.
\begin{lem}
\label{lem:phik_d0}Let $\{(z_{k},\tilde{v}_{k},\tilde{\varepsilon}_{k})\}$
be generated by an instance of the GIPPF for some $\{\lam_{k}\}_{k\geq1}$.
Then, for every $u\in{\cal Z}$, it holds that 
\[
\phi(z_{k})\leq\phi(u)+\frac{1}{2(1-\sigma)\lambda_{k}}\|z_{k-1}-u\|^{2}\quad\forall k\geq1.
\]
\end{lem}

\begin{proof}
Using some simple algebraic manipulation, it is easy to see that \eqref{eq:err_crit_GIPP}
yields 
\begin{equation}
\inner{\tilde{v}_{k}}{z_{k}-z_{k-1}}+\frac{1}{\sigma}\tilde{\varepsilon}_{k}-\frac{1}{2}\|z_{k-1}-z_{k}\|^{2}\le-\frac{1-\sigma}{2\sigma}\|\tilde{v}_{k}\|^{2}.\label{eq:epstilde}
\end{equation}
Now, letting $\theta:=(1-\sigma)/\sigma>0$, recalling the definition
of the approximate subdifferential, using \eqref{eq:GIPPF} and \eqref{eq:epstilde},
and the fact that $\langle v,v'\rangle\leq(\theta/2)\|v\|^{2}+(1/2\theta)\|v'\|^{2}$
for all $v,v'\in{\cal Z}$, we conclude that
\begin{align*}
\lambda_{k}[\phi(z_{k})-\phi(u)] & \le\frac{1}{2}\|z_{k-1}-u\|^{2}+\left\langle \tilde{v}_{k},z_{k}-u\right\rangle +\tilde{\varepsilon}_{k}-\frac{1}{2}\|z_{k}-z_{k-1}\|^{2}\\
 & \leq\frac{1}{2}\|z_{k-1}-u\|^{2}+\left\langle \tilde{v}_{k},z_{k-1}-u\right\rangle -\frac{1-\sigma}{2\sigma}\|\tilde{v}_{k}\|^{2}\\
 & \le\frac{1}{2}\|z_{k-1}-u\|^{2}+\left(\frac{\theta}{2}\|\tilde{v}_{k}\|^{2}+\frac{1}{2\theta}\|z_{k-1}-u\|^{2}\right)-\frac{1-\sigma}{2\sigma}\|\tilde{v}_{k}\|^{2},
\end{align*}
and hence that the conclusion of the lemma holds due to the definition
of $\theta$. 

Before proceeding, we define the following useful quantity
\begin{equation}
R_{\lam}\psi(z_{0}):=\inf_{u\in{\cal Z}}\left[R_{\lam}\psi(u;z_{0}):=\frac{1}{2}\|z_{0}-u\|^{2}+\lam\left[\psi(u)-\inf_{\tilde{z}\in{\cal Z}}\psi(\tilde{z})\right]\right]\label{eq:Rphi_def}
\end{equation}
for any function $\psi:{\cal Z}\mapsto(-\infty,\infty]$, scalar $\lam\geq0$,
and point $z_{0}\in{\cal Z}$. Clearly, $R_{\lam}\psi(u;z_{0})\in\r_{+}$
for all $u\in Z$ and hence $R_{\lam}\psi(z_{0})\in\r_{+}$ as well.
Moreover, it is easy to see that 
\begin{equation}
R_{\lam}\psi(z_{0})=\lam\left[e_{\lam}\psi(z_{0})-\inf_{u\in{\cal Z}}\psi(u)\right]\leq\lam\left[\psi(z_{0})-\inf_{u\in{\cal Z}}\psi(u)\right],\label{eq:R_moreau}
\end{equation}
where $e_{\lam}\psi(z_{0})$ denotes the $\lam$-Moreau envelope of
$\psi$ at $z_{0}$. 

We now show that the sequence $\{\|z_{k-1}-z_{k}+\tilde{v}_{k}\|/\lam_{k}\}_{k\geq1}$
contains a subsequence that tends to zero.
\end{proof}
\begin{prop}
\label{prop:gipp_descent}Let $\{(z_{k},\tilde{v}_{k},\tilde{\varepsilon}_{k})\}_{k\geq1}$
be generated by an instance of the GIPPF for some $\{\lam_{k}\}_{k\geq1}$.
Then, the following statements hold: 
\begin{itemize}
\item[(a)] for every $k\geq1$, 
\begin{equation}
\frac{1-\sigma}{2\lambda_{k}}\left\Vert z_{k-1}-z_{k}+\tilde{v}_{k}\right\Vert ^{2}\leq\phi(z_{k-1})-\phi(z_{k});\label{eq:gipp_descent}
\end{equation}
\item[(b)] for every $k\geq2$, there exists $i\leq k$ such that 
\begin{equation}
\frac{1}{\lambda_{i}^{2}}\|z_{i-1}-z_{i}+\tilde{v}_{i}\|^{2}\leq\frac{2R_{\lam_{1}}\phi(z_{0})}{(1-\sigma)^{2}\lam_{1}(\Lambda_{k}-\lambda_{1})}=\frac{2\left[e_{\lam_{1}}\phi(z_{0})-\phi_{*}\right]}{(1-\sigma)^{2}(\Lambda_{k}-\lambda_{1})}\label{eq:gipp_Rbd}
\end{equation}
where $\Lambda_{k}$ and $R_{\lam_{1}}\phi(z_{0})$ are as in \eqref{def:rk_Lambak_deltak}
and \eqref{eq:Rphi_def}, respectively. 
\end{itemize}
\end{prop}

\begin{proof}
(a) This follows immediately from \eqref{eq:pre_complex} and the
fact that \eqref{eq:err_crit_GIPP} is equivalent to $\delta_{k}=0$.

(b) It follows from definitions of $\phi_{*}$ and $R_{\lam_{1}}\phi(\cdot;z_{0})$
in \ref{asmp:nco3} and \eqref{eq:Rphi_def}, respectively, part (a)
and \prettyref{lem:phik_d0} with $k=1$ that for all $u\in{\cal Z}$,
\begin{align*}
\frac{R_{\lam_{1}}\phi(u;z_{0})}{(1-\sigma)\lam_{1}} & =\left(\frac{1}{1-\sigma}\right)\left[\frac{1}{2\lam_{1}}\|z_{0}-u\|^{2}+\phi(u)-\phi_{*}\right]\\
 & \geq\frac{1}{2\lam_{1}(1-\sigma)}\|z_{0}-u\|^{2}+\phi(u)-\phi_{*}\\
 & \ge\phi(z_{1})-\phi_{*}\geq\sum_{i=2}^{k}[\phi(z_{i-1})-\phi(z_{i})]\\
 & \geq(1-\sigma)\sum_{i=2}^{k}\frac{\left\Vert z_{i-1}-z_{i}+\tilde{v}_{i}\right\Vert ^{2}}{2\lam_{i}}\\
 & \geq\frac{(1-\sigma)(\Lambda_{k}-\lambda_{1})}{2}\min_{i\le k}\frac{1}{\lambda_{i}^{2}}\left\Vert z_{i-1}-z_{i}+\tilde{v}_{i}\right\Vert ^{2}
\end{align*}
and hence the first inequality of \eqref{eq:gipp_Rbd} holds in view
of the definition of $R_{\lam_{1}}\phi(z_{0})$ in \eqref{eq:Rphi_def}.
The second inequality follows from \eqref{eq:R_moreau}.

Note that the above proposition shows the GIPPF enjoys the descent
property in \prettyref{prop:gipp_descent}, which many frameworks
and/or algorithms for finding approximate stationary points of \ref{prb:eq:nco}
also share, e.g. \prettyref{alg:cgm}. It is worth noting that, under
the assumption that $\phi$ is a KL-function, frameworks and/or algorithms
sharing this property have also been developed for example in \citep{Attouch2009,Attouch2011,Frankel2015,Chouzenoux2016}
where it is shown that the generated sequence $\{z_{k}\}_{k\geq1}$
converges to some stationary point of \ref{prb:eq:nco} with a well-characterized
asymptotic (but not global) convergence rate, as long as $\{z_{k}\}_{k\geq1}$
has an accumulation point.

The following result, which follows immediately from \prettyref{prop:gipp_descent},
considers the instances of the GIPPF where $\lam_{k}$ is constant
for every $k\geq1$. For the purpose of stating it, define
\begin{equation}
d_{0}:=\inf_{z^{*}\in{\cal Z}}\left\{ \|z_{0}-z^{*}\|:\phi(z^{*})=\phi_{*}\right\} .\label{eq:dist0}
\end{equation}
Note that $d_{0}<\infty$ if and only if \ref{prb:eq:nco} has an
optimal solution, in which case the above infimum can be replaced
by a minimum in view of the first assumption following \ref{prb:eq:nco}.
\end{proof}
\begin{cor}
\label{cor:rasc1}Let $\{(z_{k},\tilde{v}_{k},\tilde{\varepsilon}_{k})\}$
be generated by an instance the GIPPF with $\lambda_{k}=\lambda$
for every $k\ge1$, and define $\{(v_{k},\varepsilon_{k},r_{k})\}$
as in \eqref{def:rvepsk}. Then, the following statements hold: 
\begin{itemize}
\item[(a)] for every $k\geq2$, there exists $i\leq k$ such that 
\begin{equation}
\frac{1}{\lambda^{2}}\|z_{i-1}-z_{i}+\tilde{v}_{i}\|^{2}\leq\frac{2R_{\lam}\phi(z_{0})}{\lambda^{2}(1-\sigma)^{2}(k-1)}\leq\frac{\min\left\{ 2[\phi(z_{0})-\phi_{*}],\frac{d_{0}^{2}}{\lambda(1-\sigma)}\right\} }{\lam(1-\sigma)(k-1)}\label{eq:corGIPP_complex2}
\end{equation}
where $R_{\lam}\phi(z_{0})$ and $d_{0}$ are as in \eqref{eq:Rphi_def}
and \eqref{eq:dist0}, respectively;
\item[(b)] for any $\tau>0$, the GIPPF generates a quadruple $(z^{-},z,\tilde{v},\tilde{\varepsilon})$
such that
\begin{equation}
\tilde{v}\in\pt_{\tilde{\varepsilon}}\left(\lam\phi+\frac{1}{2}\|\cdot-z^{-}\|^{2}\right)(z),\quad\frac{1}{\lam}\|z^{-}-z+\tilde{v}\|\leq\tau,\quad\frac{1}{\lam}\tilde{\varepsilon}\leq\left(\frac{\sigma\lam}{2}\right)\tau^{2},\label{eq:prox_approx}
\end{equation}
 in a number of iterations bounded by 
\begin{equation}
\left\lceil \frac{2R_{\lam}\phi(z_{0})}{\lambda^{2}(1-\sigma)^{2}\tau^{2}}+1\right\rceil .\label{eq:complex_gipp}
\end{equation}
\end{itemize}
\end{cor}

\begin{proof}
(a) The proof of the first inequality follows immediately from \prettyref{prop:gipp_descent}(b)
and the fact that $\lambda_{k}=\lambda$ for every $k\ge1$. Now,
note that due to \eqref{eq:Rphi_def}, we have $R_{\lam}\phi(z_{0})\leq R_{\lam}\phi(z_{0};z_{0})=\lambda[\phi(z_{0})-\phi_{*}${]}
and $R_{\lam}\phi(z_{0})\leq R_{\lam}\phi(z^{*};z_{0})=\|z^{*}-z_{0}\|^{2}/2$
for any $z^{*}$ satisfying $\phi(z^{*})=\phi_{*}$. The second inequality
now follows from the previous observation and the definition of $d_{0}$
in \eqref{eq:dist0}.

(b) This statement follows immediately from the first inequality in
(a) and \eqref{eq:err_crit_GIPP}. 
\end{proof}
In the above analysis, we have assumed that $\phi$ is quite general.
For the remainder of this chapter, we derive results that use the
composite structure underlying $\phi$, i.e. $\phi=f+h$ where $f$
and $h$ satisfy conditions \ref{asmp:nco1}--\ref{asmp:nco3}.

\subsection{Generating Stationary Points}

In the previous subsection, we established that the GIPPF is able
to generate a quadruple $(z^{-},z,\tilde{v},\tilde{\varepsilon})$
which satisfies \eqref{eq:prox_approx} for any $\tau>0$. In this
subsection, we present a refinement procedure that uses the above
quadruple to generate a pair $(\hat{z},\hat{v})\in Z\times{\cal Z}$
which, for sufficiently small enough $\tau>0$, satisfies \eqref{eq:rho_approx_nco}.

We begin by presenting the aforementioned procedure in \prettyref{alg:cref}.

\begin{mdframed}
\mdalgcaption{CR Procedure}{alg:cref}
\begin{smalgorithmic}
	\Require{$h \in \cConv({\cal Z}), \enskip f \in {\cal C}(Z), \enskip z \in {\cal Z}, \enskip L > 0, \enskip  \lam > 0$;}
	\Initialize{$L_{\lam} \gets L + \lam^{-1}$;}
	\vspace*{.5em}
	\Procedure{CREF}{$f, h, z, L, \lam$}
		\StateEq{$z_{r} \gets \argmin_{u\in{\cal Z}}\left\{ \ell_f(u;z) + h(u) +\frac{L_{\lam}}{2}\|u-z\|^2 \right\}$}
		\StateEq{$q_{r} \gets L_{\lam}(z-z_r)$}
		\StateEq{$v_{r} \gets q_r + \nabla f(z_r) - \nabla f(z) $}
		\StateEq{$\varepsilon_{r} \gets h(z) - h(z_r) - \inner{q_r - \nabla f(z)}{z - z_r}$}
		\StateEq{\Return{$(z_{r}, q_{r}, v_{r}, \varepsilon_{r})$}}
	\EndProcedure
\end{smalgorithmic}
\end{mdframed}

The result below, whose proof can be found in \prettyref{app:ref_props},
presents some important properties about the CR procedure (CRP).
\begin{prop}
\label{prop:crp_props}Let $(z_{r},q_{r},v_{r},\varepsilon_{r})$
and $L_{\lam}$ be generated by the CRP where $(f,h)$ satisfy assumptions
\ref{asmp:nco1}--\ref{asmp:nco2}. Then, the following statements
hold:
\begin{itemize}
\item[(a)] $q_{r}\in\nabla f(z)+\pt_{\varepsilon_{r}}h(z)$ and $\varepsilon_{r}\geq0$;
\item[(b)] $v_{r}\in\nabla f(z_{r})+\pt h(z_{r})$ and 
\[
(f+h)(z)-(f+h)(z_{r})\geq\frac{L_{\lam}}{2}\|z-z_{r}\|^{2};
\]
\item[(c)] if the inputs $f$, $h$, $\lam$, and $z$ satisfy
\begin{equation}
\begin{gathered}\tilde{v}\in\pt_{\tilde{\varepsilon}}\left(\lam\left[f+h\right]+\frac{1}{2}\|\cdot-z^{-}\|^{2}\right)(z),\\
\frac{1}{\lam}\|z^{-}-z+\tilde{v}\|\leq\bar{\rho},\quad\frac{1}{\lam}\tilde{\varepsilon}\leq\bar{\varepsilon},
\end{gathered}
\label{eq:rho_eps_approx}
\end{equation}
for some $(\bar{\rho},\bar{\varepsilon})\in\r_{++}^{2}$ and $(z^{-},\tilde{v},\tilde{\varepsilon})\in{\cal Z}\times{\cal Z}\times\r_{+}$,
then
\begin{equation}
\|v_{r}\|\leq\left(1+\frac{\max\{m,M\}}{L_{\lam}}\right)\|q_{r}\|,\quad\|q_{r}\|\leq\bar{\rho}+\sqrt{2\bar{\varepsilon}L_{\lam}}.\label{eq:cref_resid_bds}
\end{equation}
\end{itemize}
\end{prop}

The above proposition shows that if $(\tilde{v},\tilde{\varepsilon},z,z^{-})$
satisfies the inclusion in \eqref{eq:rho_eps_approx} and the residuals
$\tilde{\varepsilon}/\lam$ and $\|z^{-}-z+\tilde{v}\|/\lam$ are
sufficiently small enough relative to some tolerance $\hat{\rho}$,
then the CRP generates a pair $(\hat{z},\hat{v})$ that solves \prettyref{prb:approx_nco}.
Since, \prettyref{cor:rasc1} shows that instances of the GIPPF are
able to send the aforementioned residuals to zero along some subsequence,
one approach is to iterate an instance of the GIPPF and check if the
output of a call to the CRP, as above, satisfies \eqref{eq:rho_approx_nco}.
The AIPPM is essentially one method that implements this approach.

\subsection{Instances of the GIPPF}

In this subsection, we briefly discuss some specific instances of
the GIPPF. 

Recall that, for given stepsize $\lambda>0$ and initial point $z_{0}\in Z$,
the CGM in \prettyref{alg:cgm} for solving \ref{prb:eq:nco} recursively
computes a sequence $\{z_{k}\}_{k\geq1}$ given by 
\begin{equation}
z_{k}=\argmin_{u\in{\cal Z}}\left\{ \lam\left[\ell_{g}(u;z_{k-1})+h(u)\right]+\frac{1}{2}\left\Vert z-z_{k-1}\right\Vert ^{2}\right\} .\label{eq:gradientmethod}
\end{equation}
Note that if $h$ is the indicator function of a closed convex set
then the above scheme reduces to the classical projected gradient
method. 

The following result, whose proof can be found in \prettyref{app:prox_vartn_props},
shows that the CGM with $\lam$ sufficiently small is a special case
of the GIPPF in which $\lambda_{k}=\lambda$ for all $k\geq1$.
\begin{prop}
\label{prop:gradient method}Let $\{z_{k}\}_{k\geq1}$ be generated
by the CGM with $\lam_{k}=\lambda\leq1/m$ and $\lambda<2/M$ for
every $k\geq1$, and define $\tilde{v}_{k}:=z_{k-1}-z_{k}$ and
\begin{equation}
\tilde{\varepsilon}_{k}:=\lambda\left[g(z_{k})-\ell_{g}(z_{k};z_{k-1})+\frac{1}{2\lambda}\|z_{k}-z_{k-1}\|^{2}\right].\label{eq:statCGM}
\end{equation}
Then, for every $k\geq1$, the quadruple $(\lambda_{k},z_{k},\tilde{v}_{k},\tilde{\varepsilon}_{k})$
satisfies the inclusion \eqref{eq:GIPPF} with $\phi=g+h$, and the
relative error condition \eqref{eq:err_crit_GIPP} with $\sigma:=(\lambda M+2)/4$.
Thus, the CGM can be seen as an instance of the GIPPF. 
\end{prop}

Under the assumption that $\lambda<2/M$ and $g\in{\cal C}_{M}({\cal Z})$,
it is well-known that the CGM solves \prettyref{prb:approx_nco} in
$\mathcal{O}([\phi(z_{0})-\phi_{*}]/[\lambda\hat{\rho}^{2}])$ iterations.
On the other hand, under the assumption that $\lambda\le1/M$ and
$g\in{\cal C}_{M}({\cal Z})$, we can easily see that the above result
together with \prettyref{cor:rasc1}(b) imply that the CGM solves
\prettyref{prb:approx_nco} in $\mathcal{O}(R_{\lam}\phi(z_{0})/[\lambda^{2}\hat{\rho}^{2}])$
iterations.

We now make a few general remarks about our discussion in this subsection
so far. First, the condition on the stepsize $\lambda$ of \prettyref{prop:gradient method}
forces it to be ${\cal O}(1/M)$ and hence quite small whenever $M\gg m$.
Second, \prettyref{cor:rasc1}(b) implies that the larger $\lambda$
is, the smaller the complexity bound \eqref{eq:complex_gipp} becomes.
Third, letting $\lambda_{k}=\lam$ in the GIPPF for some $\lambda\le1/m$
guarantees that the function $\lambda_{k}\phi+\|\cdot-z_{k-1}\|^{2}/2$
that appears in \eqref{eq:GIPPF} is convex.

In the remaining part of this subsection, we briefly outline the ideas
behind an accelerated instance of the GIPPF which chooses $\lambda=\mathcal{O}(1/m)$.
First, note that when $\sigma=0$, \eqref{eq:GIPPF} and \eqref{eq:err_crit_GIPP}
imply that $(\tilde{v}_{k},\tilde{\varepsilon}_{k})=(0,0)$ and 
\begin{equation}
0\in\partial\left(\lambda_{k}\phi+\frac{1}{2}\|\cdot-z_{k-1}\|^{2}\right)(z_{k}).\label{eq:inclusion_GIPPF'}
\end{equation}
and hence that $z_{k}$ is an optimal solution of the prox-subproblem
\begin{equation}
z_{k}=\argmin_{z\in{\cal Z}}\left\{ \lambda_{k}\phi(z)+\frac{1}{2}\left\Vert z-z_{k-1}\right\Vert ^{2}\right\} .\label{eq:prox_sub}
\end{equation}
More generally, assuming that \eqref{eq:err_crit_GIPP} holds for
some $\sigma>0$ gives us an interpretation of $z_{k}$, together
with $(\tilde{v}_{k},\tilde{\varepsilon}_{k})$, as being an approximate
solution of \eqref{eq:prox_sub} where its (relative) accuracy is
measured by the $\sigma$-criterion \eqref{eq:err_crit_GIPP}. Obtaining
such an approximate solution is generally difficult unless the objective
function of the prox-subproblem \eqref{eq:prox_sub} is convex. This
suggests choosing $\lambda_{k}=\lambda$ for some $\lam\le1/m$ which,
according to a remark in the previous paragraph, ensures that $\lambda_{k}\phi+(1/2)\|\cdot\|^{2}$
is convex for every $k$, and then applying an ACGM, e.g. \prettyref{alg:acgm},
to the (convex) prox-subproblem \eqref{eq:prox_sub} to obtain $z_{k}$
and a certificate pair $(\tilde{v}_{k},\tilde{\varepsilon}_{k})$
satisfying \eqref{eq:err_crit_GIPP}. An accelerated prox-instance
of the GIPPF obtained in this manner will be the subject of \prettyref{sec:aipp}.

\section{Accelerated Inexact Proximal Point (AIPP) Method}

\label{sec:aipp}

The main goal of this section is to present another instance of the
GIPPF where the triples $(z_{k},\tilde{v}_{k},\tilde{\varepsilon}_{k})$
are obtained by applying an ACGM, e.g. \prettyref{alg:acgm}, to the
subproblem \eqref{eq:prox_sub}. It contains two subsections. The
first one discusses some new results of the ACGM which will be useful
in the analysis of the accelerated GIPP instance. The second one presents
the accelerated GIPP instance for solving \ref{prb:eq:nco} and derives
its corresponding iteration complexity bound.

\subsection{Key Properties of the ACGM}

The main role of the ACGM is to find an approximate solution $z_{k}$
of subproblem \eqref{eq:GIPPF} together with a certificate pair $(\tilde{v}_{k},\tilde{\varepsilon}_{k})$
satisfying \eqref{eq:GIPPF} and \eqref{eq:err_crit_GIPP}. Indeed,
since \eqref{eq:prox_sub} is a special case of \ref{prb:eq:co},
we can apply the ACGM (see \prettyref{alg:acgm}) with $x_{0}=z_{k-1}$
to obtain the triple $(z_{k},\tilde{v}_{k},\tilde{\varepsilon}_{k})$
satisfying \eqref{eq:GIPPF} and \eqref{eq:err_crit_GIPP}.

The following result analyzes the iteration complexity of computing
the aforementioned triple.
\begin{lem}
\label{lem:nest_complex} Let $\{(A_{j},y_{j},r_{j},\eta_{j})\}_{j\geq1}$
be the sequence generated by the ACGM applied to \ref{prb:eq:co},
where: 
\begin{itemize}
\item[(i)] $\psi_{n}\in\cConv({\cal Z})$ and $\psi_{s}\in{\cal F}_{\mu,L}(\dom\psi_{n})$
for some $L>0$ and $\mu\geq0$; 
\item[(ii)] $\lam_{k}=1/L$ for every $k\geq1$. 
\end{itemize}
Then, for any $\sigma>0$ and index $j$ such that $A_{j}\geq2(1+\sqrt{\sigma})^{2}/\sigma$,
we have

\begin{equation}
\|r_{j}\|^{2}+2\eta_{j}\le\sigma\|y_{0}-y_{j}+r_{j}\|^{2}.\label{ineq:Nest_vksigma}
\end{equation}
As a consequence, the ACGM obtains a triple $(y,r,\eta)=(y_{j},r_{j},\eta_{j})$
satisfying 
\[
r\in\partial_{\eta}(\psi_{s}+\psi_{n})(y)\quad\|r\|^{2}+2\eta\le\sigma\|y_{0}-y+r\|^{2}
\]
in at most 
\[
\min\left\{ \left\lceil \frac{2\sqrt{2L}(1+\sqrt{\sigma})}{\sqrt{\sigma}}\right\rceil ,\left\lceil 1+\sqrt{\frac{2L}{\mu}}\log_{1}^{+}\left(\frac{2L\left[1+\sqrt{\sigma}\right]^{2}}{\sigma}\right)\right\rceil \right\} 
\]
iterations, where $\log_{1}^{+}(\cdot):=\max\{\log(\cdot),1\}$.
\end{lem}

\begin{proof}
See \prettyref{app:acgm_vartn_props}.
\end{proof}
Note that the above lemma holds for any $\mu\ge0$. On the other hand,
the next two results hold only for $\mu>0$ and derive some important
relations satisfied by two distinct iterates of the ACGM.
\begin{lem}
\label{lem:Nest_prefut}Let $\{(A_{j},y_{j},r_{j},\eta_{j})\}_{j\geq1}$
and $(\psi_{s},\psi_{n})$ be as in \prettyref{lem:nest_complex}
with $\mu>0$. Then, 
\begin{equation}
\left(1-\left[A_{i}\mu\right]^{-1/2}\right)\|y^{*}-y_{0}\|\leq\|y_{j}-y_{0}\|\leq\left(1+\left[A_{j}\mu\right]^{-1/2}\right)\|y^{*}-y_{0}\|\qquad\forall j\geq1,\label{ineq:Nestmu}
\end{equation}
where $y^{*}$ is the unique solution of \ref{prb:eq:co}. As a consequence,
for all indices $i,j\ge1$ such that $A_{i}\mu>1$, we have
\begin{equation}
\|y_{j}-x_{0}\|\leq\left(\frac{1+\left[A_{j}\mu\right]^{-1/2}}{1-\left[A_{i}\mu\right]^{-1/2}}\right)\|x_{i}-x_{0}\|.\label{eq:prefut}
\end{equation}
\end{lem}

\begin{proof}
First note our assumption on $\psi_{s}$ combined with \ref{prb:eq:co}
imply that $\psi\in{\cal F}_{\mu}(Z)$. Hence, it follows from \prettyref{prop:acgm_vartn}(d)
that 
\[
\frac{\mu}{2}\|y_{j}-y^{*}\|^{2}\leq\psi(y_{j})-\psi(y^{*})\le\frac{1}{2A_{j}}\|y^{*}-y_{0}\|^{2}
\]
and hence that 
\begin{equation}
\|y_{j}-y^{*}\|\leq\frac{1}{\sqrt{A_{j}\mu}}\|y^{*}-y_{0}\|.\label{eq:auxnest00}
\end{equation}
The inequalities 
\[
\|y^{*}-x_{0}\|-\|y_{j}-y^{*}\|\leq\|y_{j}-y_{0}\|\leq\|y_{j}-y^{*}\|+\|y^{*}-y_{0}\|,
\]
which are due to the triangle inequality, together with \eqref{eq:auxnest00}
clearly imply \eqref{ineq:Nestmu}. The last statement of the lemma
follows immediately from \eqref{ineq:Nestmu}. 
\end{proof}
As a consequence of \prettyref{lem:Nest_prefut}, the following result
obtains several important relations on certain quantities corresponding
to two arbitrary iterates of the ACGM.
\begin{lem}
\label{lem:Nestxjxi}Let $\{(A_{j},y_{j},r_{j},\eta_{j})\}_{j\geq1}$
and $(\psi_{s},\psi_{n})$ be as in \prettyref{lem:nest_complex}
with $\mu>0$. Let $i$ be an index such that $A_{i}\geq\max\{8,9/\mu\}$.
Then, for every $j\geq i$, we have 
\begin{equation}
\|y_{j}-y_{0}\|\leq2\|y_{i}-y_{0}\|,\quad\|r_{j}\|\leq\frac{4}{A_{j}}\|y_{i}-y_{0}\|,\quad\eta_{j}\leq\frac{2}{A_{j}}\|y_{i}-y_{0}\|^{2},\label{eq:lemNestaaa}
\end{equation}
\begin{equation}
\|y_{0}-y_{j}+r_{j}\|\leq\left(4+\frac{8}{A_{j}}\right)\|y_{0}-y_{i}+r_{i}\|,\qquad\eta_{j}\leq\frac{1}{A_{j}}8\|y_{0}-y_{i}+r_{i}\|^{2}.\label{eq:corNest01}
\end{equation}
\end{lem}

\begin{proof}
The first inequality in \eqref{eq:lemNestaaa} follows from \eqref{eq:prefut}
and the assumption that $A_{i}\mu\geq9$. Now, using \prettyref{prop:acgm_vartn}(b)
and the triangle inequality for norms, we easily see that 
\[
\|r_{j}\|\leq\frac{2}{A_{j}}\|y_{j}-y_{0}\|,\quad\eta_{j}\leq\frac{1}{2A_{j}}\|y_{j}-y_{0}\|^{2}
\]
which, combined with the first inequality in \eqref{eq:lemNestaaa},
prove the second and the third inequalities in \eqref{eq:lemNestaaa}.
Noting that $A_{i}\geq8$ by assumption, \prettyref{lem:nest_complex}
implies that \eqref{ineq:Nest_vksigma} holds with $\sigma=1$ and
$j=i$, and hence that 
\begin{equation}
\|r_{i}\|\leq\|y_{0}-y_{i}+r_{i}\|.\label{eq:lemNestbbb}
\end{equation}
Using the triangle inequality, the first two inequalities in \eqref{eq:lemNestaaa}
and relation \eqref{eq:lemNestbbb}, we conclude that 
\begin{align*}
\|y_{0}-y_{j}+r_{j}\| & \leq\|y_{0}-y_{j}\|+\|r_{j}\|\leq\left(2+\frac{4}{A_{j}}\right)\|y_{0}-y_{i}\|\\
 & \leq\left(2+\frac{4}{A_{j}}\right)\left(\|y_{0}-y_{i}+r_{i}\|+\|r_{i}\|\right)\leq\left(4+\frac{8}{A_{j}}\right)\|y_{0}-y_{i}+r_{i}\|,
\end{align*}
and that the first inequality in \eqref{eq:corNest01} holds. Now,
the last inequality in \eqref{eq:lemNestaaa}, combined with the triangle
inequality for norms and the relation $\|a+b\|^{2}\leq2\|a\|^{2}+2\|b\|^{2}$
for every $a,b\in{\cal Z}$, imply that
\[
\eta_{j}\leq\frac{2}{A_{j}}\|y_{0}-y_{i}\|^{2}\leq\frac{4}{A_{j}}\left(\|y_{0}-y_{i}+r_{i}\|^{2}+\|r_{i}\|^{2}\right).
\]
Hence, in view of \eqref{eq:lemNestbbb}, the last inequality in \eqref{eq:corNest01}
follows.
\end{proof}

\subsection{Statement and Properties of the AIPPM}

\label{subsec:AIPPmet} 

This subsection presents and analyzes the AIPPM for solving \prettyref{prb:approx_nco}.
The main results of this subsection are \prettyref{thm:AIPPcomplexity}
and \prettyref{cor:spec_aipp_compl} which give the iteration complexity
of the AIPPM.

In order to state the method, we first state two ACG instances in
\prettyref{alg:aipp_acgm} that use terminations which are related
to \eqref{eq:err_crit_GIPP}.

\begin{mdframed}
\mdalgcaption{ACG Instances for the AIPPM}{alg:aipp_acgm}
\begin{smalgorithmic}
	\Require{$\sigma \geq 0, \enskip (\mu,L)\in\r_{++}^2, \enskip \psi_n \in \cConv({\cal Z}), \enskip \psi_n \in {\cal F}_{\mu,L}(Z), \enskip y_0 \in Z$;}
	\vspace*{.5em}
	\Procedure{ACG1}{$\psi_s, \psi_n, y_0, \sigma, \mu, L$}
	\For{$k=1,...$}
		\StateEq{$\lam_k \gets 1/L$}
		\StateEq{Generate $(A_k, y_k, r_k, \eta_k)$ according to \prettyref{alg:acgm}.}
		\If{$\|r_k\|^2 + 2\eta_k \leq \sigma \|y_0 - y_k + r_k\|^2$ \textbf{and} $A_{k}\geq\max\left\{8,9/\mu\right\}$}
			\StateEq{\Return{$(y_k, r_k)$}}
		\EndIf
	\EndFor
	\EndProcedure
\end{smalgorithmic}
\vspace*{2em}
\begin{smalgorithmic}
	\Require{$(\bar{\eta},\sigma) \in \r_{+}^2, \enskip (\mu,L)\in\r_{++}^2, \enskip \psi_n \in \cConv({\cal Z}), \enskip \psi_n \in {\cal F}_{\mu,L}(\dom \psi_n), \enskip y_0 \in Z$;}
	\vspace*{.5em}
	\Procedure{ACG2}{$\psi_s, \psi_n, y_0, \sigma, \bar{\eta}, \mu, L$}
	\For{$k=1,...$}
		\StateEq{$\lam_k \gets 1/L$}
		\StateEq{Generate $(y_k, r_k, \eta_k)$ according to \prettyref{alg:acgm}.}
		\If{$\|r_k\|^2 + 2\eta_k \leq \sigma \|y_0 - y_k + r_k\|^2$ \textbf{and} $\eta_k \leq \bar{\eta}$}
			\StateEq{\Return{$(y_k, r_k, \eta_k)$}}
		\EndIf
	\EndFor
	\EndProcedure
\end{smalgorithmic}
\end{mdframed}

We now state the AIPPM in \prettyref{alg:aippm}, which uses the ACGM
instances in \prettyref{alg:aipp_acgm} and the CRP in \prettyref{alg:cref}.
Given a starting point $z_{0}\in Z$ and stepsize $\lam\in(0,1/m)$,
its main idea is to repeatedly apply the ACGM at its $k^{{\rm th}}$
iteration to approximately solve the subproblem
\[
\min_{z\in Z}\left\{ \lam(f+h)(z)+\frac{1}{2}\|z-z_{k-1}\|^{2}\right\} .
\]
This process is iterated until the residuals $\|z_{k-1}-z_{k}+\tilde{v}_{k}\|/\lam$
and $\tilde{\varepsilon}_{k}$, generated by the ACG call, are sufficiently
small relative to $\hat{\rho}$. A call to the CRP is then made to
generate a pair $(\hat{z},\hat{v})$ that solves \prettyref{prb:approx_nco}.

\begin{mdframed}
\mdalgcaption{AIPP Method}{alg:aippm}
\begin{smalgorithmic}
	\Require{$\hat{\rho} > 0, \enskip \sigma \in (0,1), \enskip (m,M)\in\r_{+}^2, \enskip h \in \cConv(Z), \enskip f \in {\cal C}_{m,M}(Z), \enskip \lam \in (0, 1/m), \enskip z_0 \in Z$;}
	\Initialize{$\mu \gets 1 - \lam m, \enskip L \gets 1 +\lam M, \enskip \bar{\rho} \gets \hat{\rho}/4, \enskip \bar{\varepsilon} \gets \hat{\rho}^2/(32[\max\{m,M\}+\lam^{-1}]);$}
	\vspace*{.5em}
	\Procedure{AIPP}{$f, h, z_0, \lam, m, M, \sigma, \hat{\rho}$}
	\For{$k=1,...$}
		\StateStep{\algpart{1}\textbf{Attack} the $k^{\rm th}$ prox subproblem.}
		\StateEq{$\psi_s^k \Lleftarrow \lam f + \|\cdot - z_{k-1}\|^2 / 2$}
		\StateEq{$(z_k, \tilde{v}_k, \tilde{\varepsilon}_k) \gets \text{ACG1}(\psi_s^k, \lam h, z_{k-1}, \sigma, \mu, L)$} \label{ln:aipp_acgm1}
		\If{$\|z_{k-1}-z_k + \tilde{v}_k\| \leq \lam\bar{\rho}/{5}$} \label{ln:aipp_stop}
			\StateStep{\algpart{2}\textbf{Attack} the last prox subproblem.}
			\StateEq{$(z,\tilde{v},\tilde{\varepsilon}) \gets \text{ACG2}(\psi_s^k, \psi_n^k, z_{k-1}, \sigma, \lam\bar{\varepsilon}, \mu, L)$} \label{ln:aipp_acgm2}
			\StateEq{$(\hat{z}, \hat{q}, \hat{v}, \hat{\varepsilon}) \gets \text{CREF}(f, h, z, \max\{m,M\}, \lam)$} \label{ln:aipp_cref}
			\StateEq{\Return{$(\hat{z}, \hat{v})$}}
		\EndIf
	\EndFor
	\EndProcedure
\end{smalgorithmic}
\end{mdframed}

Some comments about the AIPPM are in order. To ease the discussion,
let us refer to the ACG iterations performed in \prettyref{ln:aipp_acgm1}
and \prettyref{ln:aipp_acgm2} of the method as \textbf{inner iterations}
and the iterations over the indices $k$ as\emph{ }\textbf{outer iterations}.
First, in view of the last statement of \prettyref{lem:nest_complex}
and the termination conditions given in \prettyref{alg:aipp_acgm},
each ACGM call always stops and outputs a triple $(z,\tilde{v},\tilde{\varepsilon})$
satisfying 
\begin{equation}
\tilde{v}\in\pt_{\tilde{\varepsilon}}\left(\lam\left[f+h\right]+\frac{1}{2}\|\cdot-z_{k-1}\|^{2}\right)(z),\quad\|\tilde{v}\|^{2}+2\tilde{\varepsilon}\leq\sigma\|z_{k-1}-z+\tilde{v}\|^{2}\label{eq:aippm_hpe}
\end{equation}
at the $k^{{\rm th}}$ outer iteration. Second, in view of the first
comment, the outer iterations can be viewed as iterations of the GIPPF
applied to \ref{prb:eq:nco}. Finally, the goal of the ACGM call in
\prettyref{ln:aipp_acgm2} is to obtain a triple $(z,\tilde{v},\tilde{\varepsilon})$
with a possibly smaller $\tilde{\varepsilon}$ while preserving the
quality of the quantity $\|z_{k-1}-\tilde{z}+\tilde{v}\|/\lam$, which
at its start is bounded by $(\lambda\bar{\rho})/5$ and, throughout
its inner iterations, can be shown to be bounded by $\lambda\bar{\rho}$
(see \eqref{eq:lemAIPPestimates}).

The next proposition summarizes some basic facts about the AIPPM.
\begin{lem}
\label{lem:AIPPmethod}Let $(\bar{\rho},\bar{\varepsilon})$ be as
in the initialization phase of the AIPPM. Then, the following statements
about the AIPPM hold:
\end{lem}

\begin{itemize}
\item[(a)] at each outer iteration, its call to the ACGM in \prettyref{ln:aipp_acgm1}
stops and finds a triple $(z,\tilde{v},\tilde{\varepsilon})$ satisfying
\eqref{eq:aippm_hpe} in at most 
\begin{equation}
k_{I}:=\left\lceil \max\left\{ \frac{2\sqrt{2}(1+\sqrt{\sigma})}{\sqrt{\sigma}},\frac{6}{\sqrt{1-\lam m}}\right\} \sqrt{1+\lam M}\right\rceil \label{eq:AIPPcomplexinner}
\end{equation}
inner iterations;
\item[(b)] its last call to the ACGM in \prettyref{ln:aipp_acgm2} stops with
an output triple $(z,\tilde{v},\tilde{\varepsilon})$ satisfying
\begin{equation}
\tilde{v}\in\pt_{\tilde{\varepsilon}}\left(\lam\phi+\frac{1}{2}\|\cdot-z_{k-1}\|^{2}\right)(z),\quad\frac{1}{\lam}\|z_{k-1}-z+\tilde{v}\|\leq\bar{\rho},\quad\tilde{\varepsilon}\leq\lam\bar{\varepsilon}\label{eq:aipp_phase2_prop}
\end{equation}
 in at most
\begin{equation}
k_{L}:=\left\lceil 2\sqrt{2\left(\frac{1+\lambda M}{1-\lam m}\right)}\ \log_{1}^{+}\left(\frac{2\bar{\rho}\sqrt{2(\lambda M+1)\lambda}}{5\sqrt{\bar{\varepsilon}}}\right)+1\right\rceil \label{eq:lemAIPPcomplex3}
\end{equation}
inner iterations, where $\log_{1}^{+}(\cdot):=\max\{\log(\cdot),1\}$;
\item[(c)] it is a special implementation of the GIPPF in which $\lambda_{k}=\lambda$
for every $k\ge1$; 
\item[(d)] it stops with an output pair $(\hat{z},\hat{v})$ that solves \prettyref{prb:approx_nco}
in at most
\begin{equation}
k_{O}:=\left\lceil \frac{25R_{\lam}\phi(z_{0})}{(1-\sigma)^{2}\lambda^{2}\bar{\rho}^{2}}+1\right\rceil \label{def:T_p}
\end{equation}
outer iterations, where $R_{\lam}\phi(\cdot)$ is as defined in \eqref{eq:Rphi_def};
\item[(e)] for every $k\geq1$, its sequence of iterates $\{z_{k}\}_{k\geq1}$
and output point $\hat{z}$ satisfy $\phi(z_{1})\geq\phi(z_{k})\geq\phi(\hat{z})$.
\end{itemize}
\begin{proof}
All line numbers referenced in this proof are with respect to the
AIPPM in \prettyref{alg:aippm}. Moreover, let $(\mu,L)$ be as in
the initialization phase of the AIPPM.

(a) In view of assumptions \ref{asmp:nco1}--\ref{asmp:nco2}, it
holds that for every $k\geq1$ we have $\psi_{s}^{k}\in{\cal F}_{\mu,L}(Z)$
and $\psi_{n}^{k}\in\cConv(Z)$. Hence, it follows from the last statement
of \prettyref{lem:nest_complex} and the definition of $L$ that the
ACGM obtains a triple $(z,\tilde{v},\tilde{\varepsilon})$ satisfying
\eqref{eq:aippm_hpe} in at most 
\[
\left\lceil \left(\frac{2\sqrt{2}[1+\sqrt{\sigma}]}{\sqrt{\sigma}}\right)\sqrt{1+\lam M}\right\rceil 
\]
inner iterations. On the other hand, in view of \prettyref{prop:acgm_vartn}(c)
with $\lam_{i}=1/L$ for every $i\geq1$ and the definitions of $\mu$
and $L$, the condition $A_{k}\geq\max\{8,9/\mu\}$ requires at most
\[
\left\lceil \left(\frac{6}{\sqrt{1-\lam m}}\right)\sqrt{1+\lam M}\right\rceil 
\]
inner iterations. Combining the previous two inner iteration bounds
yields the desired conclusion.

(b) Consider the triple $(z_{k},\tilde{v}_{k},\tilde{\varepsilon}_{k})$
obtained in the last call to \prettyref{ln:aipp_acgm1}. In view of
the termination criteria in this call, there exists an index $k\ge1$
such that $(z_{k},\tilde{v}_{k},\tilde{\varepsilon}_{k})$ is the
$j^{{\rm th}}$ iterate of the ACGM started from $y_{0}=z_{k-1}$
with $A_{j}\geq\max\{8,9/\mu\}$, and hence, the index $j$ satisfies
the assumption of \prettyref{lem:Nestxjxi}. It then follows from
\eqref{eq:corNest01}, \prettyref{ln:aipp_stop}, the first remark
following the AIPPM, and \prettyref{prop:acgm_vartn}(c) with $\lam_{i}=1/L$
for every $i\geq1$, that the call to the ACGM in \prettyref{ln:aipp_acgm2}
stops and outputs a triple $(z,\tilde{v},\tilde{\varepsilon})$ satisfying
the inclusion in \eqref{eq:aipp_phase2_prop}, the bound
\begin{equation}
\frac{1}{\lam}\|z_{k-1}-z+\tilde{v}\|\leq\left(4+\frac{8}{A_{j}}\right)\frac{\bar{\rho}}{5}\leq\bar{\rho},\label{eq:lemAIPPestimates}
\end{equation}
and the bound
\begin{equation}
\tilde{\varepsilon}\leq\frac{8\lambda^{2}\bar{\rho}^{2}}{25A_{j}}\leq\frac{8L\lambda^{2}\bar{\rho}^{2}}{25}\left(1+\sqrt{\frac{\mu}{2L}}\right)^{-2(j-1)}.\label{eq:tilde_eps_bd}
\end{equation}
Using the stopping criterion for the ACGM instance in \prettyref{ln:aipp_acgm2},
the inequality for $\tilde{\varepsilon}$ above, the definitions of
$\mu$ and $L$, and the relation that $\log(1+t)\geq t/2$ for all
$t\in[0,1]$, we can easily see that $\tilde{\varepsilon}\leq\lam\bar{\varepsilon}$
and (b) holds.

(c) This statement is obvious.

(d) The bound on the number of outer iterations follows by combining
(c), the stopping criterion in \prettyref{ln:aipp_stop}, and \prettyref{cor:rasc1}(b)
with $\bar{\rho}$ replaced by $\bar{\rho}/5$. 

To show that the output pair $(\hat{z},\hat{v})$ solves \prettyref{prb:approx_nco},
we first note that part (b) implies that the output $(z,\tilde{v},\tilde{\varepsilon})$
of \prettyref{ln:aipp_acgm2} satisfies \eqref{eq:rho_eps_approx}
with $z^{-}=z_{k-1}$. It now follows from the call to the refinement
procedure in \prettyref{ln:aipp_cref}, \prettyref{prop:crp_props}(b)--(c)
with $(z_{r},v_{r},z^{-})=(\hat{z},\hat{v},z_{k-1})$, and the definitions
of $\bar{\rho}$ and $\bar{\varepsilon}$, that $\hat{v}\in\nabla f(\hat{z})+\pt h(\hat{z})$
and 
\[
\|\hat{v}\|\leq2\left[\bar{\rho}+\sqrt{2\bar{\varepsilon}(\max\{m,M\}+\lam^{-1})}\right]\leq2\left[\frac{\hat{\rho}}{4}+\frac{\hat{\rho}}{4}\right]\leq\hat{\rho},
\]
which is exactly \eqref{eq:rho_approx_nco}.

(e) This follows from \prettyref{ln:aipp_cref}, \prettyref{lem:phik_d0},
and \prettyref{prop:crp_props}(b) with $z_{r}=\hat{z}$.
\end{proof}
We now state one of our main results of this chapter, which is the
iteration complexity of the AIPPM for solving \prettyref{prb:approx_nco}.
Recall that the AIPPM assumes that $\lambda<1/m$.
\begin{thm}
\label{thm:AIPPcomplexity}The AIPPM outputs a pair $(\hat{z},\hat{v})$
that solves \prettyref{prb:approx_nco} in
\begin{equation}
\mathcal{O}\left(\sqrt{\frac{\lam M+1}{\min\left\{ \sigma,1-\lam m\right\} }}\ \left[\frac{R_{\lam}\phi(z_{0})}{(1-\sigma)^{2}\lam^{2}\hat{\rho}^{2}}+\log_{1}^{+}\left(\lambda M\right)\right]\right)\label{eq:lastboundAIPP}
\end{equation}
inner iterations, where $R_{\lam}\phi(\cdot)$ is as in \eqref{eq:Rphi_def}
and $\log_{1}^{+}(\cdot):=\max\{\log(\cdot),1\}$.
\end{thm}

\begin{proof}
First, note that the total number of inner iterations in a call of
the AIPPM is $k_{T}:=k_{I}k_{O}+k_{L}$ --- where $k_{I}$, $k_{O}$,
and $k_{L}$ are as in \prettyref{lem:AIPPmethod}(a), (d), and (b),
respectively. Using the fact that $\lam<1/m$, and hence $\log_{1}^{+}(\lam\max\{m,M\})={\cal O}(\log_{1}^{+}(\lam M))$,
it is straightforward to verify that $k_{T}$ is on the same order
of magnitude as in \eqref{eq:lastboundAIPP}. The fact that $(\hat{z},\hat{v})$
solves \prettyref{prb:approx_nco} follows from \prettyref{lem:AIPPmethod}(d).
\end{proof}
Note that the AIPP version in which $\lambda=1/(2m)$ and $\sigma=1/2$
yields the best complexity bound under the reasonable assumption that,
inside the squared bracket in \eqref{eq:lastboundAIPP}, the first
term is larger than the second one.

The following result describes the number of oracle calls performed
by the AIPPM with $\lambda=1/(2m)$ and $\sigma=1/2$.
\begin{cor}
\label{cor:spec_aipp_compl}The AIPPM with inputs $\lam=1/(2m)$ and
$\sigma=1/2$ outputs a pair $(\hat{z},\hat{v})$ that solves \prettyref{prb:approx_nco}
in
\begin{equation}
\mathcal{O}\left(\sqrt{\frac{M}{m}+1}\ \left[\frac{m^{2}R_{1/(2m)}\phi(z_{0})}{\hat{\rho}^{2}}+\log_{1}^{+}\left(\frac{M}{m}\right)\right]\right)\label{eq:auxcomplex00}
\end{equation}
oracle calls, where $R_{\lam}\phi(\cdot)$ is as in \eqref{eq:Rphi_def}
and $\log_{1}^{+}(\cdot):=\max\{\log(\cdot),1\}$.
\end{cor}

\begin{proof}
This is immediate from \prettyref{thm:AIPPcomplexity}, the definition
of $\log_{1}^{+}(\cdot)$, and the fact that the ACGM uses ${\cal O}(1)$
oracle calls per iteration.
\end{proof}
We now make a few remarks about the iteration complexity bound \eqref{eq:auxcomplex00}
and its relationship to two other ones obtained in the literature
under assumption that: (i) $m\leq M$; and (ii) the term $\mathcal{O}(1/\hat{\rho}^{2})$
in \eqref{eq:auxcomplex00} dominates the other one. First, using
the definition of $R_{\lam}\phi(z_{0})$ and the first assumption,
it is easy to see that the complexity bound \eqref{eq:auxcomplex00}
is majorized by
\begin{equation}
{\cal O}\left(\frac{\sqrt{mM}}{\hat{\rho}^{2}}\min\left\{ \phi(z_{0})-\phi_{*},md_{0}^{2}\right\} \right)\label{eq:aipp_reduced_compl}
\end{equation}
where $d_{0}$ is as in \eqref{eq:dist0}. Second, since the iteration
complexity bound for the CGM with $\lambda=1/M$ is ${\cal O}(M[\phi(z_{0})-\phi_{*}]/\hat{\rho}^{2})$
(see the discussion following \prettyref{prop:gradient method}),
we conclude that \eqref{eq:aipp_reduced_compl}, and hence \eqref{eq:auxcomplex00},
is better than the CGM bound by a factor of $\sqrt{M/m}$. Third,
bound \eqref{eq:aipp_reduced_compl}, and hence \eqref{eq:auxcomplex00},
is also better than the one established in \citep[Corollary 2]{Ghadimi2016}
for an ACGM applied directly to \ref{prb:eq:nco} by at least a factor
of $\sqrt{M/m}$. Note that the accelerated method of \citep{Ghadimi2016}
assumes that the diameter of $Z$ is bounded while the AIPPM does
not.

\subsection{Lower Complexity Bounds}

\label{subsec:compl_bds}

Lower complexity bounds have recently been established in \citep{Zhou2019}
for the complexity of finding solutions of \prettyref{prb:approx_nco}.
The result below gives its precise statement.
\begin{thm}
Consider any algorithm ${\cal A}$ that solves \prettyref{prb:approx_nco}
under assumptions \ref{asmp:nco1}--\ref{asmp:nco3} and the assumption
that $h\equiv0$. For an initial point $z_{0}\in Z$, if the iterates
$\{z_{k}\}_{k\geq1}$ generated by ${\cal A}$ satisfy 
\begin{equation}
z_{k}\in{\rm Lin}\left\{ z_{0},...,z_{k-1},\nabla f(z_{0}),...,\nabla f(z_{k})\right\} \quad\forall k\geq1\label{eq:lin_span_req}
\end{equation}
where ${\rm Lin}\ S$ denotes the linear span of a set of elements
$S$, then ${\cal A}$ requires 
\begin{equation}
\Omega\left(\frac{\sqrt{mM}\left[\phi(z_{0})-\phi_{*}\right]}{\hat{\rho}^{2}}\right)\label{eq:lower_compl_nco_bd}
\end{equation}
iterations to generate a solution of \prettyref{prb:approx_nco}.
\end{thm}

We now make two remarks about the above result. First, since \eqref{eq:lower_compl_nco_bd}
is a lower complexity bound for the case of $h\equiv0$ it is also
a lower complexity bound for the case of $h\in\cConv(Z)$. Second
the linear-span requirement in \eqref{eq:lin_span_req} is more restrictive
than the one considered in this chapter. Finally, in view of the remarks
following \prettyref{cor:spec_aipp_compl}, the AIPPM of this chapter
achieves the lower complexity bound \eqref{eq:lower_compl_nco_bd}
up to a multiplicative constant.

\section{Conclusion and Additional Comments}

In this chapter, we presented an accelerated inexact proximal point
method for obtaining approximate stationary points of an unconstrained
NCO problem whose objective function is the sum of two functions $h\in\cConv({\cal Z})$
and $f\in{\cal C}_{m,M}(\dom h)$ for some $(m,M)\in\r_{++}^{2}$.
The method consists of inexactly solving a sequence of proximal subproblems
using an accelerated composite gradient method. We then established
an ${\cal O}(\hat{\rho}^{-2})$ iteration complexity bound for finding
$\hat{\rho}$-approximate stationary points which was observed to
be complexity optimal in terms of $m$, $M$, and $\hat{\rho}$ for
a large class of linear-span first-order methods.

The next chapter uses the developments in this one to develop methods
for solving a class of set-constrained NCO problems.

\subsection*{Additional Comments}

We now give a few additional comments about the results in this chapter. 

First, the AIPPM improves on the complexity in \citep{Carmon2018}
by a factor of $\log(M/\rho)$. Second, the AIPPM is a variant of
the AIPP method in \citep{Kong2019}. More specifically, the AIPPM
of this chapter checks conditions \eqref{eq:GIPPF} and \eqref{eq:err_crit_GIPP}
at every inner iteration while the AIPP method in \citep{Kong2019}
merely prescribes a fixed number of inner iterations per outer iteration.

\subsection*{Future Work}

It would be worth investigating if the AIPPM also achieves the lower
complexity bound for general first-order methods which do not necessary
require condition \eqref{eq:lin_span_req}. Currently, a lower bound
\citep{Carmon2020,Carmon2021} is only known for case where $f\in{\cal C}_{L}(Z)$
for some $L>0$. Additionally, it would be interesting to see if the
behavior of the AIPPM, or a variant of its, under a stochastic oracle
(as opposed to a deterministic one). Finally, it would be worth investigating
the properties of a non-Euclidean AIPPM which is based on Bregman
distances.

\newpage{}

\chapter{Function Constrained Composite Optimization}

\label{chap:cnco}

Our main goal in this chapter is to describe and establish the iteration
complexity of two methods for finding approximate stationary points
of the function constrained NCO (CNCO) problem 
\begin{equation}
\min_{z\in{\cal Z}}\left\{ \phi(z):=f(z)+h(z):g(z)\in S\right\} \tag{\ensuremath{{\cal CNCO}}}\label{prb:eq:cnco}
\end{equation}
where ${\cal Z}$ is a finite dimensional inner product space, $h\in\cConv(Z)$
for some $Z\subseteq{\cal Z}$, $f\in{\cal C}_{m,M}(Z)$ for some
$(m,M)\in\r_{++}^{2}$, $g\in{\cal C}({\cal Z})$, and $S\subseteq{\cal R}$
is a closed convex set over some finite dimensional inner product
space ${\cal R}$. 

The first method is a \textbf{quadratic penalty} method for solved
linearly set-constrained instances of \ref{prb:eq:cnco}., i.e. $g$
is linear, whereas the second method is an \textbf{inexact proximal
augmented Lagrangian }method for solving nonlinearly cone-constrained
instances of \ref{prb:eq:cnco}, i.e. $g$ is (possibly) nonlinear
and $S$ is a closed convex cone. Throughout our presentation, it
is assumed that efficient oracles for evaluating the quantities $f(z)$,
$\nabla f(z)$, $g(z)$, $\nabla g(z)$, and $h(z)$ and for obtaining
exact solutions of the subproblems
\[
\min_{z\in{\cal Z}}\left\{ \lam h(z)+\frac{1}{2}\|z-z_{0}\|^{2}\right\} ,\quad\min_{r\in S}\|r-r_{0}\|
\]
for any $z_{0}\in{\cal Z}$, $r\in{\cal R}$, and $\lam>0$, are available.
Moreover, we define an \textbf{oracle call} to be a collection of
the above oracles of size ${\cal O}(1)$ where each of them appears
at least once.

Given tolerance pair $(\hat{\rho},\hat{\eta})\in\r_{++}^{2}$, it
is shown that both methods obtain a solution pair $([\hat{z},\hat{p}],[\hat{v},\hat{q}])$
satisfying
\begin{gather}
\begin{gathered}\hat{v}\in\nabla f(\hat{z})+\pt h(\hat{z})+\nabla g(\hat{z})\hat{p},\quad g(\hat{z})+\hat{q}\in S\\
\|\hat{v}\|\leq\hat{\rho},\quad\|\hat{q}\|\leq\hat{\eta},
\end{gathered}
\label{eq:gen_rho_eta_approx_sol}
\end{gather}
in a number of oracle calls that depends on the tolerance pair $(\hat{\rho},\hat{\eta})$.
More specifically, the quadratic penalty method obtains the above
conditions in ${\cal O}(\hat{\rho}^{-2}\hat{\eta}^{-1})$ oracle calls,
while the augmented Lagrangian method does this in ${\cal O}([\hat{\eta}^{-1/2}\hat{\rho}^{-2}+\hat{\rho}^{-3}]\log_{1}^{+}[\hat{\rho}^{-1}+\hat{\eta}^{-1}])$
oracle calls. It is worth mentioning that the no regularity conditions
are needed for the quadratic penalty method and only a Slater-like
condition is needed for the augmented Lagrangian method.

The content of this chapter is based on papers \citep{Kong2019,Kong2020b}
(joint work with Jefferson G. Melo and Renato D.C. Monteiro) and several
passages may be taken verbatim from it.

\subsection*{Related Works}

We first review methods that consider the case where $g$ is linear.
The complexity analysis of a first-order quadratic penalty method
for the case where $f$ is convex, $h$ is an indicator function,
was first given in \citep{Lan2013} and further analyzed in \citep{Aybat2009,Molinari2020,Necoara2019}.
Aside from \citep{Kong2019}, papers \citep{Kong2019a,Kong2020a,Lin2019}
are other works that establish the iteration complexity of quadratic
penalty-based methods. For the case where $S=\{b\}$, paper \citep{Jiang2019}
proposes a penalty ADMM approach which introduces an artificial variable
$y$ in \ref{prb:eq:cnco} and then penalizes $y$ to obtain the penalized
problem 
\begin{equation}
\min\left\{ f(z)+h(z)+\frac{c}{2}\|y\|^{2}:Ax+y=b\right\} ,\label{eq:penpr}
\end{equation}
which is then solved by a two-block ADMM. It is then shown in \citep[Remark 4.3]{Jiang2019}
that the overall number of composite gradient steps performed by the
aforementioned two-block ADMM penalty scheme for obtaining an approximate
stationary point as in \eqref{eq:gen_rho_eta_approx_sol} is ${\cal O}(\hat{\rho}^{-6})$
when: $\hat{\eta}=\hat{\rho}$, the level sets of $f+h$ are bounded,
and the initial triple $(z_{0},y_{0},p_{0})$ satisfies $(y_{0},p_{0})=(0,0)$,
$Az_{0}=b$, and $z_{0}\in\dom h$.

We now turn our attention to augmented Lagrangian (AL) methods that
consider general (possibly nonlinear) functions $g$. Since AL-based
methods for the convex case have been extensively studied in the literature
(see, for example, \citep{Aybat2009,Lan2013,Necoara2019,Aybat2012,Lan2016,Lu2018,Patrascu2017,Xu2019}),
we focus on papers that deal with nonconvex problems. Moreover, we
concentrate on those dealing with proximal augmented Lagrangian (PAL)
based methods, i.e. the ones for which the ``inner'' subproblems
are of (or close to) the form in \eqref{eq:approx_primal_update},
and only those that establish iteration complexities. Paper \citep{Hong2016}
studies the iteration complexity of a linearized PAL method under
the restrictive assumption that $h=0$. Paper \citep{Hajinezhad2019}
introduces a perturbed $\theta$-AL function, which agrees with the
classical one (see \eqref{eq:aug_lagr_def}) when $\theta=0$, and
studies a corresponding unaccelerated PAL method whose iteration complexity
is ${\cal O}(\hat{\eta}^{-4}+\hat{\rho}^{-4})$ under the strong condition
that the initial starting point is feasible with respect to the constraint
$g(z)\in S$. Paper \citep{Melo2020a} analyzes the iteration complexity
of an inexact proximal accelerated PAL method based on the aforementioned
perturbed AL function and shows, regardless of whether the initial
point is feasible, that an approximate stationary point as in \eqref{eq:gen_rho_eta_approx_sol}
is obtained in ${\cal O}(\hat{\eta}^{-1}\hat{\rho}^{-2}\log\hat{\eta}^{-1})$
ACG iterations and that the latter bound can be improved to ${\cal O}(\hat{\eta}^{-1/2}\hat{\rho}^{-2}\log\hat{\eta}^{-1})$
under an additional Slater-like assumption. Both papers \citep{Melo2020a,Hajinezhad2019}
assume that $\theta\in(0,1]$, and hence, their analyses do not apply
to the classical PAL method. In fact, as $\theta$ approaches zero,
the universal constants that appear in the complexity bounds obtained
in \citep{Melo2020a,Hajinezhad2019} diverge to infinity. Using a
different approach, i.e. one that does not rely on a merit function,
paper \citep{Melo2020} establishes the iteration complexity of an
accelerated PAL method based on the classical augmented Lagrangian
(see \eqref{eq:aug_lagr_def}) and Lagrange multiplier update (see
\eqref{eq:dual_update}).

For the case where $S$ is a closed convex cone $-{\cal K}$, each
component of $g$ is $\cK$-convex, and $\cK=\{0\}\times\r_{+}^{k}$,
i.e. the constraint is of the form $g(x)=0$ and/or $g(x)\le0$, papers
\citep{Li2020,Sahin2019} present PAL methods that perform Lagrange
multiplier updates only when the penalty parameter is updated. Hence,
if the penalty parameter is never updated (which usually happens when
the initial penalty parameter is chosen to be sufficiently large),
then these methods never perform Lagrange multiplier updates, and
thus they behave more like penalty methods. Paper \citep{Li2020a}
studies a hybrid penalty/augmented Lagrangian (AL) based method whose
penalty iterations are the ones which guarantee its convergence and
whose AL iterations are included with the purpose of improving its
computational efficiency. For the case where $g$ is not necessarily
$\cK$-convex and $\cK=\{0\}$, i.e. the constraint is of the form
$g(x)=0$, paper \citep{Xie2019} analyzes the complexity of a PAL
method under the strong assumption that: (i) $h=0$; (ii) the smallest
singular value of $\nabla g(x)$ is uniformly bounded away from zero
everywhere; and, optionally, (iii) the initial starting point is feasible
with respect to the constraint $g(z)\in S$.

Finally, we discuss other papers that have motivated the developments
in \citep{Kong2020b} or are tangentially related to it. Paper \citep{Boob2019}
considers a primal-dual proximal point scheme and analyzes its iteration-complexity
under strong conditions on the initial point. Papers \citep{Zhang2020,Zhang2020a}
present a primal-dual first-order algorithm for solving \ref{prb:eq:cnco}
when $h=\delta_{P}$ and $P$ is a box (in \citep{Zhang2020a}) or
more generally a polyhedron (in \citep{Zhang2020}). They also show
that the primal-dual algorithm obtains an approximate stationary point
as in \eqref{eq:gen_rho_eta_approx_sol} in ${\cal O}(\hat{\rho}^{-2})$
iterations when $\hat{\rho}=\hat{\eta}$. 

\subsection*{Organization}

This chapter contains two sections. The first one presents an accelerated
quadratic penalty method for solving linear set-constrained instances
of \ref{prb:eq:cnco}. The second one presents an accelerated augmented
Lagrangian method for solving nonlinearly cone-constrained instances
of \ref{prb:eq:cnco}. The last one gives a conclusion and some closing
comments.

\section{Composite Optimization with Linear Set Constraints}

\label{sec:qp_aipp}

The quadratic penalty method is a popular optimization method for
solving convex composite optimization problems with functional constraints
$g(x)\leq0$ where $g:\r^{n}\mapsto\r^{m}$ is convex in each of its
entries. Denoting the function
\begin{equation}
{\cal L}_{c}(x;p)=\phi(x)+\frac{1}{2c}\left[\|\max\{0,p+cg(z)\}\|^{2}-\|p\|^{2}\right]\label{eq:intro_aug_Lagr}
\end{equation}
as the augmented Lagrangian of the constrained problem $\min_{x\in\rn}\{\phi(x):g(x)\leq0\}$,
the penalty method generates iterates $\{x_{k}\}_{k\geq1}$ according
to the update
\begin{align}
x_{k} & =\argmin_{x\in\rn}{\cal L}_{c_{k}}(x;p_{k}),\label{eq:intro_qp_subprb}
\end{align}
for some sequence of penalty parameters $\{c_{k}\}_{k\geq1}$ and
multipliers $\{p_{k}\}_{k\geq1}$. For the case where $h$ is the
indicator of a closed convex set, it is known (see, for example, \citep[Proposition 4.2.1]{Bertsekas1999})
that if $0<c_{k}<c_{k+1}$ for every $k\geq1$ and $c_{k}\to\infty$
then every limit point of the sequence $\{x_{k}\}$ is a global minimum
of the constrained problem. Moreover, under some additional regularity
conditions, it can be shown (see, for example, \citep[Section 4.2.1]{Bertsekas1999})
that the sequence $\{\max\{0,p_{k}+c_{k}g(x_{k})\}\}_{k\geq1}$ converges
to a Lagrange multiplier of the constrained problem.

Our main goal in this chapter is to describe and establish the iteration
complexity of an accelerated \textbf{inexact} proximal quadratic penalty
(AIP.QP) method for finding approximate stationary points of the linearly
set-constrained NCO problem

\begin{equation}
\hat{\varphi}_{*}:=\min_{z\in{\cal Z}}\left\{ \phi(z):=f(z)+h(z):{\cal A}z\in S\right\} ,\tag{\ensuremath{{\cal CNCO}[a]}}\label{prb:eq:cnco_a}
\end{equation}
where ${\cal A}:{\cal Z}\mapsto{\cal R}$ is linear, the feasible
set is nonempty, and the functions $f$ and $h$ are as described
at the beginning of the chapter. 

The AIP.QP method (AIP.QPM) is based on the smooth quadratic penalty
function
\begin{equation}
f_{c}(z):=f(z)+\frac{c}{2}\dist^{2}({\cal A}z,S)\quad\forall z\in Z,\quad\forall c>0.\label{eq:smooth_penalty_fn}
\end{equation}
and it uses the AIPPM of \prettyref{chap:unconstr_nco} to generate
its $\ell^{{\rm th}}$ iterate: given $c_{\ell}$, find an approximate
stationary point $\hat{z}$ of the NCO problem
\begin{equation}
\hat{\varphi}_{c_{\ell}}:=\min_{z\in{\cal Z}}\left\{ \varphi_{c_{\ell}}(z):=f_{c_{\ell}}(z)+h(z)\right\} ,\label{eq:prox_penalty_subprb}
\end{equation}
and check if it is approximately feasible, i.e. $\dist({\cal A}\hat{z},S)\approx0$;
if it is not, then multiplicatively increase $c_{\ell}$ by some factor
and go the next iteration. 

For a given tolerance pair $(\hat{\rho},\hat{\eta})\in\r_{++}^{2}$
and a suitable choice of $\lam$, the main result of this chapter
shows that the AIP.QPM, started from any point $z_{0}\in Z$ obtains
a pair $([\hat{z},\hat{p}],[\hat{v},\hat{q}])$ satisfying the approximate
stationarity conditions
\begin{gather}
\hat{v}\in\nabla f(\hat{z})+\pt h(\hat{z})+{\cal A}^{*}\hat{p}\quad\|\hat{v}\|\leq\hat{\rho}\label{eq:approx_cnco_a_ln1}\\
{\cal A}\hat{z}+\hat{q}\in S\quad\|\hat{q}\|\leq\hat{\eta}\label{eq:approx_cnco_a_ln2}
\end{gather}
in at most 
\[
\mathcal{O}\left(\sqrt{\frac{\Theta_{\hat{\eta}}}{m}+1}\ \left[\frac{m\cdot\min\left\{ \hat{\varphi}_{*}-\hat{\varphi}_{\hat{c}},md_{0}^{2}\right\} }{\hat{\rho}^{2}}+\log_{1}^{+}\left(\frac{\Theta_{\hat{\eta}}}{m}\right)\right]\right)
\]
oracle calls, where $d_{0}=\min_{z\in{\cal Z}}\{\|z_{0}-z_{*}\|:\phi(z_{*})=\phi_{*}\}$,
$\log_{1}^{+}(\cdot)=\max\{1,\log(\cdot)\}$, $\Theta_{\hat{\eta}}={\cal O}(M+\|{\cal A}\|^{2}/\hat{\eta}^{2})$,
and $\hat{c}$ is a positive scalar for which $\hat{\varphi}_{\hat{c}}$
as in \eqref{eq:prox_penalty_subprb} with $c_{\ell}=\hat{c}$ is
finite. 

It is worth mentioning that this result neither assumes that $Z$
is bounded nor that \ref{prb:eq:cnco_a} has an optimal solution. 

\subsubsection*{Organization}

This section contains three subsections. The first one gives some
preliminary references and discusses our notion of a stationary point
given in \eqref{eq:approx_cnco_a_ln1} and \eqref{eq:approx_cnco_a_ln2}.
The second one presents some key properties of the penalty approach.
The last one presents the AIP.QPM and its iteration complexity. 

\subsection{Preliminaries}

This section enumerates the assumptions on problem \ref{prb:eq:cnco_a},
states the main problem of interest, and discusses the notion of an
approximate stationary point given in \eqref{eq:approx_cnco_a_ln1}
and \eqref{eq:approx_cnco_a_ln2}.

It is assumed that $\phi=f+g$ satisfies assumptions \ref{asmp:nco1}--\ref{asmp:nco2}
as well as the following assumptions: 

\stepcounter{assumption}
\begin{enumerate}
\item \label{asmp:cnco_a1}${\cal A}:{\cal Z}\mapsto{\cal R}$ is a nonzero
linear operator, $S\subseteq{\cal {\cal R}}$ is a closed convex set,
and the feasible region ${\cal F}:=\{z\in{\cal Z}:{\cal A}z\in S\}$
is nonempty; 
\item \label{asmp:cnco_a2}there exists $\hat{c}\geq0$ such that $\hat{\varphi}_{\hat{c}}>-\infty$,
where
\begin{equation}
\hat{\varphi}_{c}:=\inf_{z\in{\cal Z}}\left\{ \varphi_{c}(z):=f_{c}(z)+h(z)\right\} ,\quad\forall c\geq0,\label{eq:varphiC_def}
\end{equation}
where $f_{c}(\cdot)$ is as in \eqref{eq:smooth_penalty_fn}.
\end{enumerate}
We now make three remarks about the above assumptions. First, the
above assumptions imply that the optimal value of \ref{prb:eq:cnco_a}
is finite but not necessarily achieved. Second, assumption \ref{asmp:cnco_a2}
is quite natural in the sense that the penalty approach underlying
the AIP.QPM would not make sense without it. Third, it is well-known
that a necessary condition for $z^{*}\in Z$ to be a local minimum
of \ref{prb:eq:cnco_a} is that $z^{*}$ be a stationary point of
$f+h$, i.e. there exists $p^{*}\in{\cal R}$ such that $0\in\nabla f(z^{*})+\partial h(z^{*})+{\cal A}^{*}p^{*}$
and ${\cal A}^{*}z\in S$.

In view of the above assumptions and remarks, we are interested in
solving the problem given in \prettyref{prb:approx_cnco_a}.

\begin{mdframed}
\mdprbcaption{Find an approximate stationary point of ${\cal CNCO}$[a]}{prb:approx_cnco_a}
Given $(\hat{\rho}, \hat{\eta})\in\r_{++}^2$, find a pair $([\hat{z}, \hat{p}],[\hat{v}, \hat{q}]) \in [Z \times {\cal R}] \times [{\cal Z} \times {\cal R}]$ satisfying conditions \eqref{eq:approx_cnco_a_ln1} and \eqref{eq:approx_cnco_a_ln2}.
\end{mdframed}

\subsection{Key Properties of the Quadratic Penalty Approach}

\label{subsec:qp_props}

We begin with some basic properties about the penalty function $\varphi_{c}$
and some of its related quantities. 
\begin{lem}
\label{lem:penalty_props}Let $(f,h)$ be a pair of functions satisfying
assumptions \ref{asmp:nco1}--\ref{asmp:nco2} and \ref{asmp:cnco_a1}--
\ref{asmp:cnco_a2}, ${\cal F}$ be as in assumption \ref{asmp:cnco_a1},
$(\hat{c},\hat{\varphi}_{c},\varphi_{c})$ be as in assumption \ref{asmp:cnco_a2},
and the functions $R_{\lam}\psi(\cdot)$ and $R_{\lam}\psi(\cdot,\cdot)$
be as in \eqref{eq:Rphi_def}. Moreover, define 
\begin{equation}
R_{\lam}^{{\cal F}}\psi(z_{0}):=\inf_{u\in{\cal F}}R_{\lam}\psi(u;z_{0}),\label{eq:RFphi_def}
\end{equation}
for any function $\psi:{\cal Z}\mapsto(-\infty,\infty]$, scalar $\lam\geq0$,
and point $z_{0}\in{\cal Z}$. Then, the following statements hold
for every scalar $c\geq\hat{c}$, scalars $\lam,\hat{\lam}\in\r_{+}$
satisfying $\lam\geq\hat{\lam}$, and point $z_{0}\in{\cal Z}$:
\begin{itemize}
\item[(a)] $\hat{\varphi}_{c}\geq\hat{\varphi}_{\hat{c}}>-\infty$ and $\varphi_{c}(u)=\varphi_{\hat{c}}(u)$
for every $u\in{\cal F}$;
\item[(b)] $R_{\lam}\varphi_{c}(u;z_{0})\leq R_{\hat{\lam}}\varphi_{\hat{c}}(u;z_{0})$
for every $u\in{\cal F}$, and hence, $R_{\lam}^{{\cal F}}\varphi_{c}(z_{0})\leq R_{\hat{\lam}}^{{\cal F}}\varphi_{\hat{c}}(z_{0})$;
\item[(c)] if $z^{*}$ is an optimal solution of \ref{prb:eq:cnco_a}, then
\[
R_{\lam}^{{\cal F}}\varphi_{c}(z_{0})\leq\frac{1}{2}\|z_{0}-z^{*}\|^{2}+\lam\left[\hat{\varphi}_{*}-\hat{\varphi}_{c}\right]
\]
where $\hat{\varphi}_{*}$ is as in \ref{prb:eq:cnco_a}.
\end{itemize}
\end{lem}

\begin{proof}
(a) The fact that $\hat{\varphi}_{\hat{c}}>-\infty$ is from assumption
\ref{asmp:cnco_a2}. The fact that $\varphi_{c}(u)=\varphi_{\hat{c}}(u)$
for every $u\in{\cal F}$ immediate from the definitions of $\varphi_{c}$
and ${\cal F}$. The remaining inequality follows from the definition
of $\varphi_{c}$ and the assumption that $c\geq\hat{c}$.

(b) The first set of inequalities is immediate from part (a) and our
assumption on $(\lam,\hat{\lam})$. The second one follows from the
definition of $R_{\lam}^{{\cal F}}\varphi_{c}(\cdot)$ in \eqref{eq:RFphi_def}.

(c) This is immediate from the definition of $R_{\lam}^{{\cal F}}\varphi_{c}(\cdot)$
in \eqref{eq:RFphi_def}.
\end{proof}
Note that, similar to \eqref{eq:R_moreau}, it is straightforward
to show that the function $R_{\lam,{\cal F}}\psi(\cdot)$ in \eqref{eq:RFphi_def}
satisfies
\begin{align*}
R_{\lam}^{{\cal F}}\psi(z_{0}) & =\lam\left[e_{\lam}(\psi+\delta_{{\cal F}})(z_{0})-\inf_{u\in{\cal Z}}(\psi+\delta_{{\cal F}})(u)\right].
\end{align*}

The next result shows how a solution of \prettyref{prb:approx_nco}
with $f=f_{c}$ is related to the conditions in \prettyref{prb:approx_cnco_a}.
\begin{lem}
\label{lem:props_qp}Given $\hat{\rho}>0$ and $c>0$, let $(\hat{z},\hat{v})$
be a solution of \prettyref{prb:approx_nco} with $f=f_{c}$ as in
\eqref{eq:smooth_penalty_fn}. Moreover, define the quantities

\[
\text{\ensuremath{\hat{p}=c\left[{\cal A}\hat{z}-\Pi_{S}({\cal A}\hat{z})\right],\quad\hat{q}=\Pi_{S}({\cal A}\hat{z})-{\cal A}\hat{z}.}}
\]
Then the following statements hold:
\begin{itemize}
\item[(a)] the pair $([\hat{z},\hat{p}],[\hat{v},\hat{q}])$ satisfies \eqref{eq:approx_cnco_a_ln1}
and the inclusion in \eqref{eq:approx_cnco_a_ln2};
\item[(b)] it holds that
\[
\|\hat{q}\|^{2}\leq\frac{2\left[\varphi_{c}(\hat{z})-\hat{\varphi}_{\hat{c}}\right]}{c-\hat{c}}.
\]
\end{itemize}
\end{lem}

\begin{proof}
(a) Using \prettyref{lem:dist_props}(b) with ${\cal K}=S$ and the
Chain Rule, it follows that 
\[
\nabla f_{c}(\hat{z})=\nabla f(\hat{z})+c{\cal A}^{*}\left[{\cal A}\hat{z}-\Pi_{S}({\cal A}\hat{z})\right]=\nabla f(\hat{z})+{\cal A}^{*}\hat{p},
\]
and hence, by the definition of \prettyref{prb:approx_nco} with $f_{c}$,
it holds that $(\hat{z},\hat{p},\hat{v})$ satisfies \eqref{eq:approx_cnco_a_ln1}.
On the other hand, the inclusion \eqref{eq:approx_cnco_a_ln2} follows
immediately from the definition of $\hat{q}$.

(b) Using the definition of $\hat{\varphi}_{\hat{c}}$, it holds that
\[
\hat{\varphi}_{\hat{c}}+\left(\frac{c-\hat{c}}{2}\right)\cdot\dist^{2}({\cal A}\hat{z},S)\leq\varphi_{\hat{c}}(\hat{z})+\left(\frac{c-\hat{c}}{2}\right)\cdot\dist^{2}({\cal A}\hat{z},S)=\varphi_{c}(\hat{z}).
\]
Rearranging the above inequality and using the fact that $\|\hat{q}\|=\dist(A\hat{z},S)$,
it holds that 
\[
\|\hat{q}\|^{2}=\dist^{2}(A\hat{z},S)\leq\frac{2\left[\varphi_{c}(\hat{z})-\hat{\varphi}_{\hat{c}}\right]}{c-\hat{c}}.
\]
\end{proof}
We now describe the behavior of a GIPP instance (see \prettyref{chap:unconstr_nco})
applied to \eqref{eq:varphiC_def}.
\begin{lem}
\label{lem:qp_feas}Let $\hat{q}$, $\hat{c}$, $\varphi_{c}$, and
$R_{\lam}^{{\cal F}}\varphi_{c}(\cdot)$ be as in \prettyref{lem:props_qp},
and suppose $\{(z_{k},\tilde{v}_{k},\tilde{\varepsilon}_{k})\}_{k\geq1}$
is a sequence generated by an instance of the GIPPF (see \prettyref{alg:gippf})
for some $\{\lam_{k}\}_{k\geq1}$ and $z_{0}\in Z$ with $\phi=\varphi_{c}$
for some $c>\hat{c}$. Moreover, let $\hat{\eta}\in\r_{++}$ be given
and define 
\begin{equation}
T_{\hat{\eta}}(\lam):=\hat{c}+\left[\frac{2\cdot R_{\lam}^{{\cal F}}\varphi_{\hat{c}}(z_{0})}{\lam(1-\sigma)}\right]\hat{\eta}^{-2}\quad\forall\lam\in\r_{++},\label{eq:qp_Tdef}
\end{equation}
where $R_{\lam}^{{\cal F}}\varphi_{\hat{c}}(\cdot)$ is as in \eqref{eq:RFphi_def}.
Then, for every $\hat{z}\in{\cal Z}$ such that $\varphi_{c}(\hat{z})\leq\varphi_{c}(z_{1})$,
it holds that
\begin{equation}
\|\hat{q}\|^{2}\leq\frac{\left[T_{\hat{\eta}}(\lam_{1})-\hat{c}\right]\hat{\eta}^{2}}{c-\hat{c}}.\label{eq:qhat_prelim_bd}
\end{equation}
As a consequence, if $c\geq T_{\hat{\eta}}(\lam_{1})$ then $\|\hat{q}\|\leq\hat{\eta}$.
\end{lem}

\begin{proof}
Let $\hat{z}\in{\cal Z}$ be such that $\varphi_{c}(\hat{z})\leq\varphi_{c}(z_{1})$.
Using \prettyref{lem:phik_d0} with $k=1$, the previous bound, and
\prettyref{lem:penalty_props}(a), it holds that 
\begin{align*}
\varphi_{c}(\hat{z})-\hat{\varphi}_{\hat{c}} & \leq\varphi_{c}(z_{1})-\hat{\varphi}_{\hat{c}}\\
 & \leq\varphi_{c}(u)-\hat{\varphi}_{\hat{c}}+\frac{1}{2(1-\sigma)\lam_{1}}\|u-z_{0}\|^{2}.\\
 & =\varphi_{\hat{c}}(u)-\hat{\varphi}_{\hat{c}}+\frac{1}{2(1-\sigma)\lam_{1}}\|u-z_{0}\|^{2}\\
 & \leq\frac{1}{\lam_{1}(1-\sigma)}\left(\lam_{1}\left[\varphi_{\hat{c}}(u)-\hat{\varphi}_{\hat{c}}\right]+\frac{1}{2}\|u-z_{0}\|^{2}\right)\quad\forall u\in{\cal F}.
\end{align*}
Taking the infimum of the above bound over $u\in{\cal F}$ and using
the definition of $R_{\lam}^{{\cal F}}\varphi_{\hat{c}}(z_{0})$,
we conclude that 
\begin{equation}
\varphi_{c}(\hat{z})-\hat{\varphi}_{\hat{c}}\leq\frac{R_{\lam}^{{\cal F}}\varphi_{\hat{c}}(z_{0})}{\lam_{1}(1-\sigma)}.\label{eq:qp_potential_bd}
\end{equation}
Using \eqref{eq:qp_potential_bd}, \prettyref{lem:props_qp}(b), and
the definition in \eqref{eq:qp_Tdef} yields \eqref{eq:qhat_prelim_bd}.
The last conclusion follows immediately from \eqref{eq:qhat_prelim_bd}
and the assumption that $c\geq T_{\hat{\eta}}(\lam_{1})$.
\end{proof}
We now make some remarks about the above result. First, it does not
assume that ${\cal F}$, and hence $Z$, is bounded. Also, it does
not even assume that \ref{prb:eq:cnco_a} has an optimal solution.
Second, it implies that all iterates (excluding the starting one)
generated by an instance of the GIPPF applied to \eqref{eq:prox_penalty_subprb}
satisfy the feasibility requirement, i.e. the last inequality in \eqref{eq:approx_cnco_a_ln2},
as long as $c_{\ell}$ is sufficiently large, i.e. $c_{\ell}\ge T_{\hat{\eta}}(\lambda_{1})$.
Third, since the quantity $R_{\lam}^{{\cal F}}\varphi_{\hat{c}}(z_{0})$,
which appears in the definition of $T_{\hat{\eta}}(\lambda_{1})$
is difficult to estimate, a simple way of choosing a penalty parameter
$c_{\ell}$ such that $c_{\ell}\ge T_{\hat{\eta}}(\lambda_{1})$ is
not apparent. This is why the AIP.QPM solves instead a sequence of
penalized subproblems \eqref{eq:prox_penalty_subprb} for a strictly
increasing sequence of penalty parameters $\{c_{\ell}\}_{\ell\geq1}$.
Moreover, despite solving a sequence of penalized subproblems, it
is shown that its total number of oracle calls is the same as the
one for the ideal method corresponding to solving \eqref{eq:prox_penalty_subprb}
with $c_{1}=T_{\hat{\eta}}(\lambda_{1})$.

Recall from \prettyref{lem:AIPPmethod} and \prettyref{thm:AIPPcomplexity}
in \prettyref{chap:unconstr_nco} that the AIPPM: (i) generates its
iterates as an instance of the GIPPF; and (ii) outputs a pair $(\hat{z},\hat{v})$
that solves \prettyref{prb:approx_nco} with $\phi(\hat{z})\leq\phi(z_{1})$.
In view of these facts, \prettyref{lem:props_qp} and \prettyref{lem:qp_feas}
show that the AIPPM is a suitable candidate for solving \ref{prb:approx_cnco_a}
when it is given $f=f_{c}$ for a sufficiently large enough $c>0$.
It only remains to show that the AIPPM can be applied to \eqref{eq:varphiC_def}.
Since assumption \ref{asmp:nco1} is that $h\in\cConv Z$, we show
that $f_{c}$ satisfies the necessary smoothness requirements in the
result below.
\begin{lem}
\label{lem:fc_smoothness}Suppose $f$ satisfies assumption \ref{asmp:nco2}
and let $f_{c}$ be as in \eqref{eq:smooth_penalty_fn}. For any $c\geq0$,
it holds that $f_{c}\in{\cal C}_{m,M_{c}}(Z)$ where $M_{c}:=M+c\|{\cal A}\|^{2}.$
\end{lem}

\begin{proof}
Let $Q(z):=\dist^{2}({\cal A}z,S)/2$. Using \prettyref{lem:dist_props}(a)--(b)
with ${\cal K}=S$ and the Chain Rule, it holds that
\begin{align*}
\|\nabla Q(z)-\nabla Q(u)\| & =\left\Vert {\cal A}^{*}\left(\left[{\cal A}z-\Pi_{S}({\cal A}z)\right]-\left[{\cal A}u-\Pi_{S}({\cal A}z)\right]\right)\right\Vert \\
 & \leq\|{\cal A}\|\cdot\left\Vert \left[{\cal A}z-\Pi_{S}({\cal A}z)\right]-\left[{\cal A}u-\Pi_{S}({\cal A}u)\right]\right\Vert \\
 & \leq\|{\cal A}\|\cdot\|{\cal A}z-{\cal A}u\|\leq\|{\cal A}\|^{2}\|z-u\|,
\end{align*}
 and hence, $Q\in{\cal F}_{0,\|{\cal A}\|^{2}}(Z)$. The conclusion
now follows from assumption \ref{asmp:nco2} and the fact that $f_{c}=f+cQ.$
\end{proof}

\subsection{Statement and Properties of the AIP.QPM}

This subsection describes and establishes the iteration complexity
of the AIP.QPM. 

We first state the AIP.QPM in \prettyref{alg:qp_aippm}, which uses
the AIPPM in \prettyref{alg:aippm}. Given $(\sigma,\lam)\in(0,1)\times(0,1/m)$
and $z_{0}\in Z$, its main idea is to invoke the AIPPM to obtain
approximate stationary points of sequence of penalty subproblems of
the form 
\[
\min_{z\in{\cal Z}}\left\{ f_{c_{\ell}}(z)+h(z)\right\} 
\]
where $\{c_{\ell}\}_{\ell\geq1}$ is a strictly increasing sequence
of penalty parameters that tend to infinity. At the end of each AIPPM
call, a pair $([\hat{z},\hat{p}],[\hat{v},\hat{q}])$ is generated
that satisfies \eqref{eq:eff_approx_cnco_a_ln1} and the inclusion
in \eqref{eq:approx_cnco_a_ln1}, and the method terminates when the
inequality in \eqref{eq:approx_cnco_a_ln2} holds.

\begin{mdframed}
\mdalgcaption{AIP.QP Method}{alg:qp_aippm}
\begin{smalgorithmic}
	\Require{$(\hat{\rho},\hat{\eta}) \in \r_{++}^2, \enskip \sigma \in (0,1), \enskip (m,M)\in\r_{+}^2, \enskip h \in \cConv(Z), \enskip f \in {\cal C}_{m,M}(Z), \enskip \lam \in (0, 1/m), \enskip z_0 \in Z, \enskip {\cal A}\neq 0, \enskip S\subseteq {\cal R}, \enskip \hat{c}>0 \text{ satisfying }\ref{asmp:cnco_a2}$;}
	\Initialize{$c_1 \gets \hat{c} + (M + \lam^{-1})/\|{\cal A}\|^2;$}
	\vspace*{.5em}
	\Procedure{AIP.QP}{$f, h, {\cal A}, S, z_0, \hat{c}, \lam, m, M, \sigma, \hat{\rho}, \hat{\eta}$}
	\For{$\ell=1,...$}
		\StateStep{\algpart{1}\textbf{Attack} the $\ell^{\rm th}$ prox penalty subproblem.}
		\StateEq{$f_{c_\ell} \Lleftarrow f + \frac{c_\ell}{2} \cdot \dist^2({\cal A}(\cdot), S)$}
		\StateEq{$M_{c_\ell} \gets M + {c_\ell}\|{\cal A}\|^2$} \label{ln:qp_aippm_Mc_def}
		\StateEq{$(\hat{z}_\ell, \hat{v}_\ell) \gets \text{AIPP}(f_{c_\ell}, h, z_0, \lam, m, M_{c_\ell}, \sigma, \hat{\rho})$} \label{ln:aippm_call}
		\StateEq{$\hat{p}_\ell \gets c_\ell \left[{\cal A}\hat{z_\ell}-\Pi_{S}({\cal A}\hat{z}_\ell)\right]$} \label{ln:qp_aippm_phat_def}
		\StateEq{$\hat{q}_\ell \gets \Pi_{S}({\cal A}\hat{z}_\ell)-{\cal A}\hat{z_\ell}$}
		\StateStep{\algpart{2}Either \textbf{stop} with a nearly feasible point or \textbf{increase} $c_\ell$.}
		\If{$\|\hat{q}_\ell\| \leq \hat{\eta}$} \label{ln:qp_aipp_feas_check}
			\StateEq{\Return{$([\hat{z}_\ell, \hat{p}_\ell], [\hat{v}_\ell, \hat{q}_\ell])$}}
		\EndIf
		\StateEq{$c_{\ell+1} \gets 2 c_\ell$}
	\EndFor
	\EndProcedure
\end{smalgorithmic}
\end{mdframed}

Some comments about the AIP.QPM are in order. To ease the discussion,
let us refer to the AIPP iterations in each AIPP call as \textbf{outer
iterations}\emph{, }the ACG iterations performed inside each AIPP
call as \textbf{inner iterations}\emph{, }and the iterations over
the indices $\ell$ as\textbf{ cycles}. First, it follows from \prettyref{lem:AIPPmethod}(d)
that the pair $(\hat{z},\hat{v})=(\hat{z}_{\ell},\hat{v}_{\ell})$
solves \prettyref{prb:approx_nco} with $f=f_{c_{\ell}}$. As a consequence,
\prettyref{lem:props_qp}(a) implies that the output $([\hat{z},\hat{p}],[\hat{v},\hat{q}])$
satisfies the \eqref{eq:approx_cnco_a_ln1} and the first inequality
in \eqref{eq:approx_cnco_a_ln2}. Second, since every loop of the
AIP.QPM doubles $c_{\ell}$, the condition $c_{\ell}>T_{\hat{\eta}}(\lambda_{1})$
will be eventually satisfied. Hence, in view of the previous remark,
the $\hat{q}_{\ell}$ corresponding to this $c_{\ell}$ will satisfy
the feasibility condition $\|\hat{q}_{\ell}\|\le\hat{\eta}$ and the
AIP.QPM will stop in view of its stopping criterion in \prettyref{ln:qp_aipp_feas_check}.
Finally, in view of the previous remarks, we conclude that the AIP.QPM
terminates with a triple $([\hat{z},\hat{p}],[\hat{v},\hat{q}])$
satisfying \eqref{eq:approx_cnco_a_ln1} and \eqref{eq:approx_cnco_a_ln2}.

The next result presents some basic properties of the AIP.QPM in consideration
of the above remarks.
\begin{lem}
\label{lem:props_qp_aippm}Let $T_{\hat{\eta}}(\cdot)$ be as in \eqref{eq:qp_Tdef}.
The following statements hold about the AIP.QPM:
\begin{itemize}
\item[(a)] at the $\ell^{{\rm th}}$ cycle, its call to the AIPPM in \prettyref{ln:aippm_call}
stops in
\begin{equation}
\mathcal{O}\left(\sqrt{\frac{\lam\tilde{M}_{\ell}+1}{\min\left\{ \sigma,1-\lam m\right\} }}\ \left[\frac{R_{\lam}^{{\cal F}}\varphi_{\hat{c}}(z_{0})}{(1-\sigma)^{2}\lam^{2}\hat{\rho}^{2}}+\log_{1}^{+}\left(\lambda\tilde{M}_{\ell}\right)\right]\right)\label{eq:qp_sgl_compl}
\end{equation}
inner iterations, where $R_{\lam}^{{\cal F}}\psi(\cdot)$ is as in
\eqref{eq:RFphi_def}, $\log_{1}^{+}(\cdot):=\max\{\log(\cdot),1\}$,
and 
\begin{equation}
\tilde{M}_{i}=M+2^{i-1}c_{1}\|{\cal A}\|^{2}\quad\forall i\geq1.\label{eq:qp_MTilde_def}
\end{equation}
\item[(b)] if $\ell_{C}$ is the first cycle where $c_{\ell}\geq T_{\hat{\eta}}(\lam)$,
then the AIP.QPM stops and outputs with a pair $([\hat{z},\hat{p}],[\hat{v},\hat{q}])$
that solves \prettyref{prb:approx_cnco_a} in at most $\ell_{C}$
cycles.
\end{itemize}
\end{lem}

\begin{proof}
All line numbers referenced in this proof are with respect to the
AIP.QPM in \prettyref{alg:qp_aippm}.

(a) Let $\ell\geq1$ and $M_{c_{\ell}}$ be as in \prettyref{ln:qp_aippm_Mc_def}.
Using the initialization of $c_{1}$ in the AIP.QPM, we first remark
that 
\begin{equation}
M_{c_{\ell}}=M+c_{\ell}\|{\cal A}\|^{2}=M+2^{\ell-1}c_{1}\|{\cal A}\|^{2}={\cal O}(\tilde{M}_{\ell}).\label{eq:Mc_bd}
\end{equation}
Moreover, by the definition of $R_{\lam}^{{\cal F}}\psi(\cdot)$ and
\prettyref{lem:penalty_props}(b), it follows that $R_{\lam}\varphi_{c_{\ell}}(z_{0})\leq R_{\lam}^{{\cal F}}\varphi_{c_{\ell}}(z_{0})\leq R_{\lam}^{{\cal F}}\varphi_{\hat{c}}(z_{0})$.
The conclusion result now follows from \prettyref{lem:fc_smoothness},
\eqref{eq:Mc_bd}, the previous bound, and \prettyref{thm:AIPPcomplexity}
with $M=M_{c_{\ell}}$.

(b) This follows immediately from \prettyref{lem:props_qp}(a) and
\prettyref{lem:qp_feas}.
\end{proof}
We now state one of our main results of this section, which is the
iteration complexity of the AIP.QPM for solving \prettyref{prb:approx_cnco_a}.
Recall that the AIP.QPM assumes that $\lambda<1/m$.
\begin{thm}
\label{thm:qp_aipp_compl}Let $T_{\hat{\eta}}(\cdot)$ be as in \eqref{eq:qp_Tdef}
and define
\begin{equation}
\Theta_{\hat{\eta}}:=M+T_{\hat{\eta}}(\lam)\|{\cal A}\|^{2}\quad\forall(\hat{\eta},\lam)\in\r_{++}^{2}.\label{eq:qp_Theta_def}
\end{equation}
The AIP.QPM outputs a pair $([\hat{z},\hat{p}],[\hat{v},\hat{q}])$
that solves \prettyref{prb:approx_cnco_a} in
\begin{equation}
\mathcal{O}\left(\sqrt{\frac{\lam\Theta_{\hat{\eta}}+1}{\min\left\{ \sigma,1-\lam m\right\} }}\ \left[\frac{R_{\lam}^{{\cal F}}\varphi_{\hat{c}}(z_{0})}{(1-\sigma)^{2}\lam^{2}\hat{\rho}^{2}}+\log_{1}^{+}\left(\lambda\Theta_{\hat{\eta}}\right)\right]\right)\label{eq:qp_aippm_compl}
\end{equation}
inner iterations, where $R_{\lam}^{{\cal F}}\psi(\cdot)$ is as in
\eqref{eq:RFphi_def} and $\log_{1}^{+}(\cdot):=\max\{\log(\cdot),1\}$.
\end{thm}

\begin{proof}
The fact that the output of the AIP.QPM solves \prettyref{prb:approx_cnco_a}
is an immediate consequence of \prettyref{lem:props_qp_aippm}(b).

Let us now prove the desired complexity bound. Let $\tilde{M}_{i}$
and $\ell_{C}$ be as in \eqref{eq:qp_MTilde_def} and \prettyref{lem:props_qp_aippm}(b),
respectively. In view of the AIPP call in \prettyref{ln:aippm_call}
and \prettyref{lem:props_qp_aippm}(b), it follows that the number
of inner iterations performed by the AIP.QPM is on the order given
by the sum of the bound in \eqref{eq:qp_sgl_compl} from $\ell=1$
to $\ell_{C}$. To show that this sum is exactly \eqref{eq:qp_aippm_compl},
we prove that
\begin{equation}
\sum_{i=1}^{\ell_{C}}(\lam\tilde{M}_{i}+1)^{1/2}={\cal O}\left(\left[\lam\Theta_{\hat{\eta}}+1\right]^{1/2}\right),\quad\log_{1}^{+}\left(\lam\tilde{M}_{\ell}\right)={\cal O}\left(\log_{1}^{+}\left[\lam\Theta_{\hat{\eta}}\right]\right)\quad\forall\ell\geq1.\label{eq:qp_aux_bds}
\end{equation}
 To begin, observe that the definition of $c_{1}$ implies that
\begin{equation}
M+\lam^{-1}\leq c_{1}\|{\cal A}\|^{2}\leq2^{i-1}c_{1}\|{\cal A}\|^{2}\quad\forall i\geq1,\label{eq:qp_curv_bd}
\end{equation}
and the definitions of $\Theta_{\hat{\eta}}$, $T_{\hat{\eta}}(\cdot)$,
and $c_{1}$ yield
\begin{align}
\lam\tilde{M}_{1}+1= & \lam\left(M+\lam^{-1}+c_{1}\|{\cal A}\|^{2}\right)\leq2\lam c_{1}\|{\cal A}\|^{2}=2\lam\left(M+\lam^{-1}+\hat{c}\|{\cal A}\|^{2}\right).\nonumber \\
 & =\lam\left[M+T_{\hat{\eta}}(\lam)\|{\cal A}\|^{2}\right]+1=\lam\Theta_{\hat{\eta}}+1.\label{eq:qp_aux_bd1}
\end{align}
Using \eqref{eq:qp_aux_bd1}, it follows that the bounds in \eqref{eq:qp_aux_bds}
hold for $\ell_{C}=1$ or $\ell=1$. Suppose now that $\ell_{C}>1$.
The definition of $\ell_{C}$ implies that $c_{1}\cdot2^{\ell_{C}-1}\leq2T_{\hat{\eta}}$,
or equivalently, $2^{\ell_{C}/2}\leq2\left[T_{\hat{\eta}}(\lam)/c_{1}\right]^{1/2}$.
Using the previous bound, \eqref{eq:qp_curv_bd}, and the definition
of $\Theta_{\hat{\eta}}$, it follows that
\begin{align}
\sum_{i=1}^{\ell_{C}}(\lam\tilde{M}_{i}+1)^{1/2} & =\sum_{i=1}^{\ell_{C}}\left(\lam\left[M+\lam^{-1}+2^{i-1}c_{1}\|{\cal A}\|^{2}\right]+1\right)^{1/2}\nonumber \\
 & \leq\sum_{i=1}^{\ell_{C}}\left(2\lam\left[M+\lam^{-1}+2^{i-1}c_{1}\|{\cal A}\|^{2}\right]\right)^{1/2}\nonumber \\
 & =2\left(\lam c_{1}\|{\cal A}\|^{2}\right)^{1/2}\sum_{i=1}^{\ell_{C}}2^{(i-1)/2}={\cal O}\left(\left[\lam c_{1}\|{\cal A}\|^{2}\right]^{1/2}2^{\ell_{C}/2}\right)\nonumber \\
 & ={\cal O}\left(\left[\lam T_{\hat{\eta}}(\lam)\|{\cal A}\|^{2}\right]^{1/2}\right)={\cal O}\left(\left[\lam\Theta_{\hat{\eta}}+1\right]^{1/2}\right).\label{eq:qp_aux_bd2}
\end{align}
Similarly, using the fact that $\{c_{i}\}_{i\geq1}$ is monotone increasing,
the previous bound on $2^{\ell_{C}/2}$, \eqref{eq:qp_curv_bd}, and
the definition of $\Theta_{\hat{\eta}}$, it holds that
\begin{equation}
\log\left(\lam\tilde{M}_{i}\right)\leq\log\left(\lam\tilde{M}_{\ell_{C}}\right)=\log\left(\lam2^{\ell_{C}}c_{1}\|{\cal A}\|^{2}\right)=\log\left(\lam T_{\hat{\eta}}(\lam)\|{\cal A}\|^{2}\right)=\log\left(\lam\Theta_{\hat{\eta}}\right).\label{eq:qp_aux_bd3}
\end{equation}
Using \eqref{eq:qp_aux_bd2} and \eqref{eq:qp_aux_bd3}, it follows
that the bounds in \eqref{eq:qp_aux_bds} hold for $\ell_{C}\geq2$
or $\ell\ge2$.
\end{proof}
The following result describes the number of oracle calls performed
by the AIP.QPM with $\lambda=1/(2m)$ and $\sigma=1/2$.
\begin{cor}
\label{cor:spec_qp_aipp_compl}The AIP.QPM with inputs $\lam=1/(2m)$
and $\sigma=1/2$ outputs a $([\hat{z},\hat{p}],[\hat{v},\hat{q}])$
that solves \prettyref{prb:approx_cnco_a} in 
\[
{\cal O}\left(\sqrt{\frac{\Theta_{\hat{\eta}}}{m}+1}\left[\frac{m^{2}R_{1/(2m)}^{{\cal F}}\varphi_{\hat{c}}(z_{0})}{\hat{\rho}^{2}}+\log_{1}^{+}\left(\frac{\Theta_{\hat{\eta}}}{m}\right)\right]\right)
\]
oracle calls, where $R_{\lam}^{{\cal F}}\psi(\cdot)$ is as in \eqref{eq:RFphi_def}
and $\log_{1}^{+}(\cdot):=\max\{\log(\cdot),1\}$.
\end{cor}

\begin{proof}
This follows immediately from \prettyref{thm:qp_aipp_compl}, the
definition of $\log_{1}^{+}(\cdot)$, and the fact that every iteration
of the ACGM performs ${\cal O}(1)$ oracle calls.
\end{proof}

\section{Composite Optimization with Nonlinear Cone Constraints}

\label{sec:aipp_alm}

The augmented Lagrangian method \citep{Powell1969,Hestenes1969} is
an well-known extension of the quadratic penalty method (see \prettyref{sec:qp_aipp})
applied to the problem $\min_{x\in\rn}\{\phi(x):g(x)\leq0\}$ in which
a multiplier update is added to every iteration of the method. More
specifically, recalling the Lagrangian ${\cal L}(\cdot;\cdot)$ in
\eqref{eq:intro_aug_Lagr} and denoting 
\[
\ell_{k}(p;p_{k-1})={\cal L}_{c_{k}}(x_{k};p_{k-1})+\left\langle \nabla_{p}{\cal L}_{c_{k}}(x_{k};p_{k-1}),p-p_{k-1}\right\rangle ,
\]
to be the linear approximation of the function $p\mapsto{\cal L}_{c_{k}}(x_{k};p)$
at $p=p_{k-1}$, the multiplier update is given by 
\begin{align}
p_{k} & =\argmax_{p\geq0}\left\{ c_{k}\ell_{k}(p;p_{k-1})+\frac{1}{2}\|p-p_{k-1}\|^{2}\right\} ,\nonumber \\
 & =\max\left\{ 0,p_{k-1}+c_{k}g(x_{k})\right\} .\label{eq:intro_multiplier_update}
\end{align}
For the case where $h\equiv0$, it is known \citep[Proposition 4.2.3]{Bertsekas1999}
that if the generated sequence $\{p_{k}\}_{k\geq1}$ is bounded, the
penalty parameter $c_{k}$ is sufficiently large enough after a certain
index $k$, and some additional regularity conditions hold, then $x_{k}$
and $p_{k}$ converge to a global minimum and Lagrange multiplier
of the constrained problem, respectively.

Our main goal in this section is to describe and establish the iteration
complexity of an accelerated \textbf{inexact} proximal augmented Lagrangian
(AIP.AL) method for finding approximate stationary points of the nonlinearly
cone-constrained NCO problem 

\begin{equation}
\varphi_{*}=\min_{z\in{\cal Z}}\left\{ \phi(z)=f(z)+h(z):g(z)\preceq_{{\cal K}}0\right\} \tag{\ensuremath{{\cal CNCO}[b]}}\label{prb:eq:cnco_b}
\end{equation}
where ${\cal K}$ is a closed convex cone, the feasible set is nonempty,
and the functions $f$, $h$, and $g$ are as described in the beginning
of the chapter. We will also assume that $g$ is $\cK$-convex function,
i.e.
\[
g(tu+[1-t]z)\preceq_{\cK}tg(u)+[1-t]g(z)\quad\forall(t,u,z)\in[0,1]\times{\cal Z}\times{\cal Z},
\]
with a Lipschitz continuous gradient, $h$ is Lipschitz continuous
on its domain $Z\subseteq{\cal Z}$, the set $Z$ is convex compact,
and that we have an oracle for computing the projection onto the dual
cone of ${\cal K}$, which is denoted by ${\cal K}^{+}$ and included
in the oracles that make up the oracle call mentioned at the beginning
of this chapter. Here, the relation $g(z)\preceq_{\cK}0$ means that
$g(z)\in-\cK$. 

The AIP.AL method (AIP.ALM) is based on the generalized (cf. \citep{Lu2018}
and \citep[Section 11.K]{Rockafellar2009}) augmented Lagrangian function
\begin{equation}
{\cal L}_{c}(z;p):=f(z)+h(z)+\frac{1}{2c}\left[{\rm dist}^{2}(p+cg(z),-{\cal K})-\|p\|^{2}\right],\label{eq:aug_lagr_def}
\end{equation}
and it uses an ACGM, e.g. \prettyref{alg:acgm}, to perform the following
proximal point-type update to generate its $k^{{\rm th}}$ iterate:
given $(z_{k-1},p_{k-1})$ and $(\lam,c_{k})$, compute 
\begin{align}
z_{k} & \approx\argmin_{u}\left\{ \lam{\cal L}_{c_{k}}(u;p_{k-1})+\frac{1}{2}\|u-z_{k-1}\|^{2}\right\} ,\label{eq:approx_primal_update}\\
p_{k} & =\Pi_{\cK^{+}}(p_{k-1}+c_{k}g(z_{k})),\label{eq:dual_update}
\end{align}
where $\cK^{+}$ denotes the dual cone of $\cK$ and the inexactness
in the $z_{k}$ update is according to some \textbf{relative }inexactness
criterion. At the end of the $k^{{\rm th}}$ iteration above, it also
performs a novel test to decide whether $c_{k}$ is left unchanged
or doubled. 

Under a generalized Slater assumption\footnote{See \prettyref{prop:weak_slater}.}
and a suitable choice of the inputs $(\lam,c)$, the main result of
this section shows that for any $(\hat{\rho},\hat{\eta})\in\r_{++}^{2}$,
the AIP.ALM obtains a pair $([\hat{z},\hat{p}],[\hat{v},\hat{q}])$
satisfying 
\begin{gather}
\hat{v}\in\nabla f(\hat{z})+\pt h(\hat{z})+\nabla g(\hat{z})\hat{p},\quad\inner{g(\hat{z})+\hat{q}}{\hat{p}}=0,\quad g(\hat{z})+\hat{q}\preceq_{\cK}0,\quad\hat{p}\succeq_{\cK^{+}}0\label{eq:cone_approx_soln}\\
\|\hat{v}\|\leq\hat{\rho},\quad\|\hat{q}\|\leq\hat{\eta},\label{ineq:cone_approx_soln}
\end{gather}
in ${\cal O}([\hat{\eta}^{-1/2}\hat{\rho}^{-2}+\hat{\rho}^{-3}]\log_{1}^{+}[\hat{\rho}^{-1}+\hat{\eta}^{-1}])$
oracle calls, where $\log_{1}^{+}(\cdot)=\max\{1,\log(\cdot)\}$.
Moreover, this complexity result is shown without requiring that the
initial point $z_{0}$ be feasible with respect to the nonlinear constraint,
i.e. $g(z_{0})\preceq_{\cK}0$. A key fact about AIP.AL is that its
generated sequence of Lagrange multipliers is always bounded, and
this conclusion strongly uses the fact that its constraint function
$g$ is $\cK$-convex.

\subsubsection*{Organization}

This section contains four subsections. The first one gives some preliminary
references and discusses our notion of a stationary point given in
\eqref{eq:cone_approx_soln} and \eqref{ineq:cone_approx_soln}. The
second one presents some key properties of the augmented Lagrangian
approach. The third one presents the AIP.ALM and its iteration complexity.
The last one gives the proof of the main result in this section.

\subsection{Preliminaries}

It is assumed that $\phi=f+h$ satisfies assumptions \ref{asmp:nco1}--\ref{asmp:nco2}
with $m\leq M$, as well as the following assumptions: 

\stepcounter{assumption}
\begin{enumerate}
\item \label{asmp:cnco_b1}$h$ is also $K_{h}$-Lipschitz continuous for
some $K_{h}>0$, and $Z$ is also compact with diameter $D_{z}:=\sup_{u,z\in Z}\|u-z\|$;
\item \label{asmp:cnco_b2}$g:{\cal Z}\mapsto\r^{\ell}$ is continuously
differentiable, ${\cal K}$-convex, and there exists $L_{g}>0$ such
that
\[
\|\nabla g(u)-\nabla g(z)\|\leq L_{g}\|u-z\|\quad\forall u,z\in{\cal Z};
\]
\item \label{asmp:cnco_b3}there exists $\bar{z}\in\intr Z$ and $\tau\in(0,1]$
such that $g(\bar{z})\preceq_{{\cal K}}0$ and
\begin{equation}
\max\left\{ \|\nabla g(z)p\|,\left|\left\langle g(\bar{z}),p\right\rangle \right|\right\} \geq\tau\|p\|\quad\forall z\in Z,\quad\forall p\succeq_{{\cal K}^{+}}0;\label{eq:gen_slater}
\end{equation}
\end{enumerate}
We now give three remarks about the above assumptions. First, since
$Z$ is compact by \ref{asmp:cnco_b1}, the image of any continuous
$\r^{\ell}$-valued function on $Z$ is bounded. In view of this observation,
we introduce the useful notation for any continuously differentiable
function $\Psi:Z\mapsto\r^{\ell}$: 
\begin{equation}
B_{\Psi}^{(0)}:=\sup_{z\in{\cal H}}\|\Psi(z)\|<\infty,\quad B_{\Psi}^{(1)}:=\sup_{z\in{\cal H}}\|\nabla\Psi(z)\|<\infty.\label{eq:bd_Psi_val}
\end{equation}
Second, it is well-known that if $g$ is differentiable and $\cK$-convex,
then for every $z,u\in{\cal Z}$ it holds that
\[
g'(z)(u-z)\preceq_{{\cal K}}g(u)-g(z).
\]
Third, it is also well-known that a necessary condition for a point
$z^{*}$ to be a local minimum of \ref{prb:eq:cnco_b} is that there
exists a multiplier $p^{*}\in\r^{\ell}$ that satisfies the stationarity
conditions 
\begin{equation}
\begin{gathered}0\in\nabla f(z^{*})+\pt h(z^{*})+\nabla g(z^{*})p^{*},\quad\inner{g(z^{*})}{p^{*}}=0,\quad g(z^{*})\preceq_{\cK}0,\quad p^{*}\succeq_{{\cal K}^{+}}0.\end{gathered}
\label{eq:stationary_soln}
\end{equation}
Moreover, the last three conditions in \eqref{eq:stationary_soln}
(resp. \eqref{eq:cone_approx_soln}) are equivalent\footnote{See, for example, \citep[Example 11.4]{Rockafellar2009} with $\bar{x}=g(z^{*})$
and $\bar{v}=p^{*}$.} to the inclusion $g(z^{*})\in N_{\cK^{+}}(p^{*})$ (resp. the inequality
${\rm dist}(g(\hat{z}),N_{\cK^{+}}(\hat{p}))\leq\hat{\eta}$). In
view of the above, \eqref{eq:cone_approx_soln} and \eqref{eq:cone_approx_soln}
are clearly relaxations of \eqref{eq:stationary_soln}. For the ease
of future reference, let us formally state the problem of finding
a pair $([\hat{z},\hat{p}],[\hat{v},\hat{q}])$ satisfying \eqref{eq:cone_approx_soln}
and \eqref{eq:cone_approx_soln} in \prettyref{prb:approx_cnco_b}.

\begin{mdframed}
\mdprbcaption{Find an approximate stationary point of ${\cal CNCO}[b]$}{prb:approx_cnco_b}
Given $(\hat{\rho}, \hat{\eta}) \in \r_{++}^2 $, find a pair $([\hat{z}, \hat{p}], [\hat{v}, \hat{q}]) \in [Z \times \r^{\ell}] \times [{\cal Z} \times \r^{\ell}]$ satisfying conditions \eqref{eq:cone_approx_soln} and \eqref{ineq:cone_approx_soln}.
\end{mdframed}

It is also worth mentioning that the conditions in \ref{asmp:cnco_b3}
can be viewed as a generalization of a Slater-like assumption with
respect to $g$, as shown in \prettyref{prop:weak_slater} below.
\begin{prop}
\label{prop:weak_slater}(Slater-like Assumption) Assume that the
constraint $g(z)\preceq_{\cK}0$ is of the form 
\begin{equation}
g_{\iota}(z)\preceq_{{\cal J}}0\quad g_{e}(z)=0\label{eq:ie_cone}
\end{equation}
where ${\cal J}\subseteq\r^{s}$ is a closed convex cone, $g_{\iota}:\r^{n}\to\r^{s}$
is continuously differentiable, and $g_{e}:\r^{n}\to\r^{t}$ is an
onto affine map (and hence $g=(g_{\iota},g_{e})$ and $\cK={\cal J}\times\{0\}$).
Assume also that there exists $\bar{z}\in{\cal H}$ such that $g_{\iota}(\bar{z})\prec_{{\cal J}}0$
and $g_{e}(\bar{z})=0$. Then, there exists $\tau>0$ such that $(\bar{z},\tau)$
satisfies \eqref{eq:gen_slater}. If, in addition, $\bar{z}\in\intr Z$,
then $(\bar{z},\tau)$ satisfies \ref{asmp:cnco_b3}. 
\end{prop}

\begin{proof}
Since $g_{e}$ is affine and onto, its gradient matrix $G_{e}:=\nabla g_{e}$
is independent of $z$ and has full column rank.Hence, there exists
$\tau_{e}>0$ such that 
\begin{equation}
\|G_{e}p_{e}\|\ge\tau_{e}\|p_{e}\|_{1}\quad\forall p_{e}\in\r^{s}.\label{eq:wslater_bd1}
\end{equation}
On the other hand, the assumption that $g_{\iota}(\bar{z})\prec_{{\cal J}}0$,
and \prettyref{lem:cone_generator} with $\cK={\cal J}$ and $x=-g_{\iota}(\bar{z})\in{\cal J}$,
imply that there exists $\tau_{\iota}>0$ such that 
\[
-\left\langle p_{\iota},g_{\iota}(\bar{z})\right\rangle \ge\tau_{\iota}\|p_{\iota}\|\quad\forall p_{\iota}\in{\cal J}^{+}.
\]
Using the previous inequality and the fact that $\|\nabla g_{\iota}(z)\|$
is bounded on ${\cal H}$, we conclude that there exists $\gamma>0$
such that 
\begin{equation}
-\|\nabla g_{\iota}(z)p_{\iota}\|-2\gamma\inner{p_{\iota}}{g_{\iota}(\bar{z})}\ge[2\gamma\tau_{\iota}-\|\nabla g_{\iota}(z)\|]\cdot\|p_{\iota}\|\geq\tau_{\iota}\|p_{\iota}\|_{1}\quad\forall z\in Z\label{eq:wslater_bd2}
\end{equation}
Relations \eqref{eq:wslater_bd1}, \eqref{eq:wslater_bd2}, and the
reverse triangle inequality, then imply that for every $z\in Z$,
\begin{align*}
 & \|\nabla g(z)p\|-2\gamma\left\langle p,g(\bar{z})\right\rangle =\|\nabla g_{\iota}(z)p_{\iota}+G_{e}p_{e}\|-2\gamma\left\langle p_{\iota},g_{\iota}(\bar{z})\right\rangle \\
 & \geq\|G_{e}p_{e}\|-\|\nabla g_{\iota}(z)p_{\iota}\|-2\gamma\left\langle p_{\iota},g_{\iota}(\bar{z})\right\rangle \geq\tau_{e}\|p_{e}\|_{1}+\tau_{\iota}\|p_{\iota}\|_{1}\\
 & \geq\tau_{c}\|p\|_{1}\geq\tau_{c}\|p\|,
\end{align*}
where $\tau_{c}:=\min\{\tau_{e},\tau_{\iota},1\}$. It is now straightforward
to see that the above inequality yields inequality \eqref{eq:gen_slater}
with $\tau=\tau_{c}/(1+2\gamma)\in(0,1]$. The last part of the proposition
now follows from the statement of assumption \ref{asmp:cnco_b3} and
the previous conclusion. 
\end{proof}
Some additional comments about \prettyref{prop:weak_slater} are in
order. First, the assumption that $g_{\iota}$ is ${\cal J}$-convex
and $g_{e}$ is affine implies that $g$ is $\cK$-convex. Second,
the Slater condition is with regards to a single point $\bar{z}\in Z$,
as opposed to condition \eqref{eq:gen_slater} which involves inequality
\eqref{eq:gen_slater} at all pairs $(z,p)\in Z\times\cK^{+}$. Third,
\ref{asmp:cnco_b3} can be replaced by the Slater-like assumption
of \prettyref{prop:weak_slater} since the former is implied by the
latter. Actually, a slightly more involved analysis can be done to
show that the assumption that $g_{e}$ is onto (which is part of the
assumption of \prettyref{prop:weak_slater}) can be removed at the
expense of obtaining a weaker version of \ref{asmp:cnco_b3}, namely:
inequality \eqref{eq:gen_slater} holds for every pair $(z,p)\in Z\times({\cal J}^{+}\times{\rm Im}\,\nabla g_{e})$,
instead of $(z,p)\in Z\times({\cal J}^{+}\times\r^{t})=Z\times\cK^{+}$.
Finally, since the analysis of this chapter can be easily adapted
to this slightly weaker version of \ref{asmp:cnco_b3}, the Slater-like
condition of \prettyref{prop:weak_slater} without $g_{e}$ assumed
to be onto (or equivalently, $\nabla g_{e}$ to have full column rank)
can be used in place of \ref{asmp:cnco_b3} in order to guarantee
that all of the results derived in this chapter for the AIP.ALM hold. 

\subsection{Key Properties of the Augmented Lagrangian Approach}

This subsection presents some technical results about the augmented
Lagrangian approach.

The first result describes some properties about the smooth part of
the Lagrangian in \eqref{eq:aug_lagr_def}. 
\begin{lem}
\label{lem:dist_smoothness}Define the function 
\begin{gather}
\widetilde{{\cal L}}_{c}(z;p):=f(z)+\frac{1}{2c}\left[{\rm dist}^{2}(p+cg(z),-{\cal K})-\|p\|^{2}\right]\quad\forall(z,p,c)\in Z\times\r^{\ell}\times\r_{++}.\label{eq:L_tilde_def}
\end{gather}
Then, for every $c>0$ and $p\in\r^{\ell}$, the following properties
hold: 
\begin{itemize}
\item[(a)] $\widetilde{{\cal L}}_{c}(\cdot;p)$ is convex, differentiable, and
its gradient is given by 
\[
\nabla_{z}\widetilde{{\cal L}}_{c}(z;p)=\nabla f(z)+\nabla g(z)\Pi_{\cK^{+}}(p+cg(z))\quad\forall z\in\rn;
\]
\item[(b)] $\widetilde{{\cal L}}_{c}(\cdot;p)\in{\cal C}_{m,\widetilde{L}}(Z)$
where 
\begin{equation}
\widetilde{L}=\widetilde{L}(c,p):=M+L_{g}\|p\|+c\left(B_{g}^{(0)}L_{g}+[B_{g}^{(1)}]^{2}\right),\label{eq:Lipsc_tilde_def}
\end{equation}
and the quantities $L_{g}$ and $(B_{g}^{(0)},B_{g}^{(1)})$ are as
in \ref{asmp:cnco_b2} and \eqref{eq:bd_Psi_val}, respectively. 
\end{itemize}
\end{lem}

\begin{proof}
We first state that the case of $f\equiv0$ and $M=0$ has been previously
shown in \citep[Proposition 5]{Lu2018} under the condition that $B_{g}^{(1)}$
is a Lipschitz constant of $g$. Hence, in view of assumption \ref{asmp:cnco_b2}
and the definition of ${\cal L}_{c}$, it suffices to verify the aforementioned
condition. Indeed, using the Mean Value Inequality and the definition
of $B_{g}^{(1)}$ in \eqref{eq:bd_Psi_val} we have that 
\[
\|g(z')-g(z)\|\leq\sup_{t\in[0,1]}\|\nabla g(tz'+[1-t]z)\|\cdot\|z'-z\|\leq B_{g}^{(1)}\|z'-z\|\quad\forall z',z\in{\cal H},
\]
and hence that $g$ is $B_{g}^{(1)}$-Lipschitz continuous. 
\end{proof}
The next result, whose proof can be found in \prettyref{app:ref_props},
describes how the refinement procedure in \prettyref{alg:cref} yields
a pair $([\hat{z},\hat{p}],[\hat{v},\hat{q}])$ that \emph{nearly}
solves \prettyref{prb:approx_cnco_b} when given inputs that satisfy
conditions similar to \eqref{eq:err_crit_GIPP} and \eqref{eq:rho_eps_approx}.
\begin{prop}
\label{prop:al_refinement}Given $(c,\sigma)\in\r_{++}^{2}$, $(\lam,z,p^{-})\in\r_{++}\times Z\times\r^{\ell}$,
and $(f,h)$ satisfying assumptions \ref{asmp:nco1}--\ref{asmp:nco2},
suppose there exists $\bar{\rho}\geq0$ and $(z^{-},\tilde{v},\tilde{\varepsilon})\in Z\times{\cal Z}\times\r_{+}$
such that 
\begin{equation}
\begin{gathered}\tilde{v}\in\pt_{\tilde{\varepsilon}}\left(\lam\left[\widetilde{{\cal L}}_{c}(\cdot,p)+h\right]+\frac{1}{2}\|\cdot-z^{-}\|^{2}\right)(z)\\
\|\tilde{v}\|^{2}+2\tilde{\varepsilon}\leq\sigma\|z^{-}-z+\tilde{v}\|^{2},\quad\frac{1}{\lam}\|z^{-}-z+\tilde{v}\|^{2}\leq\bar{\rho},
\end{gathered}
\label{eq:hpe_refine_ineq}
\end{equation}
where $\widetilde{{\cal L}}_{c}(\cdot,\cdot)$ is as in \eqref{eq:L_tilde_def}.
Moreover, using $\widetilde{L}(\cdot,\cdot)$ in \eqref{eq:Lipsc_tilde_def},
define
\begin{gather}
\begin{gathered}L^{\psi}:=\lam\widetilde{L}(c,p^{-})+1,\quad p:=\Pi_{\cK^{+}}\left(p^{-}+cg(z)\right),\\
\hat{p}:=\Pi_{{\cal K}^{+}}\left(p^{-}+cg(\hat{z})\right),\quad\hat{q}:=\frac{1}{c}(p-\hat{p}),
\end{gathered}
\label{eq:refined_points}
\end{gather}
and using \prettyref{alg:cref}, consider the assigned triple
\[
(\hat{z},\hat{w},\hat{v},\varepsilon)\gets\text{CREF}(\widetilde{{\cal L}}_{c}(\cdot,p),h,z,L^{\psi},\lam).
\]
Then, the following properties hold:
\begin{itemize}
\item[(a)] the tuple $(w,\varepsilon,p,L^{\psi})$ satisfies
\begin{equation}
\begin{gathered}\hat{w}\in\nabla f(z)+\pt_{\varepsilon}h(z)+\nabla g(z)p,\\
\|\hat{w}\|\leq\left(1+\sqrt{\sigma L^{\psi}}\right)\bar{\rho},\quad\varepsilon\leq\frac{\sigma}{2}\bar{\rho}^{2};
\end{gathered}
\label{eq:weak_refine}
\end{equation}
\item[(b)] the tuples $(\hat{z},\hat{p},\hat{v},\hat{q})$ and $(p,L^{\psi})$
satisfy \eqref{eq:cone_approx_soln} and
\begin{equation}
\begin{gathered}\|\hat{v}\|\leq2\left(1+\sqrt{\sigma L^{\psi}}\right)\bar{\rho},\quad\|\hat{q}\|\leq\frac{B_{g}^{(1)}}{L^{\psi}}\left(1+\sqrt{\sigma L^{\psi}}\right)\bar{\rho}+\frac{1}{c}\|p-p^{-}\|,\end{gathered}
\label{eq:strong_refine}
\end{equation}

where $B_{g}^{(1)}$ is given by \eqref{eq:bd_Psi_val}.

\end{itemize}
\end{prop}

\begin{proof}
(a) The inclusion follows from \prettyref{prop:crp_props}(a) with
$(z_{r},q_{r})=(\hat{z},\hat{w})$ and $f=\widetilde{{\cal L}}_{c}(\cdot;p^{-})$,
\prettyref{lem:dist_smoothness}(a) and the definition of $p$ in
\eqref{eq:refined_points}. To show the bound on $\varepsilon$, observe
that its definition and the inequalities in \eqref{eq:hpe_refine_ineq}
yield 
\[
\varepsilon=\frac{1}{\lam}\tilde{\varepsilon}\leq\frac{\sigma}{2\lam}\|z^{-}-z+\tilde{v}\|^{2}\leq\frac{\sigma}{2}\bar{\rho}^{2}.
\]
To show that \prettyref{prop:crp_props}(c) with $(L_{\lam},q_{r})=(L^{\psi},\hat{w})$
and $\bar{\varepsilon}=\sigma\bar{\rho}^{2}/(2\lam)$ imply that 
\[
\|\hat{w}\|\leq\bar{\rho}+\sqrt{L^{\psi}\sigma\bar{\rho}^{2}}=\left(1+\sqrt{L^{\psi}\sigma}\right)\bar{\rho}.
\]

(b) The inclusion in \eqref{eq:cone_approx_soln} follows from \prettyref{prop:crp_props}(b)
with $(z_{r},v_{r})=(\hat{z},\hat{v})$ and $f=\widetilde{{\cal L}}_{c}(\cdot;p^{-})$,
\prettyref{lem:dist_smoothness}(a), and the definition of $\hat{p}$
in \eqref{eq:refined_points}. To show the remaining relations in
\eqref{eq:cone_approx_soln}, observe that \prettyref{lem:dist_props}(b)
with $u=p^{-}+cg(\hat{z})$ and the definitions of $\hat{q}$ and
$\hat{p}$ in \eqref{eq:refined_points} imply that 
\[
g(\hat{z})+\hat{q}=\frac{1}{c}\left[p^{-}+cg(\hat{z})-\hat{p}\right]\in N_{\cK^{+}}(\hat{p}).
\]
Combining the above relations and \prettyref{lem:dist_props}(c) with
$u=g(\hat{z})+\hat{q}$ and $p=\hat{p}$, we conclude that the remaining
relations in \eqref{eq:strong_refine} hold.

We now show the bounds in \eqref{eq:strong_refine}. The bound on
$\|\hat{v}\|$ follows immediately from part (a) and \prettyref{prop:crp_props}(a)
with $(z_{r},v_{r},q_{r},f)=(\hat{z},\hat{v},\hat{w},\widetilde{{\cal L}}_{c}(\cdot;p^{-}))$.
To show the bound on $\hat{q}$, we first use the definitions of $\hat{p}$
and $p$ in \eqref{eq:refined_points}, the definition of $B_{g}^{(1)}$
given by \eqref{eq:bd_Psi_val}, the Mean Value Inequality, and \prettyref{lem:dist_props}(a)
to obtain 
\begin{align}
\frac{1}{c}\|\hat{p}-p\| & =\frac{1}{c}\left\Vert \Pi_{{\cal K}^{+}}\left(p^{-}+cg(\hat{z})\right)-\Pi_{{\cal K}^{+}}\left(p^{-}+cg(z)\right)\right\Vert \leq\frac{1}{c}\|cg(\hat{z})-cg(z)\|\nonumber \\
 & \leq\sup_{t\in[0,1]}\|\nabla g(t\hat{z}+[1-t]z)\|\cdot\|\hat{z}-z\|\leq B_{g}^{(1)}\|\hat{z}-z\|.\label{eq:delta_phat_bd}
\end{align}
Now, since $w=L^{\psi}(\hat{z}-z)$ (see the definition of $q_{r}$
in \prettyref{alg:cref}), it follows from the triangle inequality,
the definition of $\hat{q}$ given in \eqref{eq:refined_points},
part (a), and \eqref{eq:delta_phat_bd}, that
\begin{align*}
\|\hat{q}\| & =\frac{1}{c}\|\hat{p}-p^{-}\|\leq\frac{1}{c}\|\hat{p}-p\|+\frac{1}{c}\|p-p^{-}\|\\
 & \leq B_{g}^{(1)}\|\hat{z}-z\|+\frac{1}{c}\|p-p^{-}\|\leq\frac{B_{g}^{(1)}}{L^{\psi}}\|w\|+\frac{1}{c}\|p-p^{-}\|\\
 & \leq\frac{B_{g}^{(1)}}{L^{\psi}}\left(1+\sqrt{\sigma L^{\psi}}\right)\bar{\rho}+\frac{1}{c}\|p-p^{-}\|.
\end{align*}
\end{proof}
Two comments about \prettyref{prop:al_refinement} are in order. First,
the relations in (a) will be used to establish the boundedness of
the sequence $\{p_{k}\}_{k\geq1}$. Second, in view of (b), the quadruple
$(\hat{z},\hat{p},\hat{w},\hat{q})$ always satisfies \eqref{eq:cone_approx_soln}.
Hence, in order to solve \prettyref{prb:approx_cnco_b}, it remains
only to guarantee that condition \eqref{ineq:cone_approx_soln} will
eventually be satisfied. The inequalities in \eqref{eq:strong_refine}
will be essential to show the latter fact. 

\subsection{Statement and Properties of the AIP.ALM}

This subsection describes and establishes the iteration complexity
of the AIP.ALM.

Before presenting the method, we present an ACG subroutine in \prettyref{ln:al_acgm_call}
that is used to approximate solve its key subproblems.

\begin{mdframed}
\mdalgcaption{ACGM Instance for the AIP.ALM}{alg:al_acgm}
\begin{smalgorithmic}
	\Require{$\sigma \geq 0, \enskip (\mu,L)\in\r_{++}^2, \enskip \psi_n \in \cConv({\cal Z}), \enskip \psi_n \in {\cal F}_{\mu,L}(Z), \enskip y_0 \in Z$;}
	\vspace*{.5em}
	\Procedure{ACG3}{$\psi_s, \psi_n, y_0, \sigma, \mu, L$}
	\For{$k=1,...$}
		\StateEq{$\lam_k \gets 1/L$}
		\StateEq{Generate $(A_k, y_k, r_k, \eta_k)$ according to \prettyref{alg:acgm}.}
		\If{$\|r_k\|^2 + 2\eta_k \leq \sigma \|y_0 - y_k + r_k\|^2$}
			\StateEq{\Return{$(y_k, r_k, \eta_k)$}}
		\EndIf
	\EndFor
	\EndProcedure
\end{smalgorithmic}
\end{mdframed}

We now state the AIP.ALM in \prettyref{alg:aip_alm}, which uses the
ACG subroutine in \prettyref{alg:al_acgm}, the refinement procedure
in \prettyref{alg:cref}, and the Lagrangian ${\cal L}_{c}(\cdot;\cdot)$
in \eqref{eq:aug_lagr_def}. Given $(\lam,\theta)\in\r_{++}\times(0,1/\sqrt{2}]$
and $z_{0}\in Z$, its main idea is to invoke at its $k^{{\rm th}}$
iteration an ACG variant to obtain the inexact update
\[
z_{k}\approx\argmin_{u}\left\{ \lam{\cal L}_{c_{k}}(u;p_{k-1})+\frac{1}{2}\|u-z_{k-1}\|^{2}\right\} .
\]
More specifically, this ACG call obtains a triple $(z_{k},v_{k},\varepsilon_{k})$
satisfying

\begin{equation}
\begin{gathered}v_{k}\in\partial_{\varepsilon_{k}}\left(\lambda\left[\widetilde{{\cal L}}_{c_{k}}(\cdot;p_{k-1})+h\right]+\frac{1}{2}\|\cdot-z_{k-1}\|^{2}\right)(z_{k}),\\
\|v_{k}\|^{2}+2\varepsilon_{k}\leq\frac{\theta^{2}}{L_{k-1}^{\psi}}\|v_{k}+z_{k-1}-z_{k}\|^{2}
\end{gathered}
\label{eq:prox_incl}
\end{equation}
where 
\begin{equation}
L_{k-1}^{\psi}=\lam\widetilde{L}(c_{k},p_{k-1})+1,\label{eq:LPsi_k_def}
\end{equation}
and $\widetilde{L}(\cdot,\cdot)$ and $\widetilde{{\cal L}}_{c_{k}}(\cdot;\cdot)$
are as in \eqref{eq:Lipsc_tilde_def} and \eqref{eq:L_tilde_def},
respectively. Using this triple $(z_{k},v_{k},\varepsilon_{k})$,
and the available data $(\lam,z_{k-1},p_{k-1},\theta,L_{k-1}^{\psi})$,
it then generates a refined point $(\hat{z},\hat{p},\hat{v},\hat{q})=(\hat{z}_{k},\hat{p}_{k},\hat{v}_{k},\hat{q}_{k})$
satisfying all the conditions in \eqref{eq:cone_approx_soln}. If
this quadruple also satisfies the bounds in \eqref{ineq:cone_approx_soln},
then the method stops and outputs $(\hat{z},\hat{p},\hat{v},\hat{q})$.
Otherwise, $p_{k}$ is updated according to 
\[
p_{k}=\Pi_{\cK^{+}}(p_{k-1}+c_{k}g(z_{k})),
\]
a novel test is invoked to check if $c_{k}$ needs to be doubled,
and the method continues to the $(k+1)^{{\rm th}}$ iteration.

\begin{mdframed}
\mdalgcaption{AIP.AL Method}{alg:aip_alm}
\begin{smalgorithmic}
	\Require{$(\hat{\rho}, \hat{\eta}) \in \r_{++}^2, \enskip (m,M)\in\r_{+}^3, \enskip L_g > 0, \enskip h \in \cConv(Z), \enskip f \in {\cal C}_{m,M}(Z), \enskip$ $g \text{ satisfying \ref{asmp:cnco_b2}}, \enskip \lam \in (0, 1 / m), \enskip \theta \in (0,1/\sqrt{2}], \enskip c_1 > 0, \enskip (z_0,p_0) \in Z \times \r^{\ell}, \enskip \cK \subseteq \r^{\ell}$;}
	\Initialize{$\mu \gets 1 - \lam m, \enskip \hat{k} \gets 0;$}
	\vspace*{.5em}
	\Procedure{AIP.AL}{$f, h, g, z_0, p_0, \lam, m, M, L_g, \theta, \hat{\rho}, \hat{\eta}$}
	\For{$k=1,...$}
		\StateStep{\algpart{1}\textbf{Attack} the $k^{\rm th}$ prox subproblem.}
		\StateEq{$\psi_s^k \Lleftarrow \lam \widetilde{\cal L}_{c_k}(\cdot;p_{k-1}) + \|\cdot - z_{k-1}\|^2 / 2$} \Comment{See \eqref{eq:L_tilde_def}.}
		\StateEq{$L^{\psi}_{k-1} \gets \lam \widetilde{L}(c_k, p_{k-1}) + 1$} \Comment{See \eqref{eq:Lipsc_tilde_def}.}
		\StateEq{$\sigma_{k-1} \gets \theta^2 / L^{\psi}_{k-1}$}
		\StateEq{$(z_k, \tilde{v}_k, \tilde{\varepsilon}_k) \gets \text{ACG3}(\psi_s^k, \lam h, z_{k-1}, \sigma_{k-1}, \mu, L^{\psi}_{k-1})$} \label{ln:al_acgm_call}
		\StateStep{\algpart{2}\textbf{Compute} and \textbf{check} the candidate output pair.}
		\StateEq{$(\hat{z}_k, \hat{w}_k, \hat{v}_k, \hat{\varepsilon}_k) \gets \text{CREF}({\cal L}^k, h, z_k, \max\{m,L^{\psi}_{k-1}\}, \lam)$}
		\StateEq{$p_k \gets \Pi_{\cK^{+}}\left(p_{k-1} + c_k g(z_k)\right)$}
		\StateEq{$\hat{p}_k \gets \Pi_{\cK^{+}}\left(p_{k} + c_k g(z_k)\right)$}
		\StateEq{$\hat{q}_k \gets (\hat{p}_k - p_{k-1}) / c_k$}
		\If{$\|\hat{v}_k\| \leq \hat{\rho}$ and $\|\hat{q}_k\|\leq\hat{\eta}$} \label{ln:al_term_check}
			\StateEq{\Return{$([\hat{z}_k, \hat{p}_k], [\hat{v}_k, \hat{q}_k])$}}
		\EndIf
		\StateStep{\algpart{3}\textbf{Check} if we need to increase $c_k$.}
		\StateEq{$\Delta_k \gets \left[{\cal L}_{c_k}(z_{\hat{k}+1};p_{\hat{k}+1}) - {\cal L}_{c_k}(z_{k};p_{k})\right] / (k - \hat{k}+1)$} \label{ln:al_Delta_cond}
		\If{$k > \hat{k} + 1$ and $\Delta_k \leq \lam(1-\theta)\hat{\rho}^2/36$} \label{ln:al_incr_cond}		
			\StateEq{$c_{k+1} \gets 2 c_k$}
			\StateEq{$\hat{k} \gets k$} \label{ln:khat_update}	
		\Else
			\StateEq{$c_{k+1} \gets c_k$}
		\EndIf
	\EndFor
	\EndProcedure
\end{smalgorithmic}
\end{mdframed}

Some remarks about the AIP.ALM are in order. To ease the discussion,
let us refer to the ACG iterations performed in \prettyref{ln:al_acgm_call}
as \textbf{inner iterations} and the iterations over the indices $k$
as\emph{ }\textbf{outer iterations}. First, its input $z_{0}$ can
be any element in $Z$ and does not necessarily need to be a point
satisfying the constraint $g(z_{0})\preceq_{\cK}0$. Second, its ACG
call in \prettyref{ln:al_acgm_call} generates an output $(z_{k},v_{k},\varepsilon_{k})$
that satisfies \eqref{eq:prox_incl}, which corresponds to the approximate
update in \eqref{eq:approx_primal_update}. Finally, in view of \prettyref{prop:al_refinement}(b)
and the comments following it, AIP.AL stops if and only if the quadruple
$(\hat{z},\hat{p},\hat{w},\hat{q})$ solves \prettyref{prb:approx_cnco_b}.

We now discuss the notion of a cycle. Define the $l^{{\rm th}}$ cycle
${\cal K}_{l}$ as the $l^{{\rm th}}$ set of consecutive indices
$k$ for which $c_{k}$ remains constant, i.e. 
\begin{equation}
{\cal C}_{l}:=\left\{ k\geq1:c_{k}=\tilde{c}_{l}:=2^{l-1}c_{1}\right\} .\label{def:ctilde-l}
\end{equation}
For every $l\ge1$, we let $k_{l}$ denote the largest index in ${\cal C}_{l}$.
Hence, 
\[
{\cal C}_{l}=\left\{ k_{l-1}+1,\ldots,k_{l}\right\} \quad\forall l\ge1
\]
where $k_{0}:=0$. Clearly, the different values of $\hat{k}$ that
arise in \prettyref{ln:khat_update} are exactly the indices in the
index set $\{k_{l}\}_{l\geq1}$. Moreover, in view of the test performed
in \prettyref{ln:al_incr_cond}, we have that $k_{l}-k_{l-1}\ge2$
for every $l\ge1$, or equivalently, every cycle contains at least
two indices. While generating the indices in the $l^{{\rm th}}$ cycle,
if an index $k\ge k_{l-1}+2$ satisfying the bound on $\Delta_{k}$
in \prettyref{ln:al_incr_cond} is found, $k$ becomes the last index
$k_{l}$ in the $l$-th cycle and the $(l+1)^{{\rm th}}$ cycle is
started at iteration $k_{l}+1$ with the penalty parameter set to
$\tilde{c}_{l+1}=2\tilde{c}_{l}$, where $\tilde{c}_{l}$ is as in
\eqref{def:ctilde-l}.

The following result, whose proof is deferred to \prettyref{subsec:al_technical},
describes the inner iteration complexity of AIP.AL. Its iteration
complexity bound is expressed in terms of its inputs and the following
auxiliary constants: 
\begin{gather}
\bar{d}:=\dist(\bar{z},\partial Z),\quad\phi_{*}:=\inf_{z\in{\cal Z}}\phi(z),\quad R_{\phi}:=\hat{\varphi}_{*}-\phi_{*}+\frac{1}{\lambda}D_{z}^{2},\label{eq:aux_dist_consts}\\
\kappa_{0}:=2\left[K_{h}+B_{f}^{(1)}\right]D_{z}+\left[\frac{\theta^{2}}{2(1-\theta)^{2}}+2\left(\frac{1+\theta}{1-\theta}\right)\right]\frac{D_{z}^{2}}{\lambda},\label{def:kappa00}\\
\kappa_{1}:=M+\frac{\kappa_{0}L_{g}}{\min\{1,\bar{d}\}\tau},\quad\kappa_{2}:=B_{g}^{(0)}L_{g}+[B_{g}^{(1)}]^{2},\label{def:kappa1n2}\\
\kappa_{3}:=\frac{16(1+2\theta)^{2}\max\{\|p_{0}\|,\kappa_{0}\}^{2}}{\lambda(1-\theta^{2})\min\{1,\bar{d}^{2}\}\tau^{2}},\quad\kappa_{4}:=\frac{\theta D_{z}}{\lam(1-\sigma)B_{g}^{(1)}}+\frac{2\max\{\|p_{0}\|,\kappa_{0}\}^{2}}{\min\{1,\bar{d}\}\tau},\label{def:kappa3n4}
\end{gather}
where $\hat{\varphi}_{*}$, $(B_{g}^{(0)},B_{g}^{(1)},B_{f}^{(1)})$,
$(K_{h},D_{z})$, $M$, $L_{g}$, and $(\tau,\bar{z})$ are as in
\ref{prb:eq:cnco_b}, \eqref{eq:bd_Psi_val}, \ref{asmp:cnco_b1},
\ref{asmp:nco2}, \ref{asmp:cnco_b2}, and \ref{asmp:cnco_b3}, respectively,
and $\pt Z$ denotes the boundary of the set $Z$.
\begin{thm}
\label{thm:alm_compl}Let the scalars $\{\kappa_{i}\}_{i=1}^{4}$
and $R_{\phi}$ be as in \eqref{eq:aux_dist_consts}, \eqref{def:kappa1n2},
and \eqref{def:kappa3n4}. Moreover, define
\begin{gather}
\bar{c}(\hat{\rho},\hat{\eta}):=\frac{\kappa_{3}}{\hat{\rho}^{2}}+\frac{\kappa_{4}}{\hat{\eta}},\quad{\cal T}_{\hat{\eta},\hat{\rho}}:=[\lam(\kappa_{1}+c_{1}\kappa_{2})+1]\frac{\max\left\{ c_{1},2\bar{c}(\hat{\rho},\hat{\eta})\right\} }{c_{1}}.\label{def:al_compl_consts}
\end{gather}
Then, the AIP.ALM outputs a pair $([\hat{z},\hat{p}],[\hat{v},\hat{q}])$
that solves \prettyref{prb:approx_cnco_b} in
\begin{equation}
{\cal O}\left(\frac{R_{\phi}}{\lam\hat{\rho}^{2}}\sqrt{\frac{{\cal T}_{\hat{\eta},\hat{\rho}}}{1-\lam m}}\log_{1}^{+}\left[\frac{{\cal T}_{\hat{\eta},\hat{\rho}}}{\theta}\right]\right)\label{eq:alm_compl}
\end{equation}
inner iterations, where $\log_{1}^{+}(\cdot):=\max\{\log(\cdot),1\}$.
\end{thm}

The result below presents the iteration complexity of the AIP.ALM
with inputs $\theta=1/\sqrt{2}$ and $\lam=1/(2m)$.
\begin{cor}
Let $T_{\hat{\eta},\hat{\rho}}$, $\hat{\varphi}_{*}$, $\phi_{*}$,
and $D_{z}$ be as in \prettyref{thm:alm_compl}, \ref{prb:eq:cnco_b},
\eqref{eq:aux_dist_consts}, and assumption \ref{asmp:cnco_b1}, respectively.
Then, the AIP.ALM with inputs $\lam=1/(2m)$ and $\theta=1/\sqrt{2}$
outputs a pair $([\hat{z},\hat{p}],[\hat{v},\hat{q}])$ that solves
\prettyref{prb:approx_cnco_b} in 
\begin{equation}
{\cal O}\left(\left[\frac{m\left(\hat{\varphi}_{*}-\phi_{*}\right)+m^{2}D_{z}^{2}}{\hat{\rho}^{2}}\right]\sqrt{{\cal T}_{\hat{\eta},\hat{\rho}}}\log_{1}^{+}{\cal T}_{\hat{\eta},\hat{\rho}}\right)\label{eq:spec_alm_compl}
\end{equation}
inner iterations, where $\log_{1}^{+}(\cdot):=\max\{\log(\cdot),1\}$.
\end{cor}

\begin{proof}
This follows immediately from \prettyref{thm:alm_compl} with $\lam=1/(2m)$
and $\theta=1/\sqrt{2}$ and the definition of $R_{\phi}$ in \eqref{eq:aux_dist_consts}.
\end{proof}
Note that the iteration complexity bound in \eqref{eq:alm_compl}
solely in terms of the tolerance pair $(\hat{\rho},\hat{\eta})$ is
\[
{\cal O}\left(\left[\frac{1}{\sqrt{\hat{\eta}}\hat{\rho}^{2}}+\frac{1}{\hat{\rho}^{3}}\right]\log_{1}^{+}\left[\frac{1}{\hat{\eta}}+\frac{1}{\hat{\rho}^{2}}\right]\right).
\]

\subsection{Proof of \texorpdfstring{\prettyref{thm:alm_compl}}{Key Theorem}}

\label{subsec:al_technical}

This subsection presents several technical results that are needed
to establish \prettyref{thm:alm_compl}. It is divided into two subsections.
The first one presents a bound on the sequence of multipliers $\{p_{k}\}_{k\geq1}$
generated by the AIP.AL, while the second one is devoted to proving
\prettyref{thm:alm_compl}.

\subsubsection*{Bounding on the Sequence of Lagrangian Multipliers}

The goal of this subsection is to show that the sequence of multipliers
$\{p_{k}\}_{k\geq1}$ generated by the AIP.ALM is bounded.

The first result presents a key inclusion and some basic bounds on
several residuals generated by AIP.ALM.
\begin{lem}
\label{lem:inclusion:xik-ineq:wk} Let $\{(\hat{w}_{k},v_{k},z_{k},\varepsilon_{k})\}_{k\geq1}$
and $\{\sigma_{k-1}\}_{k\geq1}$ be generated by the AIP.ALM, and
consider $D_{z}$ and $\tau$ as in assumptions \ref{asmp:cnco_b1}
and \ref{asmp:cnco_b3}, respectively. Moreover, define for every
$k\geq1$, the quantities
\begin{equation}
\xi_{k}=\hat{w}_{k}-\nabla f(z_{k})-\nabla g(z_{k})p_{k},\quad r_{k}=v_{k}+z_{k-1}-z_{k}.\label{eq:aux_al_resids}
\end{equation}
Then, for every $k\geq1$, it holds that 
\begin{equation}
\xi_{k}\in\pt_{\delta_{k}}h(z_{k}),\quad\|r_{k}\|\leq\frac{D_{z}}{1-\theta},\quad\varepsilon_{k}\leq\frac{1}{2}\left(\frac{\theta D_{z}}{1-\theta}\right)^{2},\quad\|\hat{w}_{k}\|\leq\frac{(1+\theta)D_{z}}{\lam(1-\theta)}.\label{ineq:rk-deltak-wk}
\end{equation}
\end{lem}

\begin{proof}
Let $k\geq1$ be fixed. Using \prettyref{prop:al_refinement}(a) and
the definition of $\xi_{k}$ yields the required inclusion. On the
other hand, the definitions of $r_{k}$ and $\sigma_{k-1}$, the inequality
in \eqref{eq:prox_incl}, and the fact that $z_{k},z_{k-1}\in Z$
imply that 
\[
\|r_{k}\|=\|v_{k}+z_{k-1}-z_{k}\|\leq\|v_{k}\|+D_{z}\leq\sigma_{k-1}^{1/2}\|r_{k}\|+D_{z}\leq\theta\|r_{k}\|+D_{z},
\]
which, after a simple re-arrangement, yields the desired bound on
$\|r_{k}\|$. Consequently, the definition of $\varepsilon_{k}$,
the aforementioned bound on $\|r_{k}\|$, the fact that $L_{k-1}^{\psi}\geq1$,
and the inequality in \eqref{eq:prox_incl} gives the bound on $\varepsilon_{k}$.
Finally, the definitions of $w_{k}$ and $\sigma_{k-1}$, the fact
that $\theta\leq1$, and \prettyref{prop:al_refinement}(a) with $\bar{\rho}=\|r_{k}\|/\lam$
and $(\sigma,L^{\psi})=(\sigma_{k-1},L_{k-1}^{\psi})$ yield 
\[
\|\hat{w}_{k}\|\leq\frac{1}{\lam}\left(1+\sqrt{\sigma_{k-1}L_{k-1}^{\psi}}\right)\|r_{k}\|=\frac{1+\theta}{\lam}\|r_{k}\|,
\]
which, combined with the previous bound on $\|r_{k}\|$, gives the
desired bound on $\|\hat{w}_{k}\|$. 
\end{proof}
The next result presents some important properties about the iterates
generated by the AIP.ALM. 
\begin{lem}
\label{lem:sk} Let $\{(z_{k},p_{k},c_{k})\}_{k\geq1}$ be generated
by the AIP.ALM and define, for every $k\geq1$, 
\begin{equation}
s_{k}:=\Pi_{-\cK}(p_{k-1}+c_{k}g(z_{k})).\label{eq:sk_def}
\end{equation}
Then, the following relations hold for every $k\ge1$: 
\begin{gather}
p_{k-1}+c_{k}g(z_{k})=p_{k}+s_{k},\quad\inner{p_{k}}{s_{k}}=0,\quad(p_{k},s_{k})\in\cK^{+}\times(-\cK),\label{eq:sk_moreau}\\
{\cal L}_{c_{k}}(z_{k},p_{k-1})=\phi(z_{k})+\frac{1}{2c_{k}}\left(\|p_{k}\|^{2}-\|p_{k-1}\|^{2}\right).\label{eq:declemma1}
\end{gather}
\end{lem}

\begin{proof}
Let $\cK^{-}$ denote the polar of $\cK$. The two identities in \eqref{eq:sk_moreau}
follow from the definitions of $p_{k}$ and $s_{k}$ in \eqref{eq:dual_update}
and \eqref{eq:sk_def}, respectively, the fact that $(\cK^{+})^{-}=-\cK$,
and \citep[Exercise 2.8]{Ruszczynski2011} with $\cK=\cK^{-}$ and
$x=p_{k-1}+c_{k}g(z_{k})$. On the other hand, using the definitions
of $\mathcal{L}_{c}(\cdot;\cdot)$ and $s_{k}$ in \eqref{eq:aug_lagr_def}
and \eqref{eq:sk_def}, respectively, it holds that
\[
{\cal L}_{c_{k}}(z_{k},p_{k-1})=\phi(z_{k})+\frac{1}{2c_{k}}\left[\|p_{k-1}+c_{k}g(z_{k})-s_{k}\|^{2}-\|p_{k-1}\|^{2}\right]
\]
which, in view of the first identity in \eqref{eq:sk_moreau}, immediately
implies \eqref{eq:declemma1}. 
\end{proof}
The following technical result, whose proof can be found in \citep[Lemma 4.7]{Melo2020},
plays an important role in the proof of \prettyref{prop:mainprop-boundpk}
below.
\begin{lem}
\label{lem:bound_xiN} Let $h$ be a function as in \ref{asmp:cnco_b1}.
Then, for every $u,z\in Z$, $\delta>0$, and $\xi\in\partial_{\delta}h(z)$,
we have 
\[
\|\xi\|{\rm dist}(u,\partial Z)\le\left[{\rm dist}(u,\partial Z)+\|z-u\|\right]K_{h}+\inner{\xi}{z-u}+\delta,
\]
where $\partial Z$ denotes the boundary of the set $Z$. 
\end{lem}

We are now ready to prove the main result of this subsection, namely,
that the sequence $\{p_{k}\}_{k\geq1}$ is bounded. 
\begin{prop}
\label{prop:mainprop-boundpk}Consider the sequence $\{(p_{k},c_{k})\}_{k\geq1}$
generated by the AIP.ALM and let $\kappa_{0}$, $\tau$, and $\bar{d}$
be as in \eqref{def:kappa00}, \ref{asmp:cnco_b3}, and \eqref{eq:aux_dist_consts},
respectively. Then, the following statements hold: 
\begin{itemize}
\item[(a)] for every $k\geq1$, we have 
\begin{align*}
\min\{1,\bar{d}\}\tau\|p_{k}\|+\frac{\|p_{k}\|^{2}}{c_{k}} & \le\kappa_{0}+\frac{1}{c_{k}}\inner{p_{k}}{p_{k-1}};
\end{align*}
\item[(b)] for every $k\geq0$, we have 
\begin{equation}
\|p_{k}\|\leq C_{0}:=\frac{\max\{\|p_{0}\|,\kappa_{0}\}}{\min\{1,\bar{d}\}\tau}.\label{ineq:pkbounded}
\end{equation}
\end{itemize}
\end{prop}

\begin{proof}
(a) Let $k\geq1$ be fixed and $\bar{z}$ be as in \ref{asmp:cnco_b3}.
Moreover, let $(\xi_{k},\varepsilon_{k},z_{k})$ be as in \prettyref{lem:inclusion:xik-ineq:wk}.
It follows from \prettyref{lem:inclusion:xik-ineq:wk} that $\xi_{k}\in\partial_{\varepsilon_{k}}h(z_{k})$
for every $k\geq1$. Hence, assumption \ref{asmp:cnco_b1}, the fact
that $\bar{d}\leq D_{z}$ and $z_{k}\in Z$, \prettyref{lem:bound_xiN}
with $\xi=\xi_{k}$, $z=z_{k}$, $u=\bar{z}$ and $\delta=\varepsilon_{k}$,
and the bound on $\|\varepsilon_{k}\|$ in \prettyref{lem:inclusion:xik-ineq:wk},
imply that 
\begin{equation}
\bar{d}\|\xi_{k}\|\leq2D_{z}K_{h}+\frac{\theta^{2}D_{z}^{2}}{2\lam(1-\theta)^{2}}+\inner{\xi_{k}}{z_{k}-\bar{z}}.\label{ineq:aux-xik00}
\end{equation}
On the other hand, using the assumption that $g$ is $\cK$-convex
(see \ref{asmp:cnco_b2}), the fact that $p_{k}\in\cK^{+}$, the definition
of $\xi_{k}$ in \eqref{eq:aux_al_resids}, the bound on $\|\hat{w}_{k}\|$
in \prettyref{lem:inclusion:xik-ineq:wk}, and the Cauchy-Schwarz
inequality, we conclude that 
\begin{align}
\inner{\xi_{k}}{z_{k}-\bar{z}} & =\inner{\hat{w}_{k}-\nabla f(z_{k})-\nabla g(z_{k})p_{k}}{z_{k}-\bar{z}}\nonumber \\
 & =\inner{\hat{w}_{k}-\nabla f(z_{k})}{z_{k}-\bar{z}}+\inner{p_{k}}{g'(z_{k})(\bar{z}-z_{k})}\nonumber \\
 & \le\inner{\hat{w}_{k}-\nabla f(z_{k})}{z_{k}-\bar{z}}+\inner{p_{k}}{g(\bar{z})-g(z_{k})}\nonumber \\
 & \le B_{f}^{(1)}D_{h}+\frac{(1+\theta)D_{h}^{2}}{\lam(1-\theta)}+\inner{p_{k}}{g(\bar{z})-g(z_{k})}\label{eq:aux_xi_bd1}
\end{align}
where $B_{f}^{(1)}$ is as in \eqref{eq:bd_Psi_val}. Now, defining
\begin{equation}
\kappa:=\left[2K_{h}+B_{f}^{(1)}\right]D_{h}+\left[\frac{\theta^{2}}{2(1-\theta)^{2}}+\frac{1+\theta}{1-\theta}\right]\frac{D_{z}^{2}}{\lambda},\label{def:auxkappa}
\end{equation}
and using \eqref{ineq:aux-xik00}, \eqref{eq:aux_xi_bd1}, together
with the relations in \eqref{eq:sk_moreau}, we conclude that 
\begin{align*}
\bar{d}\|\xi_{k}\|-\inner{p_{k}}{g(\bar{z})} & \le\kappa-\inner{p_{k}}{g(z_{k})}=\kappa-\frac{1}{c_{k}}\inner{p_{k}}{s_{k}+p_{k}-p_{k-1}}\\
 & =\kappa-\frac{\|p_{k}\|^{2}}{c_{k}}+\frac{1}{c_{k}}\inner{p_{k}}{p_{k-1}}
\end{align*}
where $s_{k}$ is as in \eqref{eq:sk_def}. Noting that the definition
of $\xi_{k}$ and the reverse triangle inequality yield 
\[
\|\xi_{k}\|=\|\nabla f(z_{k})-\hat{w}_{k}+\nabla g(z_{k})p_{k}\|\ge-\|\nabla f(z_{k})-\hat{w}_{k}\|+\|\nabla g(z_{k})p_{k}\|,
\]
it follows that 
\begin{align}
\bar{d}\|\nabla g(z_{k})p_{k}\|-\inner{p_{k}}{g(\bar{z})} & \le\kappa-\frac{\|p_{k}\|^{2}}{c_{k}}+\frac{1}{c_{k}}\inner{p_{k}}{p_{k-1}}+\bar{d}\|\nabla f(z_{k})-\hat{w}_{k}\|.\label{eq:aux_xi_bd2}
\end{align}
Using now the triangle inequality, assumption \ref{asmp:cnco_b3},
\eqref{def:auxkappa}, \eqref{eq:aux_xi_bd2}, the fact that $\bar{d}\leq D_{z}$,
and the definition of $\kappa_{0}$ in \eqref{def:kappa00}, we finally
conclude that 
\[
\min\{1,\bar{d}\}\tau\|p_{k}\|+\frac{\|p_{k}\|^{2}}{c_{k}}\le\kappa+B_{f}^{(1)}D_{h}+\frac{(1+\sigma)D_{h}^{2}}{\lam(1-\sigma)}+\frac{1}{c_{k}}\inner{p_{k}}{p_{k-1}}=\kappa_{0}+\frac{1}{c_{k}}\inner{p_{k}}{p_{k-1}}.
\]

(b) This statement is proved by induction. Since $\tau\leq1$, inequality
\eqref{ineq:pkbounded} trivially holds for $k=0$. Assume that \eqref{ineq:pkbounded}
holds with $k=i-1$ for some $i\geq1$. This assumption, together
with the bound obtained in the latter result and the Cauchy-Schwarz
inequality, then imply that 
\begin{align*}
\left(\min\{1,\bar{d}\}\tau+\frac{\|p_{i}\|}{c_{i}}\right)\|p_{i}\| & \leq\kappa_{0}+\frac{\|p_{i}\|\cdot\|p_{i-1}\|}{c_{i}}\leq\kappa_{0}+\frac{\|p_{i}\|C_{0}}{c_{i}}\\
 & \leq\left(\min\{1,\bar{d}\}\tau+\frac{\|p_{i}\|}{c_{i}}\right)C_{0},
\end{align*}
which implies that $\|p_{i}\|\leq C_{0}$. Then, \eqref{ineq:pkbounded}
also holds with $k=i$ and hence, by induction, we conclude that \eqref{ineq:pkbounded}
holds for the whole sequence $\{p_{k}\}_{k\geq1}$. 
\end{proof}

\subsubsection*{Proving \prettyref{thm:alm_compl}}

The main goal of this sub-subsection is to present the proof of \prettyref{thm:alm_compl}.

The proof of \prettyref{thm:alm_compl} requires several technical
results. The first one characterizes the change in the augmented Lagrangian
between consecutive iterations of the AIP.ALM.
\begin{lem}
The sequence $\{(z_{k},p_{k})\}_{k\geq1}$ generated by AIP.AL satisfies
the relations 
\begin{align}
{\cal L}_{c_{k}}(z_{k};p_{k}) & \leq{\cal L}_{c_{k}}(z_{k};p_{k-1})+\frac{1}{c_{k}}\|p_{k}-p_{k-1}\|^{2},\label{eq:declemma2}\\
{\cal L}_{c_{k}}(z_{k};p_{k}) & \leq{\cal L}_{c_{k}}(z_{k-1};p_{k-1})-\left(\frac{1-\theta^{2}}{2\lambda}\right)\|r_{k}\|^{2}+\frac{1}{c_{k}}\|p_{k}-p_{k-1}\|^{2},\label{eq:declemma3}
\end{align}
for every $k\geq1$, where $r_{k}$ is as in \eqref{eq:refine_aux_defs}. 
\end{lem}

\begin{proof}
Let $s_{k}$ be as in \eqref{eq:sk_def}. Using \eqref{eq:declemma1},
the definition of $\mathcal{L}_{c}(\cdot;\cdot)$ in \eqref{eq:aug_lagr_def},
the fact that $s_{k}\in-\cK$ and $p_{k-1}+c_{k}g(z_{k})=p_{k}+s_{k}$
in view of \eqref{eq:sk_moreau}, we have that 
\begin{align*}
 & {\cal L}_{c_{k}}(z_{k},p_{k})-{\cal L}_{c_{k}}(z_{k},p_{k-1})\\
 & ={\cal L}_{c_{k}}(z_{k},p_{k})-\phi(z_{k})-\frac{1}{2c_{k}}\left(\|p_{k}\|^{2}-\|p_{k-1}\|^{2}\right)\\
 & =\frac{1}{2c_{k}}\left({\rm dist}^{2}(p_{k}+c_{k}g(z_{k}),-{\cal K})-\|p_{k}\|^{2}\right)-\frac{1}{2c_{k}}\left(\|p_{k}\|^{2}-\|p_{k-1}\|^{2}\right)\\
 & \le\frac{1}{2c_{k}}\left(\|p_{k}+c_{k}g(z_{k})-s_{k}\|^{2}-\|p_{k}\|^{2}\right)-\frac{1}{2c_{k}}\left(\|p_{k}\|^{2}-\|p_{k-1}\|^{2}\right)\\
 & =\frac{1}{2c_{k}}\left(\|2p_{k}-p_{k-1}\|^{2}-2\|p_{k}\|^{2}+\|p_{k-1}\|^{2}\right),
\end{align*}
which immediately implies \eqref{eq:declemma2}. Now, in view of the
definition of the approximate subdifferential and the fact that $(z_{k},v_{k},\varepsilon_{k})$
satisfies both the inclusion and the inequality in \eqref{eq:prox_incl},
we conclude that 
\begin{align}
 & \lambda{\cal L}_{c_{k}}(z_{k},p_{k-1})-\lambda{\cal L}_{c_{k}}(z_{k-1},p_{k-1})\leq-\frac{1}{2}\|z_{k}-z_{k-1}\|^{2}+\inner{v_{k}}{z_{k}-z_{k-1}}+\varepsilon_{k}\nonumber \\
 & =-\frac{1}{2}\|v_{k}+z_{k}-z_{k-1}\|^{2}+\frac{1}{2}\|v_{k}\|^{2}+\varepsilon_{k}\leq-\left(\frac{1-\sigma_{k-1}}{2}\right)\|r_{k}\|^{2}\le-\left(\frac{1-\theta^{2}}{2}\right)\|r_{k}\|^{2},\label{eq:aux_DeltaLagr_bd1}
\end{align}
where the last inequality follows from the fact that $\sigma_{k-1}\leq\theta$.
Inequality \eqref{eq:declemma3} now follows by combining \eqref{eq:declemma2}
with \eqref{eq:aux_DeltaLagr_bd1}. 
\end{proof}
Recall that the $l^{{\rm th}}$ cycle ${\cal C}_{l}$ and the penalty
constants $\{\tilde{c}_{l}\}_{l\geq1}$ are defined in \eqref{def:ctilde-l}.
The next results present some properties of the iterates generated
during an AIP.AL cycle. The first one below establishes an upper bound
on the augmented Lagrangian function along the iterates within an
AIP.AL cycle.
\begin{lem}
Consider the sequences $\{(z_{k},p_{k})\}_{k\in{\cal C}_{l}}$ and
$\{\tilde{c}_{l}\}_{l\geq1}$ generated during the $l^{{\rm th}}$
cycle of the AIP.ALM. Then, for every $k\in{\cal C}_{l}$, we have
\begin{equation}
{\cal L}_{\tilde{c}_{l}}(z_{k};p_{k})\leq R_{\phi}+\phi_{*}+\frac{4C_{0}^{2}}{\tilde{c}_{l}},\label{aux:ineqTz1p0}
\end{equation}
where $(\phi_{*},R_{\phi})$, $\tilde{c}_{l}$, and $C_{0}$ are as
in \eqref{eq:aux_dist_consts}, \eqref{def:ctilde-l}, and \eqref{ineq:pkbounded},
respectively. 
\end{lem}

\begin{proof}
First note that for any $k\in{\cal C}_{l}$, we have $c_{k}=\tilde{c}_{l}=2^{l-1}c_{1}$.
Moreover, $(\lambda,z_{k},v_{k},\varepsilon_{k},\theta)$ satisfies
the inclusion and the inequality in \eqref{eq:prox_incl}. Hence,
it follows from \prettyref{lem:auxNewNest2} with $s=1$, $\tilde{\sigma}=\sigma_{k-1}$
and $\tilde{\phi}=\lambda{\cal L}_{\tilde{c}_{l}}(\cdot,p_{k-1})$,
and assumption \ref{asmp:cnco_b1} that for every $z\in Z$, we have
\begin{align}
\lambda{\cal L}_{\tilde{c}_{l}}(z_{k},p_{k-1})+\frac{1-2\sigma_{k-1}^{2}}{2}\|r_{k}\|^{2} & \leq\lambda{\cal L}_{\tilde{c}_{l}}(z,p_{k-1})+\|z-z_{k-1}\|^{2}\nonumber \\
 & \leq\lambda{\cal L}_{\tilde{c}_{l}}(z,p_{k-1})+D_{z}^{2}\label{ineq:Lz1p0}
\end{align}
where $r_{k_{0}}$ is as in \eqref{eq:aux_al_resids} with $k=k_{0}$.
Now, observe that the definitions of $\sigma_{k-1}$ and $L_{k-1}^{\psi}$
imply that $\sigma_{k-1}\le\theta\in(0,1/\sqrt{2}]$ and that the
definition of ${\cal L}_{c}$ in \eqref{eq:aug_lagr_def} implies
that ${\cal L}_{\tilde{c}_{l}}(z,p_{k-1})\leq\phi(z)$ for every $z\in\mathcal{F}:=\{z\in Z:g(z)\preceq_{\mathcal{K}}0\}$.
Using then the definition of $\hat{\varphi}_{*}$ given in \ref{prb:eq:cnco_b},
the aforementioned observations, and the minimization of the right-hand-side
of \eqref{ineq:Lz1p0} with respect to $z\in{\cal F}$, we get 
\[
{\cal L}_{\tilde{c}_{l}}(z_{k},p_{k-1})\leq\hat{\varphi}_{*}+\frac{D_{h}^{2}}{\lambda}=R_{\phi}+\phi_{*}
\]
where the last equality is due to the definition of $R_{\phi}$ in
\eqref{eq:aux_dist_consts}. Combining the above inequality, \eqref{eq:declemma2}
and the bound $(a+b)^{2}\leq2a^{2}+2b^{2}$ for every $a,b\in\r$,
we have 
\begin{align*}
{\cal L}_{\tilde{c}_{l}}(z_{k},p_{k}) & \leq{\cal L}_{\tilde{c}_{l}}(z_{k},p_{k-1})+\frac{1}{\tilde{c}_{l}}\|p_{k}-p_{k-1}\|^{2}\\
 & \leq{\cal L}_{\tilde{c}_{l}}(z_{k},p_{k-1})+\frac{2}{\tilde{c}_{l}}\left(\|p_{k}\|^{2}+\|p_{k-1}\|^{2}\right)\\
 & \leq R_{\phi}+\phi_{*}+\frac{4C_{0}^{2}}{\tilde{c}_{l}},
\end{align*}
and hence the conclusion of the lemma follows. 
\end{proof}
The next result presents some bounds on the sequences $\{\|r_{k}\|\}_{k\in{\cal C}_{l}}$
and $\{\Delta_{k}\}_{k\in{\cal C}_{l}}$.
\begin{lem}
\label{lem:minrk-Deltak} Let $\{(z_{k},v_{k},\varepsilon_{k},\Delta_{k})\}_{k\in{\cal C}_{l}}$
and $\{\tilde{c}_{l}\}_{l\geq1}$ be generated during the $l^{{\rm th}}$
cycle of the AIP.ALM and consider $\{r_{k}\}_{k\in{\cal C}_{l}}$
as in \eqref{eq:aux_al_resids}. Then, for every $k\in{\cal C}_{l}$
such that $k\geq k_{l-1}+2$, we have 
\begin{equation}
\min_{k_{l-1}+2\leq j\leq k}\|r_{j}\|^{2}\leq\frac{2\lambda}{1-\theta^{2}}\left(\Delta_{k}+\frac{4C_{0}^{2}}{\tilde{c}_{l}}\right),\label{main-ineq-dreasingLag}
\end{equation}
\begin{equation}
\Delta_{k}\le\frac{1}{k-k_{l-1}-1}\left(R_{\phi}+\frac{9C_{0}^{2}}{2\tilde{c}_{l}}\right),\label{ineq:Deltak}
\end{equation}
where $C_{0}$ is as in \eqref{ineq:pkbounded}.
\end{lem}

\begin{proof}
Relations \eqref{ineq:pkbounded}, \eqref{eq:declemma3}, the fact
that $c_{k}=\tilde{c}_{l}$ for every $k\in{\cal C}_{l}$, and the
inequality $\|p_{k}-p_{k-1}\|^{2}\leq2\|p_{k}\|^{2}+2\|p_{k-1}\|^{2}$,
imply that for any $k\in{\cal C}_{l}$ such that $k\geq k_{l-1}+2$
the following inequalities hold: 
\begin{align*}
 & \frac{(1-\theta^{2})(k-k_{l-1}-1)}{2\lambda}\min_{k_{l-1}+2\leq j\leq k}\|r_{j}\|^{2}\leq\frac{(1-\theta^{2})}{2\lambda}\sum_{j=k_{l-1}+2}^{k}\|r_{j}\|^{2}\\
 & \leq{\cal L}_{\tilde{c}_{l}}(z_{k_{l-1}+1};p_{k_{l-1}+1})-{\cal L}_{\tilde{c}_{l}}(z_{k};p_{k})+\frac{1}{\tilde{c}_{l}}\sum_{j=k_{l-1}+2}^{k}\|p_{j}-p_{j-1}\|^{2}\\
 & \leq{\cal L}_{\tilde{c}_{l}}(z_{k_{l-1}+1};p_{k_{l-1}+1})-{\cal L}_{\tilde{c}_{l}}(z_{k};p_{k})+\frac{4(k-k_{l-1}-1)C_{0}^{2}}{\tilde{c}_{l}},
\end{align*}
and hence that \eqref{main-ineq-dreasingLag} holds, in view of the
definition of $\Delta_{k}$. Now, in view of the definitions of $\mathcal{L}_{c}$
and $\phi_{*}$ given in \eqref{eq:aug_lagr_def} and \eqref{eq:aux_dist_consts},
respectively, we have 
\[
{\cal L}_{\tilde{c}_{l}}(z_{k};p_{k})=\phi(z_{k})+\frac{1}{2\tilde{c}_{l}}\left[{\rm dist}^{2}(p_{k}+\tilde{c}_{l}g(z_{k}),-{\cal K})-\|p_{k}\|^{2}\right]\geq\phi_{*}-\frac{\|p_{k}\|^{2}}{2\tilde{c}_{l}}.
\]
It follows from the above inequality, \eqref{aux:ineqTz1p0} with
$k=k_{l-1}+1$, and the definition of $\Delta_{k}$ that 
\begin{align*}
\Delta_{k}\le\frac{1}{k-k_{l-1}-1}\left(R_{\phi}+\phi_{*}+\frac{4C_{0}^{2}}{\tilde{c}_{l}}+\frac{\|p_{k}\|^{2}}{2\tilde{c}_{l}}-\phi_{*}\right),
\end{align*}
which proves \eqref{ineq:Deltak} in view of \eqref{ineq:pkbounded}. 
\end{proof}
The next technical lemma presents some additional properties of the
refined iterates generated by the AIP.ALM.
\begin{lem}
\label{lem:lastaux}Consider the sequences $\{(c_{k},z_{k},p_{k},v_{k},\varepsilon_{k})\}_{k\in{\cal C}_{l}}$,
$\{(\sigma_{k-1},L_{k-1}^{\psi})\}_{k\in C_{l}}$, $\{\tilde{c}_{l}\}_{l\geq1}$,
and $\{(\hat{z}_{k},\hat{p}_{k},\hat{v}_{k},\hat{q}_{k})\}_{k\in{\cal C}_{l}}$
generated during the $l^{{\rm th}}$ cycle of the AIP.ALM. Then, the
following statements hold: 
\begin{itemize}
\item[(a)] for every $k\in{\cal C}_{l}$, the quadruple $(\hat{z},\hat{p},\hat{v},\hat{q})=(\hat{z}_{k},\hat{p}_{k},\hat{v}_{k},\hat{q}_{k})$
satisfies \eqref{eq:cone_approx_soln} and \eqref{eq:strong_refine}
with 
\[
(c,p^{-},\sigma,L^{\psi}),=(c_{k},p_{k-1},\sigma_{k-1},L_{k-1}^{\psi}),\quad\bar{\rho}=\frac{1}{\lam}\|z_{k-1}-z_{k}+v_{k}\|;
\]
\item[(b)] for every $k\in{\cal C}_{l}$ and $k\geq k_{l-1}+2$, there exists
an index $i\in\{k_{l-1}+2,\ldots,k\}$ such that 
\begin{equation}
\|\hat{v}_{i}\|^{2}\le\frac{2(1+2\sigma)^{2}R_{\phi}}{\lambda(1-\sigma^{2})(k-k_{l-1}-1)}+\frac{\kappa_{3}}{2\tilde{c}_{l}},\quad\|\hat{q}_{i}\|\le\frac{\kappa_{4}}{\tilde{c}_{l}},\label{ineq:boundingDelta-Wk}
\end{equation}
where $R_{\phi}$ and $(\kappa_{3},\kappa_{4})$ are as in \eqref{eq:aux_dist_consts}
and \eqref{def:kappa3n4}, respectively. 
\end{itemize}
\end{lem}

\begin{proof}
(a) In view of the ACG call in \prettyref{ln:al_acgm_call} of the
method, we have that $(\lam,\theta)$, $L_{k-1}^{\psi}$, $(z_{k-1},p_{k-1})$,
and $(z_{k},v_{k},\varepsilon_{k})$ satisfy \eqref{eq:prox_incl}.
The conclusion now follows from \prettyref{prop:al_refinement}(b)--(c).

(b) Let $k\in{\cal C}_{l}$ such that $k\geq k_{l-1}+2$. In view
of \prettyref{lem:minrk-Deltak}, there exists an index $i\in\{k_{l-1}+2,\ldots,k\}$
such that 
\begin{equation}
\|r_{i}\|^{2}\le\frac{2\lambda}{1-\sigma^{2}}\left[\frac{R_{\phi}}{k-k_{l-1}-1}+\frac{4C_{0}^{2}}{\tilde{c}_{l}}\right],\label{eq:r_i_spec_bd}
\end{equation}
where $C_{0}$ is as in \eqref{ineq:pkbounded}. The bound on $\|\hat{w}_{i}\|^{2}$
now follows from combining \eqref{eq:r_i_spec_bd}, the first inequality
in \eqref{eq:strong_refine}, the definitions of $\kappa_{3}$ and
$C_{0}$ in \eqref{def:kappa3n4} and \eqref{ineq:pkbounded}, and
the fact that $\sigma_{k-1}L_{k-1}^{\psi}=\theta^{2}$.

Now, recall that for any $k\in\mathcal{C}$ it holds that $c_{k}=\tilde{c}_{l}$.
Hence, in view of the second inequality in \eqref{eq:strong_refine},
\eqref{ineq:pkbounded}, the triangle inequality for norms, and the
facts that $\sigma_{k-1}L_{k-1}^{\psi}=\theta$ and $L_{k-1}^{\psi}\geq\lam\tilde{c}_{l}[B_{g}^{(1)}]^{2}$
(see their definitions in the AIP.ALM), we have 
\begin{align}
\|\hat{q}_{i}\| & \leq\frac{B_{g}^{(1)}\sigma_{k-1}}{\sqrt{L_{k-1}^{\psi}}}\|r_{i}\|+\frac{1}{\tilde{c}_{l}}\left(\|p_{i}\|+\|p_{i-1}\|\right)\leq\frac{B_{g}^{(1)}\theta}{L_{k-1}^{\psi}}\|r_{i}\|+\frac{2C_{0}}{\tilde{c}_{l}}\nonumber \\
 & \leq\frac{\theta D_{h}}{\lam(1-\theta)B_{g}^{(1)}\tilde{c}_{l}}+\frac{2C_{0}}{\tilde{c}_{l}}=\frac{\kappa_{4}}{\tilde{c}_{l}},\label{ineq:aux0022}
\end{align}
where the last relation is due to the definitions of $\kappa_{4}$
and $C_{0}$ in \eqref{def:kappa3n4} and \eqref{ineq:pkbounded},
respectively. 
\end{proof}
The next result establishes some bounds on the number of inner and
outer iterations performed during an AIP.AL cycle. It also shows that
if the penalty parameter is sufficiently large, then AIP.AL generates
a solution of \prettyref{prb:approx_cnco_b}.
\begin{lem}
\label{lem:StaticIPAAL}Let $R_{\phi}$, $(\kappa_{1},\kappa_{2})$,
and $\bar{c}(\hat{\rho},\hat{\eta})$ be as in \eqref{eq:aux_dist_consts},
\eqref{def:kappa1n2}, and \eqref{def:al_compl_consts}, respectively.
Then, the following statements hold about the AIP.ALM: 
\begin{itemize}
\item[(a)] at every outer iteration $k$ within the $l^{{\rm th}}$ cycle, it
performs at most
\[
\left\lceil 1+4\sqrt{\frac{2\left[1+\lam(\kappa_{1}+\tilde{c}_{l}\kappa_{2})\right]}{1-\lam m}}\ \log_{1}^{+}\left(\frac{4\left[1+\lam(\kappa_{1}+\tilde{c}_{l}\kappa_{2})\right]}{\theta}\right)\right\rceil 
\]
inner iterations, where $\log_{1}^{+}(\cdot):=\max\{\log(\cdot),1\}$; 
\item[(b)] every cycle performs ${\cal O}(R_{\phi}/[\lam\hat{\rho}^{2}])$ outer
iterations; 
\item[(c)] if $\tilde{c}_{l}\geq\bar{c}(\hat{\rho},\hat{\eta})$ then the AIP.ALM
must stop in the $l^{{\rm th}}$ cycle with a pair $([\hat{z},\hat{p}],[\hat{v},\hat{q}])$
that solves \prettyref{prb:approx_cnco_b}. 
\end{itemize}
\end{lem}

\begin{proof}
(a) Note that within the $l^{{\rm th}}$ cycle, $c_{k}=\tilde{c}_{l}$.
Hence, in view of \eqref{ineq:pkbounded} and the definitions of $L_{k-1}^{\psi}$,
($\kappa_{1}$, $\kappa_{2}$), and $\bar{L}_{c}^{\psi}$, we have
\begin{align}
L_{k-1}^{\psi} & =\lam\left[L_{f}+L_{g}\|p_{k-1}\|+c_{k}\left(B_{g}^{(0)}L_{g}+[B_{g}^{(1)}]^{2}\right)\right]+1\nonumber \\
 & \leq\lam\left[L_{f}+L_{g}C_{0}+\tilde{c}_{l}\left(B_{g}^{(0)}L_{g}+[B_{g}^{(1)}]^{2}\right)\right]+1\nonumber \\
 & =\lambda(\kappa_{1}+\kappa_{2}\tilde{c}_{l})+1.\label{eq:auxestimatesLpsi-tc}
\end{align}
Using the fact that the AIP.ALM invokes \prettyref{alg:al_acgm} in
\prettyref{ln:al_acgm_call} with $(L,\mu)=(L_{k-1}^{\psi},1-\lam m)$,
\eqref{eq:auxestimatesLpsi-tc}, the fact that $\sigma_{k-1}=\theta^{2}/L_{k-1}^{\psi}\leq1$,
and \prettyref{lem:nest_complex}, it holds that the number of inner
iterations performed within this cycle is at most
\begin{align*}
 & \left\lceil 1+\sqrt{\frac{2L_{k-1}^{\psi}}{1-\lam m}}\log_{1}^{+}\left(\frac{2L_{k-1}^{\psi}\left[1+\sqrt{\sigma_{k-1}}\right]^{2}}{\sigma_{k-1}}\right)\right\rceil \\
 & \leq\left\lceil 1+\sqrt{\frac{2L_{k-1}^{\psi}}{1-\lam m}}\ \log^{+}\left(16\left[\frac{L_{k-1}^{\psi}}{\theta}\right]^{2}\right)\right\rceil \leq\left\lceil 1+2\sqrt{\frac{2L_{k-1}^{\psi}}{1-\lam m}}\log^{+}\left(\frac{4L_{k-1}^{\psi}}{\theta}\right)\right\rceil \\
 & \leq\left\lceil 1+4\sqrt{\frac{2\left[1+\lam(\kappa_{1}+\tilde{c}_{l}\kappa_{2})\right]}{1-\lam m}}\ \log_{1}^{+}\left(\frac{4\left[1+\lam(\kappa_{1}+\tilde{c}_{l}\kappa_{2})\right]}{\theta}\right)\right\rceil .
\end{align*}

(b) Fix a cycle $l\geq1$ and let $C_{0}$ be as in \eqref{ineq:pkbounded}.
It follows from \eqref{ineq:Deltak} that, for every $k\in\mathcal{C}_{l}$,
we have $k\geq k_{l-1}+2$, and
\[
\Delta_{k}\le\frac{1}{k-k_{l-1}-1}\left(R_{\phi}+\frac{9C_{0}^{2}}{2\tilde{c}_{l}}\right).
\]
Hence, since $\tilde{c}_{l}\geq c_{1}$, it is easy to see that if
$k$ satisfies 
\[
k>k_{l-1}+1+\frac{4(1+2\theta)^{2}}{\lambda(1-\theta^{2})\hat{\rho}^{2}}\left(R_{\phi}+\frac{9C_{0}^{2}}{2c_{1}}\right)
\]
then the condition on $\Delta_{k}$ in \prettyref{ln:al_incr_cond}
of the method holds, ending the $l^{{\rm th}}$ cycle. Since the cycle
starts at $k_{l-1}+1$, statement (b) follows immediately from the
above bound.

(c) From the definition of $\bar{c}(\cdot,\cdot)$ and the fact that
$\tilde{c}_{l}\geq\bar{c}(\cdot,\cdot)$, we have 
\begin{equation}
\tilde{c}_{l}\geq\frac{\kappa_{3}}{\hat{\rho}^{2}},\quad\tilde{c}_{l}\geq\frac{\kappa_{4}}{\hat{\eta}},\label{eq:lower_c_bds}
\end{equation}
where $\kappa_{3}$ and $\kappa_{4}$ are as in \eqref{def:kappa3n4}.
Now, let $\bar{k}\geq k_{l-1}+2$ be the smallest index such that
\begin{equation}
\frac{2(1+2\sigma)^{2}R_{\phi}}{\lambda(1-\theta^{2})(\bar{k}-k_{l-1}-1)}\leq\frac{\hat{\rho}^{2}}{2}.\label{eq:bar_k_bd}
\end{equation}
Hence, in view of \eqref{eq:lower_c_bds}, \eqref{eq:bar_k_bd}, and
\prettyref{lem:lastaux}(b), there exists an index $i\in\{k_{l-1}+2,\ldots,\bar{k}\}$
such that 
\[
\|\hat{v}_{i}\|\le\hat{\rho},\quad\|\hat{q}_{i}\|\le\hat{\eta}
\]
which implies that the AIP.ALM must stop at iteration $i$, in view
its \prettyref{ln:al_term_check}. Hence, the proof of the statement
in (c) follows. 
\end{proof}
We are now ready give the proof of \prettyref{thm:alm_compl}.
\begin{proof}[Proof of \prettyref{thm:alm_compl}]
For a fixed $(\hat{\rho},\hat{\eta})\in\r_{++}^{2}$, first define
\[
\bar{c}=\bar{c}(\hat{\rho},\hat{\eta}),\quad L_{\tilde{c}_{l}}^{\psi}=1+\lam(\kappa_{1}+\tilde{c}_{l}\kappa_{2}),\quad\forall l\geq1,
\]
where $\bar{c}(\cdot,\cdot)$ and $\tilde{c}_{l}$ are as in \eqref{def:al_compl_consts}
and \eqref{def:ctilde-l}, respectively. Moreover, let $\bar{l}$
be the first index $l$ such that $\tilde{c}_{l}\geq\bar{c}$, and
recall from \eqref{def:ctilde-l} that in the $l^{{\rm th}}$ cycle
of the AIP.ALM, we have $c_{k}=\tilde{c}_{l}=2^{l-1}c_{1}$, for every
$l\ge1$. In view of \prettyref{lem:StaticIPAAL}(c), we see that
the AIP.AL obtains a solution of \prettyref{prb:approx_cnco_b} within
the $\bar{l}^{{\rm th}}$ cycle. Moreover, it follows by \prettyref{lem:StaticIPAAL}(a)--(b)
that the total number of inner iterations performed by AIP.ALM is
$\mathcal{O}(T_{I})$ where 
\begin{equation}
T_{I}:=\frac{R_{\phi}}{\lam\hat{\rho}^{2}}\sum_{l=1}^{\bar{l}}\sqrt{\frac{\bar{L}_{\tilde{c}_{l}}^{\psi}}{1-\lam m}}\log_{1}^{+}\left[\frac{\bar{L}_{\tilde{c}_{l}}^{\psi}}{\theta}\right].\label{eq:totalacg}
\end{equation}
Since $c_{k}$ is doubled every time the cycle is changed, we have
in view of the definitions of $\tilde{c}_{l}$ and $\bar{l}$ that
\begin{equation}
\tilde{c}_{l}\leq\max\left\{ c_{1},2\bar{c}\right\} ,\quad\forall l=1,\ldots,\bar{l}.\label{eq:aux-clbar}
\end{equation}
Hence, it holds that
\begin{align}
\bar{L}_{\tilde{c}_{l}}^{\psi} & =1+\lam(\kappa_{1}+\tilde{c}_{l}\kappa_{2})
\nonumber \\
 & 
\leq[\lam(\kappa_{1}+c_{1}\kappa_{2})+1]\frac{\max\left\{ c_{1},2\bar{c}\right\} }{c_{1}}.\label{eq:estimatingThetal}
\end{align}
Moreover, using \eqref{eq:aux-clbar} and the fact that $\tilde{c}_{l}=2^{l-1}c_{1}$,
it holds that 
\begin{align*}
\sum_{l=1}^{\bar{l}}\sqrt{\bar{L}_{\tilde{c}_{l}}^{\psi}} & =\sum_{l=1}^{\bar{l}}\sqrt{\lam(\kappa_{1}+\tilde{c}_{l}\kappa_{2})+1}\leq\sqrt{\lam(\kappa_{1}+c_{1}\kappa_{2})+1}\sum_{l=1}^{\bar{l}}\sqrt{2}^{l-1}\\
 & \leq8\sqrt{\lam(\kappa_{1}+c_{1}\kappa_{2})+1}\left(\frac{\bar{c}_{\bar{l}}}{c_{1}}\right)^{1/2}\\
 & \leq8\sqrt{\lam(\kappa_{1}+c_{1}\kappa_{2})+1}\left(\frac{\max\left\{ c_{1},2\bar{c}\right\} }{c_{1}}\right)^{1/2}.
\end{align*}
Hence, \eqref{eq:alm_compl} then follows by combining \eqref{def:al_compl_consts},
\eqref{eq:totalacg}, \eqref{eq:estimatingThetal}, and the above
inequalities. 
\end{proof}

\section{Conclusion and Additional Comments}

In this chapter, we presented two optimization methods for finding
approximate stationary points for two classes of set-constrained optimization
problems with constraints of the form $g(z)\in S\subseteq{\cal R}$.
More specifically, a quadratic penalty method was proposed for a class
of linear set-constrained NCO problems and a proximal augmented Lagrangian
method was proposed for a class of nonlinearly cone-constrained NCO
problems. We then established ${\cal O}(\hat{\eta}^{-1}\hat{\rho}^{-2})$
and ${\cal O}([\hat{\eta}^{-1/2}\hat{\rho}^{-2}+\hat{\rho}^{-3}]\log_{1}^{+}[\hat{\rho}^{-1}+\hat{\eta}^{-1}])$
iteration complexity bounds, in each of the respective methods, for
finding $\hat{\rho}$-approximate stationary points that are $\hat{\eta}$
feasible, i.e. points $\bar{z}$ satisfying $\dist(g(\bar{z}),S)\leq\hat{\eta}$.

The next chapter continues the developments in \prettyref{chap:unconstr_nco}
to develop a smoothing method for solving min-max NCO problems.

\subsection*{Additional Comments}

We now give some additional comments about the results and assumptions
in this chapter. 

First, it is worth stressing that the regularity condition in assumption
\ref{asmp:cnco_b3}, which is a generalization of the weak Slater
condition (see \prettyref{prop:weak_slater}), is generally easier
to verify compared to other conditions in the literature. For example,
paper \citep{Li2020} requires a regularity condition to hold at every
point generated by their proposed algorithm and paper \citep{Boob2019}
requires either the Mangasarian-Fromovitz constraint qualification
or strong feasibility to hold. It is worth mentioning that we do \textbf{not}
assume any regularity conditions on the linear set constraints in
\prettyref{sec:qp_aipp}.

Second, we comment on the contributions of the AIP.QPM to the literature.
The AIP.QPM and the QP-AIPP method from \citep{Kong2019} appear to
be the first methods to consider an infeasible starting point with
a guaranteed complexity bound under the general assumptions in this
chapter. Moreover, these methods have substantially improved on the
previous state-of-art complexity bound of ${\cal O}(\hat{\rho}^{-6})$
which was obtained in \citep{Jiang2019} under the assumption that
$Z$ is bounded and $\hat{\rho}=\hat{\eta}$. 

Third, we comment on how the AIP.ALM compares with the works \citep{Xie2019,Melo2020,Hong2016,Sahin2019,Li2020,Hajinezhad2019}.
The IAPIAL method of \citep{Melo2020} is designed to solve the special
instance of \ref{prb:eq:cnco_b} in which $\cK=\{0\}$. In contrast
to the AIP.ALM, the IAPIAL method sets $p_{k}$ to $p_{0}$ every
time the penalty parameter $c_{k}$ is increased, and hence it is
not a full warm-start proximal augmented Lagrangian method. Compared
to \citep{Li2020,Sahin2019}, the multiplier update in \eqref{eq:dual_update}
is performed at every prox iteration, regardless of whether the penalty
parameter is updated. Unlike the methods in \citep{Xie2019,Hajinezhad2019,Hong2016},
which require the initial point $z_{0}$ to be feasible, i.e. $g(z_{0})\preceq_{\cK}0$,
the AIP.ALM only requires $z_{0}$ to be in $Z$.

\subsection*{Future Work}

Several recent works present improved complexity bounds (compared
to the ones in this chapter) for obtaining approximate stationary
points of linearly-constrained \citep{Zhang2020,Zhang2020a} and nonlinearly-constrained
\citep{Li2020,Lin2019} NCO problems under different conditions and
multiplier updates. For example, papers \citep{Zhang2020,Zhang2020a}
assume that $h$ is the indicator of a polyhedron and \citep{Li2020}
requires the Lagrange multiplier and penalty updates be performed
simultaneously. It would be worth investigating whether the methods
in this chapter, or some variant of them, can obtain these improved
rates. Comparing the AIP.ALM to the AIP.QPM, the former assumes that
the composite function $h$ has bounded domain and is Lipschitz continuous,
whereas the latter does not. It would be interesting to see if the
AIP.ALM, or some variant of it, can still obtain approximate stationary
points when the above conditions are removed. 

\newpage{}

\chapter{Efficient Implementation Strategies}

\label{chap:practical}

The main goal of this chapter is to present efficient implementation
strategies of some procedures and methods presented in prior chapters
for smooth NCO problems. For the iterative methods, in particular,
the variants in this chapter consider two key improvements. First,
they apply efficient line search subroutines to adaptively choose
parameters that directly affect convergence rates, such as stepsize
parameters. Second, the convex subproblems that are solved in each
of the iterative method are relaxed to (possibly) nonconvex subproblems.
The degree of relaxation in these subproblems is determined by checking
a finite set of novel descent inequalities which are guaranteed to
hold when the subproblems are convex. We then demonstrate the effectiveness
of these strategies on many of optimization problems in the literature.

The content of this chapter is based on paper \citep{Kong2020a} (joint
work with Jefferson G. Melo and Renato D.C. Monteiro) and several
passages may be taken verbatim from it.

\subsection*{Organization}

This chapter contains six sections. The first one presents an efficient
refinement procedure. The second one presents a relaxed ACG variant.
The third one presents a relaxed AIPP variant and its iteration complexity.
The fourth one presents a relaxed AIP.QPM variant and its iteration
complexity. The fifth one presents a large collection of numerical
experiments The last one gives a conclusion and some closing comments.

\section{Proximal Refinement Procedure}

This section presents a refinement procedure that is generally more
effective in practice than the refinement procedure (the CRP) in \prettyref{alg:cref}.

We first state the procedure in \prettyref{alg:pref}, which follows
a similar approach as in the CRP.

\begin{mdframed}
\mdalgcaption{PR Procedure}{alg:pref}
\begin{smalgorithmic}
	\Require{$h \in \cConv({\cal Z}), \enskip f \in {\cal C}(Z), \enskip (z, z^{-}, v) \in {\cal Z}^3, \enskip L > 0, \enskip  \lam > 0$;}
	\Initialize{$L_{\lam} \gets \lam L + 1, \enskip f_{\lam} \gets \lam f + \frac{1}{2}\|\cdot - z^{-}\|^2 - \inner{v}{\cdot}, \enskip h_{\lam} \gets \lam h;$}
	\vspace*{.5em}
	\Procedure{PREF}{$f, h, z, z^{-}, v, L, \lam$}
		\StateEq{$z_{r} \gets \argmin_{u\in{\cal Z}}\left\{ \ell_{f_\lam}(u;z) + h_{\lam}(u) +\frac{L_{\lam}}{2}\|u-z\|^2 \right\}$}
		\StateEq{$v_{r} \gets \frac{1}{\lam}\left[(v + z^{-} - z) + L_{\lam}(z-z_r)\right] + \nabla f(z_r) - \nabla f(z)$}
		\StateEq{$\varepsilon_{r} \gets (f_\lam + h_\lam)(z) - (f_\lam + h_\lam)(z_r)$}
		\StateEq{\Return{$(z_{r}, v_{r}, \varepsilon_r)$}}
	\EndProcedure
\end{smalgorithmic}
\end{mdframed}

The result below, whose proof can be found in \prettyref{app:ref_props},
presents the some important properties of the PR procedure (PRP).
\begin{prop}
\label{prop:eff_refine}Let $(z_{r},v_{r},\varepsilon_{r})$ and $L_{\lam}$
be generated by the PRP where $(f,h)$ satisfy assumptions \ref{asmp:nco1}--\ref{asmp:nco2}.
Then, the following properties hold:
\begin{itemize}
\item[(a)] $\varepsilon_{r}\geq L_{\lam}\|z-z_{r}\|^{2}/2$;
\item[(b)] $v_{r}\in\nabla f(z_{r})+\pt h(z_{r})$ and 
\[
\|v_{r}\|\leq\frac{1}{\lam}\|v+z^{-}-z\|+\left(\frac{1}{\lam}+\frac{\max\{m,M\}}{L_{\lam}}\right)\sqrt{2\varepsilon_{r}L_{\lam}};
\]
\item[(c)] if the inputs $f$, $h$, $\lam$, and $(z,z^{-},v)$ satisfy 
\begin{equation}
\begin{gathered}v\in\pt_{\varepsilon_{r}}\left(\lam\left[f+h\right]+\frac{1}{2}\|\cdot-z^{-}\|^{2}\right)(z),\\
\frac{1}{\lam}\|z^{-}-z+v\|\leq\bar{\rho},\quad\frac{1}{\lam}\varepsilon\leq\bar{\varepsilon},
\end{gathered}
\label{eq:prp_incl}
\end{equation}
for some $(\bar{\rho},\bar{\varepsilon})\in\r_{++}^{2}$ and $\varepsilon>0$,
it holds that
\begin{equation}
\|v_{r}\|\leq\bar{\rho}+\left(\frac{1}{\lam}+\frac{\max\left\{ m,M\right\} }{L_{\lam}}\right)\sqrt{2\lam\bar{\varepsilon}L_{\lam}}.\label{eq:prp_prox_refine_bd}
\end{equation}
\end{itemize}
\end{prop}

The result above is analogous to \prettyref{prop:crp_props}, which
describes properties of the CRP. In view of this link, we now make
a comparison between the PRP and the aforementioned CRP. First, the
PRP requires two extra points, $z^{-}$ and $v$, as part of its input
compared to the CRP. Second, \prettyref{prop:crp_props}(b) shows
that the CRP obtains a point $v_{r}$ satisfying the inclusion in
\prettyref{prop:eff_refine}(b). Finally, under the same conditions
in \eqref{eq:prp_incl}, \prettyref{prop:crp_props}(c) shows that
the point $v_{r}$ obtained by the CRP satisfies 
\begin{equation}
\|v_{r}\|\leq\left(1+\frac{\max\{m,M\}}{L_{\lam}}\right)\left(\bar{\rho}+\sqrt{2\bar{\varepsilon}L_{\lam}}\right),\label{eq:crp_prox_refine_bd}
\end{equation}
which is analogous to the bound in \eqref{eq:prp_prox_refine_bd}.
Note that, compared to \eqref{eq:prp_prox_refine_bd}, the above bound
has a larger constant in front of $\bar{\rho}$ and a possibly larger
constant in front of $\bar{\varepsilon}$ depending on the relationships
between $\lam$, $M$, $m$, and $L$. 

\section{Relaxed ACG (R.ACG) Method }

This section presents a relaxed ACG (R.ACG) variant that is generally
more efficient in practice than the ACGM in \prettyref{alg:acgm}. 

We first state the R.ACG variant in \prettyref{alg:r_acgm}. It main
idea is to start with a possibly large stepsize $\lam_{1}$ and adaptively
update this stepsize by checking a particular descent inequality at
every iteration.

\begin{mdframed}
\mdalgcaption{R.ACG Method}{alg:r_acgm}
\begin{smalgorithmic}
	\Require{$\psi_n \in \cConv({\cal Z}), \enskip \psi_n \in {\cal C}(\dom \psi_n), \enskip y_0 \in \dom \psi_n, \enskip (\mu,L_{\rm est}) \in \r_{++}^2, \enskip L_{\min} \in (0, L_{\rm est}]$;}
	\Initialize{$L_1 \gets L_{\rm est}$}
	\vspace*{.5em}
	\Procedure{R.ACG}{$\psi_s, \psi_n, y_0, \mu, L_{\min}, L_{\rm est}$}
	\For{$k=1,...$}
		\StateEq{$L \gets L_k$}
		\Do
			\StateEq{$\lam_k \gets 1/L$}
			\StateEq{Generate $(A_k, y_k, \tilde{x}_{k-1}, r_k, \eta_k)$ according to \prettyref{alg:acgm}.}
			\StateEq{$L \gets 2(L - L_{\min}) + L_{\min}$}
		\doWhile{$\psi_s(y_k) - \ell_{\psi_s}(y_k;\tilde{x}_{k-1}) > \frac{L}{2} \|y_k - \tilde{x}_{k-1}\|^2$} \label{ln:r_acgm_check}
		\StateEq{$L_{k+1} \gets L$}
	\EndFor
	\EndProcedure
\end{smalgorithmic}
\end{mdframed}

We now make two remarks about the above R.ACG method (R.ACGM). First,
if $\psi_{s}\in{\cal C}_{m,\overline{L}}(\dom\psi_{n})$ for some
$(m,\overline{L})\in\r_{++}^{2}$ and $L_{{\rm est}}\geq\overline{L}$,
then $L_{k}=L_{{\rm est}}$ for every $k\geq1$. On the other hand,
if $L_{{\rm est}}<\overline{L}$ then $L_{k}$ is doubled at most
\[
\left\lceil 1+\log_{2}\left(\frac{L-L_{{\rm est}}}{L_{{\rm est}}-L_{\min}}\right)\right\rceil 
\]
times and $L_{1}\leq L_{k}\leq2\overline{L}$ for every $k\geq1$.
Second, if $(L-L_{{\rm est}})/(L_{{\rm est}}-L_{\min})={\cal O}(1)$,
then the iteration complexities of the R.ACGM and ACGM in \prettyref{alg:acgm}
are on the same order of magnitude when given a common termination
condition.

It is worth mentioning that the above line search idea has been explored
in many other works in the literature. For example, \citep{Nesterov2013}
considers applying a similar line search subroutine in which the stepsize
parameter $\lam_{k}$ is increased whenever a key descent inequality
holds and decreased otherwise.

\section{Relaxed AIPP (R.AIPP) Method}

\label{sec:r_aipp}

This section establishes an iteration complexity bound for a relaxed
AIPPM (R.AIPPM) that is generally more efficient in practice than
the AIPPM in \prettyref{alg:aippm}.

Before proceeding, we first state the main problem of the R.AIPPM
and its key assumptions. Consider the NCO problem

\begin{equation}
\phi_{*}=\min_{z\in{\cal Z}}\left[\phi(z):=f(z)+h(z)\right],\tag{\ensuremath{{\cal NCO}}}\label{prb:eq:eff_nco}
\end{equation}
where ${\cal Z}$ is a finite dimensional inner product space, and
it is assumed that

\stepcounter{assumption}
\begin{enumerate}
\item \label{asmp:eff_nco1}$h\in\cConv(Z)$ for some nonempty convex set
$Z\subseteq{\cal Z}$;
\item \label{asmp:eff_nco2}$f\in{\cal C}_{M}(Z)$ for some $M>0$;
\item \label{asmp:eff_nco3}$\phi_{*}>-\infty$.
\end{enumerate}
Moreover, like in \prettyref{chap:unconstr_nco}, assume that efficient
oracles for evaluating the quantities $f(z)$, $\nabla f(z)$, and
$h(z)$ and for obtaining exact solutions of the subproblem 
\[
\min_{z\in{\cal Z}}\left\{ \lam h(z)+\frac{1}{2}\|z-z_{0}\|^{2}\right\} ,
\]
for any $z_{0}\in{\cal Z}$ and $\lam>0$, are available.

The AIPPM considers finding approximate stationary points of \ref{prb:eq:eff_nco}
as in \prettyref{prb:approx_nco}, i.e. given $\hat{\rho}>0$, find
$(\hat{z},\hat{v})\in Z\times{\cal Z}$ satisfying 
\begin{equation}
\hat{v}\in\nabla f(\hat{z})+\pt h(\hat{z}),\quad\|\hat{v}\|\leq\hat{\rho}.\label{eq:approx_eff_nco}
\end{equation}
For the sake of future referencing, let us state the problem of finding
$(\hat{z},\hat{v})$ satisfying \eqref{eq:approx_eff_nco} in \prettyref{prb:approx_eff_nco}.

\begin{mdframed}
\mdprbcaption{Find an approximate stationary point of \ref{prb:eq:eff_nco}}{prb:approx_eff_nco}
Given $\hat{\rho} > 0$, find a pair $(\hat{z},\hat{v}) \in Z \times {\cal Z}$ satisfying condition \eqref{eq:approx_eff_nco}.
\end{mdframed}

To ease the notation in later sections, let us conclude by defining
the useful quantity
\begin{equation}
\underline{m}:=\inf_{m>0}\left\{ f(u)-\ell_{f}(u;z)\ge-\frac{m}{2}\|u-z\|^{2}\quad\forall u,z\in Z\right\} .\label{eq:m_lower_def}
\end{equation}

\subsection{General Descent (GD) Framework}

\label{subsec:gd_framework}

This subsection presents a general descent (GD) framework that relaxes
the GIPP framework from \prettyref{chap:unconstr_nco}. We later show
that the R.AIPPM is a special instance of GD framework (GDF) in which
each prox subproblem is approximate solved by invoking the R.ACGM
in \prettyref{alg:acgm}.

Recall that for an IPP framework with stepsizes $\{\lam_{k}\}_{k\geq1}$,
the larger $\lam_{k}$ is the faster the IPP framework converges to
a desirable approximate solution. While $\lam_{k}$ is required to
be at most $1/\underline{m}$ in the GIPPF of \prettyref{chap:unconstr_nco},
the GDF of this subsection considers choosing $\lam_{k}$ significantly
larger than $1/\underline{m}$ despite a possible loss of convexity.
More specifically, it adaptively chooses its stepsizes based on two
key inequalities that are checked at the end of its iterations.

We first start by stating the GDF in \prettyref{alg:gdf}.

\begin{mdframed}
\mdalgcaption{GD Framework}{alg:gdf}
\begin{smalgorithmic}
	\Require{$h \in \cConv(Z), \enskip f \in {\cal C}(Z), \enskip z_0 \in Z, \enskip (\theta, \tau) \in \r_{++}^2,  \enskip L>0, \enskip \{\lam_k\}_{k\geq 1} \subseteq \r_{++}$;}
	\Initialize{$L_\lam \gets \lam M + 1, \enskip \phi \Lleftarrow f + h;$}
	\vspace*{.5em}
	\Procedure{GD}{$f, h, z_0, \theta, \tau, M$}
		\For{$k=1,...$}
			\StateEq{\textbf{Find} $(z_k, v_k, \lam_k) \in Z \times {\cal Z} \times \r_{++}$ such that its corresponding refined triple
			\begin{gather}
				(\hat z_k, \hat v_k, \hat{\varepsilon}_k) \gets \text{PREF}(f, h, z_k, z_{k-1}, v_k, M, \lam_k) \label{eq:gdf_refine}
			\end{gather}} \label{ln:gd_black_box}
			\State {satisfies the bounds
			\begin{gather}
				\|v_k + z_{k-1} - z_k\|^2 \leq \theta \lam_k \left[ \phi(z_{k-1}) - \phi(z_k) \right], \label{eq:bd_prox-approx} \\ 
				2 L_\lam \hat{\varepsilon}_k \leq \tau \|v_k + z_{k-1} - z_k\|^2. \label{eq:eps_gsm_bd}
			\end{gather}}	
	\EndFor
	\EndProcedure
\end{smalgorithmic}
\end{mdframed}

We now give two remarks about the above framework. First, no termination
criterion is added so as to be able to discuss convergence rate results
about its generated sequence. A discussion of how to terminate it
is given after \prettyref{prop:gipp10} below. Second, its \prettyref{ln:gd_black_box}
should be viewed as an oracle in that it does not specify how to compute
the triple $(\lambda_{k},z_{k},v_{k})$. Third, \prettyref{cor:exact_prox}
below shows that if the stepsize $\lambda_{k}$ is chosen so that
the prox subproblem 
\begin{equation}
\min_{z\in{\cal Z}}\left\{ \lam_{k}(f+h)(z)+\frac{1}{2}\|z-z_{k-1}\|^{2}\right\} \label{eq:r_aipp_subprb}
\end{equation}
is a strongly convex composite problem, i.e. $\lam_{k}\in(0,1/\underline{m})$,
the point $z_{k}$ is chosen as its unique optimal solution, and $v_{k}$
is set to zero, then the triple $(\lambda_{k},z_{k},v_{k})$ satisfies
\eqref{eq:bd_prox-approx} and \eqref{eq:eps_gsm_bd} with $\theta=2$
and $\tau=0$. Thus, when $(\theta,\tau)\in[2,\infty)\times[0,\infty)$,
we conclude that: (i) there always exists a triple satisfying \eqref{eq:bd_prox-approx}
and \eqref{eq:eps_gsm_bd}; and (ii) the GD framework can be viewed
as an IPP method. 

In \prettyref{subsec:r_aipp}, we show that the R.AIPPM is a special
instance of the GD framework, and hence, can be viewed as a relaxed
IPP method which chooses $(\theta,\tau)$ in the open rectangle $(2,\infty)\times(0,\infty)$.
In particular, it applies an instance of the R.ACGM in \prettyref{alg:r_acgm}
to problem \eqref{eq:r_aipp_subprb} in order to obtain a triple $(\lambda_{k},z_{k},v_{k})$
satisfying \eqref{eq:bd_prox-approx} and \eqref{eq:eps_gsm_bd}.

We now present an important property about the sequence of iterates
$\{(\lam_{k},\hat{z}_{k},\hat{v}_{k})\}_{k\geq1}$.
\begin{prop}
\label{prop:gipp10} The sequences of stepsizes $\{\lambda_{k}\}_{k\geq1}$
and iterate pairs $\{(\hat{z}_{k},\hat{v}_{k})\}_{k\geq1}$ satisfy
\begin{equation}
\hat{v}_{k}\in\nabla g(\hat{z}_{k})+\pt h(\hat{z}_{k}),\quad\min_{i\leq k}\|\hat{v}_{i}\|^{2}\leq\theta\left(1+2\sqrt{\tau}\right)^{2}\frac{[\phi(z_{0})-\phi_{*}]}{\Lambda_{k}},\label{eq:corGD_complex_b}
\end{equation}
for every $k\ge1$, where $\Lambda_{k}:=\sum_{i=1}^{k}\lam_{i}$.
\end{prop}

\begin{proof}
Let $k\geq1$ be fixed. The inclusion in \eqref{eq:corGD_complex_b}
follows from \prettyref{prop:eff_refine} with $(\hat{z},\hat{v})=(\hat{z}_{k},\hat{v}_{k})$
and the definitions of $\hat{z}_{k}$ and $\hat{v}_{k}$ in \eqref{eq:gdf_refine}.
To show the inequality in \eqref{eq:corGD_complex_b}, first observe
that \eqref{eq:bd_prox-approx} and the definition of $\phi_{*}$
in \ref{prb:eq:eff_nco} implies that 
\begin{align}
\phi(z_{0})-\phi_{*} & \geq\sum_{i=1}^{k}[\phi(z_{i-1})-\phi(z_{i})]\geq\sum_{i=1}^{k}\frac{\|v_{i}+z_{i-1}-z_{i}\|^{2}}{\theta\lam_{i}}\nonumber \\
 & \geq\frac{\Lambda_{k}}{\theta}\min_{i\leq k}\frac{1}{\lam_{i}^{2}}\|v_{i}+z_{i-1}-z_{i}\|^{2}.\label{eq:GD_complex_prf_a}
\end{align}
Now, let $i\geq1$ be arbitrary. Using \eqref{eq:gdf_refine}, \eqref{eq:eps_gsm_bd}
with $k=i$, and \prettyref{prop:eff_refine} with $\lam=\lam_{i}$,
$(z^{-},z)=(z_{i-1},z_{i})$, and $(v,v_{r})=(v_{i},\hat{v}_{i})$,
it holds that
\begin{align}
\|\hat{v}_{i}\| & \leq\frac{1}{\lam_{i}}\|v_{i}+z_{i-1}-z_{i}\|+\left(\frac{1}{\lam_{i}}+\frac{M}{\lam_{i}M+1}\right)\sqrt{2(\lam_{i}M+1)\hat{\varepsilon}_{i}}\nonumber \\
 & \leq\frac{1}{\lam_{i}}\|v_{i}+z_{i-1}-z_{i}\|+\frac{2}{\lam_{i}}\sqrt{2(\lam_{i}M+1)\hat{\varepsilon}_{i}}\\
 & \leq\left(\frac{1+2\sqrt{\tau}}{\lam_{i}}\right)\|v_{i}+z_{i-1}-z_{i}\|.\label{eq:v_hat_bd}
\end{align}
The inequality in \eqref{eq:corGD_complex_b} now follows by combining
\eqref{eq:GD_complex_prf_a} and \eqref{eq:v_hat_bd}. 
\end{proof}
We now make three additional remarks about the GDF in light of \prettyref{prop:gipp10}.
First, if the GDF stops when a pair $(\hat{z}_{k},\hat{v}_{k})$ such
that $\|\hat{v}_{k}\|\leq\hat{\rho}$ is found, then it follows from
the inclusion in \eqref{eq:corGD_complex_b} that $(\hat{z}_{k},\hat{v}_{k})$
solves \prettyref{prb:approx_eff_nco}. Second, if the sequence of
stepsizes $\{\lambda_{i}\}$ satisfies $\lim_{k\to\infty}\Lambda_{k}=\infty$,
then it follows from the inequality in \eqref{eq:corGD_complex_b}
and assumption \ref{asmp:eff_nco3} that the GDF indeed stops according
to the above termination criterion. Third, \eqref{eq:corGD_complex_b}
indicates that the larger the stepsizes $\lam_{k}$ are, the faster
the quantity $\min_{i\leq k}\|\hat{v}_{i}\|$ approaches zero.

For the remainder of this section, our goal is to show that the GDF
can be seen as a relaxation of the GIPPF from \prettyref{sec:gippf}.
The proof of this fact is not essential in establishing any results
pertaining to the R.AIPPM in this section and may skipped without
any loss of continuity.

Recall that, for a given $z_{0}\in Z$ and $\sigma\in[0,1)$, the
GIPPF in \prettyref{sec:gippf} considers a sequences $\{\lam_{k}\}_{k\geq1}$
and $\{(z_{k},v_{k},\varepsilon_{k})\}_{k\geq1}$ satisfying 
\begin{gather}
v_{k}\in\partial_{\varepsilon_{k}}\left(\lambda_{k}\phi+\frac{1}{2}\|\cdot-z_{k-1}\|^{2}\right)(z_{k}),\quad\|v_{k}\|^{2}+2\varepsilon_{k}\leq\sigma\|v_{k}+z_{k-1}-z_{k}\|^{2},\label{eq:gipp_main}
\end{gather}
for every $k\ge1$. We begin by presenting a simple technical result
that will be used both here and in the analysis of the R.AIPPM.
\begin{lem}
\label{lem:Del_bd}Assume that $\varepsilon\geq0$ and $(\lam,z^{-},z,v)\in\R_{++}\times{\cal Z}\times Z\times{\cal Z}$
satisfy 
\begin{gather}
v\in\partial_{\varepsilon}\left(\lam\phi+\frac{1}{2}\|\cdot-z^{-}\|^{2}\right)(z).\label{eq:Del_bd_incl}
\end{gather}
Then, the quantity $\varepsilon_{r}$ computed in \prettyref{alg:pref}
satisfies $\varepsilon_{r}\leq\varepsilon$.
\end{lem}

\begin{proof}
Let $(z_{r},\varepsilon_{r})$ be computed as in \prettyref{alg:pref}.
It follows from the definition of the approximate subdifferential
and \eqref{eq:Del_bd_incl} that 
\[
\lambda\phi(u)+\frac{1}{2}\|u-z^{-}\|^{2}\geq\lambda\phi(z)+\frac{1}{2}\|z-z^{-}\|^{2}+\langle v,u-z\rangle-\varepsilon\quad\forall u\in{\cal Z}.
\]
Considering the above inequality at the point $u=z_{r}$, along with
some algebraic manipulation, we have 
\begin{align*}
\varepsilon & \geq\left[\lam\phi(z)+\frac{1}{2}\|z-z^{-}\|^{2}-\langle v,z\rangle\right]-\left[\lam\phi(z_{r})+\frac{1}{2}\|z_{r}-z^{-}\|^{2}-\langle v,z_{r}\rangle\right]=\varepsilon_{r},
\end{align*}
where the last equality is due to the definitions $f_{\lam},h_{\lam},$
and $\varepsilon_{r}$ given in \prettyref{alg:pref}.
\end{proof}
The following result shows the relationship between the GIPPF of \prettyref{sec:gippf}
and the GDF of this section.
\begin{prop}
\label{prop:GIPP1} If, for some $z_{k-1}\in Z$, constant $\sigma\in[0,1)$,
and index $k\ge1$, the quadruple $(\lam_{k},z_{k},v_{k},\varepsilon_{k})$
satisfies \eqref{eq:gipp_main}, then $(\lam_{k},z_{k},v_{k})$ satisfies
\eqref{eq:bd_prox-approx} and \eqref{eq:eps_gsm_bd} for any $\theta\geq2/(1-\sigma)$
and $\tau\geq\sigma(\lam_{k}M+1)$. As a consequence, if $\sup_{k\geq1}\lam_{k}<\infty$,
then every instance of the GIPPF is an instance of the GDF for any
$(\theta,\tau)$ satisfying 
\begin{equation}
\theta\geq\frac{2}{1-\sigma},\quad\tau\geq\sup_{k\geq1}\left[\sigma(\lam_{k}M+1)\right].\label{eq:gipp_gd_cond}
\end{equation}
\end{prop}

\begin{proof}
The fact that $(\lambda_{k},z_{k},v_{k})$ satisfies \eqref{eq:bd_prox-approx}
with $\theta=2/(1-\sigma)$ follows from \prettyref{prop:gipp_descent}(a).
Now, let $k\geq1$ and observe that from \prettyref{lem:Del_bd} with
$(\lam,z^{-},z,v)=(\lam_{k},z_{k-1},z_{k},v_{k})$ and $\varepsilon=\varepsilon_{k}$
we have $\hat{\varepsilon}_{k}\leq\varepsilon_{k}$. It follows from
the last inequality and the inequality in \eqref{eq:gipp_main} that
$2\hat{\varepsilon}_{k}\leq\sigma\|v_{k}+z_{k-1}-z_{k}\|^{2}$. Combining
the previous inequality with the assumption on $\tau$ now shows that
$(\lambda_{k},z_{k},v_{k})$ satisfies \eqref{eq:eps_gsm_bd}. The
second part of the proposition follows immediately from the first
part and condition \eqref{eq:gipp_gd_cond}. 
\end{proof}
{[}add remarks{]}
\begin{cor}
\label{cor:exact_prox}Let $z_{k-1}\in Z$ and $\lam_{k}\in(0,1/\underline{m})$
be given, where $\underline{m}$ is as in \eqref{eq:m_lower_def}.
Then, \ref{prb:eq:eff_nco} has a unique global minimum $z_{k}$ and
the triple $(\lam_{k},z_{k},v_{k})$ satisfies \eqref{eq:bd_prox-approx}
and \eqref{eq:eps_gsm_bd} with $(\theta,\tau,v_{k})=(2,0,0)$. 
\end{cor}

\begin{proof}
The existence and unique uniqueness of $z_{k}$ follows from the fact
that $\phi+\|\cdot-z_{k-1}\|^{2}/\lam_{k}$ is strongly convex. Moreover,
the fact that $z_{k}$ is the unique global minimum of \ref{prb:eq:eff_nco}
implies that the quadruple $(\lam_{k},z_{k},v_{k},\varepsilon_{k})$,
where $(v_{k},\varepsilon_{k})=(0,0)$, satisfies \eqref{eq:gipp_main}
with $\sigma=0$. The conclusion of the corollary now follows immediately
from the first part of \prettyref{prop:GIPP1} with $\sigma=0$. 
\end{proof}

\subsection{Key Properties of the R.ACGM}

This subsection describes how the R.ACGM in \prettyref{alg:r_acgm}
can be used to implement a single iteration of the GDF in \prettyref{subsec:gd_framework}.

Consider the R.ACGM inputs
\begin{equation}
\begin{gathered}\psi_{s}=\lam f+\frac{1}{2}\|\cdot-z_{k-1}\|^{2},\quad\psi_{n}=\lam h,\quad y_{0}=z_{k-1},\\
\mu=1,\quad L_{\min}=1,\quad L_{{\rm est}}=L_{\lam}:=\lam M+1,
\end{gathered}
\label{eq:r_aipp_acg_inputs}
\end{equation}
and the termination criteria
\begin{gather}
2\max\left\{ 0,L_{\lam}\eta_{j}\right\} \leq\tau\|y_{0}-y_{j}+r_{j}\|^{2},\label{eq:r_acg_stop1}\\
\|y_{0}-y_{j}+r_{j}\|^{2}\leq\lam\theta\left[\phi(y_{0})-\phi(y_{j})\right],\label{eq:r_acg_stop2}
\end{gather}
for some $(\theta,\tau)\in\r_{++}^{2}$. In the following lemma, we
show that if the conditions \eqref{eq:r_acg_stop1} and \eqref{eq:r_acg_stop2}
\begin{gather}
\|A_{i}r_{i}+y_{i}-y_{0}\|^{2}+2\max\left\{ 0,A_{i}\eta_{i}\right\} \leq\|y_{i}-y_{0}\|^{2},\label{eq:r_acg_check1}\\
\psi(y_{0})\geq\psi(y_{i})+\inner{r_{i}}{y_{0}-y_{i}}-\max\left\{ 0,\eta_{i}\right\} ,\label{eq:r_acg_check2}
\end{gather}
hold at every iteration of R.ACGM, then the conditions \eqref{eq:r_acg_stop1}
and \eqref{eq:r_acg_stop2} will be obtained in a finite number of
R.ACG iterations.
\begin{lem}
\label{lem:basic_r_acgm_props}Let $\phi=f+h$ be a function satisfying
assumptions \ref{asmp:eff_nco1}--\ref{asmp:eff_nco2}, $L_{\lam}$
be as in \eqref{eq:r_aipp_acg_inputs}, and $(z_{k-1},\lam)\in Z\times\r_{++}$.
Moreover, suppose the R.ACGM is called with $(\psi_{s},\psi_{n},y_{0},L_{\min},L_{{\rm est}})$
as in \eqref{eq:r_aipp_acg_inputs} and generates the sequence of
iterates $\{(A_{i},y_{i},r_{i},\eta_{i})\}_{i\geq1}$. Then, the following
statements hold:
\begin{itemize}
\item[(a)] if the inequalities \eqref{eq:r_acg_check1} and \eqref{eq:r_acg_check2}
hold for every $i\geq1$, then for any $\theta>2$ and $\tau>0$ the
R.ACGM generates an iterate $(y_{j},r_{j},\eta_{j})$ satisfying \eqref{eq:r_acg_stop1}
and \eqref{eq:r_acg_stop2} in 
\begin{equation}
\left\lceil 1+\sqrt{2L_{\lam}}\,\log_{1}^{+}(2C_{\theta,\tau}L_{\lam})\right\rceil \label{eq:r_acgm_compl}
\end{equation}
iterations, where
\begin{equation}
C_{\theta,\tau}:=\max\left\{ \left[1+\sqrt{\frac{L_{\lam}}{\tau}}\right]^{2},\left[1+\sqrt{\frac{\theta}{\theta-2}}\right]^{2}\right\} .\label{eq:C_theta_tau_def}
\end{equation}
\item[(b)] if $\lam\leq1/\underline{m}$, then \eqref{eq:r_acg_check1}, \eqref{eq:r_acg_check2},
and the inclusion $r_{j}\in\pt_{\max\{\eta_{j},0\}}\psi(y_{j})$ hold
for every $j\geq1$.
\end{itemize}
\end{lem}

\begin{proof}
(a) See \prettyref{app:oth_acgm_props}.

(b) If $\lam\leq1/\underline{m}$, it follows that $\psi_{s}\in{\cal F}_{0,L_{\lam}}(Z)$,
and hence, $(\psi_{s},\psi_{n})$ satisfy the requirements of the
ACGM (see \prettyref{alg:acgm}) with $L=L_{\lam}.$ The conclusion
now follows from \prettyref{prop:acgm_vartn} and the definition of
the approximate subdifferential.
\end{proof}
The next result shows that conditions \eqref{eq:r_acg_stop1} and
\eqref{eq:r_acg_stop2} are sufficient to implement a single iteration
of the GDF in \prettyref{subsec:gd_framework}.
\begin{lem}
Let $\phi=f+h$ and $(z_{k-1},\lam)$ be as in \prettyref{lem:basic_r_acgm_props}
and $(\psi_{s},\psi_{n},y_{0})$ be as in \eqref{eq:r_aipp_acg_inputs}.
If $(y_{j},r_{j},\eta_{j})$ satisfy \eqref{eq:r_acg_stop1}, \eqref{eq:r_acg_stop2},
and $r_{j}\in\pt_{\max\{\eta_{j},0\}}(y_{j})$, then the assigned
triples
\begin{equation}
(\lam_{k},z_{k},v_{k})\gets(\lam,y_{j},r_{j}),\quad(\hat{z}_{k},\hat{v}_{k},\hat{\varepsilon}_{k})\gets\text{PREF}(f,h,y_{j},z_{k-1},r_{j},M,\lam)\label{eq:r_acgm_asgn}
\end{equation}
satisfy \eqref{eq:bd_prox-approx} and \eqref{eq:eps_gsm_bd}. 
\end{lem}

\begin{proof}
The fact that $(\lam_{k},z_{k-1},z_{k},v_{k})$ satisfies \eqref{eq:bd_prox-approx}
follows immediately from \eqref{eq:r_acg_stop2} and \eqref{eq:r_acgm_asgn}.
On the other hand, using \prettyref{lem:Del_bd} with $(z,v,\varepsilon)=(y_{j},r_{j},\max\{\eta_{j},0\})$
and the definition of $(\psi_{s},\psi_{n})$, we have that $\hat{\varepsilon}_{k}\leq r_{j}$.
Using the previous bound and \eqref{eq:r_acg_stop1} yields \eqref{eq:eps_gsm_bd}.
\end{proof}
We now conclude by discussing alternative choices for the R.ACG input
$(L_{{\rm est}},L_{\min})$ in \eqref{eq:r_aipp_acg_inputs}. First,
note that if $L_{{\rm est}}=\lam\alpha M+1$ for some $\alpha\in(0,1)$
and $L_{\min}=1$, then 
\[
\frac{L_{\lam}-L_{\min}}{L_{{\rm est}}-L_{\min}}=\frac{\lam M+1-1}{\lam\alpha M+1-1}=\frac{1}{\alpha}.
\]
Hence, in view of the above identity and the discussion following
the R.ACGM in \prettyref{alg:acgm}, choosing $L_{{\rm est}}=\lam\alpha M+1$
with $\alpha^{-1}={\cal O}(1)$ in an R.ACG call yields an complexity
that is on the same order of magnitude as an R.ACG call with $L_{{\rm est}}=\lam M+1$.

\subsection{Statement and Properties of the R.AIPPM}

\label{subsec:r_aipp}

This subsection describes and gives the iteration complexity of the
R.AIPPM.

We first state the R.ACG instance in \prettyref{alg:r_aippm_r_acgm}
that implements the approach described in the preceding subsection.
More specifically, this instance chooses $L_{\min}=1$ and $L_{{\rm est}}=\lam M/100+1$,
uses the termination conditions \eqref{eq:r_acg_stop1} and \eqref{eq:r_acg_stop2},
and uses \eqref{eq:r_acg_check1} and \eqref{eq:r_acg_check2} to
check for failure of the method. The variable $\pi_{S}$ is used to
store the termination status of the method where $\pi_{S}=$\texttt{
true} if the method outputs a solution satisfying \eqref{eq:r_acg_stop1}
and \eqref{eq:r_acg_stop2} and $\pi_{S}=$\texttt{ false} otherwise.

\begin{mdframed}
\mdalgcaption{R.ACG Instance for the R.AIPPM}{alg:r_aippm_r_acgm}
\begin{smalgorithmic}
	\Require{$\psi_n \in \cConv({\cal Z}), \enskip \psi_n \in {\cal C}(Z), \enskip y_0 \in Z, \enskip (\theta,\tau)\in\r_{++}^2, \enskip L_1 > 0, \enskip L_{\min} \in (0, L_1)$;}
	\Initialize{$L_1 \gets L_{\rm est}, \enskip \pi_S \gets {\tt true}, \enskip \psi \Lleftarrow \psi_s + \psi_n,$}
	\vspace*{.5em}
	\Procedure{R.ACG1}{$\psi_s, \psi_n, \phi, y_0, \theta, \tau, L_{\min}, L_{\max}, L_{\rm est}$}
	\For{$k=1,...$}
		\StateEq{Generate $(A_k, y_k, r_k, \eta_k)$ according to \prettyref{alg:r_acgm}.}
		\StateEq{$\eta_k^{+} \gets \max \{\eta_k, 0\}$}
		\If{\eqref{eq:r_acg_stop1} \textbf{ and } \eqref{eq:r_acg_stop2} hold with $j=k$}
			\StateEq{\Return{$(y_k, r_k, \eta_k^{+}, \pi_S)$}}
		\EndIf
		\If{\eqref{eq:r_acg_check1} \textbf{ or } \eqref{eq:r_acg_check2} do not hold with $i=k$}
			\StateEq{$\pi_S \gets {\tt false}$}
			\StateEq{\Return{$(y_0, \infty, \infty, \pi_S)$}}
		\EndIf
	\EndFor
	\EndProcedure
\end{smalgorithmic}
\end{mdframed}

Using the R.ACGM instance in \prettyref{alg:r_aippm_r_acgm} and the
refinement procedure in \prettyref{alg:pref}, we now state the R.AIPPM
in \prettyref{alg:r_aippm}. Given $\lam_{0}>0$ and $z_{0}\in Z$,
its main idea is to apply the R.ACGM to obtain the approximate update
for the $k^{{\rm th}}$ iteration
\[
z_{k}\approx\min_{z\in{\cal Z}}\left\{ \lam_{k}(f+h)(z)+\frac{1}{2}\|z-z_{k-1}\|^{2}\right\} 
\]
for a suitable stepsize $\lam_{k}$, and implement one iteration of
the GDF in \prettyref{subsec:gd_framework}. The iterate $z_{k}$
is then refined using the PRP in \prettyref{alg:pref} and termination
of the method occurs when a refined iterate solving \prettyref{prb:approx_eff_nco}
is found.

\begin{mdframed}
\mdalgcaption{R.AIPP Method}{alg:r_aippm}
\begin{smalgorithmic}
	\Require{$\hat{\rho} > 0, \enskip M>0, \enskip h \in \cConv(Z), \enskip f \in {\cal C}_{M}(Z), \enskip \lam_0 > 0, \enskip (\theta, \tau) \in (2,\infty) \times \r_{++}, \enskip z_0 \in Z$;}
	\Initialize{$\phi \Lleftarrow f + h;$}
	\vspace*{.5em}
	\Procedure{R.AIPP}{$f, h, z_0, \lam_0, \theta, \tau, M, \hat{\rho}$}
	\For{$k=1,...$}
		\StateStep{\algpart{1}\textbf{Find} the right $\lam_k$ and \textbf{attack} the $k^{\rm th}$ prox subproblem.}
		\StateEq{$\lam \gets \lam_{k-1}$} \label{ln:lam_asn}
		\StateEq{$\psi_s^k \Lleftarrow \lam f + \frac{1}{2}\|\cdot - z_{k-1}\|^2$}
		\Repeat
			\StateEq{$(L_{\min}, L_{\max}, L_{\rm est}) \gets (1, \lam M + 1, \lam [M/100] + 1)$}
			\StateEq{$(z_k, v_k, \varepsilon_k, \pi_k^{\rm acg}) \gets \text{R.ACG1}(\psi_s^k, \lam h, \phi, y_0, \theta, \tau, L_{\min}, L_{\max},  L_{\rm est})$} \label{ln:r_acg_call}
			\StateEq{$(\hat{z}_k, \hat{v}_k, \hat{\varepsilon}_k) \gets \text{PREF}(f, h, z_k, z_{k-1}, v_k, M, \lam)$} \label{ln:pref_call}
			\If{$\lnot (\pi_k^{\rm acg})$ \textbf{ or } $2 L_{\max} \hat{\varepsilon}_k > \tau \|v_k + z_{k-1} - z_k\|^2$} \label{ln:r_aipp_check}
				\StateEq{$\lam \gets \lam / 2$}
			\EndIf
		\Until{$\pi_k^{\rm acg}$ \text{ and } $2 L_{\max} \hat{\varepsilon}_k \leq \tau \|v_k + z_{k-1} - z_k\|^2$} \label{ln:eff_r_acgm_stop}
		\StateEq{$\lam_k \gets \lam$} \label{ln:eff_lam_asn}

		\StateStep{\algpart{2}\textbf{Check} the termination condition.}
		\If{$\|\hat{v}_k\| \leq \hat{\rho}$} \label{ln:eff_vk_termination}
			\StateEq{\Return{$(\hat{z}_k, \hat{v}_k)$}}
		\EndIf
	\EndFor
	\EndProcedure
\end{smalgorithmic}
\end{mdframed}

Some comments about the R.AIPPM are in order. To ease the discussion,
let us refer to the ACG iterations performed in \prettyref{ln:r_acg_call}
of the method as \textbf{inner iterations} and the iterations over
the indices $k$ as\emph{ }\textbf{outer iterations}. First, the failure
checks in the R.ACG instances and \prettyref{ln:r_aipp_check} of
the method immediately imply that a single iteration of the R.AIPPM
 implements a single iterations of the GDF. Second, the termination
condition in \prettyref{ln:eff_vk_termination} and \prettyref{prop:eff_refine}(b),
with $(\lam,z^{-},z,v)=(\lam_{k},z_{k-1},z_{k},v_{k})$, imply that
the required solution, i.e. a pair $(\hat{z},\hat{v})$ that solves
\prettyref{prb:approx_eff_nco}, is obtained when the R.AIPPM terminates.
Third, since the R.AIPP iterates correspond to iterates of the GDF,
and the sequence $\{\lam_{k}\}$ is bounded below (see \prettyref{lem:adaptiveMet}(c)
below), \prettyref{prop:gipp10} implies that the sequence $\{\hat{v}_{k}\}$
generated by the R.AIPPM has a subsequence approaching zero, and thus
the method must terminate in \prettyref{ln:eff_vk_termination}. Fifth,
although the R.AIPPM does not necessarily generate proximal subproblems
with convex objective functions, it is shown in \prettyref{thm:r_aipp_compl}
below that it has an iteration complexity similar to that of the AIPPM
of \prettyref{sec:aipp}. Finally, in contrast to the aforementioned
AIPPM, the R.AIPPM neither requires an upper bound on the quantity
$\underline{m}$ in \eqref{eq:m_lower_def} as part of its input nor
does it place any restriction on the initial stepsize $\lam_{0}$.

Each iteration of the R.AIPPM may call the R.ACGM multiple times (possibly
just one time). Invocations of the R.ACGM algorithm that stop with
$\pi_{k}^{{\rm acg}}=$ \texttt{true} are said to be of type $S$
while the other invocations are said to be of type $F$. Let $k_{S}$
(resp., $k_{F}$) denote the total number of R.ACG calls of type $S$
(resp., type $F$). The following technical result provides some basic
facts about $k_{S}$, $k_{F}$, and the sequence of stepsizes $\{\lambda_{k}\}_{k\geq1}$.
\begin{lem}
\label{lem:adaptiveMet} The following statements hold for the R.AIPPM: 
\begin{itemize}
\item[(a)] if the stepsize $\lam_{\bar{k}}\le1/\underline{m}$ for some $\bar{k}\geq1$,
then every iteration $k\geq\bar{k}$ is of type $S$ and, as a consequence,
$\lambda_{k}=\lambda_{\bar{k}}$ for every $k>\bar{k}$; 
\item[(b)] $k_{F}$ can be bounded as $2^{k_{F}}\leq\max\{1,2\lambda_{0}\underline{m}\}$; 
\item[(c)] $\{\lambda_{k}\}_{k\geq1}$ is non-increasing and satisfies 
\begin{equation}
\xi:=\max\left\{ \frac{1}{\lam_{0}},2\underline{m}\right\} \geq\frac{1}{\lam_{k}}\quad\forall k\geq1.\label{eq:eff_xi_def}
\end{equation}
\end{itemize}
\end{lem}

\begin{proof}
(a) Since $\lambda_{\bar{k}}\leq1/\underline{m}$, the definition
of $\underline{m}$ in \eqref{eq:m_lower_def} implies that $\lam_{\bar{k}}f+\|\cdot-z_{k-1}\|^{2}/2$
is convex. Hence, it follows from \prettyref{lem:basic_r_acgm_props}(b)
that $\pi_{k}^{{\rm acg}}=$ \texttt{true}, which is to say that this
iteration is of type $S$. Since $\{\lambda_{k}\}_{k\geq1}$ is clearly
nonincreasing, the same conclusion holds true for every iteration
$k\geq\bar{k}$. Moreover, as $\lam$ is not halved for subsequent
iterations following $\bar{k}$, it follows that $\lambda_{k}=\lambda_{\bar{k}}$
for every $k>\bar{k}$.

(b) Using the fact that immediately before each iteration of type
$F$, the stepsize $\lambda$ is halved, we see that the condition
$\lambda_{\bar{k}}\leq1/\underline{m}$ in part (a) would eventually
be satisfied for some iteration $\bar{k}\geq1$, and hence $k_{F}$
is finite. Now, note that if $k_{F}=0$ then the inequality in part
(b) follows immediately. Assume then that $k_{F}\geq1$. It now follows
from part (a) and the definition of $k_{F}$ that $\lam_{0}/2^{k_{F}-1}>1/\underline{m}$,
which clearly implies the inequality in part (b).

(c) The first statement follows trivially from the update rule of
$\lambda_{k}$ in the R.AIPPM. Now, note that the definition of $k_{F}$
together with the update rule for $\lambda_{k}$ imply, for every
$k\geq1$, that $\lambda_{0}/2^{k_{F}}\leq\lambda_{k}.$ The inequality
in part (c) then follows from the inequality in part (b). 
\end{proof}
In view of \prettyref{lem:adaptiveMet}(a), choosing an initial stepsize
$\lam_{0}$ satisfying $\lam_{0}\leq1/(2\underline{m})$ results in
an R.AIPP variant with constant stepsize, which resembles the AIPPM
described in \prettyref{sec:aipp}.

The following theorem presents a worst-case iteration complexity bound
on the number of inner iterations of the R.AIPPM with respect to the
inputs $M,$ $\lam_{0}$, $z_{0}$, the quantity $\underline{m}$
in \eqref{eq:m_lower_def}, and the tolerance $\hat{\rho}$.
\begin{thm}
\label{thm:r_aipp_compl}The R.AIPPM outputs a pair $(\hat{z},\hat{v})$
that solves \prettyref{prb:approx_eff_nco} in 
\begin{equation}
{\cal O}\left(\sqrt{M+\xi}\left(\frac{\sqrt{\xi}\theta\left[1+\tau\right]\left[\phi(z_{0})-\phi_{*}\right]}{\hat{\rho}^{2}}+\sqrt{\lam_{0}}\right)\log_{1}^{+}\left[C_{\theta,\tau}\lam_{0}M\right]\right)\label{eq:r_aipp_compl}
\end{equation}
inner iterations, where $C_{\theta,\tau}$ and $\xi$ are as \eqref{eq:C_theta_tau_def}
and \eqref{eq:eff_xi_def}, respectively.
\end{thm}

\begin{proof}
The fact that its output solves \prettyref{prb:approx_eff_nco} follows
from the termination condition in \prettyref{ln:eff_vk_termination},
\prettyref{ln:pref_call}, and \prettyref{prop:eff_refine}. 

To show the desired complexity, we let ${\rm TI}_{S}$ (resp. ${\rm TI}_{F}$)
denote the total number of inner iterations performed during all calls
of type $S$ (resp. type $F$) (see the paragraph preceding \prettyref{lem:adaptiveMet}).
Clearly, the total number of inner iterations is ${\rm TI}:={\rm TI}_{S}+{\rm TI}_{F}$.
We now bound each one of the quantities ${\rm TI}_{S}$ and ${\rm TI}_{F}$
separately by using the fact that the inputs given to every R.ACG
call and \prettyref{lem:basic_r_acgm_props}(a) imply that the number
of inner iterations performed during each R.ACG call is 
\[
{\cal O}\left(\sqrt{\lam M+1}\log_{1}^{+}\left[C_{\theta,\tau}(\lam M+1)\right]\right),
\]
where $\lam$ is the value of $\lam$ just before the call and $C$
is as in \eqref{eq:C_theta_tau_def} with $L_{\lam}=\lam M+1$.

We first consider ${\rm TI}_{F}$. Note that \prettyref{lem:adaptiveMet}(b)
implies that $k_{F}$ is finite. Since ${\rm TI}_{F}=0$ when $k_{F}=0$,
we may assume without loss of generality that $k_{F}\geq1$. Note
that the values of $\lam$ just before the $k_{F}$ calls of type
$F$ are exactly $\lam_{0},\lam_{0}/2,\ldots,\lam_{0}/2^{k_{F}-1}$.
Hence, we conclude that 
\begin{align}
{\rm TI}_{F} & ={\cal O}\left(\sum_{i=1}^{k_{F}}\sqrt{\frac{\lam_{0}M}{2^{i-1}}+1}\ \log_{1}^{+}\left[C_{\theta,\tau}\left(\frac{\lam_{0}M}{2^{i-1}}\right)\right]\right)\nonumber \\
 & ={\cal O}\left(\sum_{i=1}^{k_{F}}\sqrt{\frac{\lam_{0}\left(M+\xi\right)}{2^{i-1}}}\ \log_{1}^{+}\left[C_{\theta,\tau}\lam_{0}M\right]\right)\nonumber \\
 & ={\cal O}\left(\sqrt{\lam_{0}\left(M+\xi\right)}\log_{1}^{+}\left[C_{\theta,\tau}\lam_{0}M\right]\right)\label{eq:numbinner00}
\end{align}
where the second identity is due the fact that \prettyref{lem:adaptiveMet}(b)
implies $2^{i-1}\le2^{k_{F}-1}\leq2\lam_{0}\xi$ for every $i\le k_{F}$. 

We now bound ${\rm TI}_{S}$. Suppose that $k_{S}>1$ and observe
that the termination criterion $\|\hat{v}_{k}\|\leq\hat{\rho}$ is
not satisfied in the first $k_{S}-1$ iterations. Since the R.AIPPM
is an instance of the GDF, it follows from \prettyref{prop:gipp10}
that 
\begin{equation}
\hat{\rho}^{2}<\min_{j\leq k_{S}-1}\|\hat{v}_{j}\|^{2}\leq\theta\left(1+2\sqrt{\tau}\right)^{2}\frac{\left[\phi(z_{0})-\phi_{*}\right]}{\sum_{j=1}^{k_{S}-1}\lam_{j}}.\label{eq:r_aipp_prf_L}
\end{equation}
Using the fact that \prettyref{lem:adaptiveMet}(c) implies $1/\lam_{j}\leq\xi$
and $\lam_{j}\leq\lam_{0}$ for every $j\geq1$, we obtain 
\begin{align}
{\rm TI}_{S} & ={\cal O}\left(\sum_{j=1}^{k_{S}}\sqrt{\lam_{j}M+1}\ \log_{1}^{+}\left[C_{\theta,\tau}\lam_{j}M\right]\right)\nonumber \\
 & ={\cal O}\left(\sum_{j=1}^{k_{S}}\sqrt{\lam_{j}(M+\xi)}\log_{1}^{+}\left[C_{\theta,\tau}\lam_{0}M\right]\right)\nonumber \\
 & ={\cal O}\left(\sqrt{M+\xi}\left[\sum_{j=1}^{k_{S}}\sqrt{\lam_{j}}\right]\log_{1}^{+}\left[C_{\theta,\tau}\lam_{0}M\right]\right)\nonumber \\
 & ={\cal O}\left(\sqrt{M+\xi}\left[\sum_{j=1}^{k_{S}}\frac{\lam_{j}}{\sqrt{\lam_{j}}}+\sqrt{\lam_{0}}\right]\log_{1}^{+}\left[C_{\theta,\tau}\lam_{0}M\right]\right).\label{eq:prelim_r_aipp_compl1}
\end{align}
Now, using \ref{eq:r_aipp_prf_L}, the bound $(a+b)^{2}\leq2a^{2}+2b^{2}$
for every $a,b\in\r$, and the previous bound $1/\lam_{j}\leq\xi$
for every $j\geq1$, it holds that 
\begin{align}
\sum_{j=1}^{k_{S}-1}\frac{\lam_{j}}{\sqrt{\lam_{j}}} & \leq\sqrt{\xi}\sum_{j=1}^{k_{S}-1}\lam_{j}={\cal O}\left(\frac{\sqrt{\xi}\theta\left(1+\tau\right)\left[\phi(z_{0})-\phi_{*}\right]}{\hat{\rho}^{2}}\right).\label{eq:prelim_r_aipp_compl2}
\end{align}
Hence, combining \eqref{eq:prelim_r_aipp_compl1} and \eqref{eq:prelim_r_aipp_compl2},
we conclude that 
\begin{align}
{\rm TI}_{S}={\cal O}\left(\sqrt{M+\xi}\left(\frac{\sqrt{\xi}\theta\left[1+\tau\right]\left[\phi(z_{0})-\phi_{*}\right]}{\hat{\rho}^{2}}+\sqrt{\lam_{0}}\right)\log_{1}^{+}\left[C_{\theta,\tau}\lam_{0}M\right]\right).\label{eq:aux00010}
\end{align}
It can be easily seen that the bound in \eqref{eq:aux00010} trivially
holds when $k_{S}\leq1$ in view of the last term in it. Indeed, to
prove this, just assume that $\sum_{j=1}^{k_{S}-1}\lambda_{j}=0$
in the above argument bounding ${\rm TI}_{S}$. Now, since ${\rm TI}={\rm TI}_{F}+{\rm TI}_{S}$,
the bound in \eqref{eq:r_aipp_compl} follows by adding \eqref{eq:numbinner00}
and \eqref{eq:aux00010}.
\end{proof}
The result below presents the iteration complexity of the R.AIPPM
with inputs $(\theta,\tau)=(4,2)$ and $\lam=1/\underline{m}$.
\begin{cor}
\label{cor:spec_r_aipp_compl}The R.AIPPM with inputs $(\theta,\tau)=(4,2)$
and $\lam=1/\underline{m}$ outputs a pair $(\hat{z},\hat{v})$ that
solves in \prettyref{prb:approx_eff_nco} in 
\[
{\cal O}\left(\sqrt{\frac{M}{\underline{m}}+1}\left[\frac{\underline{m}\left[\phi(z_{0})-\phi_{*}\right]}{\hat{\rho}^{2}}+1\right]\log_{1}^{+}\left[\frac{M}{\underline{m}}\right]\right)
\]
oracle calls, where $\xi$ is as in \eqref{eq:eff_xi_def}.
\end{cor}

\begin{proof}
This follows immediately from \prettyref{thm:r_aipp_compl} with $(\theta,\tau)=(4,2)$
and $\lam=1/\underline{m}$ together with the fact that the R.ACGM
uses ${\cal O}(1)$ oracle calls at the end of every one of its iterations.
\end{proof}
We now briefly discuss alternative update rules for the stepsize $\lam_{k}$.
To begin, one could consider an update in which the intermediate variable
$\lam$ in \prettyref{ln:lam_asn} of the R.AIPPM is initialized with
$\lam\gets\beta\lam_{k-1}$ for some $\beta>1$. For larger values
of $\beta$, this might result in larger number of inner iterations
per outer iteration due to a (possibly) large number of R.ACG calls
that result in $\pi_{k}^{{\rm acg}}=$ \texttt{false}. A modification
of this approach is to fix this multiplier $\beta$ to be 1 for all
iterations following one in which $\pi_{k}^{{\rm acg}}=$ \texttt{false}.
This modification results in a bitonic stepsize sequence (as opposed
to a monotonic one) and is only slightly more conservative than the
first approach. The second approach will be used in our computational
experiments in \prettyref{chap:numerical}.

\section{Relaxed AIP.QP (R.AIP.QP) Method}

\label{sec:rqp_aipp}

This section establishes an iteration complexity bound for a relaxed
AIP.QPM (R.AIP.QPM) that is generally more efficient in practice than
the AIP.QPM in \prettyref{sec:qp_aipp}.

Before proceeding, we first recall the main problem of the R.AIP.QPM
and its key assumptions. Consider the CNCO problem

\begin{equation}
\hat{\varphi}_{*}:=\min_{z\in{\cal Z}}\left\{ \phi(z):=f(z)+h(z):{\cal A}z\in S\right\} \tag{\ensuremath{{\cal CNCO}[a]}}\label{prb:eq:eff_cnco_a}
\end{equation}
where ${\cal Z}$ is a finite dimensional inner product space and
it is assumed that $\phi=f+h$ satisfies assumptions \ref{asmp:eff_nco1}--\ref{asmp:eff_nco3}
and:

\stepcounter{assumption}
\begin{enumerate}
\item \label{asmp:eff_cnco_a1}${\cal A}:{\cal Z}\mapsto{\cal R}$ is a
nonzero linear operator for some finite dimensional inner product
space ${\cal R}$, the quantity $S\subseteq{\cal {\cal R}}$ is a
closed convex set, and ${\cal F}:=\{z\in{\cal Z}:{\cal A}z\in S\}\neq\emptyset;$
\item \label{asmp:eff_cnco_a2}$Z$ is compact.
\end{enumerate}
Moreover, like in \prettyref{chap:cnco}, it is assumed that efficient
oracles for evaluating the quantities $f(z)$, $\nabla f(z)$, ${\cal A}z$,
and $h(z)$ and for obtaining exact solutions of the subproblems
\[
\min_{z\in{\cal Z}}\left\{ \lam h(z)+\frac{1}{2}\|z-z_{0}\|^{2}\right\} ,\quad\min_{r\in S}\|r-r_{0}\|
\]
for any $z_{0}\in{\cal Z}$, $r\in{\cal R}$, and $\lam>0$, are available.

The R.AIP.QPM considers finding approximate stationary points of \ref{prb:approx_eff_cnco_a}
as in \prettyref{prb:approx_cnco_a}, i.e. given $(\hat{\rho},\hat{\eta})\in\r_{++}^{2}$,
find $([\hat{z},\hat{p}],[\hat{v},\hat{q}])\in[Z\times{\cal R}]\times[{\cal Z}\times{\cal R}]$
satisfying 
\begin{gather}
\hat{v}\in\nabla f(\hat{z})+\pt h(\hat{z})+{\cal A}^{*}\hat{p}\quad\|\hat{v}\|\leq\hat{\rho},\label{eq:eff_approx_cnco_a_ln1}\\
{\cal A}\hat{z}+\hat{q}\in S\quad\|\hat{q}\|\leq\hat{\eta}.\label{eq:eff_approx_cnco_a_ln2}
\end{gather}
For the sake of future referencing, let us state the problem of finding
$(\hat{z},\hat{v})$ satisfying \eqref{eq:eff_approx_cnco_a_ln1}
in \prettyref{prb:approx_eff_nco}.

\begin{mdframed}
\mdprbcaption{Find an approximate stationary point of ${\cal CNCO}$[a]}{prb:approx_eff_cnco_a}
Given $(\hat{\rho}, \hat{\eta})\in\r_{++}^2$, find a pair $([\hat{z}, \hat{p}],[\hat{v}, \hat{q}]) \in [Z \times {\cal R}] \times [{\cal Z} \times {\cal R}]$ satisfying  conditions \eqref{eq:eff_approx_cnco_a_ln1} and \eqref{eq:eff_approx_cnco_a_ln2}.
\end{mdframed}

\subsection{Key Properties of the Quadratic Penalty Approach}

This subsection presents some key properties of a quadratic penalty
function that is used in the R.AIP.QPM. Its properties mirror those
in \prettyref{subsec:qp_props}.

We first introduce some useful quantities. First, the diameter of
$Z$ is denoted by 
\begin{equation}
D_{z}:=\sup_{u,z\in Z}\|u-z\|.\label{eq:diamZ_def}
\end{equation}
We define the following important quantity for future reference:
\begin{equation}
\hat{\varphi}_{c}:=\inf_{z\in Z}\left\{ \varphi_{c}(z):=f_{c}(z)+h(z)\right\} .\label{eq:eff_varPhiC_def}
\end{equation}
where $f_{c}(\cdot)$ is a quadratic penalty function given by
\begin{equation}
f_{c}(z):=f(z)+\frac{c}{2}\dist^{2}({\cal A}z,S)\quad\forall z\in Z.\label{eq:eff_penalty_fn}
\end{equation}
Note that using \prettyref{lem:penalty_props}(a) and the definition
of $\hat{\varphi}_{*}$ is as in \ref{prb:eq:eff_cnco_a}, it is easily
seen that 
\begin{equation}
\hat{\varphi}_{*}\geq\hat{\varphi}_{\bar{c}}\geq\hat{\varphi}_{c}\quad\forall\bar{c}>c\ge0.\label{eq:pen_topo}
\end{equation}

The following result shows how a solution of \prettyref{prb:approx_eff_nco}
with $f=f_{c}$ yields a solution of \prettyref{prb:approx_eff_cnco_a}
when the penalty parameter $c$ is sufficiently large.
\begin{prop}
\label{prop:eff_penalty_props}Given $\hat{\rho}>0$ and $c>0$, let
$(\hat{z},\hat{v})$ be a solution of \prettyref{prb:approx_nco}
with $f=f_{c}$ as in \eqref{eq:eff_penalty_fn}. Moreover, let $\underline{m}$
be as in \eqref{eq:m_lower_def} and define the quantities

\begin{gather}
\begin{gathered}\hat{p}:=c\left[{\cal A}\hat{z}-\Pi_{S}({\cal A}\hat{z})\right],\quad\hat{q}:=\Pi_{S}({\cal A}\hat{z})-{\cal A}\hat{z},\\
\ensuremath{T_{\hat{\eta}}:=\left[2(\hat{\varphi}_{*}-\hat{\varphi}_{0}+\hat{\rho}D_{z})+\underline{m}D_{z}^{2}\right]\hat{\eta}^{-2},\quad M_{c}:=M+c\|{\cal A}\|^{2}}
\end{gathered}
\label{eq:pen_c_def}
\end{gather}
where $\hat{\varphi}_{*}$, $\hat{\varphi}_{0}$, and $D_{z}$ are
as in \ref{prb:eq:eff_cnco_a}, \eqref{eq:eff_varPhiC_def}, and \eqref{eq:diamZ_def},
respectively. Then the following statements hold:
\begin{itemize}
\item[(a)] it holds that $f_{c}\in{\cal C}_{m,M_{c}}(Z)$;
\item[(b)] the pair $([\hat{z},\hat{p}],[\hat{v},\hat{q}])$ satisfies \eqref{eq:approx_cnco_a_ln1},
the inclusion in \eqref{eq:approx_cnco_a_ln2}, and 
\[
\|\hat{q}\|^{2}\leq\frac{1}{c}\left(2\left[\hat{\varphi}_{*}-\varphi(\hat{z})+\hat{\rho}D_{z}\right]+\underline{m}D_{z}^{2}\right);
\]
\item[(c)] if $c\geq T_{\hat{\eta}}$, then $\|\hat{q}\|\leq\hat{\eta}$.
\end{itemize}
\end{prop}

\begin{proof}
(a) See \prettyref{lem:fc_smoothness}.

(b) See \prettyref{lem:props_qp}(a) for the proof that the pair $([\hat{z},\hat{p}],[\hat{v},\hat{q}])$
satisfies \eqref{eq:approx_cnco_a_ln1}, the inclusion in \eqref{eq:approx_cnco_a_ln2}.
To show the desired inequality on $\|\hat{q}\|$, let $p_{S}(z)=(c/2)\dist^{2}({\cal A}z,S)$
for every $z\in Z$. Using the inclusion in \eqref{eq:approx_cnco_a_ln2},
the convexity of $p_{S}$, the definition of $\hat{p}$, and \prettyref{lem:dist_props}(b),
it follows that $\hat{v}\in\nabla f(\hat{z})+\pt\left[h+p_{S}\right](\hat{z})$,
or equivalently,
\begin{equation}
h(u)+p_{S}(u)\geq h(\hat{z})+p_{S}(\hat{z})+\left\langle \hat{v}-\nabla f(\hat{z}),u-\hat{z}\right\rangle \quad\forall u\in Z.\label{eq:pen_incl}
\end{equation}
Considering \eqref{eq:pen_incl} at any $u\in{\cal F}$ and using
the fact that $p_{S}(u)=0$ for any $u\in{\cal F}$, the definition
of $\underline{m}$ in \eqref{eq:m_lower_def}, and the definitions
of $p_{S}$ and $\hat{q}$, we conclude that 
\begin{align*}
\frac{c}{2}\|\hat{q}\|^{2}= & \frac{c}{2}\|\Pi_{S}({\cal A}\hat{z})-{\cal A}\hat{z}\|^{2}=p_{S}({\cal A}\hat{z})\\
 & \leq h(u)-h(\hat{z})+\left\langle \nabla f(\hat{z}),u-\hat{z}\right\rangle -\left\langle \hat{v},u-\hat{z}\right\rangle \\
 & \leq(f+h)(u)-(f+h)(\hat{z})+\|\hat{v}\|\|u-\hat{z}\|+\frac{1}{2}\left(\underline{m}\|u-\hat{z}\|^{2}\right)\\
 & \leq\varphi(u)-\varphi(\hat{z})+\hat{\rho}D_{h}+\frac{1}{2}\left(\underline{m}D_{h}^{2}\right).
\end{align*}
Taking the infimum over $u\in{\cal F}$ immediately yields the desired
bound.

(c) Suppose $c\geq T_{\hat{\eta}}$. Using the previous bound on $c$,
the fact that $\varphi(\hat{z})\geq\hat{\varphi}_{0}$ , and the definition
of $T_{\hat{\eta}}$, it follows from part (b) that 
\[
\|\hat{q}\|^{2}\leq\frac{1}{c}\left(2\left[\hat{\varphi}_{*}-\hat{\varphi}_{0}+\hat{\rho}D_{h}\right]+\underline{m}D_{h}^{2}\right)=\frac{1}{c}\left[\hat{\eta}^{2}T_{\hat{\eta}}\right]\leq\hat{\eta}^{2}.
\]
\end{proof}
In view of the above proposition, we now outline a static penalty
method for solving \prettyref{prb:approx_eff_cnco_a}. First, let
$z_{0}\in Z$ be given and select a penalty parameter $c={\cal O}(\hat{\eta}^{-2})$
satisfying $c\geq T_{\hat{\eta}}$. Second, obtain a point $(\hat{z},\hat{v})$
solving \prettyref{prb:approx_eff_nco} with $f=f_{c}$ (see \eqref{eq:eff_penalty_fn})
using the R.AIPPM of \prettyref{sec:r_aipp} with $z_{0}$, $(m,M)=(m,M_{c})$,
$(\theta,\tau)=(4,2)$, and $\lam=1/\underline{m}$, where $M_{c}$
is as in \prettyref{prop:eff_penalty_props}(b). Finally, compute
the pair $(\hat{p},\hat{q})$ according to \eqref{eq:pen_c_def} and
output the pair $([\hat{z},\hat{p}],[\hat{v},\hat{q}])$, which solves
\prettyref{prb:approx_eff_cnco_a} in view of \prettyref{prop:eff_penalty_props}(a)
and (d). Using the fact that $c={\cal O}(\hat{\eta}^{-2})$, and \prettyref{cor:spec_r_aipp_compl}
with $(f,M)=(f_{c},M_{c})$, it is easy to see that the inner iteration
complexity of the outlined method is 
\begin{equation}
{\cal O}\left(\sqrt{\frac{M_{c}}{\underline{m}}+1}\left[\frac{\underline{m}\left[\varphi_{c}(z_{0})-\hat{\varphi}_{c}\right]}{\hat{\rho}^{2}}+1\right]\log_{1}^{+}\left[\frac{M_{c}}{\underline{m}}\right]\right)={\cal O}\left(\hat{\rho}^{-2}\hat{\eta}^{-3}\log_{1}^{+}\hat{\eta}^{-1}\right),\label{eq:naive_qp_compl}
\end{equation}
where the last quantity ignores any constants aside from the tolerances.
A drawback of this static penalty method is that it requires in its
first step the selection of a single parameter $c$, which is generally
difficult to obtain. This issue can be circumvented by considering
a dynamic cold-started penalty method in which the static penalty
method is repeated for a sequence of increasing values of $c$ and
common starting point $z_{0}$. It can be shown that the resulting
cold-started dynamic penalty method has an ACG iteration complexity
that is still on the same order as \eqref{eq:naive_qp_compl}. Note
that the bound \eqref{eq:naive_qp_compl} is actually ${\cal O}(\hat{\rho}^{-2}\hat{\eta}^{-1}\log_{1}^{+}\hat{\eta}^{-1})$
when $z_{0}\in{\cal F}$, but our interest lies in the case where
$z_{0}\notin{\cal F}$ since an initial point $z_{0}\in{\cal F}$
is generally not known.

The AIP.QPM of \prettyref{sec:qp_aipp} is a modified cold-started
dynamic penalty method like the one just outlined, but which replaces
each R.AIPP call of the static penalty method with the AIPPM of \prettyref{sec:r_aipp}.
It has been shown in \prettyref{thm:qp_aipp_compl} that its inner
iteration complexity bound for solving is ${\cal O}(\hat{\rho}^{-2}\hat{\eta}^{-1})$.
This bound is established without assuming that $Z$ is bounded and
is clearly better than the one in \eqref{eq:naive_qp_compl}.

The next subsection considers a warm-started dynamic penalty method,
similar to the one described immediately after \prettyref{prop:eff_penalty_props},
in which the input $z_{0}$ to the R.AIPP call for solving the next
penalty subproblem is chosen to be the output $\hat{z}$ from the
R.AIPP call for solving the current one. It is shown in \prettyref{thm:rqp_aipp_compl}
that its inner iteration complexity is ${\cal O}(\hat{\rho}^{-2}\hat{\eta}^{-1}\log_{1}^{+}\hat{\eta}^{-1})$,
which is the same as the one for the AIP.QPM up to a logarithmic factor.
As a side remark, we note that although a warm-started version of
the AIP.QPM in \prettyref{sec:qp_aipp} can be also considered, the
aforementioned ${\cal O}(\hat{\rho}^{-2}\hat{\eta}^{-1})$ inner iteration
complexity bound was derived for its cold-started version.

\subsection{Statement and Properties of the R.AIP.QPM}

This subsection describes and establishes the iteration complexity
of the R.AIP.QPM.

We first state the R.AIP.QPM in \prettyref{alg:rqp_aippm}. Given
$(\theta,\tau)\in(2,\infty)\times\r_{++}$ and $z_{0}\in Z$, its
main idea is to call the R.AIPPM of \prettyref{sec:r_aipp} to obtain
approximate stationary points for a sequence of penalty subproblems
of the form 
\[
\min_{z\in{\cal Z}}\left\{ f_{c_{\ell}}(z)+h(z)\right\} 
\]
where $\{c_{\ell}\}_{\ell\ge1}$ is a strictly increasing sequence
that tends to infinity. At the end of each R.AIPPM call, a pair $([\hat{z},\hat{p}],[\hat{v},\hat{q}])$
is generated that satisfies \eqref{eq:eff_approx_cnco_a_ln1} and
the inclusion in \eqref{eq:eff_approx_cnco_a_ln2}, and the method
terminates when the inequality in \eqref{eq:eff_approx_cnco_a_ln2}
holds.

\begin{mdframed}
\mdalgcaption{R.AIP.QP Method}{alg:rqp_aippm}
\begin{smalgorithmic}
	\Require{$(\hat{\rho},\hat{\eta}) \in \r_{++}^2, \enskip M>0, \enskip h \in \cConv(Z), \enskip f \in {\cal C}_{M}(Z), \enskip \lam >0, \enskip (\theta, \tau) \in (2,\infty) \times \r_{++}, \enskip z_0 \in Z, \enskip {\cal A}\neq 0, \enskip S \subseteq {\cal R}$;}
	\Initialize{$\c_1 \gets M /\|{\cal A}\|^2, \hat{z}_0 \gets z_0;$}
	\vspace*{.5em}
	\Procedure{R.AIP.QP}{$f, h, {\cal A}, S, z_0, c_1, \lam, m, M, \theta, \tau, \hat{\rho}, \hat{\eta}$}
	\For{$\ell=1,...$}
		\StateStep{\algpart{1}\textbf{Attack} the $\ell^{\rm th}$ prox penalty subproblem.}
		\StateEq{$f_{c_\ell} \Lleftarrow f + \frac{c_\ell}{2} \cdot \dist^2({\cal A}(\cdot), S)$}
		\StateEq{$M_{c_\ell} \gets M + {c_\ell}\|{\cal A}\|^2$} \label{ln:rqp_aippm_Mc_def}
		\StateEq{$(\hat{z}_\ell, \hat{v}_\ell) \gets \text{R.AIPP}(f_{c_\ell}, h, \hat{z}_{\ell -1}, \lam, \theta, \tau, M_{c_\ell}, \hat{\rho})$} \label{ln:r_aippm_call}
		\StateEq{$\hat{p}_\ell \gets c_\ell \left[{\cal A}\hat{z_\ell}-\Pi_{S}({\cal A}\hat{z}_\ell)\right]$}
		\StateEq{$\hat{q}_\ell \gets \Pi_{S}({\cal A}\hat{z}_\ell)-{\cal A}\hat{z_\ell}$}
		\StateStep{\algpart{2}Either \textbf{stop} with a nearly feasible point or \textbf{increase} $c_\ell$.}
		\If{$\|\hat{q}_\ell\| \leq \hat{\eta}$} \label{ln:r_aippm_term}
			\StateEq{\Return{$([\hat{z}_\ell, \hat{p}_\ell], [\hat{v}_\ell, \hat{q}_\ell])$}}
		\EndIf
		\StateEq{$c_{\ell+1} \gets 2 c_\ell$} \label{ln:r_aippm_cdbl}
	\EndFor
	\EndProcedure
\end{smalgorithmic}
\end{mdframed}

We now make three comments about the R.AIP.QPM. To ease the discussion,
let us refer to the R.AIPP iterations in each R.AIPP call as \textbf{outer
iterations}\emph{, }the R.ACG iterations performed inside each R.AIPP
call as \textbf{inner iterations}\emph{, }and the iterations over
the indices $\ell$ as\textbf{ cycles}. First, it follows from \prettyref{prop:eff_penalty_props}(b)
that, for every $\ell\geq1$, the pair $([\hat{z},\hat{p}],[\hat{v},\hat{q}])=([\hat{z}_{\ell},\hat{p}_{\ell}],[\hat{v}_{\ell},\hat{q}_{\ell}])$
satisfies \eqref{eq:eff_approx_cnco_a_ln1} and the first inclusion
of \eqref{eq:eff_approx_cnco_a_ln2}. Second, since every cycle of
the R.AIP.QPM doubles $c$, the condition on $c$ in \prettyref{prop:eff_penalty_props}(c)
will be eventually satisfied. Hence, the residual $\hat{q}$ corresponding
to this $c$ will satisfy the condition $\|\hat{q}\|\le\hat{\eta}$
and the R.AIP.QPM will stop in view of its stopping criterion in \prettyref{ln:r_aippm_term}.
Finally, in view of the first and second comments, we conclude that
the R.AIP.QPM always outputs a pair $([\hat{z},\hat{p}],[\hat{v},\hat{q}])$
that solves \prettyref{prb:approx_eff_cnco_a}.

Before deriving the inner iteration complexity of the R.AIP.QPM, we
note that the number of inner iterations needed in the $(\ell+1)^{{\rm th}}$
execution of the R.AIPPM depends on the quantity $\varphi_{c_{\ell}}(\hat{z}_{l})-\hat{\varphi}_{c_{\ell}}$
(see the left-hand-side of \eqref{eq:naive_qp_compl} with $(c,z_{0})=(c_{\ell},\hat{z}_{\ell})$).
The result below shows that the warm-start strategy in \prettyref{ln:r_aippm_call}
of the method together with the boundedness of $Z$ imply that the
aforementioned quantity has an upper bound that is independent of
the size of the parameter $c_{\ell}$.
\begin{lem}
\label{lem:poten_bd} Let $c_{1}$ be as in the initialization of
the R.AIP.QPM and define 
\begin{equation}
\begin{gathered}S(z_{0}):=\varphi_{c_{1}}(z_{0})-\hat{\varphi}_{c_{1}},\\
Q(z_{0}):=S(z_{0})+2\left[\hat{\varphi}_{*}-\hat{\varphi}_{0}+\hat{\rho}D_{z}+\frac{1}{2}\underline{m}D_{z}^{2}\right],
\end{gathered}
\label{eq:SQ_def}
\end{equation}
where $\hat{\varphi}_{*}$ and $\hat{\varphi}_{0}$are as in \ref{prb:eq:eff_cnco_a}
and \eqref{eq:eff_varPhiC_def},, respectively. Then, for every $\ell\geq1$,
we have
\begin{align}
\varphi_{c_{\ell}}(\hat{z}_{\ell})-\hat{\varphi}_{c_{\ell}}\leq Q(\hat{z}_{0}).\label{eq:poten_bd}
\end{align}
\end{lem}

\begin{proof}
All line numbers referenced in this proof are with respect to \prettyref{alg:rqp_aippm}.
The case in which $\ell=1$ follows trivially from the definition
of $S(z_{0})$. Consider now the case in which $\ell\geq2$. Remark
that \prettyref{ln:r_aippm_cdbl} and \prettyref{prop:eff_penalty_props}
respectively imply that $c_{\ell}=2c_{\ell-1}$ and $(\hat{z}_{\ell},\hat{v}_{\ell})$
solves \prettyref{prb:approx_eff_nco} with $f=f_{c_{\ell-1}}$. It
now follows from the aforementioned remarks, the last inequality in
\eqref{eq:pen_topo} with $c=c_{\ell}$, the definition of $\hat{q}_{\ell}$,
and \prettyref{prop:eff_penalty_props}(b) with $(\hat{z},c)=(\hat{z}_{\ell},c_{\ell})$,
that 
\begin{align}
\varphi_{c_{\ell}}(\hat{z}_{\ell})-\hat{\varphi}_{c_{\ell}} & \leq\varphi_{c_{\ell}}(\hat{z}_{\ell})-\hat{\varphi}_{0}=\varphi(\hat{z}_{\ell})+2\left[\frac{c_{\ell-1}}{2}\|\hat{q}_{\ell}\|^{2}\right]-\hat{\varphi}_{0}\nonumber \\
 & \leq\varphi(\hat{z}_{\ell})+2\left[\hat{\varphi}_{*}-\varphi(\hat{z}_{\ell})+\hat{\rho}D_{z}+\frac{1}{2}\underline{m}D_{z}^{2}\right]-\hat{\varphi}_{0}\nonumber \\
 & =2\hat{\varphi}_{*}-\varphi(\hat{z}_{\ell})-\hat{\varphi}_{0}+2\left[\hat{\rho}D_{z}+\frac{1}{2}\underline{m}D_{z}^{2}\right]\nonumber \\
 & \leq2\left[\hat{\varphi}_{*}-\hat{\varphi}_{0}+\hat{\rho}D_{z}+\frac{1}{2}\underline{m}D_{z}^{2}\right]\leq Q(z_{0}).\label{eq:poten_bd1}
\end{align}
\end{proof}
We now establish the iteration complexity of the R.AIP.QPM in the
following result.
\begin{thm}
\label{thm:rqp_aipp_compl}Let $T_{\hat{\eta}}$ be as in \eqref{eq:pen_c_def}
and define 
\begin{equation}
\Xi_{\hat{\eta}}:=M+T_{\hat{\eta}}\|A\|^{2}\quad\forall\hat{\eta}>0.\label{eq:pen_compl_const}
\end{equation}
Then, the R.AIP.QPM outputs a pair $([\hat{z},\hat{p}],[\hat{v},\hat{q}])$
that solves \prettyref{prb:approx_eff_cnco_a} in
\begin{equation}
{\cal O}\left(\sqrt{\Xi_{\hat{\eta}}+\xi}\left(\frac{\sqrt{\xi}\theta\left[1+\tau\right]Q(z_{0})}{\hat{\rho}^{2}}+\sqrt{\lam_{0}}\right)\log_{1}^{+}\left(C_{\theta,\tau}\lam_{0}\Xi_{\hat{\eta}}\right)\right),\label{eq:r_qp_aipp_compl}
\end{equation}
inner iterations, where $C_{\theta,\tau}$, $\xi$, and $Q(z_{0})$
are as in \eqref{eq:C_theta_tau_def}, \eqref{eq:eff_xi_def}, and
\eqref{eq:SQ_def}, respectively. 
\end{thm}

\begin{proof}
The fact that the output of the R.AIP.QPM solves \prettyref{prb:approx_eff_cnco_a}
is an immediate consequence of \prettyref{prop:eff_penalty_props}
and the termination condition in \prettyref{ln:r_aippm_term} of the
method.

Let us now prove the desired complexity bound. Let $\bar{\ell}\geq1$
be the smallest index such that $c_{\bar{\ell}}\geq T_{\hat{\eta}}$.
Since the R.AIP.QPM calls the R.AIPPM with $(M,f)=(M_{c_{\ell}},f_{c_{\ell}})$
at every cycle, it follows from \prettyref{lem:poten_bd} and \prettyref{thm:r_aipp_compl},
with $M=M_{c_{\ell}}$, that the total number of inner iterations
at the $\ell^{{\rm th}}$ cycle of the R.AIP.QPM is on the order of
\begin{equation}
{\cal O}\left(\sqrt{\left[1+\frac{\xi}{M}\right]\Xi_{\hat{\eta}}}\left[\frac{\sqrt{\xi}\theta\left[1+\tau\right]Q(z_{0})}{\hat{\rho}^{2}}+\sqrt{\lam_{0}}\right]\log_{1}^{+}\left[C_{\theta,\tau}\lam_{0}M_{c_{\ell}}\right]\right).\label{eq:qp_compl_stage}
\end{equation}
Hence, the R.AIP.QPM method stops in a total number of inner iterations
bounded above by the sum of the quantity in \eqref{eq:qp_compl_stage}
over $\ell=1,\ldots,\bar{\ell}$.

We now focus on simplifying some quantities in the aforementioned
sum. Using the fact that $M=c_{1}\|{\cal A}\|^{2}$, we obtain the
bound 
\begin{align}
M_{c_{\ell}} & =M+c_{\ell}\|{\cal A}\|^{2}=M+2^{\ell-1}c_{1}\|{\cal A}\|^{2}\nonumber \\
 & \leq2^{\ell-1}\left(M+c_{1}\|{\cal A}\|^{2}\right)=2^{\ell}c_{1}\|{\cal A}\|^{2}.\label{eq:pen_curv_bd}
\end{align}
Now, if $\bar{\ell}=1$, then the above inequality implies that $M_{c_{\bar{\ell}}}\leq2c_{1}\|{\cal A}\|^{2}=2M={\cal O}(\Xi_{\hat{\eta}})$.
Assume then that $\bar{\ell}\geq2$. Observe that the definition of
$\bar{\ell}$ implies that $2^{\bar{\ell}-1}c_{1}=c_{\bar{\ell}}\leq T_{\hat{\eta}}$
or, equivalently, $\sqrt{c_{1}}\sqrt{2}^{\bar{l}}\leq\sqrt{2T_{\hat{\eta}}}$.
Combining the previous inequality with \eqref{eq:pen_curv_bd}, we
conclude that 
\begin{align}
 & \sum_{k=1}^{\bar{\ell}}\sqrt{M_{c_{k}}+\xi}\leq\sum_{k=1}^{\bar{\ell}}\sqrt{2^{k}c_{1}\|{\cal A}\|^{2}+\xi}\leq\sqrt{2}^{\bar{\ell}}\left(1+\sqrt{2}\right)\sqrt{2c_{1}\|{\cal A}\|^{2}+\xi}\nonumber \\
 & \leq8\sqrt{T_{\hat{\eta}}\left(\|{\cal A}\|^{2}+\frac{\xi}{c_{1}}\right)}=\sqrt{\|{\cal A}\|^{2}T_{\hat{\eta}}\left(1+\frac{\xi}{M}\right)}\nonumber \\
 & ={\cal O}\left(\sqrt{\left[1+\frac{\xi}{M}\right]\Xi_{\hat{\eta}}}\right),\label{eq:lam_Mk_bd1}
\end{align}
and also that
\begin{align}
\log_{1}^{+}\left(M_{c_{\ell}}\right)\leq\log_{1}^{+}\left(2^{\bar{\ell}}c_{0}\|A\|^{2}\right)\leq\log_{1}^{+}\left(T_{\hat{\eta}}\|A\|^{2}\right)={\cal O}\left(\log_{1}^{+}\Xi_{\hat{\eta}}\right).\label{eq:lam_Mk_bd2}
\end{align}
It now follows from \eqref{eq:qp_compl_stage}, \eqref{eq:lam_Mk_bd1},
and \eqref{eq:lam_Mk_bd2} that the R.AIP.QPM stops in a total number
of inner bounded by the quantity in \eqref{eq:r_qp_aipp_compl}.
\end{proof}
The result below presents the iteration complexity of the R.AIP.QPM
with inputs $(\theta,\tau)=(4,2)$ and $\lam=1/\underline{m}$.
\begin{cor}
\label{cor:spec_rqp_aipp_compl}The R.AIP.QPM with inputs $(\theta,\tau)=(4,2)$
and $\lam=1/\underline{m}$ outputs a pair $([\hat{z},\hat{p}],[\hat{v},\hat{q}])$
that solves in \prettyref{prb:approx_eff_cnco_a} in 
\begin{equation}
{\cal O}\left(\sqrt{\left[1+\frac{\underline{m}}{M}\right]\Xi_{\hat{\eta}}}\left[\frac{\sqrt{\underline{m}}Q(z_{0})}{\hat{\rho}^{2}}+\frac{1}{\sqrt{\underline{m}}}\right]\log_{1}^{+}\left[\frac{\Xi_{\hat{\eta}}}{\underline{m}}\right]\right)\label{eq:spec_r_qp_aipp_compl}
\end{equation}
inner iterations, where $\Xi_{\hat{\eta}}$ and $Q(z_{0})$ are as
in \eqref{eq:pen_compl_const} and \eqref{eq:SQ_def}, respectively.
\end{cor}

Note that in terms of the tolerance pair $(\hat{\rho},\hat{\eta})$,
it is ${\cal O}(\hat{\rho}^{-2}\hat{\eta}^{-1}\log_{1}^{+}\hat{\eta}^{-1})$,
which improves upon the complexity in \eqref{eq:naive_qp_compl} by
a $\Theta(\hat{\eta}^{-2})$ multiplicative factor. 

We now end this section by discussing how the above R.AIP.QP instance
in \prettyref{cor:spec_rqp_aipp_compl} compares to the AIP.QP instance
in \prettyref{cor:spec_qp_aipp_compl}. First, recall that the AIP.QPM
requires the knowledge of an upper bound $m$ on $\underline{m}$.
Under the same assumptions of this section, it can be shown, using
the bound $m\leq M$ and \prettyref{thm:qp_aipp_compl} with $\hat{c}=0$,
that AIP.QPM instance iteration complexity is 
\begin{align}
 & {\cal O}\left(\sqrt{\frac{\Xi_{\hat{\eta}}}{m}}\left[\frac{mQ(z_{0})}{\hat{\rho}^{2}}+\log_{1}^{+}\left(\frac{\Xi_{\hat{\eta}}}{m}\right)\right]\right)\nonumber \\
 & ={\cal O}\left(\frac{\sqrt{m\Xi_{\hat{\eta}}}Q(z_{0})}{\hat{\rho}^{2}}+\sqrt{\frac{\Xi_{\hat{\eta}}}{m}}\log_{1}^{+}\left[\frac{\Xi_{\hat{\eta}}}{m}\right]\right)\label{eq:qp_aipp_compl}
\end{align}
On the other hand, using the bound $m\leq M$ it can be shown that
\eqref{eq:spec_r_qp_aipp_compl} reduces to 
\begin{align}
 & {\cal O}\left(\sqrt{\Xi_{\hat{\eta}}}\left[\frac{\sqrt{\underline{m}}Q(z_{0})}{\hat{\rho}^{2}}+\sqrt{\frac{1}{\underline{m}}}\right]\log_{1}^{+}\left[\frac{\Xi_{\hat{\eta}}}{\underline{m}}\right]\right)\nonumber \\
 & ={\cal O}\left(\frac{\sqrt{\underline{m}\Xi_{\hat{\eta}}}Q(z_{0})}{\hat{\rho}^{2}}\log_{1}^{+}\left[\frac{\Xi_{\hat{\eta}}}{\underline{m}}\right]+\sqrt{\frac{\Xi_{\hat{\eta}}}{\underline{m}}}\log_{1}^{+}\left[\frac{\Xi_{\hat{\eta}}}{\underline{m}}\right]\right).\label{eq:r_qp_aipp_spec_compl1}
\end{align}
Note that \eqref{eq:r_qp_aipp_spec_compl1} is as good as \eqref{eq:qp_aipp_compl}
when $\Xi_{\hat{\eta}}/\underline{m}={\cal O}(1)$ and is only worse
by a factor of $\log\hat{\eta}^{-1}$ when $m=\bar{m}$.

\section{Numerical Experiments}

\label{chap:numerical}

This section presents several numerical experiments that use the various
algorithms developed in this chapter. All experiments are run on Linux
64-bit machines each containing Xeon E5520 processors and at least
8 GB of memory using MATLAB 2020a. Supporting code for some of the
benchmarked solvers was generously denoted by the original authors
Jiaming Liang, Saeed Ghadimi, and Guanghui ``George'' Lan. It is
worth mentioning that the complete code for reproducing the experiments
is freely available online\footnote{See the code in \texttt{./tests/thesis/} from the GitHub repository
\href{https://github.com/wwkong/nc_opt/}{https://github.com/wwkong/nc\_opt/}}.

\subsection{Unconstrained Optimization Problems}

\label{subsec:num_unconstr}

This subsection examines the performance of several solvers for finding
approximate stationary points of \ref{prb:eq:nco} where $(f,h)$
satisfy assumptions \ref{asmp:nco1}--\ref{asmp:nco3} of \prettyref{chap:unconstr_nco}. 

The algorithms benchmarked in this section are as follows.
\begin{itemize}
\item \textbf{AIPP}: a variant of \prettyref{alg:r_aippm} with starting
inputs $\lam_{0}=1/m$, $\theta=2$, and $\tau=10(\lam_{0}M+1)$.
More specifically, this variant adaptively changes the value of $\tau$
based on the update rule in \citep{Kong2019a}, uses the bitonic stepsize
update rule described at the end of \prettyref{subsec:r_aipp}, and
initializes $L_{0}$ for each R.ACGM call as follows: at the $k^{{\rm th}}$
outer iteration, if $L_{-1}$ denotes either $\lam_{0}M+1$ for $k=1$
or the last obtained estimate of $L_{k}$ from a previous R.ACG call
for $k>1$, then $L_{0}$ of the current R.ACG call is set to $L_{0}=\lam_{k}(L_{1}-1)/[100\lam_{\max\{k-1,1\}}]+1$.
Moreover, at the $k^{{\rm th}}$ iterate, it uses the $z_{k-1}$ as
the initial starting point for its $k^{{\rm th}}$ R.ACG call.
\item \textbf{AG}: an instance of \citep[Algorithm 2]{Ghadimi2016} in which
$L_{\Psi}=\max\{m,M\}$ and the sequences $\{\alpha_{k}\}_{k\geq1}$,
$\{\beta_{k}\}_{k\geq1}$, and $\{\lam_{k}\}_{k\geq1}$ are as in
\citep[Corollary 1]{Ghadimi2016}.
\item \textbf{NC-FISTA}: an instance of the algorithm in \citep[Section]{Liang2019}
in which, defining $\xi=1.05m$, we have $A_{0}=2\xi(\xi+m)/(\xi-m)^{2}$.
\item \textbf{UPFAG}: an instance of \citep[Algorithm 1]{Ghadimi2019} in
which $H=\max\{m,M\}$, $\nu=1$, $(\gamma_{1},\gamma_{2},\gamma_{3})=(0.4,0.4,1)$,
$\delta=10^{-3}$, $\hat{\lam}_{0}=H$, $\hat{\beta}_{0}=1/H$, and
the line search method the Barzilai-Borwein method given in \citep[Equation 2.12]{Ghadimi2019}
with $\sigma=10^{-10}$.
\end{itemize}
Given a tolerance $\hat{\rho}>0$ and an initial point $z_{0}\in Z$,
each algorithm above seeks a pair $(\hat{z},\hat{v})\in Z\times{\cal Z}$
satisfying 
\begin{equation}
\hat{v}\in\nabla g(\hat{z})+\pt h(\hat{z}),\quad\frac{\|\hat{v}\|}{\|\nabla g(z_{0})\|+1}\leq\hat{\rho}.\label{eq:term_unconstr}
\end{equation}
Moreover, each algorithm is given a time limit of 4000 seconds. Iteration
counts are not reported for instances which were unable to obtain
$(\hat{z},\hat{v})$ as above. The bold numbers in each of the tables
in this section highlight the algorithm that performed the most efficiently
in terms of iteration count or total runtime.

\subsubsection{Quadratic Matrix Problem}

\label{subsec:qmp}

This sub-subsection presents computational results for the unconstrained
quadratic matrix (QM) problem considered in \citep{Kong2019}. More
specifically, given a pair of dimensions $(l,n)\in\n^{2}$, scalar
pair $(\alpha_{1},\alpha_{2})\in\r_{++}^{2}$, linear operators ${\cal B}:\mathbb{S}_{+}^{n}\mapsto\r^{n}$
and ${\cal C}:\mathbb{S}_{+}^{n}\mapsto\r^{l}$ defined by 
\[
\left[{\cal B}(z)\right]_{j}=\left\langle B_{j},z\right\rangle _{F},\quad\left[{\cal C}(z)\right]_{i}=\left\langle C_{i},z\right\rangle _{F},
\]
for matrices $\{B_{j}\}_{j=1}^{n},\{C_{i}\}_{i=1}^{l}\subseteq\r^{n\times n}$
, positive diagonal matrix $D\in\r^{n\times n}$, and vector $d\in\r^{l}$,
we consider the QM problem
\begin{align*}
\min_{Z\in\r^{n\times n}}\  & \frac{\alpha_{1}}{2}\|{\cal C}Z-d\|^{2}-\frac{\alpha_{2}}{2}\|D{\cal B}Z\|^{2}\\
\text{subject to}\  & Z\in P_{n}
\end{align*}
where $P_{n}=\{Z\in\mathbb{S}_{+}^{n}:\trc z=1\}$ denotes the $n$-dimensional
spectraplex.

We now describe the experiment parameters for the instances considered.
First, the dimensions were set to be $(l,n)=(50,200)$ and only 2.5\%
of the entries of the submatrices $B_{j}$ and $C_{i}$ being nonzero.
Second, the entries of $B_{j},C_{i}$, and $d$ (resp., $D$) are
generated by sampling from the uniform distribution ${\cal U}[0,1]$
(resp., ${\cal U}\{1,...,1000\}$). Third, the initial starting point
is $z_{0}=I_{n}/n$. Fourth, with respect to the termination criterion
\eqref{eq:term_unconstr}, the key problem parameters, for every $Z\in\mathbb{S}_{+}^{n}$,
are 
\[
f(Z)=\frac{\alpha_{1}}{2}\|{\cal C}Z-d\|^{2}-\frac{\alpha_{2}}{2}\|D{\cal B}Z\|^{2},\quad h(z)=\delta_{P_{n}}(z),\quad\hat{\rho}=10^{-7}.
\]
Finally, each problem instance considered is based on a specific curvature
pair $(m,M)\in\r_{++}^{2}$ for which the scalar pair $(\alpha_{1},\alpha_{2})\in\r_{++}^{2}$
is selected so that $M=\lambda_{\max}(\nabla^{2}g)$ and $-m=\lambda_{\min}(\nabla^{2}g)$.
In \prettyref{app:comp_details}, we describe how to generate the
pair $(\alpha_{1},\alpha_{2})\in\r_{++}^{2}$ under the reasonable
assumption that ${\cal B}$, ${\cal C}$, and $D$ are nonzero.

The table of iteration counts and total runtimes are given in \prettyref{tab:qmp_iter}
and \prettyref{tab:qmp_runtime}, respectively. 
\begin{center}
\begin{table}[th]
\begin{centering}
\begin{tabular}{|>{\centering}p{0.7cm}>{\centering}p{0.7cm}|>{\centering}p{1.8cm}>{\centering}p{1.8cm}>{\centering}p{1.8cm}>{\centering}p{1.8cm}|}
\hline 
\multicolumn{2}{|c|}{\textbf{\small{}$(m,M)$}} & \multicolumn{4}{c|}{\textbf{\small{}Iteration Count}}\tabularnewline
\hline 
{\footnotesize{}$m$} & {\footnotesize{}$M$} & {\footnotesize{}UPFAG} & {\footnotesize{}NC-FISTA} & {\footnotesize{}AG} & {\footnotesize{}AIPP}\tabularnewline
\hline 
{\footnotesize{}$10^{1}$} & {\footnotesize{}$10^{3}$} & {\footnotesize{}4766} & \textbf{\footnotesize{}1463} & {\footnotesize{}4139} & {\footnotesize{}2420}\tabularnewline
{\footnotesize{}$10^{1}$} & {\footnotesize{}$10^{4}$} & {\footnotesize{}7768} & \textbf{\footnotesize{}1820} & {\footnotesize{}3439} & {\footnotesize{}1851}\tabularnewline
{\footnotesize{}$10^{1}$} & {\footnotesize{}$10^{5}$} & {\footnotesize{}10452} & {\footnotesize{}3873} & {\footnotesize{}3326} & \textbf{\footnotesize{}898}\tabularnewline
{\footnotesize{}$10^{1}$} & {\footnotesize{}$10^{6}$} & {\footnotesize{}11422} & {\footnotesize{}4432} & {\footnotesize{}3316} & \textbf{\footnotesize{}801}\tabularnewline
\hline 
\end{tabular}
\par\end{centering}
\caption{Iteration counts for QM problems.\label{tab:qmp_iter}}
\end{table}
\par\end{center}

\begin{center}
\begin{table}[th]
\begin{centering}
\begin{tabular}{|>{\centering}p{0.7cm}>{\centering}p{0.7cm}|>{\centering}p{1.8cm}>{\centering}p{1.8cm}>{\centering}p{1.8cm}>{\centering}p{1.8cm}|}
\hline 
\multicolumn{2}{|c|}{\textbf{\small{}$(m,M)$}} & \multicolumn{4}{c|}{\textbf{\small{}Runtime}}\tabularnewline
\hline 
{\footnotesize{}$m$} & {\footnotesize{}$M$} & {\footnotesize{}UPFAG} & {\footnotesize{}NC-FISTA} & {\footnotesize{}AG} & {\footnotesize{}AIPP}\tabularnewline
\hline 
{\footnotesize{}$10^{1}$} & {\footnotesize{}$10^{3}$} & {\footnotesize{}242.67} & \textbf{\footnotesize{}32.83} & {\footnotesize{}123.54} & {\footnotesize{}71.42}\tabularnewline
{\footnotesize{}$10^{1}$} & {\footnotesize{}$10^{4}$} & {\footnotesize{}377.05} & \textbf{\footnotesize{}40.57} & {\footnotesize{}102.11} & {\footnotesize{}54.86}\tabularnewline
{\footnotesize{}$10^{1}$} & {\footnotesize{}$10^{5}$} & {\footnotesize{}485.79} & {\footnotesize{}89.18} & {\footnotesize{}102.01} & \textbf{\footnotesize{}26.24}\tabularnewline
{\footnotesize{}$10^{1}$} & {\footnotesize{}$10^{6}$} & {\footnotesize{}499.48} & {\footnotesize{}107.1} & {\footnotesize{}106.56} & \textbf{\footnotesize{}26.37}\tabularnewline
\hline 
\end{tabular}
\par\end{centering}
\caption{Runtimes for QM problems.\label{tab:qmp_runtime}}
\end{table}
\par\end{center}

\subsubsection{Support Vector Machine Problem}

This sub-subsection presents computational results for the support
vector machine (SVM) considered in \citep{Ghadimi2019}. More specifically,
given a pair of dimensions $(n,k)\in\n^{2}$, matrix $U\in\r^{n\times k},$
and vector $v\in\{-1,+1\}^{n},$ this subsection considers the (sigmoidal)
SVM problem 
\[
\min_{z\in\r^{n}}\ \frac{1}{k}\sum_{i=1}^{k}\left[1-\tanh\left(v_{i}\left\langle u_{i},z\right\rangle \right)\right]+\frac{1}{2k}\|z\|^{2},
\]
where $u_{i}$ denotes the $i^{{\rm th}}$ column of $U$.

We now describe the experiment parameters for the instances considered.
First, the entries of $U$ are generated by sampling from the uniform
distribution ${\cal U}[0,1]$, with only 5\% of the entries being
nonzero, and $v=\mathrm{sgn}(U^{T}x)$ where the entries of $x$ are
sampled from the uniform distribution over the $k$--dimensional
ball centered at 0 with radius 50. Second, the initial starting point
is $z_{0}=0$. Third, the curvature parameters for each problem instance
are $m=M=(4\sqrt{3}\|U\|_{F}^{2})/(9k)+1/k.$ Fourth, with respect
to the termination criterion \eqref{eq:term_unconstr}, the key problem
parameters, for every $z\in\r^{n}$, are 
\[
f(z)=\frac{1}{k}\sum_{i=1}^{k}\left[1-\tanh\left(v_{i}\left\langle u_{i},z\right\rangle \right)\right]+\frac{1}{2k}\|z\|^{2},\quad h(z)=0,\quad\hat{\rho}=10^{-3}.
\]
Finally, each problem instance considered is based on a specific dimension
pair $(n,k)\in\n^{2}$.

The table of iteration counts and total runtimes are given in \prettyref{tab:qmp_iter}
and \prettyref{tab:qmp_runtime}, respectively. 
\begin{center}
\begin{table}[th]
\begin{centering}
\begin{tabular}{|>{\centering}p{1cm}>{\centering}p{1cm}|>{\centering}p{1.8cm}>{\centering}p{1.8cm}>{\centering}p{1.8cm}>{\centering}p{1.8cm}|}
\hline 
\multicolumn{2}{|c|}{\textbf{\small{}$(n,k)$}} & \multicolumn{4}{c|}{\textbf{\small{}Iteration Count}}\tabularnewline
\hline 
{\footnotesize{}$n$} & {\footnotesize{}$k$} & {\footnotesize{}UPFAG} & {\footnotesize{}NC-FISTA} & {\footnotesize{}AG} & {\footnotesize{}AIPP}\tabularnewline
\hline 
{\footnotesize{}1000} & {\footnotesize{}500} & \textbf{\footnotesize{}80} & {\footnotesize{}3024} & {\footnotesize{}782} & {\footnotesize{}145}\tabularnewline
{\footnotesize{}2000} & {\footnotesize{}1000} & \textbf{\footnotesize{}194} & {\footnotesize{}8360} & {\footnotesize{}1191} & {\footnotesize{}234}\tabularnewline
{\footnotesize{}4000} & {\footnotesize{}2000} & {\footnotesize{}1112} & {\footnotesize{}22485} & {\footnotesize{}1346} & \textbf{\footnotesize{}392}\tabularnewline
{\footnotesize{}8000} & {\footnotesize{}4000} & \textbf{\footnotesize{}327} & {\footnotesize{}-} & {\footnotesize{}1646} & {\footnotesize{}782}\tabularnewline
\hline 
\end{tabular}
\par\end{centering}
\caption{Iteration counts for SVM problems.\label{tab:svm_iter}}
\end{table}
\par\end{center}

\begin{center}
\begin{table}[th]
\begin{centering}
\begin{tabular}{|>{\centering}p{1cm}>{\centering}p{1cm}|>{\centering}p{1.8cm}>{\centering}p{1.8cm}>{\centering}p{1.8cm}>{\centering}p{1.8cm}|}
\hline 
\multicolumn{2}{|c|}{\textbf{\small{}$(n,k)$}} & \multicolumn{4}{c|}{\textbf{\small{}Iteration Count}}\tabularnewline
\hline 
{\footnotesize{}$n$} & {\footnotesize{}$k$} & {\footnotesize{}UPFAG} & {\footnotesize{}NC-FISTA} & {\footnotesize{}AG} & {\footnotesize{}AIPP}\tabularnewline
\hline 
{\footnotesize{}1000} & {\footnotesize{}500} & {\footnotesize{}5.46} & {\footnotesize{}71.64} & {\footnotesize{}19.11} & \textbf{\footnotesize{}5.03}\tabularnewline
{\footnotesize{}2000} & {\footnotesize{}1000} & {\footnotesize{}35.88} & {\footnotesize{}570.19} & {\footnotesize{}84.85} & \textbf{\footnotesize{}21.14}\tabularnewline
{\footnotesize{}4000} & {\footnotesize{}2000} & {\footnotesize{}775.77} & {\footnotesize{}3447.60} & {\footnotesize{}179.31} & \textbf{\footnotesize{}66.26}\tabularnewline
{\footnotesize{}8000} & {\footnotesize{}4000} & \textbf{\footnotesize{}659.85} & {\footnotesize{}4000.00} & {\footnotesize{}1286.05} & {\footnotesize{}780.07}\tabularnewline
\hline 
\end{tabular}
\par\end{centering}
\caption{Runtimes for SVM problems.\label{tab:svm_runtime}}
\end{table}
\par\end{center}

\subsection{Function Constrained Optimization Problems}

\label{subsec:num_constr}

This section examines the performance of several solvers for finding
approximate stationary points of \ref{prb:eq:cnco} where $(f,h,g,S)$
satisfy \ref{asmp:nco1}--\ref{asmp:nco2} and either \ref{asmp:cnco_a1}--\ref{asmp:cnco_a2}
or \ref{asmp:cnco_b1}--\ref{asmp:cnco_b3} of \prettyref{chap:cnco}. 

The algorithms benchmarked in this section are as follows.
\begin{itemize}
\item \textbf{AIP.QP}: a variant of \prettyref{alg:rqp_aippm} in which
the R.AIPPM is replaced with the R.AIPP variant described in \prettyref{subsec:num_unconstr}
and $c_{0}=\max\left\{ 1,\hat{c}+L_{f}/[B_{g}^{(1)}]^{2}\right\} $.
\item \textbf{AIP.AL}: an variant of \prettyref{alg:aip_alm} in which the
parameter inputs for the S.ACGM and the variant are given by 
\begin{gather*}
c_{1}=\max\left\{ 1,\frac{L_{f}}{\left[B_{g}^{(1)}\right]^{2}}\right\} ,\quad\theta=\frac{1}{\sqrt{2}},\quad\sigma=\min\left\{ \frac{\nu}{\sqrt{L_{k-1}^{\psi}}},\theta\right\} ,\\
\nu=\sqrt{\theta\left(\lam M+1\right)},\quad\lam=\frac{1}{2m},\quad p_{0}=0,
\end{gather*}
and the condition on $\Delta_{k}$ in \prettyref{ln:al_Delta_cond}
of \prettyref{alg:aip_alm} is replaced by 
\[
\Delta_{k}\leq\frac{\lam(1-\theta^{2})\hat{\rho}^{2}}{4(1+2\nu)^{2}}.
\]
\item \textbf{AG.QP}: a variant of \prettyref{alg:rqp_aippm} in which the
R.AIPPM is replaced with the AG method described in \prettyref{subsec:num_unconstr}
and $c_{0}=\max\left\{ 1,\hat{c}+L_{f}/[B_{g}^{(1)}]^{2}\right\} $.
\item \textbf{iALM}: an instance of \citep[Algorithm 3]{Li2020} in which
\[
\sigma=5,\quad\beta_{0}=\max\left\{ 1,\frac{L_{f}}{\left[B_{g}^{(1)}\right]^{2}}\right\} ,\quad w_{0}=1,\quad\boldsymbol{y}^{0}=0,\quad\gamma_{k}=\frac{\left(\log2\right)\|c(x^{1})\|}{(k+1)\left[\log(k+2)\right]^{2}},
\]
for every $k\geq1$, and the starting point given to the $k^{{\rm th}}$
APG call is set to be $\boldsymbol{x}^{k-1}$, which is the prox center
for the $k^{{\rm th}}$ prox subproblem.
\end{itemize}
Given a tolerance pair $(\hat{\rho},\hat{\eta})\in\r_{++}^{2}$ and
an initial point $z_{0}\in Z$, each algorithm in this section seeks
a pair $([\hat{z},\hat{p}],[\hat{p},\hat{q}])\in[Z\times{\cal R}]\times[{\cal Z}\times{\cal R}]$
satisfying 
\begin{gather}
\begin{gathered}\hat{v}\in\nabla f(\hat{z})+\pt h(\hat{z})+\nabla g(\hat{z})\hat{p},\quad g(\hat{z})+\hat{q}\in S\\
\|\hat{v}\|\leq\hat{\rho},\quad\|\hat{q}\|\leq\hat{\eta}.
\end{gathered}
\label{eq:term_lin_constr}
\end{gather}
For cone-constrained problems, i.e. where $S$ is a closed convex
cone $-{\cal K}$, the following additional conditions are also required:
\[
\inner{g(\hat{z})+\hat{q}}{\hat{p}}=0,\quad\hat{p}\succeq_{{\cal K}^{+}}0,
\]
 where ${\cal K}^{+}$ denotes the dual cone of ${\cal K}$. Moreover,
each algorithm is given a time limit of 4000 seconds. Iteration counts
are not reported for instances which were unable to obtain $([\hat{z},\hat{p}],[\hat{p},\hat{q}])$
as above. The bold numbers in each of the tables in this section highlight
the algorithm that performed the most efficiently in terms of iteration
count or total runtime.

It is worth mentioning that for problems where $S$ is a pointed convex
cone ${\cal -K}$, the iALM method attempts to solve the equivalent
problem with equality constraints under an additional slack variable.
More specifically, it introduces an additional slack variable $s$,
and considers the equivalent problem 
\[
\min_{(z,s)\in{\cal Z}\times{\cal R}}\left\{ f(z)+h(z):c(z)+s=0,s\succeq_{{\cal K}}0\right\} .
\]

\subsubsection{Linearly-Constrained Quadratic Matrix Problem}

\label{subsec:lc_qmp}

This sub-subsection presents computational results for the linearly-constrained
quadratic matrix (LC-QM) problem considered in \citep{Kong2019}.
More specifically, given a pair of dimensions $(l,n)\in\n^{2}$, scalar
pair $(\alpha_{1},\alpha_{2})\in\r_{++}^{2}$, linear operators ${\cal A}:\mathbb{S}_{+}^{n}\mapsto\r^{l}$
, ${\cal B}:\mathbb{S}_{+}^{n}\mapsto\r^{n}$, and ${\cal C}:\mathbb{S}_{+}^{n}\mapsto\r^{l}$
defined by 
\[
\left[{\cal A}Z\right]_{i}=\left\langle A_{i},Z\right\rangle _{F},\quad\left[{\cal B}Z\right]_{j}=\left\langle B_{j},Z\right\rangle _{F},\quad\left[{\cal C}Z\right]_{i}=\left\langle C_{i},Z\right\rangle _{F},
\]
for matrices $\{A_{i}\}_{i=1}^{l},\{B_{j}\}_{j=1}^{n},\{C_{i}\}_{i=1}^{l}\subseteq\r^{n\times n}$,
positive diagonal matrix $D\in\r^{n\times n}$, and vector pair $(b,d)\in\r^{l}\times\r^{l}$,
we consider the LC-QM problem
\begin{align*}
\min_{Z\in\r^{n\times n}}\  & \frac{\alpha_{1}}{2}\|{\cal C}Z-d\|^{2}-\frac{\alpha_{2}}{2}\|D{\cal B}Z\|^{2}\\
\text{subject to}\  & {\cal A}Z\in\{b\},\\
 & Z\in P_{n},
\end{align*}
where $P_{n}=\{Z\in\mathbb{S}_{+}^{n}:\trc Z=1\}$ denotes the $n$-dimensional
spectraplex.

We now describe the experiment parameters for the instances considered.
First, the dimensions were set to be $(\ell,n)=(10,50)$ and only
1.0\% of the entries of the submatrices $A_{i},B_{j},$ and $C_{i}$
being nonzero. Second, the entries of $A_{i},B_{j},C_{i},b$, and
$d$ (resp., $D$) were generated by sampling from the uniform distribution
${\cal U}[0,1]$ (resp., ${\cal U}\{1,...,1000\}$). Third, the initial
starting point $z_{0}$ was chosen to be a random point in $\mathbb{S}_{+}^{n}$.
More specifically, three unit vectors $\nu_{1},\nu_{2},\nu_{3}\in\r^{n}$
and three scalars $e_{1},e_{2},e_{2}\in\r_{+}$ are first generated
by sampling vectors $\widetilde{\nu}_{i}\sim{\cal U}^{n}[0,1]$ and
scalars $\widetilde{d}_{i}\sim{\cal U}[0,1]$ and setting $\nu_{i}=\widetilde{\nu}_{i}/\|\widetilde{\nu}_{i}\|$
and $e_{i}=\widetilde{e}_{i}/(\sum_{j=1}^{3}\widetilde{e}_{i})$ for
$i=1,2,3$. The initial iterate for the first subproblem is then set
to $z_{0}=\sum_{i=1}^{3}e_{i}\nu_{i}\nu_{i}^{T}$. Fourth, key problem
parameters, for every $z\in S_{+}^{n}$, are 
\begin{gather*}
f(Z)=\frac{\alpha_{1}}{2}\|{\cal C}Z-d\|^{2}-\frac{\alpha_{2}}{2}\|D{\cal B}Z\|^{2},\quad h(Z)=\delta_{P_{n}}(Z),\\
g(Z)={\cal A}(z),\quad S=\{b\},\quad\hat{\rho}=10^{-4},\quad\hat{\eta}=10^{-4}.
\end{gather*}
Sixth, using the fact that $\|Z\|_{F}\leq1$ for every $Z\in P_{n}$,
the constant hyperparameters for the AIP.ALM and iALM are 
\[
L_{g}=0,\quad B_{g}^{(1)}=\|{\cal A}\|,\quad L_{j}=0,\quad\rho_{j}=0,\quad B_{j}=\|A_{j}\|_{F}.
\]
Finally, each problem instance considered is based on a specific curvature
pair $(m,M)\in\r_{++}^{2}$ for which the scalar pair $(\alpha_{1},\alpha_{2})\in\r_{++}^{2}$
is selected so that $M=\lambda_{\max}(\nabla^{2}f)$ and $-m=\lambda_{\min}(\nabla^{2}f)$.
More specifically, the pair $(\alpha_{1},\alpha_{2})\in\r_{++}^{2}$
is generated using the approach in \prettyref{subsec:qmp}.

The table of iteration counts and total runtimes are given in \prettyref{tab:lc_qmp_iter}
and \prettyref{tab:lc_qmp_runtime}, respectively. 
\begin{center}
\begin{table}[th]
\begin{centering}
\begin{tabular}{|>{\centering}p{0.7cm}>{\centering}p{0.7cm}|>{\centering}p{1.8cm}>{\centering}p{1.8cm}>{\centering}p{1.8cm}>{\centering}p{1.8cm}|}
\hline 
\multicolumn{2}{|c|}{\textbf{\small{}$(m,M)$}} & \multicolumn{4}{c|}{\textbf{\small{}Iteration Count}}\tabularnewline
\hline 
{\footnotesize{}$m$} & {\footnotesize{}$M$} & {\footnotesize{}iALM} & {\footnotesize{}AIP.QP} & {\footnotesize{}AG.QP} & {\footnotesize{}AIP.AL}\tabularnewline
\hline 
{\footnotesize{}$10^{1}$} & {\footnotesize{}$10^{2}$} & {\footnotesize{}65780} & {\footnotesize{}2211} & {\footnotesize{}6891} & \textbf{\footnotesize{}366}\tabularnewline
{\footnotesize{}$10^{1}$} & {\footnotesize{}$10^{3}$} & {\footnotesize{}34629} & {\footnotesize{}1839} & {\footnotesize{}6672} & \textbf{\footnotesize{}217}\tabularnewline
{\footnotesize{}$10^{1}$} & {\footnotesize{}$10^{4}$} & {\footnotesize{}54469} & {\footnotesize{}1906} & {\footnotesize{}6667} & \textbf{\footnotesize{}644}\tabularnewline
{\footnotesize{}$10^{1}$} & {\footnotesize{}$10^{5}$} & {\footnotesize{}136349} & \textbf{\footnotesize{}1966} & {\footnotesize{}6667} & {\footnotesize{}2175}\tabularnewline
{\footnotesize{}$10^{1}$} & {\footnotesize{}$10^{6}$} & {\footnotesize{}371276} & \textbf{\footnotesize{}2222} & {\footnotesize{}6666} & {\footnotesize{}13831}\tabularnewline
\hline 
\end{tabular}
\par\end{centering}
\caption{Iteration Counts for LC-QM problems.\label{tab:lc_qmp_iter}}
\end{table}
\par\end{center}

\begin{center}
\begin{table}[th]
\begin{centering}
\begin{tabular}{|>{\centering}p{0.7cm}>{\centering}p{0.7cm}|>{\centering}p{1.8cm}>{\centering}p{1.8cm}>{\centering}p{1.8cm}>{\centering}p{1.8cm}|}
\hline 
\multicolumn{2}{|c|}{\textbf{\small{}$(m,M)$}} & \multicolumn{4}{c|}{\textbf{\small{}Runtime}}\tabularnewline
\hline 
{\footnotesize{}$m$} & {\footnotesize{}$M$} & {\footnotesize{}iALM} & {\footnotesize{}AIP.QP} & {\footnotesize{}AG.QP} & {\footnotesize{}AIP.AL}\tabularnewline
\hline 
{\footnotesize{}$10^{1}$} & {\footnotesize{}$10^{2}$} & {\footnotesize{}407.46} & {\footnotesize{}23.71} & {\footnotesize{}76.17} & \textbf{\footnotesize{}5.02}\tabularnewline
{\footnotesize{}$10^{1}$} & {\footnotesize{}$10^{3}$} & {\footnotesize{}214.04} & {\footnotesize{}19.81} & {\footnotesize{}73.39} & \textbf{\footnotesize{}2.88}\tabularnewline
{\footnotesize{}$10^{1}$} & {\footnotesize{}$10^{4}$} & {\footnotesize{}337.36} & {\footnotesize{}20.58} & {\footnotesize{}72.81} & \textbf{\footnotesize{}7.59}\tabularnewline
{\footnotesize{}$10^{1}$} & {\footnotesize{}$10^{5}$} & {\footnotesize{}971.32} & \textbf{\footnotesize{}21.35} & {\footnotesize{}73.82} & {\footnotesize{}25.00}\tabularnewline
{\footnotesize{}$10^{1}$} & {\footnotesize{}$10^{6}$} & {\footnotesize{}2493.30} & \textbf{\footnotesize{}25.35} & {\footnotesize{}77.11} & {\footnotesize{}162.56}\tabularnewline
\hline 
\end{tabular}
\par\end{centering}
\caption{Runtimes for LC-QM problems.\label{tab:lc_qmp_runtime}}
\end{table}
\par\end{center}

\subsubsection{Sparse Principal Component Analysis }

This subsection presents computational results for the sparse principal
component analysis (SPCA) problem considered in \citep{Gu2014}. More
specifically, given an integer $k$, positive scalar pair $(\nu,b)\in\r_{++}^{2}$,
and matrix $\Sigma\in S_{+}^{n}$, we consider the SPCA problem 
\begin{align*}
\min_{\Pi,\Phi}\  & \left\langle \Sigma,\Pi\right\rangle _{F}+\sum_{i,j=1}^{n}q_{\nu}(\Phi_{ij})+\nu\sum_{i,j=1}^{n}|\Phi_{ij}|\\
\text{subject to}\  & \Pi-\Phi=0,\\
 & (\Pi,\Phi)\in{\cal F}^{k}\times\r^{n\times n},
\end{align*}
where ${\cal F}^{k}=\{z\in S_{+}^{n}:0\preceq z\preceq I,\trc M=k\}$
denotes the $k$--Fantope and $q_{\nu}$ is the min-max concave penalty
function given by 
\[
q_{\nu}(t):=\begin{cases}
-t^{2}/(2b), & \text{if }|t|\leq b\nu,\\
b\nu^{2}/2-\nu|t|, & \text{if }|t|>b\nu,
\end{cases}\quad\forall t\in\r.
\]

We now describe the experiment parameters for the instances considered.
First, the scalar parameters are chosen to be $(\nu,n,k,b)=(100,100,1,0.1)$.
Second, the matrix $\Sigma$ is generated according to an eigenvalue
decomposition $\Sigma=P\Lambda P^{T}$, based on a parameter pair
$(s,k)$, where $k$ is as in the problem description and $s$ is
a positive integer. In particular, we choose $\Lambda=(100,1,...,1)$,
the first column of $P$ to be a sparse vector whose first $s$ entries
are $1/\sqrt{s}$, and the other entries of $P$ to be sampled randomly
from the standard Gaussian distribution. Third, the initial starting
point is $(\Pi_{0},\Phi_{0})=(D_{k},0)$ where $D_{k}$ is a diagonal
matrix whose first $k$ entries are 1 and whose remaining entries
are 0. Fourth, the curvature parameters for each problem instance
are $m=M=1/b.$ Fifth, the key problem parameters, the inputs, for
every $(\Pi,\Phi)\in S_{+}^{n}\times\r^{n\times n}$, are 
\begin{gather*}
f(\Pi,\Phi)=\left\langle \Sigma,\Pi\right\rangle _{F}+\sum_{i,j=1}^{n}q_{\nu}(\Phi_{ij}),\quad h(\Pi,\Phi)=\delta_{{\cal F}^{k}}(\Pi)+\nu\sum_{i,j=1}^{n}|\Phi_{ij}|,\\
g(\Pi,\Phi):=\Pi-\Phi,\quad S=\{0\},\quad\hat{\eta}=10^{-3},\quad\hat{\rho}=10^{-6}.
\end{gather*}
Finally, each problem instance considered is based on a specific choice
of $s$ (see the description above).

The table of iteration counts and total runtimes are given in \prettyref{tab:spca_iter}
and \prettyref{tab:spca_runtime}, respectively. 
\begin{center}
\begin{table}[th]
\begin{centering}
\begin{tabular}{|>{\centering}p{0.7cm}|>{\centering}p{1.8cm}>{\centering}p{1.8cm}|}
\hline 
 & \multicolumn{2}{c|}{\textbf{\small{}Iteration Count}}\tabularnewline
\hline 
$s$ & {\footnotesize{}AIP.QP} & {\footnotesize{}AG.QP}\tabularnewline
\hline 
{\footnotesize{}5} & \textbf{\footnotesize{}5254} & {\footnotesize{}25871}\tabularnewline
{\footnotesize{}10} & \textbf{\footnotesize{}5328} & {\footnotesize{}27074}\tabularnewline
{\footnotesize{}15} & \textbf{\footnotesize{}5492} & {\footnotesize{}26664}\tabularnewline
\hline 
\end{tabular}
\par\end{centering}
\caption{Iteration counts for SPCA problems.\label{tab:spca_iter}}
\end{table}
\par\end{center}

\begin{center}
\begin{table}[th]
\begin{centering}
\begin{tabular}{|>{\centering}p{0.7cm}|>{\centering}p{1.8cm}>{\centering}p{1.8cm}|}
\hline 
 & \multicolumn{2}{c|}{\textbf{\small{}Runtime}}\tabularnewline
\hline 
$s$ & {\footnotesize{}AIP.QP} & {\footnotesize{}AG.QP}\tabularnewline
\hline 
{\footnotesize{}5} & \textbf{\footnotesize{}76.81} & {\footnotesize{}295.78}\tabularnewline
{\footnotesize{}10} & \textbf{\footnotesize{}72.88} & {\footnotesize{}310.87}\tabularnewline
{\footnotesize{}15} & \textbf{\footnotesize{}86.89} & {\footnotesize{}361.03}\tabularnewline
\hline 
\end{tabular}
\par\end{centering}
\caption{Runtimes for SPCA problems.\label{tab:spca_runtime}}
\end{table}
\par\end{center}

\subsubsection{Box-Constrained Matrix Completion}

This subsection presents computational results for the box-constrained
matrix completion (BC-MC) problem considered in \citep{Yao2017}.
More specifically, given a dimension pair $(p,q)\in\n^{2}$, positive
scalar triple $(\beta,\mu,\theta)\in\r_{++}^{3}$, scalar pair $(u,l)\in\r^{2}$,
matrix $A\in\r^{p\times q}$, and indices $\Omega$, we consider the
BC-MC problem: 
\begin{align*}
\min_{X}\  & \frac{1}{2}\|P_{\Omega}(X-A)\|^{2}+\mu\sum_{i=1}^{\min\{p,q\}}\left[\kappa(\sigma_{i}(X))-\kappa_{0}\sigma_{i}(X)\right]+\mu\kappa_{0}\|X\|_{*}\\
\text{s.t.}\  & l\leq X_{ij}\leq u\quad\forall(i,j)\in\{1,...,p\}\times\{1,...,q\},
\end{align*}
where $\|\cdot\|_{*}$ denotes the nuclear norm, the function $P_{\Omega}$
is the linear operator that zeros out any entry not in $\Omega$,
the function $\sigma_{i}(X)$ denotes the $i^{{\rm th}}$ largest
singular value of $X$, and 
\[
\kappa_{0}:=\frac{\beta}{\theta},\quad\kappa(t):=\beta\log\left(1+\frac{|t|}{\theta}\right)\quad\forall t\in\r.
\]

We now describe the experiment parameters for the instances considered.
First, the matrix $A$ is the user-movie ratings data matrix of the
Jester dataset\footnote{The ratings in the file ``jester\_dataset\_1\_1.zip'' from \url{http://eigentaste.berkeley.edu/dataset/}..},
the index set $\Omega$ is the set of nonzero entries in $A$, the
dimension pair is set to be $(p,q)=(24938,100)$, and the fixed scalar
parameters are $(\mu,\theta)=(2,\sqrt{2})$. Second, the initial starting
point was chosen to be $X_{0}=0$. Third, the curvature parameters
for each problem instance are $m=2\beta\mu/\theta^{2}$ and $M=\max\left\{ 1,m\right\} $
and the bounds are set to $(l,u)=(0,5)$. Fourth, the key problem
parameters, for every $X\in\r^{n\times n}$, are 
\begin{gather*}
f(X)=\frac{1}{2}\|P_{\Omega}(X-A)\|^{2}+\mu\sum_{i=1}^{\min\{p,q\}}\left[\kappa(\sigma_{i}(X))-\kappa_{0}\sigma_{i}(X)\right],\quad h(X)=\mu\kappa_{0}\|X\|_{*},\\
g(X)=X,\quad S=\left\{ Z\in\r^{p\times q}:l\leq Z_{ij}\leq u,\:(i,j)\in\{1,...,p\}\times\{1,...,q\}\right\} ,\\
\hat{\eta}=10^{-2},\quad\hat{\rho}=10^{-2}.
\end{gather*}
Finally, each problem instance considered is based on a specific scalar
parameter $\beta>0$.

The table of iteration counts and total runtimes are given in \prettyref{tab:bcmc_iter}
and \prettyref{tab:bcmc_runtime}, respectively. 

\begin{table}[th]
\begin{centering}
\begin{tabular}{|>{\centering}m{0.7cm}|>{\centering}p{1.7cm}>{\centering}p{1.7cm}|}
\hline 
\multirow{1}{0.7cm}{} & \multicolumn{2}{c|}{\textbf{\small{}Iteration Count}}\tabularnewline
\hline 
$\beta$ & {\footnotesize{}AIP.QP} & {\footnotesize{}AG.QP}\tabularnewline
\hline 
{\footnotesize{}$1/2$} & {\footnotesize{}470} & \textbf{\footnotesize{}100}\tabularnewline
{\footnotesize{}$1$} & {\footnotesize{}447} & \textbf{\footnotesize{}100}\tabularnewline
{\footnotesize{}$2$} & {\footnotesize{}420} & \textbf{\footnotesize{}100}\tabularnewline
\hline 
\end{tabular}
\par\end{centering}
\caption{Iteration counts for BC-MC problems.\label{tab:bcmc_iter}}
\end{table}

\begin{table}[th]
\begin{centering}
\begin{tabular}{|>{\centering}m{0.7cm}|>{\centering}p{1.7cm}>{\centering}p{1.7cm}|}
\hline 
\multirow{1}{0.7cm}{} & \multicolumn{2}{c|}{\textbf{\small{}Runtime}}\tabularnewline
\hline 
$\beta$ & {\footnotesize{}AIP.QP} & {\footnotesize{}AG.QP}\tabularnewline
\hline 
{\footnotesize{}$1/2$} & {\footnotesize{}509.79} & \textbf{\footnotesize{}98.563}\tabularnewline
{\footnotesize{}$1$} & {\footnotesize{}466.05} & \textbf{\footnotesize{}124.45}\tabularnewline
{\footnotesize{}$2$} & {\footnotesize{}486.5} & \textbf{\footnotesize{}117.26}\tabularnewline
\hline 
\end{tabular}
\par\end{centering}
\caption{Iteration counts for BC-MC problems.\label{tab:bcmc_runtime}}
\end{table}

\subsubsection{Quadratically-Constrained Quadratic Matrix Problem}

This subsection presents computational results for the nonconvex quadratically
constrained quadratic matrix (QC-QM) problem considered in \citep{Kong2020b}.
More specifically, given a dimension pair $(\ell,n)\in\mathbb{N}^{2}$,
matrices $P,Q,R\in\r^{n\times n}$, and the quantities $(\alpha,\beta)$,
${\cal {\cal B}}$, ${\cal {\cal C}}$, $\{B_{j}\}_{j=1}^{n},\{C_{i}\}_{i=1}^{\ell}$,
$D$, $d$ as in \prettyref{subsec:lc_qmp}, we consider the QC-QM
problem
\begin{align*}
\begin{alignedat}{2}\min_{Z}\  & -\frac{\alpha}{2}\|D{\cal B}(Z)\|^{2}+\frac{\beta}{2}\|{\cal C}(Z)-d\|^{2}\\
\text{s.t.}\  & \frac{1}{2}(PZ)^{*}PZ+\frac{1}{2}Q^{*}QZ+\frac{1}{2}ZQ^{*}Q\preceq R^{*}R,\\
 & 0\leq\lam_{i}(Z)\leq\frac{1}{\sqrt{n}}, & i\in\{1,...,n\},\\
 & Z\in\mathbb{S}_{+}^{n},
\end{alignedat}
\end{align*}
where $\lam_{i}(Z)$ denotes the $i^{{\rm th}}$ largest eigenvalue
of $Z$ and the constraint $M\preceq0$ is equivalent to $-M\in\mathbb{S}_{+}^{n}$. 

We now describe the experiment parameters for the instances considered.
First, the dimensions are set to $(\ell,n)=(10,50)$. Second, the
quantities ${\cal B}$, ${\cal C}$, $D$, and $d$ were generated
in the same way as \prettyref{subsec:lc_qmp}. On the other hand,
the matrix $R$ is set to $I/n$ and the entries of matrices $P$
and $Q$ are sampled from the uniform distributions ${\cal U}[0,1/\sqrt{n}]$
and ${\cal U}[0,1/n]$, respectively. Third, the initial starting
point $z_{0}$ is set to the zero matrix. Fourth, the key problem
parameters, for every $Z\in\r^{n\times n}$ are
\begin{gather*}
f(Z)=-\frac{\alpha_{1}}{2}\|D{\cal B}(Z)\|^{2}+\frac{\alpha_{2}}{2}\|{\cal C}(Z)-d\|^{2},\quad h(Z)=\delta_{S}(z),\\
g(Z)=\frac{1}{2}(PZ)^{*}PZ+\frac{1}{2}Q^{*}QZ+\frac{1}{2}ZQ^{*}Q\preceq R^{*}R,\\
{\cal K}=\mathbb{S}_{+}^{n},\quad\hat{\rho}=10^{-2},\quad\hat{\eta}=10^{-2},
\end{gather*}
where $S=\{Z\in\mathbb{S}_{+}^{n}:0\leq\lam_{i}(Z)\leq1/\sqrt{n},i=1,...,n\}$.
Fifth, using the fact that $\|Z\|_{F}\leq1$ for every $Z\in S$,
the constant hyperparameters for the iALM and AIP.AL are
\begin{gather*}
L_{g}=\|P\|_{F}^{2},\quad B_{g}^{(1)}=\frac{1}{2}\|P\|_{F}^{2}+\|Q\|_{F}^{2},\\
L_{ij}=\left|\left[P^{*}P\right]_{ij}\right|,\quad\rho_{ij}=0,\quad B_{j}=\frac{1}{2}\left|\left[P^{*}P\right]_{ij}\right|+\left|\left[Q^{*}Q\right]_{ij}\right|,
\end{gather*}
for $1\leq i,j\leq n$. Finally, each problem instance considered
was based on a specific curvature pair $(m,M)$ for which the scalar
pair $(\alpha_{1},\alpha_{2})$ is selected so that $M=\lam_{\max}(\nabla^{2}f)$
and $-m=\lam_{\min}(\nabla^{2}f)$. More specifically, the pair $(\alpha_{1},\alpha_{2})\in\r_{++}^{2}$
is generated using the approach in \prettyref{subsec:qmp}.

The table of iteration counts and total runtimes are given in \prettyref{tab:qc_qmp_iter}
and \prettyref{tab:qc_qmp_runtime}, respectively. 
\begin{center}
\begin{table}[th]
\begin{centering}
\begin{tabular}{|>{\centering}p{0.7cm}>{\centering}p{0.7cm}|>{\centering}p{1.8cm}>{\centering}p{1.8cm}|}
\hline 
\multicolumn{2}{|c|}{\textbf{\small{}$(m,M)$}} & \multicolumn{2}{c|}{\textbf{\small{}Iteration Count}}\tabularnewline
\hline 
{\footnotesize{}$m$} & {\footnotesize{}$M$} & {\footnotesize{}iALM} & {\footnotesize{}AIP.ALM}\tabularnewline
\hline 
{\footnotesize{}$10^{1}$} & {\footnotesize{}$10^{2}$} & \textbf{\footnotesize{}2127} & {\footnotesize{}2373}\tabularnewline
{\footnotesize{}$10^{1}$} & {\footnotesize{}$10^{3}$} & {\footnotesize{}4196} & \textbf{\footnotesize{}283}\tabularnewline
{\footnotesize{}$10^{1}$} & {\footnotesize{}$10^{4}$} & {\footnotesize{}10075} & \textbf{\footnotesize{}1130}\tabularnewline
{\footnotesize{}$10^{1}$} & {\footnotesize{}$10^{5}$} & {\footnotesize{}21428} & \textbf{\footnotesize{}5657}\tabularnewline
\hline 
\end{tabular}
\par\end{centering}
\caption{Iteration counts for QC-QM problems.\label{tab:qc_qmp_iter}}
\end{table}
\par\end{center}

\begin{center}
\begin{table}[th]
\begin{centering}
\begin{tabular}{|>{\centering}p{0.7cm}>{\centering}p{0.7cm}|>{\centering}p{1.8cm}>{\centering}p{1.8cm}|}
\hline 
\multicolumn{2}{|c|}{\textbf{\small{}$(m,M)$}} & \multicolumn{2}{c|}{\textbf{\small{}Runtime}}\tabularnewline
\hline 
{\footnotesize{}$m$} & {\footnotesize{}$M$} & {\footnotesize{}iALM} & {\footnotesize{}AIP.ALM}\tabularnewline
\hline 
{\footnotesize{}$10^{1}$} & {\footnotesize{}$10^{2}$} & \textbf{\footnotesize{}21.46} & {\footnotesize{}42.24}\tabularnewline
{\footnotesize{}$10^{1}$} & {\footnotesize{}$10^{3}$} & {\footnotesize{}41.60} & \textbf{\footnotesize{}4.53}\tabularnewline
{\footnotesize{}$10^{1}$} & {\footnotesize{}$10^{4}$} & {\footnotesize{}97.28} & \textbf{\footnotesize{}18.61}\tabularnewline
{\footnotesize{}$10^{1}$} & {\footnotesize{}$10^{5}$} & {\footnotesize{}216.33} & \textbf{\footnotesize{}88.40}\tabularnewline
\hline 
\end{tabular}
\par\end{centering}
\caption{Runtimes for QC-QM problems. \label{tab:qc_qmp_runtime}}
\end{table}
\par\end{center}

\subsection{Discussion of the Results}

We see that the methods in this chapter are competitive against other
modern solvers and that they especially perform well when the curvature
ratio $M/m$ is large. Additionally, we see that each method scales
well across problem dimensions and parameters. Comparing the AIP.QPM
with the AIP.ALM, in particular, we see that the former is scales
better across different curvature ratios whereas the latter performs
substantially better on some problem instances than others.

We conjecture that the efficiency of these efficient methods is attributed
to three facts: (i) the use of efficient ACGM subroutines which adaptively
choose the sequence of stepsizes; (ii) the implementation of several
termination criteria that allow certain methods to stop early; and
(iii) the relaxation of certain convex proximal subproblems to nonconvex
ones (which is generally known to improve convergence). 

\section{Conclusion and Additional Comments}

In this chapter, we presented several implementation strategies of
the methods in previous chapters. More specifically, we presented
practical variants of the CRP in \prettyref{alg:cref}, the ACGM in
\prettyref{alg:acgm}, the AIPPM in \prettyref{alg:aippm}, and the
AIP.QPM in \prettyref{alg:qp_aippm}. For the iterative methods in
particular, we devised new schemes in which the ``stepsize'' parameter
is chosen in a relaxed and adaptive manner. Additionally, for the
AIP.QPM variant, we showed how using a warm-start strategy between
penalty prox subproblems made substantial improvements to the derived
complexity (compared to using a simple cold-start strategy). Finally,
numerical experiments were given to validate the efficacy of our implementation
strategies

\subsection*{Additional Comments}

We now make several comments about the results in this chapter.

Similar to how the R.AIPPM (resp. R.AIP.QPM) of \prettyref{sec:r_aipp}
(resp. \prettyref{sec:rqp_aipp}) is a relaxation of the AIPPM of
\prettyref{sec:aipp} (resp. AIP.QPM of \prettyref{sec:qp_aipp})
that uses an efficient R.ACGM (resp. R.AIPPM) to solve its key subproblems,
one could also consider similarly relaxed versions of methods in prior
chapters. We briefly describe some of these relaxations. First, recall
that the AIPP.SM in \prettyref{sec:aipp_smoothing} uses a single
AIPP call to obtain an approximate stationary point as in \prettyref{prb:approx_eff_nco}.
Hence, one could consider a relaxation of the AIPP.SM in which the
single AIPP call is replaced by an R.AIPP call. Second, recall that
the AIP.QP.SM in \prettyref{sec:qp_aipp_smoothing} uses a single
AIP.QP call to obtain an approximate stationary point as in \eqref{prb:approx_eff_cnco_a}.
Hence, similar to the first relaxation, on could consider a relaxation
of the AIPP.QPM in which the single AIP.QP call is replaced by an
R.AIP.QP call. 

Observing the arguments used in the proofs of \prettyref{prop:eff_penalty_props},
\prettyref{lem:poten_bd}, and \prettyref{thm:rqp_aipp_compl}, it
is straightforward to see that the assumption of $Z$ being bounded
can be relaxed to assuming that the iterates $\{\hat{z}_{\ell}\}_{\ell\geq1}$
generated by R.AIP.QPM be bounded. Explicitly assuming that the iterates
satisfy $\|\hat{z}_{\ell}\|\leq B$, for every $\ell\geq1$ and some
$B>0$, the resulting oracle complexity of R.AIP.QPM method is \eqref{eq:r_qp_aipp_compl}
with $Q(z_{0})$ replaced by the quantity 
\[
\varphi_{c_{1}}(z_{0})-\hat{\varphi}_{c_{1}}+\hat{\varphi}_{*}-\hat{\varphi}_{0}+\hat{\rho}\left[d_{0}+B\right]+\underline{m}\left[d_{0}^{2}+B^{2}\right],
\]
where $d_{0}:=\inf\{\|u-\hat{z}_{0}\|:z\in{\cal F}\}$ and the quantity
$\underline{m}$ is as in \eqref{eq:m_lower_def}. 

Note that the description of the R.AIPPM (resp. R.AIP.QPM) does not
actually require knowledge of an upper bound $m$ on the parameter
$\underline{m}$ in \eqref{eq:m_lower_def}. This is in contrast to
the AIPPM (resp. AIP.QPM) method of \prettyref{chap:unconstr_nco}
(resp. \prettyref{chap:cnco}), which requires $m$ in order to establish
its validity and iteration complexity. In addition, one could consider
an R.AIPPM and AIP.QPM variant in which the quantity $M$ is adaptively
inferred from its iterates rather than requiring knowledge of its
value beforehand. While for the sake of brevity we omit the formal
description and analysis of such a variant in this thesis, we conjecture
that the iteration complexity of the R.AIPPM (resp. R.AIP.QPM) variant
is as in \eqref{eq:r_aipp_compl} (resp. \eqref{eq:r_qp_aipp_compl})
with $M$ replaced with a quantity that lower bounds it, e.g. the
maximum of the lower estimates of $M$ which are inferred by the generated
iterates.

\subsection*{Future Work}

A future avenue of research is to investigate whether the iteration
complexity of R.AIP.QPM can still be established when $Z$ is unbounded.

\newpage{}

\chapter{Nonconvex-Concave Min-Max Composite Optimization}

\label{chap:min_max}

Smoothing methods are a broad class of optimization algorithms that
consider applying an smooth optimization method to a smooth approximation
of a nonsmooth optimization problem. An important class of optimization
problems that have particularly benefited from the use of smoothing
methods is the class of convex-concave min-max problems of the form
$\min_{x\in\r^{n}}\max_{y\in\r^{m}}\Phi(x,y)$. In particular, several
works \citep{Zang1980,Nesterov2005,Tsoukalas2009} consider smoothing
the nonsmooth primal function $p(x)=\max_{y\in\r^{m}}\Phi(x,y)$ and
applying an efficient solver to the resulting smooth problem under
a careful choice of the smoothing parameter.

Our main goal in this chapter is to describe and establish the iteration
complexity of an accelerated \textbf{inexact} proximal point \textbf{smoothing}
(AIPP.S) method for finding approximate stationary points of the \textbf{nonconvex}-concave
min-max composite optimization (MCO) problem 
\begin{equation}
\hat{p}_{*}:=\min_{x\in X}\left\{ \hat{p}(x):=p(x)+h(x)\right\} \tag{\ensuremath{{\cal MCO}}}\label{prb:eq:min_max_co}
\end{equation}
where $X$ is a nonempty convex set, $h\in\cConv(X)$, and $p$ is
a max function given by 
\begin{equation}
p(x):=\max_{y\in Y}\,\Phi(x,y)\quad\forall x\in X,\label{eq:p_def}
\end{equation}
for some nonempty compact convex set $Y$ and function $\Phi$ which,
for some scalar $m>0$ and open set $\Omega\supseteq X$, is such
that: (i) $\Phi$ is continuous on $\Omega\times Y$; (ii) the function
$-\Phi(x,\cdot)\in\cConv(Y)$ for every $x\in X$; and (iii) for every
$y\in Y$, the function $\Phi(\cdot,y)+m\|\cdot\|^{2}$ is convex,
differentiable, and its gradient is Lipschitz continuous on $X\times Y$.
Here, the objective function is the sum of a convex function $h$
and the pointwise supremum of differentiable functions which is generally
a nonsmooth function.

When $Y$ is a singleton, the max term in \ref{prb:eq:min_max_co}
becomes smooth and \ref{prb:eq:min_max_co} reduces to the smooth
NCO problem in \prettyref{chap:unconstr_nco} which may be solved
by the AIPPM in \prettyref{sec:aipp}. When $Y$ is not a singleton,
\ref{prb:eq:min_max_co} can no longer be directly solved by the AIPPM
due to the nonsmoothness of the max term. The AIPP.S method (AIPP.SM)
developed in this chapter is instead based on a perturbed version
of \ref{prb:eq:min_max_co} in which the max term in \ref{prb:eq:min_max_co}
is replaced by a smooth approximation and the resulting smooth NCO
problem is solved by the aforementioned AIPPM.

Throughout our presentation, it is assumed that efficient oracles
for evaluating the quantities $\Phi(x,y)$, $\nabla_{x}\Phi(x,y)$,
and $h(x)$ and for obtaining exact solutions of the problems 
\begin{equation}
\min_{x\in X}\left\{ \lam h(x)+\frac{1}{2}\|x-x_{0}\|^{2}\right\} ,\quad\max_{y\in Y}\left\{ \lam\Phi(x_{0},y)-\frac{1}{2}\|y-y_{0}\|^{2}\right\} \label{eq:prox_oracles}
\end{equation}
for any $(x_{0},y_{0})$ and $\lam>0$, are available. Throughout
this chapter, the terminology \textbf{oracle call} is used to refer
to a collection of the above oracles of size ${\cal O}(1)$ where
each of them appears at least once. 

We first develop an instance of the AIPP.SM that obtains a stationary
point based on a primal-dual formulation of \ref{prb:eq:min_max_co}.
More specifically, given a tolerance pair $(\rho_{x},\rho_{y})\in\r_{++}^{2}$,
it is shown that an instance of this scheme obtains a pair $([\bar{x},\bar{y}],[\bar{u},\bar{v}])$
such that 
\begin{equation}
\left(\begin{array}{c}
\bar{u}\\
\bar{v}
\end{array}\right)\in\left(\begin{array}{c}
\nabla_{x}\Phi(\bar{x},\bar{y})\\
0
\end{array}\right)+\left(\begin{array}{c}
\partial h(\bar{x})\\
\pt\left[-\Phi(\bar{x},\cdot)\right](\bar{y})
\end{array}\right),\quad\|\bar{u}\|\leq\rho_{x},\quad\|\bar{v}\|\leq\rho_{y}\label{eq:sp_approx_sol}
\end{equation}
in ${\cal O}(\rho_{x}^{-2}\rho_{y}^{-1/2})$ oracle calls. We then
show that another instance of the AIPP.SM can obtain an approximate
stationary point based on the directional derivative of $\hat{p}$.
More specifically, given a tolerance pair $\delta>0$, it is shown
that this instance computes a point $x\in X$ such that 
\begin{equation}
\exists\hat{x}\in X\text{ s.t. }\inf_{\|d\|\leq1}\hat{p}'(\hat{x};d)\geq-\delta,\quad\|\hat{x}-x\|\leq\delta,\label{eq:dd_approx_sol}
\end{equation}
in ${\cal O}(\delta^{-3})$ oracle calls.

A secondary goal of this chapter is to develop an accelerated inexact
proximal quadratic penalty smoothing (AIP.QP.S) method to obtain approximate
stationary points of a linearly constrained version of \ref{prb:eq:min_max_co},
namely the min-max constrained composite optimization (MCCO) problem
\begin{align}
\min_{x\in X}\left\{ p(x)+h(x):{\cal A}x=b\right\} \tag{\ensuremath{{\cal MCCO}}}\label{prb:eq:constr_min_max_co}
\end{align}
where $p$ is as in \eqref{eq:p_def}, ${\cal A}$ is a linear operator,
and $b$ is in the range of ${\cal A}$. Similar to the approach used
for the AIPP.SM, the AIP.QP.S method (AIP.QP.SM) considers a perturbed
variant of \ref{prb:eq:constr_min_max_co} in which the objective
function is replaced by a smooth approximation and the resulting CNCO
problem is solved by the AIP.QPM in \prettyref{sec:qp_aipp}. For
a given tolerance triple $(\rho_{x},\rho_{y},\eta)\in\r_{++}^{3}$,
it is shown that the method computes a pair $([\bar{x},\bar{y},\bar{r}],[\bar{u},\bar{v}])$
satisfying 
\begin{equation}
\begin{array}{c}
\left(\begin{array}{c}
\bar{u}\\
\bar{v}
\end{array}\right)\in\left(\begin{array}{c}
\nabla_{x}\Phi(\bar{x},\bar{y})+{\cal A}^{*}\bar{r}\\
0
\end{array}\right)+\left(\begin{array}{c}
\partial h(\bar{x})\\
\pt\left[-\Phi(\bar{x},\cdot)\right](\bar{y})
\end{array}\right),\\
\|\bar{u}\|\leq\rho_{x},\quad\|\bar{v}\|\leq\rho_{y},\quad\|{\cal A}\bar{x}-b\|\leq\eta.
\end{array}\label{eq:lc_sp_approx_sol}
\end{equation}
in ${\cal O}(\rho_{x}^{-2}\rho_{y}^{-1/2}+\rho_{x}^{-2}\eta^{-1})$
oracle calls.

It is worth mentioning that all the above complexities are obtained
under the mild assumption that the optimal values of the optimization
problems \ref{prb:eq:min_max_co} and \ref{prb:eq:constr_min_max_co}
are bounded below. Moreover, it is neither assumed that $X$ be bounded
nor that \ref{prb:eq:min_max_co} or \ref{prb:eq:constr_min_max_co}
have an optimal solution.

The content of this chapter is based on paper \citep{Kong2019a} (joint
work with Renato D.C. Monteiro) and several passages may be taken
verbatim from it.

\subsection*{Related Works}

Since the case when $\Phi(\cdot,\cdot)$ is convex-concave has been
well-studied in the literature (see, for example, \citep{Arrow1958,Nemirovski2005,Kolossoski2017,He2016,Rockafellar1997,Nesterov2005}),
we will make no more mention of it here. Instead, we will focus on
papers that consider the case where $\Phi(\cdot,y)$ is differentiable
and nonconvex for every $y\in Y$ and there are mild conditions on
$\Phi(x,\cdot)$ for every $x\in X$.

Denoting $\rho=\min\{\rho_{x},\rho_{y}\}$, $D_{x}$ to be the diameter
of $x$, and $C$ to be a general closed convex set, we present \prettyref{tab:comp1,tab:comp2},
which compare our contributions to past \citep{Rafique2018,Nouiehed2019}
and subsequent \citep{Lin2020,Ostrovskii2020,Thekumparampil2019}
works. It is worth mentioning that the above works consider termination
conditions that are slightly different from the ones in this chapter.
In \prettyref{sec:minmax_prelim_asmp}, we show that they are equivalent
to the ones in this chapter up to a multiplicative constant that is
independent of the tolerances, i.e.  $\rho_{x}$, $\rho_{y}$, $\delta$.

\begin{table}[ht]
\begin{centering}
\begin{tabular}{|>{\centering}p{2.4cm}|>{\centering}p{2.8cm}|>{\centering}p{1.5cm}|>{\centering}p{1.5cm}|>{\centering}p{1.5cm}|>{\centering}p{2cm}|}
\hline 
\multirow{2}{2.4cm}{\centering{}{\footnotesize{}Algorithm}} & \multirow{2}{2.8cm}{\centering{}{\footnotesize{}Oracle Complexity}} & \multicolumn{4}{c|}{{\footnotesize{}Use Cases}}\tabularnewline
\cline{3-6} \cline{4-6} \cline{5-6} \cline{6-6} 
 &  & {\footnotesize{}$D_{x}=\infty$} & {\footnotesize{}$h\equiv0$} & {\footnotesize{}$h\equiv\delta_{C}$} & {\footnotesize{}$h\in\cConv X$}\tabularnewline
\hline 
{\scriptsize{}PGSF \citep{Nouiehed2019}} & {\scriptsize{}${\cal O}(\rho^{-3})$} & {\scriptsize{}\XSolidBrush{}} & {\scriptsize{}\Checkmark{}} & {\scriptsize{}\Checkmark{}} & {\scriptsize{}\XSolidBrush{}}\tabularnewline
{\scriptsize{}Minimax-PPA \citep{Lin2020}} & {\scriptsize{}${\cal O}(\rho^{-2.5}\log^{2}(\rho^{-1}))$} & {\scriptsize{}\XSolidBrush{}} & {\scriptsize{}\Checkmark{}} & {\scriptsize{}\Checkmark{}} & {\scriptsize{}\XSolidBrush{}}\tabularnewline
{\scriptsize{}FNE Search \citep{Ostrovskii2020}} & {\scriptsize{}${\cal O}(\rho_{x}^{-2}\rho_{y}^{-1/2}\log(\rho^{-1}))$} & {\scriptsize{}\Checkmark{}} & {\scriptsize{}\Checkmark{}} & {\scriptsize{}\Checkmark{}} & {\scriptsize{}\XSolidBrush{}}\tabularnewline
\textbf{\scriptsize{}AIPP.S} & {\scriptsize{}${\cal O}(\rho_{x}^{-2}\rho_{y}^{-1/2})$} & {\scriptsize{}\Checkmark{}} & {\scriptsize{}\Checkmark{}} & {\scriptsize{}\Checkmark{}} & {\scriptsize{}\Checkmark{}}\tabularnewline
\hline 
\end{tabular}
\par\end{centering}
\caption{Comparison of iteration complexities and possible use cases under
notions equivalent to \eqref{eq:sp_approx_sol} with $\rho:=\min\{\rho_{x},\rho_{y}\}$.\label{tab:comp1}}
\end{table}

\begin{table}[ht]
\begin{centering}
\begin{tabular}{|>{\centering}p{2.4cm}|>{\centering}p{2.8cm}|>{\centering}p{1.5cm}|>{\centering}p{1.5cm}|>{\centering}p{1.5cm}|>{\centering}p{2cm}|}
\hline 
\multirow{2}{2.4cm}{\centering{}{\footnotesize{}Algorithm}} & \multirow{2}{2.8cm}{\centering{}{\footnotesize{}Oracle Complexity}} & \multicolumn{4}{c|}{{\footnotesize{}Use Cases}}\tabularnewline
\cline{3-6} \cline{4-6} \cline{5-6} \cline{6-6} 
 &  & {\footnotesize{}$D_{x}=\infty$} & {\footnotesize{}$h\equiv0$} & {\footnotesize{}$h\equiv\delta_{C}$} & {\footnotesize{}$h\in\cConv X$}\tabularnewline
\hline 
{\scriptsize{}PG-SVRG \citep{Rafique2018}} & {\scriptsize{}${\cal O}(\delta^{-6}\log\delta^{-1})$} & {\scriptsize{}\XSolidBrush{}} & {\scriptsize{}\Checkmark{}} & {\scriptsize{}\Checkmark{}} & {\scriptsize{}\Checkmark{}}\tabularnewline
{\scriptsize{}Minimax-PPA \citep{Lin2020}} & {\scriptsize{}${\cal O}(\delta^{-3}\log^{2}(\delta^{-1}))$} & {\scriptsize{}\XSolidBrush{}} & {\scriptsize{}\Checkmark{}} & {\scriptsize{}\Checkmark{}} & {\scriptsize{}\XSolidBrush{}}\tabularnewline
{\scriptsize{}Prox-DIAG \citep{Thekumparampil2019}} & {\scriptsize{}${\cal O}(\delta^{-3}\log^{2}(\delta^{-1}))$} & {\scriptsize{}\Checkmark{}} & {\scriptsize{}\Checkmark{}} & {\scriptsize{}\XSolidBrush{}} & {\scriptsize{}\XSolidBrush{}}\tabularnewline
\textbf{\scriptsize{}AIPP.S} & {\scriptsize{}${\cal O}(\delta^{-3})$} & {\scriptsize{}\Checkmark{}} & {\scriptsize{}\Checkmark{}} & {\scriptsize{}\Checkmark{}} & {\scriptsize{}\Checkmark{}}\tabularnewline
\hline 
\end{tabular}
\par\end{centering}
\caption{Comparison of iteration complexities and possible use cases under
notions equivalent to \eqref{eq:dd_approx_sol}.\label{tab:comp2}}
\end{table}

To the best of our knowledge, this chapter and \citep{Kong2019a}
are the first works to analyze the complexity of a smoothing scheme
for finding approximate stationary points as in \eqref{eq:lc_sp_approx_sol}. 

\subsection*{Organization}

This chapter contains six sections. The first one gives some preliminary
references and discusses our notion of an approximate stationary point
given in \eqref{eq:sp_approx_sol} and \eqref{eq:dd_approx_sol}.
The second one presents properties of a smooth approximation to the
primal function $p$ in \eqref{eq:p_def}. The third one presents
the AIPP.SM and its iteration complexity. The fourth one presents
the AIP.QP.SM and its iteration complexity. The fifth one presents
some numerical experiments. The last one gives a conclusion and some
closing comments.

\section{Preliminaries}

\label{sec:minmax_prelim_asmp}

This section describes the assumptions and four notions of stationary
points for problem \ref{prb:eq:min_max_co}. It is worth mentioning
that the complexities of the smoothing method of this chapter are
presented with respect to two of these notions. In order to understand
how these results can be translated to the other two alternative notions,
which have been used in a few papers dealing with problem \ref{prb:eq:min_max_co},
we also present a few results discussing some useful relations between
all these notions.

Throughout our presentation, we let ${\cal X}$ and ${\cal Y}$ be
finite dimensional inner product spaces. We also make the following
assumptions on problem \ref{prb:eq:min_max_co}: 

\stepcounter{assumption}
\begin{enumerate}
\item \label{asmp:mco_a1}$X\subset{\cal X}$ and $Y\subset{\cal Y}$ are
nonempty convex sets, and $Y$ is also compact; 
\item \label{asmp:mco_a2}there exists an open set $\Omega\supseteq X$
such that $\Phi(\cdot,\cdot)$ is finite and continuous on $\Omega\times Y$;
moreover, $\nabla_{x}\Phi(x,y)$ exists and is continuous at every
$(x,y)\in\Omega\times Y$; 
\item \label{asmp:mco_a3}$h\in\overline{{\rm Conv}}(X)$ and $-\Phi(x,\cdot)\in\overline{{\rm Conv}}(Y)$
for every $x\in\Omega$; 
\item \label{asmp:mco_a4}there exist scalars $(L_{x},L_{y})\in\r_{++}^{2},$
and $m\in(0,L_{x}]$ such that 
\begin{gather}
\Phi(x,y)-\left[\Phi(x',y)+\left\langle \nabla_{x}\Phi(x',y),x-x'\right\rangle \right]\geq-\frac{m}{2}\|x-x'\|^{2},\label{eq:sp_lower_curv}\\
\|\nabla_{x}\Phi(x,y)-\nabla_{x}\Phi(x',y')\|\leq L_{x}\|x-x'\|+L_{y}\|y-y'\|,\label{eq:M_L_xy}
\end{gather}
for every $x,x'\in X$ and $y,y'\in Y$; 
\item \label{asmp:mco_a5}$\hat{p}_{*}>-\infty$; 
\end{enumerate}
We make three remarks about the above assumptions. First, it is well-known
that condition \eqref{eq:M_L_xy} implies that 
\begin{equation}
\Phi(x',y)-\left[\Phi(x,y)+\left\langle \nabla_{x}\Phi(x,y),x'-x\right\rangle \right]\leq\frac{L_{x}}{2}\|x'-x\|^{2},\label{eq:sp_upper_curv}
\end{equation}
for every $(x',x,y)\in X\times X\times Y$. Third, the weak convexity
condition in \ref{asmp:mco_a4} implies that, for any $y\in Y$, the
function $\Phi(\cdot,y)+m\|\cdot\|^{2}/2$ is convex, and hence $p+m\|\cdot\|^{2}/2$
is as well. Note that while $\hat{p}$ is generally nonconvex and
nonsmooth, it has the nice property that $\hat{p}+m\|\cdot\|^{2}/2$
is convex.

We now discuss two stationarity conditions of \ref{prb:eq:min_max_co}
under assumptions \ref{asmp:mco_a1}--\ref{asmp:mco_a3}. First,
denoting 
\begin{equation}
\hat{\Phi}(x,y):=\Phi(x,y)+h(x)\quad\forall(x,y)\in X\times Y,\label{eq:Phi_hat_def}
\end{equation}
it is well-known that problem \ref{prb:eq:min_max_co} is related
to the saddle-point problem which consists of finding a pair $(x^{*},y^{*})\in X\times Y$
such that 
\begin{equation}
\hat{\Phi}(x^{*},y)\leq\hat{\Phi}(x^{*},y^{*})\leq\hat{\Phi}(x,y^{*}),\label{eq:saddle_point}
\end{equation}
for every $(x,y)\in X\times Y$. More specifically, $(x^{*},y^{*})$
satisfies \eqref{eq:saddle_point} if and only if $x^{*}$ is an optimal
solution of \ref{prb:eq:min_max_co}, $y^{*}$ is an optimal solution
of the dual of \ref{prb:eq:min_max_co}, and there is no duality gap
between the two problems. Using the composite structure described
above for $\hat{\Phi}$, it can be shown that a necessary condition
for \eqref{eq:saddle_point} to hold is that $(x^{*},y^{*})$ satisfy
the stationarity condition 
\begin{equation}
\left(\begin{array}{c}
0\\
0
\end{array}\right)\in\left(\begin{array}{c}
\nabla_{x}\Phi(x^{*},y^{*})\\
0
\end{array}\right)+\left(\begin{array}{c}
\partial h(x^{*})\\
\pt\left[-\Phi(x^{*},\cdot)\right](y^{*})
\end{array}\right).\label{eq:crit_conv_incl}
\end{equation}
When $m=0$, the above condition also becomes sufficient for \eqref{eq:saddle_point}
to hold. Second, it can be shown that $p'(x^{*};d)$ is well-defined
for every $d\in{\cal X}$ and that a necessary condition for $x^{*}\in X$
to be a local minimum of \ref{prb:eq:min_max_co} is that it satisfies
the stationarity condition 
\begin{equation}
\inf_{\|d\|\leq1}\hat{p}'(x^{*};d)\geq0.\label{eq:ddir_conv}
\end{equation}
When $m=0$, the above condition also becomes sufficient for $x^{*}$
to be a global minimum of \ref{prb:eq:min_max_co}. Moreover, in view
of \prettyref{lem:approx_unconstr} in \prettyref{app:statn_notions}
with $(\bar{u},\bar{v},\bar{x},\bar{y})=(0,0,x^{*},y^{*})$, it follows
that $x^{*}$ satisfies \eqref{eq:ddir_conv} if and only if there
exists $y^{*}\in Y$ such that $(x^{*},y^{*})$ satisfies \eqref{eq:crit_conv_incl}. 

Note that finding points that satisfy \eqref{eq:crit_conv_incl} or
\eqref{eq:ddir_conv} exactly is generally a difficult task. Hence,
in this section and the next one, we only consider approximate versions
of \eqref{eq:crit_conv_incl} or \eqref{eq:ddir_conv}, which are
\eqref{eq:sp_approx_sol} and \eqref{eq:dd_approx_sol}, respectively.
For ease of future reference, we say that: 
\begin{itemize}
\item[(i)] a pair $([\bar{x},\bar{y}],[\bar{u},\bar{v}])$ is a \textbf{$(\rho_{x},\rho_{y})$-primal-dual
stationary point} of \ref{prb:eq:min_max_co} if it satisfies \eqref{eq:sp_approx_sol}; 
\item[(ii)] a point $\hat{x}$ is a \textbf{$\delta$-directional stationary
point} of \ref{prb:eq:min_max_co} if it satisfies the first inequality
in \eqref{eq:dd_approx_sol}.
\end{itemize}
It is worth mentioning that \eqref{eq:dd_approx_sol} is generally
hard to verify for a given point $x\in X$. This is primarily because
the definition requires us to check an infinite number of directional
derivatives for a (potentially) nonsmooth function at points $\hat{x}$
near $\bar{x}$. In contrast, the definition of an approximate primal-dual
stationary point is generally easier to verify because the quantities
$\|\bar{u}\|$ and $\|\bar{v}\|$ can be measured directly, and the
inclusions in \eqref{eq:sp_approx_sol} are easy to verify when the
prox oracles for $h$ and $\Phi(x,\cdot)$, for every $x\in X$, are
readily available.

The next result, whose proof is given in \prettyref{app:statn_notions},
shows that a $(\rho_{x},\rho_{y})$-primal-dual stationary point,
for small enough $\rho_{x}$ and $\rho_{y}$, yields a point $x$
satisfying \eqref{eq:dd_approx_sol}. Its statement makes use of the
diameter of $Y$ defined as 
\begin{equation}
D_{y}:=\sup_{y,y'\in Y}\|y-y'\|.\label{eq:Dy_def}
\end{equation}

\begin{prop}
\label{prop:impl1_statn}If the pair $([\bar{x},\bar{y}],[\bar{u},\bar{v}])$
is a $(\rho_{x},\rho_{y})$-primal-dual stationary point of \ref{prb:eq:min_max_co},
then there exists a point $\hat{x}\in X$ such that 
\[
\inf_{\|d\|\leq1}\hat{p}'(\hat{x};d)\geq-\rho_{x}-2\sqrt{2mD_{y}\rho_{y}},\quad\|\bar{x}-\hat{x}\|\leq\sqrt{\frac{2D_{y}\rho_{y}}{m}}.
\]
\end{prop}

The iteration complexities in this chapter (see \prettyref{sec:aipp_smoothing})
are stated with respect to the two notions of stationary points \eqref{eq:sp_approx_sol}
and \eqref{eq:dd_approx_sol}. However, it is worth discussing below
two other notions of stationary points that are common in the literature
as well as some results that relate all four notions.

Given $(\lambda,\varepsilon)\in\r_{++}^{2}$, a point $x$ is said
to be a $(\lambda,\varepsilon)$-prox stationary point of \ref{prb:eq:min_max_co}
if the function $\hat{p}+\|\cdot\|^{2}/(2\lambda)$ is strongly convex
and 
\begin{equation}
\frac{1}{\lambda}\|x-x_{\lambda}\|\leq\varepsilon,\quad x_{\lambda}=\argmin_{u\in{\cal X}}\left\{ \hat{P}_{\lam}(u):=\hat{p}(u)+\frac{1}{2\lambda}\|u-x\|^{2}\right\} .\label{eq:prox_stn_point}
\end{equation}
The above notion is considered, for example, in \citep{Lin2020,Rafique2018,Thekumparampil2019}.
The result below, whose proof is given in \prettyref{app:statn_notions},
shows how it is related to \eqref{eq:dd_approx_sol}.
\begin{prop}
\label{prop:prox_statn}For any given $\lambda\in(0,1/m)$, the following
statements hold: 
\begin{itemize}
\item[(a)] for any $\varepsilon>0$, if $x\in X$ satisfies \eqref{eq:dd_approx_sol}
and 
\begin{equation}
0<\delta\leq\frac{\lambda^{3}\varepsilon}{\lambda^{2}+2(1-\lambda m)(1+\lambda)},\label{eq:spec_delta_bd}
\end{equation}
then $x$ is a $(\lambda,\varepsilon)$-prox stationary point; 
\item[(b)] for any $\delta>0$, if $x\in X$ is a $(\lambda,\varepsilon)$-prox
stationary point for some $\varepsilon\leq\delta\cdot\min\{1,1/\lambda\}$,
then $x$ satisfies \eqref{eq:dd_approx_sol} with $\hat{x}=x_{\lambda}$,
where $x_{\lambda}$ is as in \eqref{eq:prox_stn_point}. 
\end{itemize}
\end{prop}

Note that for a fixed $\lam\in(0,1/m)$ such that $\max\{\lam^{-1},(1-\lam m)^{-1}\}={\cal O}(1)$,
the largest $\delta$ in part (a) is ${\cal O}(\varepsilon)$. Similarly,
for part (b), if $\lam^{-1}={\cal O}(1)$ then largest $\varepsilon$
in part (b) is ${\cal O}(\delta)$. Combining these two observations,
it follows that \eqref{eq:prox_stn_point} and \eqref{eq:dd_approx_sol}
are equivalent (up to a multiplicative factor) under the assumption
that $\delta=\Theta(\varepsilon)$.

Given $(\rho_{x},\rho_{y})\in\r_{++}^{2}$, a pair $(\bar{x},\bar{y})$
is said to be a $(\rho_{x},\rho_{y})$-first-order Nash equilibrium
point of \ref{prb:eq:min_max_co} if 
\begin{equation}
\inf_{\|d_{x}\|\leq1}{\cal S}_{\bar{y}}'(\bar{x};d_{x})\geq-\rho_{x},\quad\sup_{\|d_{y}\|\leq1}{\cal S}_{\bar{x}}'(\bar{y};d_{y})\leq\rho_{y},\label{eq:nash_stn_point}
\end{equation}
where ${\cal S}_{\bar{y}}:=\Phi(\cdot,\bar{y})+h(\cdot)$ and ${\cal S}_{\bar{x}}:=\Phi(\bar{x},\cdot)$.
The above notion is considered, for example, in \citep{Nouiehed2019,Lin2020,Ostrovskii2020}.
The next result, whose proof is given in \prettyref{app:statn_notions},
shows that \eqref{eq:nash_stn_point} is equivalent to \eqref{eq:sp_approx_sol}.
\begin{prop}
\label{prop:ne_statn_pt}A pair $(\bar{x},\bar{y})$ is a $(\rho_{x},\rho_{y})$-first-order
Nash equilibrium point if and only if there exists $(\bar{u},\bar{v})\in{\cal X}\times{\cal Y}$
such that $([\bar{x},\bar{y}],[\bar{u},\bar{v}])$ is a $(\rho_{x},\rho_{y})$-primal-dual
stationary point. 
\end{prop}

We now briefly discuss some approaches for finding approximate stationary
points of \ref{prb:eq:min_max_co}. One approach is to apply a proximal
descent type method directly to problem \ref{prb:eq:min_max_co},
but this would lead to subproblems with nonsmooth convex composite
functions. A second approach is based on first applying a smoothing
method to \ref{prb:eq:min_max_co} and then using a prox-convexifying
descent method such as the AIPPM in \prettyref{sec:aipp} to solve
the perturbed unconstrained smooth problem. An advantage of the second
approach, which is the one pursued in this chapter, is that it generates
subproblems with smooth convex composite objective functions. The
next subsection describes one possible way to smooth the (generally)
nonsmooth function $p$ in \ref{prb:eq:min_max_co}.

Before ending this section, we formally state the problem of finding
primal-dual and directional stationary points in \prettyref{prb:approx_sp_mco}
and \prettyref{prb:approx_dd_mco}, respectively .

\begin{mdframed}
\mdprbcaption{Find an approximate primal-dual stationary point of ${\cal MCO}$}{prb:approx_sp_mco}
Given $(\rho_x, \rho_y) \in \r_{++}^2$, find a pair $([\bar{x},\bar{y}],[\bar{u},\bar{v}]) \in [X\times Y] \times [{\cal X} \times {\cal Y}]$ satisfying condition \eqref{eq:sp_approx_sol}.
\end{mdframed}

\begin{mdframed}
\mdprbcaption{Find an approximate directional stationary point of ${\cal MCO}$}{prb:approx_dd_mco}
Given $\delta > 0$, find a point $x \in X$ satisfying condition \eqref{eq:dd_approx_sol}.
\end{mdframed}

\section{Smooth Approximation}

This subsection presents a smooth approximation of the function $p$
in \ref{prb:eq:min_max_co}.

For every $\xi>0$, consider the smoothed function $p_{\xi}$ defined
by 
\begin{align}
p_{\xi}(x) & :=\max_{y\in Y}\left\{ \Phi_{\xi}(x,y):=\Phi(x,y)-\frac{1}{2\xi}\|y-y_{0}\|^{2}\right\} \quad\forall x\in X,\label{eq:p_xi_def}
\end{align}
for some $y_{0}\in Y$. The following proposition presents the key
properties of $p_{\xi}$ and its related quantities.
\begin{prop}
\label{prop:xi_facts} Let $\xi>0$ be given and assume that the function
$\Phi$ satisfies conditions \ref{asmp:mco_a1}--\ref{asmp:mco_a4}.
Let $p_{\xi}(\cdot)$ and $\Phi_{\xi}(\cdot,\cdot)$ be as defined
in \eqref{eq:p_xi_def} and define 
\begin{equation}
y_{\xi}(x):=\argmax_{y'\in Y}\Phi_{\xi}(x,y')\quad\forall x\in X.\label{eq:y_xi_def}
\end{equation}
Then, the following properties hold: 
\begin{itemize}
\item[(a)] $y_{\xi}(\cdot)$ is $Q_{\xi}$-Lipschitz continuous on $X$ where
\begin{equation}
Q_{\xi}:=\xi L_{y}+\sqrt{\xi(L_{x}+m)};\label{eq:Q_xi}
\end{equation}
\item[(b)] $p_{\xi}(\cdot)$ is continuously differentiable on $X$ and $\nabla p_{\xi}(x)=\nabla_{x}\Phi(x,y_{\xi}(x))$
for every $x\in X$; 
\item[(c)] $\nabla p_{\xi}(\cdot)$ is $L_{\xi}$-Lipschitz continuous on $X$
where 
\begin{equation}
L_{\xi}:=L_{y}Q_{\xi}+L_{x}\leq\left(L_{y}\sqrt{\xi}+\sqrt{L_{x}}\right)^{2};\label{eq:L_xi_def}
\end{equation}
\item[(d)] for every $x,x'\in X$, we have 
\begin{equation}
p_{\xi}(x)-\left[p_{\xi}(x')+\left\langle \nabla p_{\xi}(x'),x-x'\right\rangle \right]\geq-\frac{m}{2}\|x-x'\|^{2};\label{eq:g_xi_lower_curv}
\end{equation}
\end{itemize}
\end{prop}

\begin{proof}
The inequality in \eqref{eq:L_xi_def} follows from (a), the fact
that $m\leq L_{x}$, and the bound 
\[
L_{\xi}=L_{y}\left[\xi L_{y}+\sqrt{\xi(L_{x}+m)}\right]+L_{x}\leq\xi L_{y}^{2}+2\sqrt{\xi L_{x}}+L_{x}=\left(L_{y}\sqrt{\xi}+\sqrt{L_{x}}\right)^{2}.
\]
The other conclusions of (a)--(c) follow from \prettyref{lem:diff_danskin}
and \prettyref{prop:Psi_global_ext} in \prettyref{app:smoothing}
with $(\Psi,q,y)=(\Phi_{\xi},p_{\xi},y_{\xi})$. We now show that
the conclusion of (d) is true. Indeed, if we consider \eqref{eq:sp_lower_curv}
at $(y,x')=(y_{\xi}(x'),x')$, the definition of $\Phi_{\xi}$, and
use the definition of $\nabla p_{\xi}$ in (b), then 
\begin{align*}
 & -\frac{m}{2}\|x-x'\|^{2}\leq\Phi(x',y_{\xi}(x))-\left[\Phi(x,y_{\xi}(x))+\left\langle \nabla_{x}\Phi(x,y_{\xi}(x)),x'-x\right\rangle \right]\\
 & =\Phi_{\xi}(x',y_{\xi}(x))-\left[p_{\xi}(x)+\left\langle \nabla p_{\xi}(x),x'-x\right\rangle \right]\\
 & \leq p_{\xi}(x')-\left[p_{\xi}(x)+\left\langle \nabla p_{\xi}(x),x'-x\right\rangle \right],
\end{align*}
where the last inequality follows from the optimality of $y$. 
\end{proof}
We now make two remarks about the above properties. First, the Lipschitz
constants of $y_{\xi}$ and $\nabla p_{\xi}$ depend on the value
of $\xi$ while the weak convexity constant $m$ in \eqref{eq:g_xi_lower_curv}
does not. Second, as $\xi\to\infty$, it holds that $p_{\xi}\to p$
pointwise and $Q_{\xi},L_{\xi}\to\infty$. These remarks are made
more precise in the next result.
\begin{lem}
\label{lem:smoothing_relation}For every $\xi>0$, it holds that 
\[
-\infty<p(x)-\frac{D_{y}^{2}}{2\xi}\leq p_{\xi}(x)\leq p(x)\quad\forall x\in X,
\]
where $D_{y}$ is as in \eqref{eq:Dy_def}. 
\end{lem}

\begin{proof}
The fact that $p(x)>-\infty$ follows immediately from assumption
\ref{asmp:mco_a5}. To show the other bounds, observe that for every
$y_{0}\in Y$, we have 
\[
\Phi(x,y)+h(x)\geq\Phi(x,y)-\frac{1}{2\xi}\|y-y_{0}\|^{2}+h(x)\geq\Phi(x,y)-\frac{D_{y}^{2}}{2\xi}+h(x)
\]
for every $(x,y)\in X\times Y$. Taking the supremum of the bounds
over $y\in Y$ and using the definitions of $p$ and $p_{\xi}$ yields
the remaining bounds. 
\end{proof}

\section{Accelerated Inexact Proximal Point Smoothing (AIPP.S) Method}

\label{sec:aipp_smoothing}

This section presents the AIPP.SM for finding stationary points of
\ref{prb:eq:min_max_co} as in \eqref{eq:sp_approx_sol} and \eqref{eq:dd_approx_sol}.

We first state the AIPP.SM in \prettyref{alg:aipp_sm}. Given $(x_{0},y_{0})\in X\times Y$,
its main idea is to apply an instance of the AIPPM in \prettyref{sec:aipp}
to the NCO problem 
\begin{equation}
\min_{x\in X}\left\{ \hat{p}_{\xi}(x):=p_{\xi}(x)+h(x)\right\} ,\label{eq:smoothed_intro_prb}
\end{equation}
where $p_{\xi}$ is as in \eqref{eq:p_xi_def} and $\xi$ is a positive
scalar that will depend on the tolerances in \eqref{eq:sp_approx_sol}
and \eqref{eq:dd_approx_sol}. It is stated in an incomplete manner
in the sense that it does not specify how the parameter $\xi$ and
the tolerance $\rho$ used in its AIPPM call are chosen. Two invocations
of this method, with different choices of $\xi$ and $\rho$, are
considered in \prettyref{prop:sp_aipp_facts,prop:dd_aipp_facts},
which describe the iteration complexities for finding approximate
stationary points as in \eqref{eq:sp_approx_sol} and \eqref{eq:dd_approx_sol},
respectively. 

\begin{mdframed}
\mdalgcaption{AIPP.S Method}{alg:aipp_sm}
\begin{smalgorithmic}
	\Require{$\rho > 0, \enskip \xi >0,  \enskip (m,L_x,L_y)\in\r_{+}^3, \enskip h \in \cConv(Z), \enskip \Phi \text{ satisfying } \ref{asmp:mco_a2}--\ref{asmp:mco_a4}, \enskip (x_0, y_0) \in X \times Y$;}
	\Initialize{$\lam \gets 1/(2m),  \enskip \sigma \gets 1/2$}
	\vspace*{.5em}
	\Procedure{AIPP.S}{$\Phi, h, x_0, y_0, m, L_x, L_y, \rho$}
		\StateEq{$p_\xi \Lleftarrow  \max_{y\in Y} \Phi_{\xi}(\cdot,y)$} \Comment{See \eqref{eq:p_xi_def}.}
		\StateEq{$L_\xi \gets L_{y}\left[\xi L_{y}+\sqrt{\xi(L_{x}+m)}\right]+L_{x}$}
		\StateEq{$(x, u) \gets \text{AIPP}(p_\xi, h, x_0, \lam, m, L_\xi, \sigma, \rho)$} \label{ln:smoothing_aipp_call}
		\StateEq{\Return{$(x, u)$}}
	\EndProcedure
\end{smalgorithmic}
\end{mdframed}

We now give four remarks about the above method. First, it follows
from \prettyref{cor:spec_aipp_compl} that the AIPPM invoked in \prettyref{ln:smoothing_aipp_call}
stops and outputs a pair $(x,u)$ satisfying
\begin{equation}
u\in\nabla p_{\xi}(x)+\pt h(x),\quad\|u\|\leq\rho.\label{eq:approx_smoothed}
\end{equation}
Second, since the AIPP.SM is a one-pass method (as opposed to an iterative
method), the complexity of the AIPP.SM is essentially that of the
AIPPM. Third, similar to the smoothing scheme of \citep{Nesterov2005}
which assumes $m=0$, the AIPP.SM is also a smoothing scheme for the
case in which $m>0$. On the other hand, in contrast to the algorithm
of \citep{Nesterov2005} which uses an ACG variant, the AIPP.SM invokes
the AIPPM to solve \eqref{eq:smoothed_intro_prb} due to its nonconvexity.
Finally, while the AIPPM in \prettyref{ln:smoothing_aipp_call} is
called with $(\sigma,\lam)=(1/2,1/(2m))$, it can also be called with
any $\sigma\in(0,1)$ and $\lam\in(0,1/m)$ to establish the desired
termination.

For the remainder of this subsection, our goal will be to show that
a careful selection of the parameter $\xi$ and the tolerance $\rho$
will allow the AIPP.SM to generate approximate stationary points as
in \eqref{eq:dd_approx_sol} and \eqref{eq:sp_approx_sol}.

We first recall the quantity $R_{\lam}\psi(z_{0})$ in \eqref{eq:Rphi_def}
of \prettyref{chap:unconstr_nco}. The result below presents a bound
on $R_{\lam}\hat{p}_{\xi}(x_{0})$ in terms of the data in problem
\ref{prb:eq:min_max_co}. 
\begin{lem}
\label{lem:R_bd}For every $\xi>0$ and $\lambda\geq0$, it holds
that
\begin{equation}
R_{\lam}\hat{p}_{\xi}(x_{0})\leq R_{\lam}\hat{p}(x_{0})+\frac{\lambda D_{y}^{2}}{2\xi},\label{eq:R_sp_bd}
\end{equation}
where $R_{\lam}\psi(\cdot)$ and $D_{y}$ are as in \eqref{eq:Rphi_def}
and \eqref{eq:Dy_def}, respectively. 
\end{lem}

\begin{proof}
Using \prettyref{lem:smoothing_relation} and the definitions of $\hat{p}$
and $\hat{p}_{\xi}$, it holds that 
\begin{equation}
\hat{p}_{\xi}(x)-\inf_{x'}\hat{p}_{\xi}(x')\leq\hat{p}(x)-\inf_{x'}\hat{p}(x')+\frac{D_{y}^{2}}{2\xi},\quad\forall x\in X.\label{eq:primal_R}
\end{equation}
Multiplying the above expression by $(1-\sigma)\lambda$ and adding
the quantity $\|x_{0}-x\|^{2}/2$ yields the inequality 
\begin{align}
 & \frac{1}{2}\|x_{0}-x\|^{2}+(1-\sigma)\lambda\left[\hat{p}_{\xi}(x)-\inf_{x'}\hat{p}_{\xi}(x')\right]\nonumber \\
 & \leq\frac{1}{2}\|x_{0}-x\|^{2}+(1-\sigma)\lambda\left[\hat{p}(x)-\inf_{\tilde{x}}\hat{p}(x')\right]+(1-\sigma)\frac{\lambda D_{y}^{2}}{2\xi}\quad\forall x\in X,\label{eq:R_sp_pre_bd}
\end{align}
Taking the infimum of the above expression, and using the definition
of $R_{\lam}\psi(\cdot)$ in \eqref{eq:Rphi_def} yields the desired
conclusion. 
\end{proof}
We now show how the AIPP.SM generates a $(\rho_{x},\rho_{y})$-primal-dual
stationary point, i.e. a pair that solves \prettyref{prb:approx_sp_mco}. 
\begin{prop}
\label{prop:sp_aipp_facts}For a given tolerance pair $(\rho_{x},\rho_{y})\in\r_{++}^{2}$,
let $(x,u)$ be the pair output by the AIPP.SM with input parameter
$\xi$ and tolerance $\rho$ satisfying 
\begin{equation}
\xi\geq\frac{D_{y}}{\rho_{y}},\quad\rho=\rho_{x}.\label{eq:sp_params}
\end{equation}
Moreover, define 
\begin{equation}
(\bar{u},\bar{v}):=\left(u,\frac{y_{0}-y_{\xi}(x)}{\xi}\right),\quad(\bar{x},\bar{y}):=(x,y_{\xi}(x)),\label{eq:spec_output_sp_aipp}
\end{equation}
where $y_{\xi}$ is as in \eqref{eq:y_xi_def}. Then, the following
statements hold: 
\begin{itemize}
\item[(a)] the AIPP.SM performs 
\begin{equation}
{\cal O}\left(\Omega_{\xi}\left[\frac{m^{2}R_{1/(2m)}\hat{p}(x_{0})}{\rho_{x}^{2}}+\frac{mD_{y}^{2}}{\xi\rho_{x}^{2}}+\log_{1}^{+}(\Omega_{\xi})\right]\right)\label{eq:sp_aipp_compl}
\end{equation}
oracle calls, where $R_{\lam}\psi(\cdot)$ and $D_{y}$ are as in
\eqref{eq:Rphi_def} and \eqref{eq:Dy_def}, respectively, $\log_{1}^{+}(\cdot):=\max\{1,\log(\cdot)\}$,
and 
\begin{equation}
\Omega_{\xi}:=1+\frac{\sqrt{\xi}L_{y}+\sqrt{L_{x}}}{\sqrt{m}};\label{eq:Omega_xi_def}
\end{equation}
\item[(b)] the pair $([\bar{x},\bar{y}],[\bar{u},\bar{v}])$ is a $(\rho_{x},\rho_{y})$-primal-dual
stationary point of \ref{prb:eq:min_max_co}, and hence, solves \prettyref{prb:approx_sp_mco}. 
\end{itemize}
\end{prop}

\begin{proof}
(a) Using the inequality in \eqref{eq:L_xi_def}, it holds that 
\begin{align}
\sqrt{\frac{L_{\xi}}{4m}+1}\leq1+\sqrt{\frac{L_{\xi}}{4m}}\leq1+\frac{\sqrt{\xi}L_{y}+\sqrt{L_{x}}}{2\sqrt{m}}=\Theta(\Omega_{\xi}).\label{eq:L_xi_sqrt_bd}
\end{align}
Moreover, using \prettyref{cor:spec_aipp_compl} with $(\phi,M)=(\hat{p}_{\xi},L_{\xi})$,
\prettyref{lem:R_bd}, and bound \eqref{eq:L_xi_sqrt_bd}, it follows
that the number of oracle calls performed by the AIPP.SM is on the
order given by \eqref{eq:sp_aipp_compl}.

(b) It follows from the definitions of $p_{\xi}$, tolerance $\rho$,
and $(\bar{y},\bar{u})$ in \eqref{eq:p_xi_def}, \eqref{eq:sp_params},
and \eqref{eq:spec_output_sp_aipp}, respectively, \prettyref{prop:xi_facts}(b),
and the inclusion in \eqref{eq:approx_smoothed} that $\|\bar{u}\|\leq\rho_{x}$
and 
\[
\bar{u}\in\nabla p_{\xi}(\bar{x})+\pt h(\bar{x})=\nabla_{x}\Phi(\bar{x},y_{\xi}(\bar{x}))+\pt h(\bar{x})=\nabla_{x}\Phi(\bar{x},\bar{y})+\pt h(\bar{x}).
\]
Hence, we conclude that the top inclusion and the upper bound on $\|\bar{u}\|$
in \eqref{eq:sp_approx_sol} hold. Next, the optimality condition
of $\bar{y}=y_{\xi}(\bar{x})$ as a solution to \eqref{eq:p_xi_def}
and the definition of $\bar{v}$ in \eqref{eq:p_xi_def} give 
\begin{equation}
0\in\pt\left[-\Phi(\bar{x},\cdot)\right](\bar{y})+\frac{\bar{y}-y_{0}}{\xi}=\pt\left[-\Phi(\bar{x},\cdot)\right](\bar{y})-\bar{v}\label{eq:v_incl}
\end{equation}
Moreover, the definition of $\xi$ implies that 
\begin{equation}
\|\bar{v}\|=\frac{\|\bar{y}-y_{0}\|}{\xi}\leq\frac{D_{y}}{D_{y}/\rho_{y}}=\rho_{y}.\label{eq:v_bar_bd_prf}
\end{equation}
Hence, combining \eqref{eq:v_incl} and \eqref{eq:v_bar_bd_prf},
we conclude that the bottom inclusion and the upper bound on $\|\bar{v}\|$
in \eqref{eq:sp_approx_sol} hold. 
\end{proof}
We now make two remarks about \prettyref{prop:sp_aipp_facts}. First,
under the assumption that \eqref{eq:sp_params} is satisfied as equality,
the complexity of AIPP.SM reduces to

\begin{equation}
{\cal O}\left(m^{3/2}R_{1/(2m)}\hat{p}(x_{0})\left[\frac{L_{x}^{1/2}}{\rho_{x}^{2}}+\frac{L_{y}D_{y}^{1/2}}{\rho_{x}^{2}\rho_{y}^{1/2}}\right]\right)\label{eq:sp_compl_spec}
\end{equation}
under the reasonable assumption that the ${\cal O}(\rho_{x}^{-2}+\rho_{x}^{-2}\rho_{y}^{-1/2})$
term in \eqref{eq:sp_aipp_compl} dominates the other terms. Second,
recall from the last remark following the previous proposition that
when $Y$ is a singleton, \ref{prb:eq:min_max_co} becomes a special
instance of \ref{prb:eq:nco} and the AIPP.SM becomes equivalent to
the AIPPM of \prettyref{sec:aipp}. It then follows that the complexity
in \eqref{eq:sp_compl_spec} reduces to 
\begin{equation}
{\cal O}\left(\frac{m^{3/2}L_{x}^{1/2}R_{1/(2m)}\hat{p}(x_{0})}{\rho_{x}^{2}}\right)\label{eq:aipp_sp_compl}
\end{equation}
and, in view of this remark, the ${\cal O}(\rho_{x}^{-2}\rho_{y}^{-1/2})$
term in \eqref{eq:sp_compl_spec} is attributed to the (possible)
nonsmoothness in \ref{prb:eq:min_max_co}.

We next show how the AIPP.SM generates a point that is \emph{near}
a $\delta$-directional stationary point, i.e. a point that solves
\prettyref{prb:approx_dd_mco}. 
\begin{prop}
\label{prop:dd_aipp_facts}Let a tolerance pair $\delta>0$ be given
and consider the AIPP.SM with input parameter $\xi$ and tolerance
$\rho$ satisfying 
\begin{equation}
\xi\geq\frac{D_{y}}{\tau},\quad\rho=\frac{\delta}{2},\quad\tau\leq\min\left\{ \frac{m\delta^{2}}{2D_{y}},\frac{\delta^{2}}{32mD_{y}}\right\} .\label{eq:dd_params}
\end{equation}
Then, the following statements hold: 
\begin{itemize}
\item[(a)] the AIPP.SM performs 
\begin{equation}
{\cal O}\left(\Omega_{\xi}\left[\frac{R_{1/(2m)}\hat{p}(x_{0})}{\lambda^{2}\delta^{2}}+\frac{D_{y}^{2}}{\lambda\xi\delta^{2}}+\log_{1}^{+}(\Omega_{\xi})\right]\right)\label{eq:dd_aipp_comp}
\end{equation}
oracle calls where $\Omega_{\xi}$, $R_{\lam}\psi(\cdot)$ , and $D_{y}$
are as in \eqref{eq:Omega_xi_def}, \eqref{eq:Rphi_def}, and \eqref{eq:Dy_def},
respectively, and $\log_{1}^{+}(\cdot):=\max\{1,\log(\cdot)\}$,; 
\item[(b)] the first argument $x$ in the pair output by the AIPP.SM satisfies
\eqref{eq:dd_approx_sol}, and hence, solves \prettyref{prb:approx_dd_mco}. 
\end{itemize}
\end{prop}

\begin{proof}
(a) Using \prettyref{prop:sp_aipp_facts} with $(\rho_{x},\rho_{y})=(\delta/2,\tau)$
and the bound on $\tau$ in \eqref{eq:dd_params} it follows that
the AIPP.SM stops in a number of oracle calls bounded above by \eqref{eq:dd_aipp_comp}.

(b) Let $(x,u)$ be the pair generated by the AIPPM with $\xi$ and
$\bar{\rho}$ satisfying \eqref{eq:dd_params}. Defining $(\bar{v},\bar{y})$
as in \eqref{eq:spec_output_sp_aipp}, it follows from \prettyref{prop:sp_aipp_facts}
with $(\rho_{x},\rho_{y})=(\delta/2,\tau)$ that $(u,\bar{v},x,\bar{y})$
is a $(\delta/2,\tau)$-primal-dual stationary point of \ref{prb:eq:min_max_co}.
As a consequence, it follows from \prettyref{prop:impl1_statn} with
$(\rho_{x},\rho_{y})=(\delta/2,\tau)$ that there exists a point $\hat{x}$
satisfying 
\begin{align}
\|\hat{x}-x\|\leq\sqrt{\frac{2D_{y}\tau}{m}},\quad\inf_{\|d\|\leq1}\hat{p}'(\hat{x};d)\geq-\frac{\delta}{2}-2\sqrt{2mD_{y}\tau}.
\end{align}
Combining the above bounds with the bound on $\tau$ in \eqref{eq:dd_params}
yields the desired conclusion in view of \eqref{eq:dd_approx_sol}. 
\end{proof}
We now give three remarks about the above result. Second, \prettyref{prop:dd_aipp_facts}(b)
states that, while $x$ not a stationary point itself, it is near
a $\delta$-directional stationary point $\hat{x}$. Second, under
the assumption that \eqref{eq:dd_params} is satisfied as equality,
the complexity of the AIPP.SM reduces to 
\begin{equation}
{\cal O}\left(m^{3/2}R_{1/(2m)}\hat{p}(x_{0})\left[\frac{L_{x}^{1/2}}{\delta^{2}}+\frac{L_{y}D_{y}}{\delta^{3}}\right]\right)\label{eq:dd_compl_spec}
\end{equation}
under the reasonable assumption that the ${\cal O}(\delta^{-2}+\delta^{-3})$
term in \eqref{eq:dd_aipp_comp} dominates the other ${\cal O}(\delta^{-1})$
terms. Fourth, when $Y$ is a singleton, it is easy to see that \ref{prb:eq:min_max_co}
becomes a special instance of \ref{prb:eq:nco}, the AIPP.SM becomes
equivalent to the AIPPM of \prettyref{sec:aipp}, and the complexity
in \eqref{eq:dd_compl_spec} reduces to 
\begin{equation}
{\cal O}\left(\frac{m^{3/2}L_{x}^{1/2}R_{1/(2m)}\hat{p}(x_{0})}{\delta^{2}}\right).\label{eq:aipp_dd_compl}
\end{equation}
In view of the last remark, the ${\cal O}(\delta^{-3})$ term in \eqref{eq:dd_compl_spec}
is attributed to the (possible) nonsmoothness in \ref{prb:eq:min_max_co}.

\section[Accelerated Inexact Proximal Quadratic Penalty Smoothing (AIP.QP.S)
Method]{Accelerated Inexact Proximal Quadratic Penalty\protect \\
Smoothing (AIP.QP.S) Method}

\label{sec:qp_aipp_smoothing}

This section presents the AIP.QP.SM for finding stationary points
of \ref{prb:eq:constr_min_max_co} as in \eqref{eq:lc_sp_approx_sol}.

Since the AIP.QP.SM applies the AIP.QPM of \prettyref{sec:qp_aipp}
to a relaxation of \ref{prb:eq:constr_min_max_co}, we assume that
$(\Phi,h,X,Y)$ satisfies assumptions \ref{asmp:mco_a1}--\ref{asmp:mco_a4}
of \prettyref{sec:aipp_smoothing} as well as the following ones:

\stepcounter{assumption}
\begin{enumerate}
\item \label{asmp:mcco_a1}${\cal A}:{\cal X}\mapsto{\cal R}$ is a nonzero
linear operator, $b$ is in the range of ${\cal A}$, and the feasible
region ${\cal F}:=\{x\in{\cal X}:{\cal A}x=b\}$ is nonempty; 
\item \label{asmp:mcco_a2}there exists $\hat{c}\geq0$ such that $\hat{\varphi}_{\hat{c}}>-\infty$,
where
\begin{equation}
\hat{\varphi}_{c,\xi}:=\inf_{z\in{\cal Z}}\left\{ \varphi_{c,\xi}(z):=p_{\xi}(x)+\frac{c}{2}\|{\cal A}x-b\|^{2}+h(z)\right\} ,\quad\forall c\geq0,\label{eq:smoothing_varphiC_def}
\end{equation}
where $p_{\xi}(\cdot)$ is as in \eqref{eq:p_xi_def}.
\end{enumerate}
For ease of referencing, we also state the problem of finding a pair
satisfying \eqref{eq:lc_sp_approx_sol} in \prettyref{prb:lc_approx_sp_mco}.

\begin{mdframed}
\mdprbcaption{Find an approximate primal-dual stationary point of ${\cal MCCO}$}{prb:lc_approx_sp_mco}
Given $(\rho_x, \rho_y, \eta) \in \r_{++}^2$, find a pair $([\bar{x},\bar{y},\bar{r}],[\bar{u},\bar{v}]) \in [X\times Y\times {\cal R}] \times [{\cal X} \times {\cal Y}]$ satisfying condition \eqref{eq:lc_sp_approx_sol}.
\end{mdframed}

We now state the AIP.QP.SM in \prettyref{alg:qp_aipp_sm}. Given $(x_{0},y_{0})\in X\times Y$
and $\hat{c}>0$, its main idea is to its main idea is to apply an
instance of the AIP.QPM in \prettyref{sec:qp_aipp} to the CNCO problem
\begin{equation}
\min_{x\in X}\left\{ \hat{p}_{\xi}(x):=p_{\xi}(x)+h(x):{\cal A}x=b\right\} ,\label{eq:smoothed_lc_prb}
\end{equation}
where $p_{\xi}$ is as in \eqref{eq:p_xi_def} and $\xi$ is a positive
scalar that will depend on the tolerances in \eqref{eq:lc_sp_approx_sol}.
The resulting output of this AIP.QP call is then similarly transformed
like the AIPP.SM of \prettyref{sec:aipp_smoothing} to obtain a pair
that solves \prettyref{prb:lc_approx_sp_mco}. 

\begin{mdframed}
\mdalgcaption{AIP.QP.S Method}{alg:qp_aipp_sm}
\begin{smalgorithmic}
	\Require{$(\rho_x, \rho_y, \eta) \in \r_{++}^3, \enskip \xi > D_y / \rho_y,  \enskip (m,L_x,L_y)\in\r_{+}^3, \enskip h \in \cConv(Z), \enskip \Phi$  as in  \ref{asmp:mco_a2}--\ref{asmp:mco_a4}, $\enskip \hat{c} > 0,  \enskip (x_0, y_0) \in X \times Y$;}
	\Initialize{$\lam \gets 1/(2m),  \enskip \sigma \gets 1/2$}
	\vspace*{.5em}
	\Procedure{AIP.QP.S}{$\Phi, h, x_0, y_0, m, L_x, L_y, \rho$}
		\StateEq{$y_\xi \Lleftarrow \argmax_{y\in Y} \Phi_{\xi}(\cdot,y)$} \Comment{See \eqref{eq:p_xi_def}.}
		\StateEq{$p_\xi \Lleftarrow  \max_{y\in Y} \Phi_{\xi}(\cdot,y)$} \Comment{See \eqref{eq:p_xi_def}.}
		\StateEq{$L_\xi \gets L_{y}\left[\xi L_{y}+\sqrt{\xi(L_{x}+m)}\right]+L_{x}$}
		\StateEq{$([\bar{x}, \bar{r}], [\bar{u}, \bar{q}]) \gets \text{AIP.QP}(p_\xi, h, {\cal A}, \{b\}, x_0, \hat{c}, \lam, m, L_\xi, \sigma, \rho_y, \eta)$} \label{ln:smoothing_qp_aipp_call}
		\StateEq{$\bar{y} \gets y_\xi(\bar{x})$} \label{ln:smoothing_ybar}
		\StateEq{$\bar{v} \gets \frac{y_{0}-y_{\xi}(x)}{\xi}$}
		\StateEq{\Return{$([\bar{x}, \bar{y}, \bar{r}], [\bar{u}, \bar{v}])$}}.
	\EndProcedure
\end{smalgorithmic}
\end{mdframed}

We give two remarks about the AIP.QP.SM. First, it follows from \prettyref{cor:spec_qp_aipp_compl}
that the AIP.QPM invoked in \prettyref{ln:smoothing_qp_aipp_call}
stops and outputs a pair $([\bar{x},\bar{r}],[\bar{u},\bar{q}])$
satisfying 
\[
\bar{u}\in\nabla p_{\xi}(\bar{x})+\pt h(\bar{x})+A^{*}\bar{r},\quad\|\bar{u}\|\leq\rho_{x},\quad\|{\cal A}\bar{x}-b\|\leq\eta.
\]
Second, since it is a one-pass algorithm (as opposed to an iterative
algorithm), the complexity of the AIP.QP.SM is essentially that of
the AIP.QPM. 

We now show how the AIP.QP.SM generates a point $([\bar{x},\bar{y},\bar{r}],[\bar{u},\bar{v}])$
satisfying \eqref{eq:lc_sp_approx_sol}.
\begin{prop}
\label{prop:qp_sp_aipp_compl}Let a tolerance triple $(\rho_{x},\rho_{x},\eta)\in\r_{++}^{3}$
be given and let $([\bar{x},\bar{y},\bar{r}],[\bar{u},\bar{v}])$
be the output obtained by the QP-AIPP.SM. Then, the following properties
hold: 
\end{prop}

\begin{itemize}
\item[(a)] the AIP.QP.SM performs 
\begin{equation}
{\cal O}\left(\Omega_{\xi,\eta}\left[\frac{m^{2}R_{1/(2m)}^{{\cal F}}\varphi_{\hat{c}}(x_{0})}{\rho_{x}^{2}}+\frac{mD_{y}^{2}}{\xi\rho_{x}^{2}}+\log_{1}^{+}\left(\Omega_{\xi,\eta}\right)\right]\right)\label{eq:comp_qp_sp_aipp}
\end{equation}
oracle calls, where 
\begin{align}
\varphi_{\hat{c}} & :=\hat{p}(x)+\frac{c}{2}\|{\cal A}x-b\|^{2},\nonumber \\
\Omega_{\xi,\eta} & :=\Omega_{\xi}+\left(R_{\hat{c}}(\hat{p};1/(4m))+\frac{D_{y}^{2}}{m\xi}\right)^{1/2}\frac{\|{\cal A}\|}{\eta},\label{eq:Theta_xi_eta_def}
\end{align}
and $\Omega_{\xi}$, $R_{\lam}^{{\cal F}}\psi(\cdot)$, and $D_{y}$
are as in \eqref{eq:Omega_xi_def}, \eqref{eq:RFphi_def}, and \eqref{eq:Dy_def},
respectively; 
\item[(b)] the pair $([\bar{x},\bar{y},\bar{r}],[\bar{u},\bar{v}])$ solves
\prettyref{prb:lc_approx_sp_mco}. 
\end{itemize}
\begin{proof}
(a) Let $\Theta_{\eta}$ be as in \eqref{eq:qp_Theta_def} with $M=L_{\xi}$.
Using the same arguments as in \prettyref{lem:R_bd}, it is easy to
see that 
\begin{equation}
R_{1/(2m)}^{{\cal F}}\varphi_{\hat{c},\xi}(x_{0})\leq R_{1/(2m)}^{{\cal F}}\varphi_{\hat{c}}(x_{0})+\frac{D_{y}^{2}}{8m\xi}.\label{eq:R_lc_sp_bd}
\end{equation}
where $\varphi_{\hat{c},\xi}$ is as in \eqref{eq:smoothing_varphiC_def}.
Hence, using \eqref{eq:L_xi_sqrt_bd} and \eqref{eq:R_lc_sp_bd},
we have 
\begin{align}
\sqrt{\frac{\Theta_{\eta}}{4m}+1} & \leq1+\sqrt{\frac{L_{\xi}}{4m}}+\sqrt{\frac{4R_{1/(2m)}^{{\cal F}}\varphi_{\hat{c},\xi}(x_{0})\|{\cal A}\|^{2}}{\eta^{2}}}\nonumber \\
 & \leq1+\frac{\sqrt{\xi}L_{y}+\sqrt{L_{x}}}{2\sqrt{m}}+2\left(R_{1/(2m)}^{{\cal F}}\varphi_{\hat{c}}(x_{0})+\frac{D_{y}^{2}}{8m\xi}\right)^{1/2}\frac{\|{\cal A}\|}{\eta}\nonumber \\
 & =\Theta(\Omega_{\xi,\eta}).\label{eq:Theta_sqrt_bd}
\end{align}
The complexity in \eqref{eq:comp_qp_sp_aipp} now follows from \prettyref{cor:spec_qp_aipp_compl}
with $(\phi,f,M)=(p,p_{\xi},L_{\xi})$, \eqref{eq:Theta_sqrt_bd},
and \eqref{eq:R_lc_sp_bd}.

(b) The top inclusion and bounds involving $\|\bar{u}\|$ and $\|{\cal A}\bar{x}-b\|$
in \eqref{eq:lc_sp_approx_sol} follow from \prettyref{prop:xi_facts}(b),
the definition of $\bar{y}$ in \prettyref{ln:smoothing_ybar} of
the method, and \prettyref{cor:spec_qp_aipp_compl} with $f=p_{\xi}$.
The bottom inclusion and bound involving $\|\bar{v}\|$ follow from
similar arguments given for \prettyref{prop:sp_aipp_facts}(b). 
\end{proof}
We now make two remarks about the above complexity bound. First, under
the assumption that $\xi=D_{y}/\rho_{y}$, the complexity of the AIP.QP.SM
reduces to 
\begin{equation}
{\cal O}\left(m^{3/2}R_{1/(2m)}^{{\cal F}}\varphi_{\hat{c}}(x_{0})\left[\frac{L_{x}^{1/2}}{\rho_{x}^{2}}+\frac{L_{y}D_{y}^{1/2}}{\rho_{y}^{1/2}\rho_{x}^{2}}+\frac{m^{1/2}\|{\cal A}\|}{\eta\rho_{x}^{2}}\sqrt{R_{1/(2m)}^{{\cal F}}\varphi_{\hat{c}}(x_{0})}\right]\right),\label{eq:qp_sp_aipp_spec_compl}
\end{equation}
under the reasonable assumption that the ${\cal O}(\rho_{x}^{-2}+\eta^{-1}\rho_{x}^{-2}+\rho_{y}^{-1/2}\rho_{x}^{-2})$
term in \eqref{eq:comp_qp_sp_aipp} dominates the other terms. Third,
when $Y$ is a singleton, it is easy to see that \ref{prb:eq:constr_min_max_co}
becomes a special instance of the CNCO problem \ref{prb:eq:cnco},
the AIP.QP.SM of this subsection becomes equivalent to the AIP.QPM
of \prettyref{sec:qp_aipp}, and the complexity in \eqref{eq:qp_sp_aipp_spec_compl}
reduces to 
\begin{equation}
{\cal O}\left(m^{3/2}R_{1/(2m)}^{{\cal F}}\varphi_{\hat{c}}(x_{0})\left[\frac{L_{x}^{1/2}}{\rho_{x}^{2}}+\frac{m^{1/2}\|{\cal A}\|}{\eta\rho_{x}^{2}}\sqrt{R_{1/(2m)}^{{\cal F}}\varphi_{\hat{c}}(x_{0})}\right]\right).\label{eq:aipp_sp_qp_compl}
\end{equation}
In view of the last remark, the ${\cal O}(\rho_{x}^{-2}\rho_{y}^{-1/2})$
term in \eqref{eq:qp_sp_aipp_spec_compl} is attributed to the (possible)
nonsmoothness in \ref{prb:eq:constr_min_max_co}.

Let us now conclude this section with a remark about the penalty subproblem
\begin{equation}
\min_{x\in X}\left\{ p_{\xi}(x)+h(x)+\frac{c}{2}\|{\cal A}x-b\|^{2}\right\} ,\label{eq:penalty_sp_prb}
\end{equation}
which is what the AIPPM considers every time it is called in the AIP.QPM
(see \prettyref{ln:smoothing_qp_aipp_call} of the AIP.QP.SM). First,
observe that \ref{prb:eq:constr_min_max_co} can be equivalently reformulated
as 
\begin{equation}
\min_{x\in X}\max_{y\in Y,r\in{\cal U}}\left[\Psi(x,y,r):=\Phi(x,y)+h(x)+\inner r{{\cal A}x-b}\right].\label{eq:Psi_def}
\end{equation}
Second, it is straightforward to verify that problem \eqref{eq:penalty_sp_prb}
is equivalent to 
\begin{equation}
\min_{x\in X}\left\{ \hat{p}_{c,\xi}(x):=p_{c,\xi}(x)+h(x)\right\} ,\label{eq:penalty_composite_prb}
\end{equation}
where the function $p_{c,\xi}:X\mapsto\r$ is given by 
\begin{equation}
p_{c,\xi}(x):=\max_{y\in Y,r\in{\cal U}}\left\{ \Psi(x,y,r)-\frac{1}{2c}\|r\|^{2}-\frac{1}{2\xi}\|y-y_{0}\|^{2}\right\} \quad\forall x\in X\label{eq:g_xi_c_def}
\end{equation}
with $\Psi$ as in \eqref{eq:Psi_def}. As a consequence, problem
\eqref{eq:penalty_composite_prb} is similar to \eqref{eq:smoothed_intro_prb}
in that a smooth approximate is used in place of the nonsmooth component
of the underlying saddle function $\Psi$. On the other hand, observe
that we cannot directly apply the smoothing scheme developed in \prettyref{sec:aipp_smoothing}
to \eqref{eq:penalty_composite_prb} as the set ${\cal U}$ is generally
unbounded. One approach that avoids this problem is to invoke the
AIPPM of \prettyref{sec:aipp} to solve a sequence subproblems of
the form in \eqref{eq:penalty_composite_prb} for increasing values
of $c$. However, in view of the equivalence of \eqref{eq:penalty_sp_prb}
and \eqref{eq:penalty_composite_prb}, this is exactly the approach
taken by the AIP.QP.SM of this section.

\section{Numerical Experiments}

\label{sec:num_min_max}

This section examines the performance of several solvers for finding
approximate stationary points of \ref{prb:eq:min_max_co} where $(X,Y,\Phi,h)$
satisfy assumptions \ref{asmp:mco_a1}--\ref{asmp:mco_a5} of \prettyref{chap:min_max}.
Each problem is chosen so that the computation of the function $y_{\xi}$
in \eqref{eq:y_xi_def} is easy, and the justification for the curvature
constants in this section, e.g. $m$, $L_{x}$, and $L_{y}$, can
be found in \prettyref{app:num_Lipschitz}. All experiments are run
on Linux 64-bit machines each containing Xeon E5520 processors and
at least 8 GB of memory using MATLAB 2020a. It is worth mentioning
that the complete code for reproducing the experiments is freely available
online\footnote{See the code in \texttt{./tests/thesis/} from the GitHub repository
\href{https://github.com/wwkong/nc_opt/}{https://github.com/wwkong/nc\_opt/}}.

The algorithms benchmarked in this section are as follows.
\begin{itemize}
\item \textbf{PGSF}: a variant of \citep[Algorithm 2]{Nouiehed2019} in
which the input parameters are as in \citep[Theorem 4.2]{Nouiehed2019}
and which explicitly evaluates the argmax function $\alpha^{*}(\cdot)$
in \citep[Section 4]{Nouiehed2019} instead of applying an ACG variant
to estimate its evaluation.
\item \textbf{AG.S}: an instance of \prettyref{alg:aipp_sm} in which the
AIPPM is replaced by the AG method in \prettyref{subsec:num_unconstr}.
\item \textbf{AIPP.S}: an instance of \prettyref{alg:aipp_sm} in which
the AIPPM is replaced by the R.AIPP variant in \prettyref{subsec:num_unconstr}.
\end{itemize}
Given a tolerance pair $(\rho_{x},\rho_{y})\in\r_{++}^{2}$ and an
initial point $(x_{0},y_{0})\in X\times Y$, each algorithm in this
section seeks a pair $([\bar{x},\bar{y}],[\bar{u},\bar{v}])\in[X\times Y]\times[{\cal X}\times{\cal Y}]$
satisfying 
\begin{align}
\begin{gathered}\left(\begin{array}{c}
\bar{u}\\
\bar{v}
\end{array}\right)\in\left(\begin{array}{c}
\nabla_{x}\Phi(\bar{x},\bar{y})\\
0
\end{array}\right)+\left(\begin{array}{c}
\pt h(\bar{x})\\
\pt\left[-\Phi(\bar{x},\cdot)\right](\bar{y})
\end{array}\right),\\
\frac{\|\bar{u}\|}{\|\nabla p_{\xi}(z_{0})\|+1}\leq\rho_{x},\quad\|\bar{v}\|\leq\rho_{y},
\end{gathered}
\label{eq:numerical_obj}
\end{align}
is obtained, where $\xi=D_{y}/\rho_{y}$ and $p_{\xi}$ is as in \eqref{eq:p_xi_def}.
Moreover, each algorithm is given a time limit of 4000 seconds. Iteration
counts are not reported for instances which were unable to obtain
$([\hat{x},\hat{y}],[\hat{u},\hat{v}])$ as above. The bold numbers
in each of the tables in this section highlight the algorithm that
performed the most efficiently in terms of iteration count or total
runtime.

\subsection{Maximum of Nonconvex Quadratic Forms}

This subsection presents computational results for the min-max quadratic
vector problem (MQV) problem considered in \citep{Kong2019a}. More
specifically, given a dimension triple $(n,l,k)\in\n^{3}$, a set
of parameters $\{(\alpha_{i},\beta_{i})\}_{i=1}^{k}\subseteq\r_{++}^{2}$,
a set of vectors $\{d_{i}\}_{i=1}^{k}\subseteq\r^{l}$, a set of diagonal
matrices $\{D_{i}\}_{i=1}^{k}\subseteq\r^{n\times n}$, and matrices
$\{C_{i}\}_{i=1}^{k}\subseteq\r^{l\times n}$ and $\{B_{i}\}_{i=1}^{k}\subseteq\r^{n\times n}$,
we consider the MQV problem
\[
\min_{x\in\r^{n}}\max_{y\in\r^{k}}\left\{ \delta_{\Delta^{n}}(x)+\sum_{i=1}^{k}y_{i}g_{i}(x):y\in\Delta^{k}\right\} ,
\]
where, for every index $1\leq i\leq k$, integer $p\in\n$, and $x\in\rn$,
\[
f_{i}(x):=\frac{\alpha_{i}}{2}\|C_{i}x-d_{i}\|^{2}-\frac{\beta_{i}}{2}\|D_{i}B_{i}x\|^{2},\quad\Delta^{p}:=\left\{ z\in\r_{+}^{p}:\sum_{i=1}^{p}z_{i}=1,z\geq0\right\} .
\]

We now describe the experiment parameters for the instances considered.
First, the dimensions are set to be $(n,l,k)=(200,10,5)$ and only
5.0\% of the entries of the submatrices $B_{i}$ and $C_{i}$ are
nonzero. Second, the entries of $B_{i},C_{i},$ and $d_{i}$ (resp.,
$D_{i}$) are generated by sampling from the uniform distribution
${\cal U}[0,1]$ (resp., ${\cal U}\{1,...,1000\}$). Third, the initial
starting point is $z_{0}=\boldsymbol{e}/n$ where $\boldsymbol{e}$
is a vector of all ones. Fourth, the key problem parameters, for every
$(x,y)\in\rn\times\r^{k}$, are 
\begin{gather*}
\Phi(x,y)=\sum_{i=1}^{k}y_{i}f_{i}(x),\quad h(x)=\delta_{\Delta^{n}}(x),\\
\rho_{x}=10^{-2},\quad\rho_{y}=10^{-1},\quad Y=\Delta^{k}.
\end{gather*}
Fifth, each problem instance considered is based on a specific curvature
pair $(m,M)\in\r_{++}^{2}$ satisfying $m\leq M$, for which each
scalar pair $(\alpha_{i},\beta_{i})\in\r_{++}^{2}$ is selected so
that $M=\lambda_{\text{\ensuremath{\max}}}(\nabla^{2}f_{i})$ and
$-m=\lambda_{\min}(\nabla^{2}f_{i})$ for $1\leq i\leq k$. Moreover,
the method for obtaining each pair $(\alpha_{i},\beta_{i})$ is the
same as in \prettyref{subsec:qmp}. Finally, the Lipschitz and curvature
constants selected are 
\begin{equation}
m=m,\quad L_{x}=M,\quad L_{y}=M\sqrt{k}+\|P\|,\label{eq:L_mm_unconstr}
\end{equation}
where $P$ is an $n$-by-$k$ matrix whose $i^{th}$ column is equal
to $\alpha_{i}C_{i}^{T}d_{i}$.

The table of iteration counts and total runtimes are given in \prettyref{tab:mqv_iter}
and \prettyref{tab:mqv_runtime}, respectively. 
\begin{center}
\begin{table}[th]
\begin{centering}
\begin{tabular}{|>{\centering}p{0.7cm}>{\centering}p{0.7cm}|>{\centering}p{1.8cm}>{\centering}p{1.8cm}>{\centering}p{1.8cm}|}
\hline 
\multicolumn{2}{|c|}{\textbf{\small{}$(m,M)$}} & \multicolumn{3}{c|}{\textbf{\small{}Iteration Count}}\tabularnewline
\hline 
{\footnotesize{}$m$} & {\footnotesize{}$M$} & {\footnotesize{}PGSF} & {\footnotesize{}AG.S} & {\footnotesize{}AIPP.S}\tabularnewline
\hline 
{\footnotesize{}$10^{1}$} & {\footnotesize{}$10^{2}$} & {\footnotesize{}21462} & {\footnotesize{}1824} & \textbf{\footnotesize{}81}\tabularnewline
{\footnotesize{}$10^{1}$} & {\footnotesize{}$10^{3}$} & {\footnotesize{}159682} & {\footnotesize{}6280} & \textbf{\footnotesize{}267}\tabularnewline
{\footnotesize{}$10^{1}$} & {\footnotesize{}$10^{4}$} & {\footnotesize{}-} & {\footnotesize{}28966} & \textbf{\footnotesize{}793}\tabularnewline
{\footnotesize{}$10^{1}$} & {\footnotesize{}$10^{5}$} & {\footnotesize{}-} & {\footnotesize{}28966} & \textbf{\footnotesize{}793}\tabularnewline
\hline 
\end{tabular}
\par\end{centering}
\caption{Iteration Counts for MQV problems.\label{tab:mqv_iter}}
\end{table}
\par\end{center}

\begin{center}
\begin{table}[th]
\begin{centering}
\begin{tabular}{|>{\centering}p{0.7cm}>{\centering}p{0.7cm}|>{\centering}p{1.8cm}>{\centering}p{1.8cm}>{\centering}p{1.8cm}|}
\hline 
\multicolumn{2}{|c|}{\textbf{\small{}$(m,M)$}} & \multicolumn{3}{c|}{\textbf{\small{}Runtime}}\tabularnewline
\hline 
{\footnotesize{}$m$} & {\footnotesize{}$M$} & {\footnotesize{}PGSF} & {\footnotesize{}AG.S} & {\footnotesize{}AIPP.S}\tabularnewline
\hline 
{\footnotesize{}$10^{1}$} & {\footnotesize{}$10^{2}$} & {\footnotesize{}358.24} & {\footnotesize{}40.17} & \textbf{\footnotesize{}1.86}\tabularnewline
{\footnotesize{}$10^{1}$} & {\footnotesize{}$10^{3}$} & {\footnotesize{}2896.70} & {\footnotesize{}179.27} & \textbf{\footnotesize{}6.36}\tabularnewline
{\footnotesize{}$10^{1}$} & {\footnotesize{}$10^{4}$} & {\footnotesize{}4000.00} & {\footnotesize{}698.52} & \textbf{\footnotesize{}15.21}\tabularnewline
{\footnotesize{}$10^{1}$} & {\footnotesize{}$10^{5}$} & {\footnotesize{}4000.00} & {\footnotesize{}835.17} & \textbf{\footnotesize{}18.76}\tabularnewline
\hline 
\end{tabular}
\par\end{centering}
\caption{Runtimes for MQV problems.\label{tab:mqv_runtime}}
\end{table}
\par\end{center}

\subsection{Truncated Robust Regression}

This subsection presents computational results for the truncated robust
regression (TRR) problem in \citep{Rafique2018}. More specifically,
given a dimension pair $(n,k)\in\n^{2}$, a set of $n$ data points
$\{(a_{j},b_{j})\}_{i=1}^{n}\subseteq\r^{k}\times\{1,-1\}$ and a
parameter $\alpha>0$, we consider the TRR problem 
\[
\min_{x\in\r^{k}}\max_{y\in\r^{n}}\left\{ \sum_{j=1}^{n}y_{j}(\phi_{\alpha}\circ\ell_{j})(x):y\in\Delta^{n}\right\} 
\]
where $\Delta^{n}$ is as in \eqref{eq:Delta_def} with $p=n$ and,
for every $(\alpha,t,x)\in\r_{++}\times\r_{++}\times\r^{k}$, 
\[
\phi_{\alpha}(t):=\alpha\log\left(1+\frac{t}{\alpha}\right),\quad\ell_{j}(x):=\log\left(1+e^{-b_{j}\left\langle a_{j},x\right\rangle }\right).
\]

We now describe the experiment parameters for the instances considered.
First, $\alpha$ is set to $10$ and the data points $\{(a_{i},b_{i})\}$
are taken from different datasets in the LIBSVM library\footnote{See https://www.csie.ntu.edu.tw/\textasciitilde cjlin/libsvmtools/datasets/binary.html.}
for which each problem instance is based off of (see the ``data name''
column in the table below, which corresponds to a particular LIBSVM
dataset). Second, the initial starting point is $z_{0}=0$. Third,
the key problem parameters, for every $(x,y)\in\r^{k}\times\rn$,
are 
\[
\begin{gathered}\Phi(x,y)=\sum_{j=1}^{n}y_{j}(\phi_{\alpha}\circ\ell_{j})(x),\quad h(x)=0,\quad\rho_{x}=10^{-5},\quad\rho_{y}=10^{-3},\end{gathered}
\quad Y=\Delta^{n}.
\]
Finally, the Lipschitz and curvature constants selected are 
\begin{equation}
m=L_{x}=\frac{1}{\alpha}\max_{1\leq j\leq n}\|a_{j}\|^{2},\quad L_{y}=\sqrt{\sum_{j=1}^{n}\|a_{j}\|^{2}}.\label{eq:L_trr}
\end{equation}

The table of iteration counts and total runtimes are given in \prettyref{tab:trr_iter}
and \prettyref{tab:trr_runtime}, respectively. 
\begin{center}
\begin{table}[th]
\begin{centering}
\begin{tabular}{|>{\centering}p{2cm}|>{\centering}p{1.8cm}>{\centering}p{1.8cm}>{\centering}p{1.8cm}|}
\hline 
\multicolumn{1}{|c|}{} & \multicolumn{3}{c|}{\textbf{\small{}Iteration Count}}\tabularnewline
\hline 
\multicolumn{1}{|c|}{{\footnotesize{}data name}} & {\footnotesize{}PGSF} & {\footnotesize{}AG.S} & {\footnotesize{}AIPP.S}\tabularnewline
\hline 
{\footnotesize{}heart} & {\footnotesize{}6415} & {\footnotesize{}1746} & \textbf{\footnotesize{}506}\tabularnewline
{\footnotesize{}diabetes} & {\footnotesize{}3721} & {\footnotesize{}1641} & \textbf{\footnotesize{}463}\tabularnewline
{\footnotesize{}ionosphere} & {\footnotesize{}54545} & {\footnotesize{}8327} & \textbf{\footnotesize{}1262}\tabularnewline
{\footnotesize{}sonar} & {\footnotesize{}-} & {\footnotesize{}96208} & \textbf{\footnotesize{}69464}\tabularnewline
\hline 
\end{tabular}
\par\end{centering}
\caption{Iteration Counts for TRR problems.\label{tab:trr_iter}}
\end{table}
\par\end{center}

\begin{center}
\begin{table}[th]
\begin{centering}
\begin{tabular}{|>{\centering}p{2cm}|>{\centering}p{1.8cm}>{\centering}p{1.8cm}>{\centering}p{1.8cm}|}
\hline 
\multicolumn{1}{|c|}{} & \multicolumn{3}{c|}{\textbf{\small{}Runtime}}\tabularnewline
\hline 
\multicolumn{1}{|c|}{{\footnotesize{}data name}} & {\footnotesize{}PGSF} & {\footnotesize{}AG.S} & {\footnotesize{}AIPP.S}\tabularnewline
\hline 
{\footnotesize{}heart} & {\footnotesize{}10.24} & {\footnotesize{}3.24} & \textbf{\footnotesize{}2.08}\tabularnewline
{\footnotesize{}diabetes} & {\footnotesize{}5.98} & {\footnotesize{}3.77} & \textbf{\footnotesize{}1.67}\tabularnewline
{\footnotesize{}ionosphere} & {\footnotesize{}104.75} & {\footnotesize{}18.94} & \textbf{\footnotesize{}4.58}\tabularnewline
{\footnotesize{}sonar} & {\footnotesize{}4000.00} & \textbf{\footnotesize{}97.56} & {\footnotesize{}107.42}\tabularnewline
\hline 
\end{tabular}
\par\end{centering}
\caption{Runtimes for TRR problems.\label{tab:trr_runtime}}
\end{table}
\par\end{center}

It is worth mentioning that \citep{Rafique2018} also presents a min-max
algorithm for obtaining a stationary point as in \eqref{eq:numerical_obj}.
However, its iteration complexity, which is ${\cal O}(\rho^{-6})$
when $\rho=\rho_{x}=\rho_{y}$, is significantly worse than the other
algorithms considered in this section and, hence, we choose not to
include this algorithm in our benchmarks.

\subsection{Power Control in the Presence of a Jammer}

This subsection presents computational results for the power control
(PC) problem in \citep{Lu2019}. More specifically, given a dimension
pair $(N,K)\in\n^{2}$, a pair of parameters $(\sigma,R)\in\r_{++}^{2}$,
a 3D tensor ${\cal A}\in\r_{+}^{K\times K\times N}$, and a matrix
$B\in\r_{+}^{K\times N}$, we consider the PC problem 
\[
\min_{X\in\r^{K\times N}}\max_{y\in\r^{N}}\left\{ \sum_{k=1}^{K}\sum_{n=1}^{N}f_{k,n}(X,y):0\leq X\leq R,0\leq y\leq\frac{N}{2},\right\} ,
\]
where, for every $(X,y)\in\r^{K\times N}\times\r^{N}$, 
\[
f_{k,n}(X,y):=-\log\left(1+\frac{{\cal A}_{k,k,n}X_{k,n}}{\sigma^{2}+B_{k,n}y_{n}+\sum_{j=1,j\neq k}^{K}{\cal A}_{j,k,n}X_{j,n}}\right).
\]

We now describe the experiment parameters for the instances considered.
First, the scalar parameters are set to be $(\sigma,R)=(1/\sqrt{2},K^{1/K})$
and the quantities ${\cal A}$ and $B$ are set to be the squared
moduli of the entries of two Gaussian sampled complex--valued matrices
${\cal H}\in\mathbb{C}^{K\times K\times N}$ and $P\in\mathbb{C}^{K\times N}$.
More precisely, the entries of ${\cal H}$ and $P$ are sampled from
the standard complex Gaussian distribution ${\cal CN}(0,1)$ and 
\[
{\cal A}_{j,k,n}=|{\cal H}_{j,k,n}|^{2},\quad B_{k,n}=|P_{k,n}|^{2}\quad\forall(j,k,n).
\]
Second, the initial starting point is $z_{0}=0$. Third, with respect
to the termination criterion \eqref{eq:numerical_obj}, the inputs,
for every $(X,y)\in\r^{K\times N}\times\r^{N}$, are 
\begin{align*}
\begin{gathered}\Phi(X,y)=\sum_{k=1}^{K}\sum_{n=1}^{N}f_{k,n}(X,y),\quad h(X)=\delta_{Q_{R}^{K\times N}}(X),\\
\rho_{x}=10^{-1},\quad\rho_{y}=10^{-1},\quad Y=Q_{N/2}^{N\times1}.
\end{gathered}
\end{align*}
where $Q_{T}^{U\times V}:=\{z\in\r^{p\times q}:0\leq z\leq T\}$ for
every $T>0$ and $(U,V)\in\n^{2}$. Fourth, each problem instance
considered is based on a specific dimension pair $(N,K)$. Finally,
the Lipschitz and curvature constants selected are 
\begin{equation}
m=L_{x}=\frac{2}{\min\{\sigma^{4},\sigma^{6}\}}\max_{\substack{1\leq k\leq K\\
1\leq n\leq N
}
}\sum_{j=1}^{K}{\cal A}_{k,j,n}^{2},\quad L_{y}=\frac{2}{\min\{\sigma^{4},\sigma^{6}\}}\max_{\substack{1\leq k\leq K\\
1\leq n\leq N
}
}\sum_{j=1}^{K}B_{j,n}{\cal A}_{k,j,n}.\label{eq:L_pc}
\end{equation}

The table of iteration counts and total runtimes are given in \prettyref{tab:pc_iter}
and \prettyref{tab:pc_runtime}, respectively. 
\begin{center}
\begin{table}[th]
\begin{centering}
\begin{tabular}{|>{\centering}p{0.7cm}>{\centering}p{0.7cm}|>{\centering}p{1.8cm}>{\centering}p{1.8cm}>{\centering}p{1.8cm}|}
\hline 
\multicolumn{2}{|c|}{\textbf{\small{}$(N,K)$}} & \multicolumn{3}{c|}{\textbf{\small{}Iteration Count}}\tabularnewline
\hline 
{\footnotesize{}$N$} & {\footnotesize{}$K$} & {\footnotesize{}PGSF} & {\footnotesize{}AG.S} & {\footnotesize{}AIPP.S}\tabularnewline
\hline 
{\footnotesize{}5} & {\footnotesize{}5} & {\footnotesize{}-} & {\footnotesize{}322831} & \textbf{\footnotesize{}38}\tabularnewline
{\footnotesize{}10} & {\footnotesize{}10} & {\footnotesize{}-} & {\footnotesize{}33398} & \textbf{\footnotesize{}62}\tabularnewline
{\footnotesize{}25} & {\footnotesize{}25} & {\footnotesize{}-} & {\footnotesize{}161716} & \textbf{\footnotesize{}187}\tabularnewline
{\footnotesize{}50} & {\footnotesize{}50} & {\footnotesize{}-} & {\footnotesize{}-} & \textbf{\footnotesize{}572}\tabularnewline
\hline 
\end{tabular}
\par\end{centering}
\caption{Iteration Counts for PC problems.\label{tab:pc_iter}}
\end{table}
\par\end{center}

\begin{center}
\begin{table}[th]
\begin{centering}
\begin{tabular}{|>{\centering}p{0.7cm}>{\centering}p{0.7cm}|>{\centering}p{1.8cm}>{\centering}p{1.8cm}>{\centering}p{1.8cm}|}
\hline 
\multicolumn{2}{|c|}{\textbf{\small{}$(N,K)$}} & \multicolumn{3}{c|}{\textbf{\small{}Runtime}}\tabularnewline
\hline 
{\footnotesize{}$N$} & {\footnotesize{}$K$} & {\footnotesize{}PGSF} & {\footnotesize{}AG.S} & {\footnotesize{}AIPP.S}\tabularnewline
\hline 
{\footnotesize{}5} & {\footnotesize{}5} & {\footnotesize{}4000.00} & {\footnotesize{}3166.40} & \textbf{\footnotesize{}0.65}\tabularnewline
{\footnotesize{}10} & {\footnotesize{}10} & {\footnotesize{}4000.00} & {\footnotesize{}509.47} & \textbf{\footnotesize{}0.74}\tabularnewline
{\footnotesize{}25} & {\footnotesize{}25} & {\footnotesize{}4000.00} & {\footnotesize{}3907.10} & \textbf{\footnotesize{}4.89}\tabularnewline
{\footnotesize{}50} & {\footnotesize{}50} & {\footnotesize{}4000.00} & {\footnotesize{}4000.00} & \textbf{\footnotesize{}30.29}\tabularnewline
\hline 
\end{tabular}
\par\end{centering}
\caption{Runtimes for PC problems.\label{tab:pc_runtime}}
\end{table}
\par\end{center}

It is worth mentioning that \citep{Lu2019} also presents a min-max
algorithm for obtaining stationary points for the aforementioned problem.
However, its termination criterion and notion of stationarity are
significantly different from what is being considered in this chapter
and, hence, we choose not to include the algorithm of \citep{Lu2019}
in our benchmarks.

\subsection{Discussion of the Results}

We see that the smoothing method in this chapter are competitive against
other modern solvers and that they especially perform well when the
curvature ratio $M/m$ is large. Additionally, we see that the method
scales well across problem dimensions and parameters. 

\section{Conclusion and Additional Comments}

In this chapter, we presented a smoothing method for finding approximate
stationary points of a class of min-max NCO problems. The method consists
of applying the accelerated method of \prettyref{chap:unconstr_nco}
to a smooth approximation of the original nonsmooth min-max problem.
We then established an ${\cal O}(\delta^{-3})$ iteration complexity
bound for finding $\delta$-directional stationary points and an ${\cal O}(\rho_{x}^{-2}\rho_{y}^{-1/2})$
iteration bound for finding $(\rho_{x},\rho_{y})$-primal-dual stationary
points. Additionally, we combined our developments with those in \prettyref{sec:qp_aipp}
to present a quadratic penalty smoothing method for finding approximate
stationary points of a linearly-constrained variant of the original
class of min-max NCO problems. We then established a ${\cal O}(\rho_{x}^{-2}[\rho_{y}^{-1/2}+\eta^{-1}])$
iteration complexity bound for finding $(\rho_{x},\rho_{y})$-primal-dual
stationary points that were $\eta$ feasible, i.e. the points $\bar{x}$
satisfy $\|{\cal A}\bar{x}-b\|\leq\eta$ for a particular linear constraint
${\cal A}x=b$.

The next chapter uses a framework similar to the one in \prettyref{chap:unconstr_nco}
to develop methods for finding stationary points of a class of spectral
NCO problems.

\subsection*{Additional Comments}

We now give a few additional comments about the results in this chapter. 

First, recall that the main idea of the AIPP.SM is to call the AIPPM
of \prettyref{chap:unconstr_nco} to obtain a pair satisfying \eqref{eq:approx_smoothed},
or equivalently\footnote{See \prettyref{lem:compl_approx2} with $f=p_{\xi}$.},
\begin{equation}
\inf_{\|d\|\leq1}(\hat{p}_{\xi})'(x;d)\geq-\rho.\label{eq:gen_aipp_term}
\end{equation}
Moreover, using \prettyref{prop:sp_aipp_facts} with $(\rho_{x},\rho_{y})=(\rho,D_{y}/\xi)$,
it straightforward to see that the number of oracle calls, in terms
of $(\xi,\rho)$, is ${\cal O}(\rho^{-2}\xi^{1/2})$, which reduces
to ${\cal O}(\rho^{-2.5})$ if $\xi$ is chosen so as to satisfy $\xi=\Theta(\rho^{-1})$.
The latter complexity bound improves upon the one obtained for an
algorithm in \citep{Nouiehed2019} which obtains a point $x$ satisfying
\eqref{eq:gen_aipp_term} with $\xi=\Theta(\rho^{-1})$ in ${\cal O}(\rho^{-3})$
oracle calls.

Second, similar to \prettyref{chap:unconstr_nco}, we neither assume
that the set $X$ in \ref{asmp:mco_a1} is bounded nor that the min-max
NCO problem \ref{prb:eq:min_max_co} has an optimal solution. Also,
both the AIPP.SM and AIP.QP.SM only require that their starting point
$x_{0}$ be in $X$ and the AIP.QP.SM, in particular, makes no assumption
about the feasibility of $x_{0}$.

\subsection*{Future Work}

It is worth investigating whether complexity results for the AIPP.SM
can be derived for the case where $Y$ is unbounded or for the case
in which assumption \ref{asmp:mco_a2} is relaxed to the condition
that there exists $m_{y}>0$ such that $-\Phi(x,\cdot)$ is $m_{y}$-weakly
convex for every $x\in X$. It would also be interesting to see if
the notions of stationary points in \prettyref{sec:minmax_prelim_asmp}
are related to first-order stationary points\footnote{See, for example, \citep[Chapter 3]{Luo1996}.}
of the related mathematical program with equilibrium constraints:
\[
\min_{(x,y)\in X\times Y}\left\{ \Phi(x,y)+h(y):0\in\pt[-\Phi(\cdot,y)](x)\right\} .
\]
Finally, it would be worth investigating if a complexity as in \prettyref{prop:dd_aipp_facts}
and \prettyref{prop:sp_aipp_facts} can still be obtained if the exact
proximal oracle for $\Phi(x,\cdot)$ in \prettyref{eq:prox_oracles}
is replaced with an inexact one. 

\newpage{}

\chapter{Spectral Composite Optimization}

\label{chap:spectral}

Over the past decade, there has been a tremendous interest \citep{Wen2020,Candes2013,Lanza2019,Mazumder2018,Loh2015,Greenewald2019}
in developing iterative optimization algorithms for solving large-scale
matrix NCO problems. Moreover, a large majority of the NCO problems
in these works are such that the composite term $h$ is a function
of the singular values of its inputs and the smooth term $f$ can
be decomposed as $f_{1}+f_{2}$ where $f_{2}$ is also a function
of the singular values of its input. In this sense, these problems
admit a sort of \textbf{spectral} decomposition.

Our main goal in this chapter is to describe and establish the iteration
complexity of two efficient inexact composite gradient (ICG) methods
for finding approximate stationary points of the spectral NCO (SNCO)
problem
\begin{equation}
\min_{U\in\r^{m\times n}}\left\{ \phi(U):=f_{1}(U)+(f_{2}^{{\cal V}}\circ\sigma)(U)+(h^{{\cal V}}\circ\sigma)(U)\right\} ,\tag{\ensuremath{{\cal SNCO}}}\label{prb:eq:snco}
\end{equation}
where, denoting $r=\min\{m,m\}$, the function $\sigma:\r^{m\times n}\mapsto\r^{r}$
maps a matrix to its singular value vector in nonincreasing order
of magnitude, $h^{{\cal V}}\in\cConv Z$ for some nonempty convex
set $Z\subseteq\r^{r}$, $f_{1}\in{\cal C}_{m_{1},M_{1}}(\r^{m\times n})$
for some $(m_{1},M_{1})\in\r_{++}^{2}$, and $f_{2}^{{\cal V}}\in{\cal C}_{m_{2},M_{2}}(Z)$
for some $(m_{2},M_{2})\in\r_{++}^{2}$. Moreover, we also assume
that both $f_{2}^{{\cal V}}$ and $h^{{\cal V}}$ are absolutely symmetric
in their arguments, i.e. they do not depend on the ordering or the
sign of their arguments. 

A standard approach for finding stationary points of \ref{prb:eq:snco}
is to apply the CGM (see \prettyref{alg:cgm}), or an accelerated
version of it, to problem \ref{prb:eq:snco} where $f=f_{1}+f_{2}^{{\cal V}}\circ\sigma$
and $h=h^{{\cal V}}\circ\sigma$. The two ICG methods in this chapter
generalize this approach by exploiting the spectral structure underlying
the objective function. For example, one of the methods, called the
accelerated ICG (AICG) method inexactly solves a sequence of matrix
prox subproblems of the form 
\begin{align}
\min_{U\in\r^{m\times n}}\left\{ \lam\left[\left\langle \nabla f_{1}(\icgMatY{k-1}),U\right\rangle +(f_{2}^{{\cal V}}\circ\sigma)(U)+(h^{{\cal V}}\circ\sigma)(U)\right]+\frac{1}{2}\|U-\icgMatY{k-1}\|^{2}\right\} \label{eq:icg_mat_prox}
\end{align}
where $\lam>0$ and the point $\icgMatY{k-1}$ is the previous iterate.
It is shown (see \prettyref{subsec:spectral_exploit}) that the effort
of finding the required inexact solution $\icgMatY k$ of \eqref{eq:icg_mat_prox}
consists of computing one SVD and applying an ACG method to the related
vector prox subproblem
\begin{equation}
\min_{u\in\r^{r}}\left\{ \lam\left[f_{2}^{{\cal V}}(u)-\left\langle c_{k-1},u\right\rangle +h^{{\cal V}}(u)\right]+\frac{1}{2}\|u\|^{2}\right\} \label{eq:icg_vec_prox}
\end{equation}
where $r=\min\{m,n\}$ and $c_{k-1}=\sigma(\icgMatY{k-1}-\lam\nabla f_{1}(\icgMatY{k-1}))$.
Note that \eqref{eq:icg_vec_prox} is a problem over the vector space
$\r^{r}$, and hence, has significantly fewer dimensions than \eqref{eq:icg_mat_prox}
which is a problem over the matrix space $\r^{m\times n}$. The other
ICG method, called the doubly accelerated ICG (D.AICG) method, solves
a similar prox subproblem as in \eqref{eq:icg_mat_prox} but with
$\icgMatY{k-1}$ selected in an accelerated manner (and hence its
qualifier of ``doubly accelerated'') and some additional mild assumptions. 

Throughout our presentation, it is assumed that efficient oracles
for evaluating the quantities $f_{1}(U)$, $f_{2}^{{\cal V}}(u)$,
$\nabla f_{1}(U)$, $\nabla f_{2}^{{\cal V}}(u)$, and $h^{{\cal V}}(u)$
and for obtaining exact solutions of the subproblems
\[
\min_{u\in\r^{r}}\left\{ \lam h^{{\cal V}}(u)+\frac{1}{2}\|u-z_{0}\|^{2}\right\} ,
\]
for any $z_{0}\in\r^{r}$ and $\lam>0$, are available. Moreover,
we define an \textbf{oracle call} to be a collection of the above
oracles of size ${\cal O}(1)$ where each of them appears at least
once. 

Given $\hat{\rho}>0$ and a suitable choice of $\lam$, the main result
of this chapter shows that both of the ICG methods, started from any
point $Z_{0}\in Z$, obtain a pair $(\hat{Z},\hat{V})$ satisfying
the approximate stationarity condition
\begin{equation}
\hat{V}\in\nabla f_{1}(\hat{Z})+\nabla\left(f_{2}^{{\cal V}}\circ\sigma\right)(\hat{Z})+\pt\left(h^{{\cal V}}\circ\sigma\right)(\hat{Z}),\quad\|\hat{V}\|\leq\hat{\rho}\label{eq:approx_snco}
\end{equation}
in ${\cal O}(\hat{\rho}^{-2})$ oracle calls. When $f_{1}$ and $f_{2}^{{\cal V}}$
are convex, it is shown that the D.AICGM obtains a pair $(\hat{Z},\hat{V})$
satisfying in ${\cal O}(\hat{\rho}^{-2/3})$ oracle calls.

It is worth mentioning that the AICG method (AICGM) can be viewed
an inexact version of the CGM applied to \ref{prb:eq:snco}, which
solves a sequence of subproblems 
\begin{equation}
\min_{U\in\r^{m\times n}}\ \left\{ \lam\left[\left\langle \nabla\left[f_{1}+f_{2}^{{\cal V}}\circ\sigma\right](\icgMatY{k-1}),U\right\rangle +(h^{{\cal V}}\circ\sigma)(U)\right]+\frac{1}{2}\|U-\icgMatY{k-1}\|^{2}\right\} ,\label{eq:ecg_prox}
\end{equation}
where $\lam>0$ and the point $\icgMatY{k-1}$ is the previous iterate.
Similarly, the D.AICG method (D.AICGM) can be viewed as an inexact
version of a monotone ACGM, which also solves a sequence of subproblems
\eqref{eq:ecg_prox} but with $\icgMatY{k-1}$ chosen in an accelerated
manner.

For high-dimensional instances of \ref{prb:eq:snco} where $\min\{m,n\}$
is large, and hence, SVDs are expensive to compute, it will be shown
that the larger the Lipschitz constant of $\nabla f_{2}^{{\cal V}}$
is, the better the performance of the ICG methods is compared to that
of their exact counterparts. This is due to the following facts: (i)
solving \eqref{eq:ecg_prox} or \eqref{eq:icg_mat_prox} involves
a single SVD computation; (ii) even though \eqref{eq:ecg_prox} requires
fewer resolvent evaluations to solve than \eqref{eq:icg_mat_prox},
the cost of solving these subproblems is comparable due to the fact
that the aforementioned SVD is the bottleneck step; and (iii) the
larger the Lipschitz constant of $\nabla f_{2}^{{\cal V}}$, is the
smaller the stepsize $\lam$ in \eqref{eq:ecg_prox} must be, and
hence, the more subproblems of form \eqref{eq:ecg_prox} need to be
solved during the execution of the exact counterparts.

The content of this chapter is based on paper \citep{Kong2020} (joint
work with Renato D.C. Monteiro) and several passages may be taken
verbatim from it. To the best of our knowledge, paper \citep{Kong2020}
is the first one to present ICG methods that exploit both the spectral
and composite structure in \ref{prb:eq:snco}.

\subsection*{Organization}

This chapter contains seven sections. The first one gives some preliminary
references and discusses our notion of a stationary point given in
\eqref{eq:approx_snco}. The second one presents some specialized
subroutines that are used in the ICG methods. The third one presents
the AICGM and its iteration complexity. The fourth one presents the
D.AICGM and its iteration complexity. The fifth one presents an ACG
variant that exploits the spectral structure underlying the subproblems,
i.e. \eqref{eq:icg_mat_prox}, that each of the ICG methods solve.
The sixth one presents some numerical experiments. The last one gives
a conclusion and some closing comments.

\section{Preliminaries}

This subsection describes the general problem that the ICG methods
solve and outlines their general structure.

The ICG methods consider the NCO problem 
\begin{equation}
\phi_{*}=\min_{u\in{\cal Z}}\left[\phi(u):=f_{1}(u)+f_{2}(u)+h(u)\right]\tag{\ensuremath{{\cal NCO}_{2}}}\label{prb:eq:nco2}
\end{equation}
where ${\cal Z}$ is an finite dimensional inner product space and
the functions $f_{1},f_{2},$ and $h$ are assumed to satisfy the
following assumptions: 

\stepcounter{assumption}
\begin{enumerate}
\item \label{asmp:snco1}$h\in\cConv(Z)$ for some nonempty convex set $Z\subseteq{\cal Z}$; 
\item \label{asmp:snco2}$f_{1}\in{\cal C}_{m_{1},M_{1}}(Z)$ and $f_{2}\in{\cal C}_{m_{2},M_{2}}(Z)$
for some $(m_{1},M_{1})\in\r^{2}$ and $(m_{2},M_{2})\in\r^{2}$;
\item \label{asmp:snco3}$\phi_{*}>-\infty$.
\end{enumerate}
We now make a few remarks about \ref{prb:eq:nco2} and the above assumptions.
First, \ref{prb:eq:snco} is an instance of \ref{prb:eq:nco2} in
which $f_{2}=f_{2}^{{\cal V}}$ and $h=h^{{\cal V}}$, and hence,
any results developed in this section immediately apply for \ref{prb:eq:snco}.
Second, it is well-known that a necessary condition for $z^{*}$ to
be a local minimum of \ref{prb:eq:snco} is that $z^{*}$ be a stationary
point of $\phi$, i.e. $0\in\nabla f_{1}(z^{*})+\nabla f_{2}(z^{*})+\partial h(z^{*})$.

In view of the above remarks, our goal is to find an approximate stationary
point $(\hat{z},\hat{v})$ of \ref{prb:eq:nco2} in the following
sense: given $\hat{\rho}>0$, find a pair $(\hat{z},\hat{v})$ that
satisfies
\begin{equation}
\hat{v}\in\nabla f_{1}(\hat{z})+\nabla f_{2}(\hat{z})+\pt h(\hat{z}),\quad\|\hat{v}\|\leq\hat{\rho}.\label{eq:rho_approx_nco2}
\end{equation}
For ease of future reference, let us state the problem of finding
this pair in \prettyref{prb:approx_nco2}.

\begin{mdframed}
\mdprbcaption{Find an approximate stationary point of ${\cal NCO}_2$}{prb:approx_nco2}
Given $\hat{\rho} > 0$, find a pair $(\hat{z}, \hat{v}) \in Z \times {\cal Z}$ satisfying condition \eqref{eq:rho_approx_nco2}.
\end{mdframed}

We now outline the ICG methods. Given a starting point $\icgY 0\in Z$
and a special stepsize $\lam>0$, each method continually calls an
ACG variant, i.e. based on \prettyref{alg:acgm}, to find an approximate
solution of a prox-linear form of \ref{prb:eq:nco2}. More specifically,
each ACG call is used to tentatively find an inexact solution of 
\begin{equation}
\min_{u\in{\cal Z}}\left\{ \lam\left[\ell_{f_{1}}(u;w)+f_{2}(u)+h(u)\right]+\frac{1}{2}\|u-w\|^{2}\right\} ,\label{eq:icg_subprb}
\end{equation}
for some reference point $w$. For the AICGM, the point $w$ is $\icgY 0$
for the first ACG call and is the last obtained point for the other
ACG calls. For the D.AICGM, the point $w$ is chosen in an accelerated
manner. From the output of the $k^{{\rm th}}$ ACG call, a refined
pair $(\hat{z},\hat{v})=(\hat{z}_{k},\hat{v}_{k})$ is generated which:
(i) always satisfies the inclusion of \eqref{eq:rho_approx_nco2};
and (ii) is such that $\min_{i\leq k}\|\hat{v}_{i}\|\to0$ as $k\to\infty$. 

The next section details the inexactness criterion considered by the
ACG variant as well as how the refined pair $(\hat{z},\hat{v})$ is
generated. Before proceeding, we introduce the function
\[
L_{\Psi}(u,z):=\begin{cases}
\dfrac{\|\nabla\Psi(u)-\nabla\Psi(z)\|}{\|u-z\|}, & u\neq z,\\
0, & u=z,
\end{cases}\quad\forall(u,z)\in Z,
\]
for any differentiable function $\Psi$ on $Z$, and the shorthand
notation 
\begin{equation}
\begin{gathered}M_{i}^{+}:=\max\{0,M_{i}\},\quad m_{i}^{+}:=\max\{0,m_{i}\},\quad L_{i}:=\max\left\{ m_{i}^{+},M_{i}^{+}\right\} \\
L_{i}(u,z)=L_{f_{i}}(u,z)\quad\forall u,z\in Z,
\end{gathered}
\label{eq:mMi_def}
\end{equation}
for $i\in\{1,2\}$, to keep the presentation of future results concise.

\section{Specialized Refinement and ACG Procedures}

\label{sec:icg_prelim}

Recall from the beginning of this chapter that our interest is in
solving \ref{prb:eq:snco} by repeated solving a sequence of prox
subproblems as in \eqref{eq:icg_mat_prox}. This subsection presents
some background material regarding \eqref{eq:icg_mat_prox}. 

Consider the NCO problem 
\begin{equation}
\min_{u\in{\cal Z}}\left\{ \psi(u)=\psi_{s}(u)+\psi_{n}(u)\right\} ,\label{eq:acg_motivating_prb}
\end{equation}
where ${\cal Z}$ is a finite dimensional inner product space, $\psi_{n}\in\cConv(Z)$,
and $\psi_{s}\in{\cal C}_{m,L}(Z)$ for some $(m,L)\in\r\times\r_{++}$.
Clearly, problem \eqref{eq:icg_subprb} and \eqref{eq:icg_mat_prox}
are special cases of \eqref{eq:acg_motivating_prb}, and hence any
definition or result that is stated in the context of \eqref{eq:acg_motivating_prb}
applies to \eqref{eq:icg_subprb} and/or \eqref{eq:icg_mat_prox}. 

We now discuss the inexactness criterion under which the subproblems
\eqref{eq:icg_mat_prox} are solved. The criterion is described in
the context of \eqref{eq:acg_motivating_prb} as follows.

\begin{mdframed}\textbf{Problem} ${\cal A}:$ Given $(\mu,\theta)\in\r_{++}\times$
and $z_{0}\in\mathcal{Z}$, find $(z,v,\varepsilon)\in Z\times\mathcal{Z}\times\mathbb{R}_{+}$
such that 
\begin{equation}
v\in\partial_{\varepsilon}\left(\psi-\frac{\mu}{2}\|\cdot-z\|^{2}\right)(z),\quad\|v\|^{2}+2\varepsilon\le\theta^{2}\|z-z_{0}\|^{2}.\label{eq:cvx_inexact}
\end{equation}
\end{mdframed}

Some remarks about the above problem are in order. First, if $(z,v,\varepsilon)$
solves Problem~${\cal A}$ with $\theta=0$, then $(v,\varepsilon)=(0,0)$,
and $z$ is an exact solution of \eqref{eq:acg_motivating_prb}. Hence,
the output $(z,v,\varepsilon)$ of Problem~${\cal A}$ can be viewed
as an inexact solution of \eqref{eq:acg_motivating_prb} when $\theta\in\r_{++}$.
Second, the input $\acgX 0$ is arbitrary for the purpose of this
section. However, the two methods described in the next two sections
for solving \ref{prb:eq:nco2} repeatedly solve \eqref{eq:icg_mat_prox}
according to Problem~${\cal A}$ with the input $z_{0}$ at the $k^{{\rm th}}$
iteration determined by the iterates generated at the $(k-1)^{{\rm th}}$
iteration. Third, defining the function 
\begin{equation}
\Delta_{\mu}(u;z,v):=\psi(z)-\psi(u)-\inner v{u-z}+\frac{\mu}{2}\|u-z\|^{2}\quad\forall u,z\in Z,\label{eq:Delta_def}
\end{equation}
another way to express the inclusion in \eqref{eq:cvx_inexact} is
$\Delta_{\mu}(u;z,v)\leq\varepsilon$ for every $u\in Z$. Finally,
the ACG variant presented later in this section will be shown to solve
Problem~${\cal A}$ when $\psi_{s}\in{\cal F}_{\mu}(Z)$. Moreover,
it solves a weaker version of Problem~${\cal A}$ involving $\Delta_{\mu}$
(see Problem~${\cal B}$ later on) whenever $\psi_{s}\notin{\cal F}_{\mu}(Z)$
and as long as some key inequalities are satisfied during its execution.

A technical issue in our analysis in this chapter lies in the ability
of refining the output of Problem~${\cal A}$ to an point $(\hat{z},\hat{v})$
satisfying the inclusion in \eqref{eq:rho_approx_nco2}, in which
$\|\hat{v}\|$ is nicely bounded. The follow two results establish
a way to obtain such a point. 

The first result presents some properties of a composite gradient
step made on \eqref{eq:acg_motivating_prb}.
\begin{lem}
\label{lem:gen_refine}Let a quadruple $(z_{0},z,v,\varepsilon)\in{\cal Z}\times Z\times{\cal Z}\in\r_{+}$
and functions $\psi_{n}\in\cConv(Z)$ and $\psi_{s}\in{\cal C}_{\mu,L}(Z)$
for some $(mu,L)\in\r\times\r_{++}$ be given. Moreover, let $\psi=\psi_{s}+\psi_{n}$,
the function $\Delta_{\mu}(\cdot;\cdot,\cdot)$ be as in \eqref{eq:Delta_def},
and consider the pair $(\hat{z},v_{r})$ given by
\begin{equation}
\begin{aligned}\hat{z} & :=\argmin_{u\in{\cal Z}}\left\{ \ell_{\psi_{s}}(u;z)-\left\langle v,u\right\rangle +\frac{L}{2}\|u-z\|^{2}+\psi_{n}(u)\right\} ,\\
v_{r} & :=v+L(z-\hat{z})+\nabla\psi_{s}(\hat{z})-\nabla\psi_{s}(z),
\end{aligned}
\label{eq:refine_def}
\end{equation}
Then, the following statements hold:
\begin{itemize}
\item[(a)] $v_{r}\in\nabla\psi_{s}(\hat{z})+\pt\psi_{n}(\hat{z})$; 
\item[(b)] for every $s\in Z$ we have $\Delta_{\mu}(u;z,v)\geq0$ and, in particular,
\begin{equation}
\Delta_{\mu}(\hat{z};z,v)\geq\frac{L}{2}\|\hat{z}-z\|^{2};\label{eq:spec_Delta_lbd}
\end{equation}
\item[(c)] if $\Delta_{\mu}(\hat{z};z,v)\le\varepsilon$ and $(z,v,\varepsilon)$
satisfy the inequality in \eqref{eq:cvx_inexact}, then 
\begin{equation}
\|v_{r}\|\leq\theta\left[1+\frac{L+L_{\psi_{s}}(z,\hat{z})}{\sqrt{L}}\right]\|z-z_{0}\|;\label{eq:vr_acg_bd}
\end{equation}
\item[(d)] if $(z,v,\varepsilon)$ solves Problem ${\cal A}$, then $\Delta_{\mu}(u;z,v)\le\varepsilon$
for every $u\in Z$, and, as a consequence, bound \eqref{eq:vr_acg_bd}
holds.
\end{itemize}
\end{lem}

\begin{proof}
(a) The optimality condition of $\hat{z}$ is 
\[
0\in\nabla\psi_{s}(z)-v+L(\hat{z}-z)+\pt\psi_{n}(\hat{z})
\]
which, together with the definition of $v_{r}$, yields the desired
inclusion.

(b) The fact that $\Delta_{\mu}(u;z,v)\geq0$ for every $u\in Z$
follows from the optimality of $\hat{z}$ and the fact that $\psi_{s}\leq\ell_{\psi_{s}}(\cdot;z)+L\|\cdot-z\|^{2}/2$.
The bound \eqref{eq:spec_Delta_lbd} follows from \prettyref{prop:cgm_ext_vartn}(c)
with $\lam=1/L$ and $(z,z^{-})=(\hat{z},z)$.

(c) Using the assumption that $\Delta_{\mu}(\hat{z};z,v)\leq\varepsilon$,
part (b), and the inequality in \eqref{eq:cvx_inexact}, we have that
\begin{equation}
\|z-\hat{z}\|\leq\sqrt{\frac{2\Delta_{\mu}(\hat{z};z,v)}{L}}\leq\sqrt{\frac{2\varepsilon}{L}}\leq\frac{\theta}{\sqrt{L}}\|z-\acgX 0\|.\label{eq:z_zr_bd}
\end{equation}
Using the triangle inequality, the definitions of $L(\cdot,\cdot)$
and $v_{r}$, \eqref{eq:z_zr_bd}, and the inequality in \eqref{eq:cvx_inexact},
we conclude that 
\begin{align*}
\|v_{r}\| & =\|v+L_{\lam}(z-\hat{z})+\nabla\psi_{s}(\hat{z})-\nabla\psi_{s}(z)\|\\
 & \leq\|v\|+\left[L+L_{\tilde{f}}(z,\hat{z})\right]\|z-\hat{z}\|\\
 & \leq\theta\left[1+\frac{L+L_{\tilde{f}}(z,\hat{z})}{\sqrt{L}}\right]\|z-\acgX 0\|.
\end{align*}
(d) The fact that $\Delta_{\mu}(u;z,v)\leq\varepsilon$ for every
$u\in Z$ follows immediately from the inclusion in \eqref{eq:cvx_inexact}
and the definition of $\Delta_{\mu}$ in \eqref{eq:Delta_def}. The
fact that \eqref{eq:vr_acg_bd} holds now follows from part (c). 
\end{proof}
The next result specializes the above lemma to the context of \ref{prb:eq:nco2}
and describes the desired pair $(\hat{z},\hat{v})$.
\begin{prop}
\label{prop:spec_refine}Let functions $f_{1}$, $f_{2}$, and $h$
functions satisfying assumptions \ref{asmp:snco1}--\ref{asmp:snco2}
and a quadruple $(z_{0},z,v)\in Z\times Z\times{\cal Z}\in\r_{+}$
be given. Moreover, let $\Delta_{\mu}(\cdot;\cdot,\cdot)$ and $(\hat{z},v_{r})$
be as in \prettyref{lem:gen_refine} with 
\[
\psi_{s}=\lam\left[\ell_{f_{1}}(\cdot;z_{0})+f_{2}\right]+\frac{1}{2}\|\cdot-z_{0}\|^{2},\quad\psi_{n}=\lam h,\quad L=\lam M_{2}^{+}+1,
\]
and define
\begin{align}
\hat{v} & :=\frac{1}{\lam}\left(v_{r}+z_{0}-\hat{z}\right)+\nabla f_{1}(\hat{z})-\nabla f_{1}(\acgX 0),\nonumber \\
C_{\lam}(u,z) & :=\frac{2+\lam\left[M_{2}^{+}+L_{1}(u,z)+L_{2}(u,z)\right]}{\sqrt{1+\lam M_{2}^{+}}},\label{eq:C_lam_fn_def}
\end{align}
for every $u,z\in{\cal Z}$. Then, the following statements hold:
\begin{itemize}
\item[(a)] $\hat{v}\in\nabla f_{1}(\hat{z})+\nabla f_{2}(\hat{z})+\pt h(\hat{z})$;
\item[(b)] if $\Delta_{\mu}(\hat{z};z,v)\le\varepsilon$ and $(z,v,\varepsilon)$
satisfy the inequality in \eqref{eq:cvx_inexact}, then it holds that
\begin{equation}
\|\hat{v}\|\leq\left[L_{1}(z_{0},z)+\frac{2+\theta C_{\lam}(z,\hat{z})}{\lam}\right]\|z-z_{0}\|;\label{eq:vhat_acg_bd}
\end{equation}
\end{itemize}
\end{prop}

\begin{proof}
(a) It follows from \prettyref{lem:gen_refine}(a) and the definition
of $\hat{v}$ that 
\begin{align*}
\hat{v} & =\frac{1}{\lam}\left(v_{r}+z_{0}-\hat{z}\right)+\nabla f_{1}(\hat{z})-\nabla f_{1}(\acgX 0)\\
 & \in\frac{1}{\lam}\left[\nabla\psi_{s}(\hat{z})+\pt\psi_{n}(\hat{z})\right]+\frac{1}{\lam}\left(v_{r}+z_{0}-\hat{z}\right)+\nabla f_{1}(\hat{z})-\nabla f_{1}(\acgX 0)\\
 & =\nabla f_{1}(\hat{z})+\nabla f_{2}(\hat{z})+\pt h(\hat{z}).
\end{align*}

(b) It follows from \eqref{eq:z_zr_bd} with $L:=\lam M_{2}^{+}+1$,
the triangle inequality that 
\begin{align*}
 & \frac{1}{\lam}\|z_{0}-\hat{z}\|+\|\nabla f_{1}(z_{0})-\nabla f_{1}(\hat{z})\|\\
 & \leq\frac{1}{\lam}\left[1+\lam L_{1}(z,\hat{z})\right]\|z-\hat{z}\|+\left[1+\lam L_{1}(z_{0},z)\right]\|z-z_{0}\|\\
 & \leq\frac{1}{\lam}\left(1+\lam L_{1}(z_{0},z)+\theta\left[\frac{1+\lam L_{1}(z,\hat{z})}{\sqrt{1+\lam M_{2}^{+}}}\right]\right)\|z-z_{0}\|.
\end{align*}
Using the above bound, \prettyref{lem:gen_refine}(c) with $L=\lam M_{2}^{+}+1$
and $L_{\psi_{s}}(\cdot,\cdot)=\lam L_{2}(\cdot,\cdot)+1$, the definition
of $C_{\lam}(\cdot,\cdot)$, and the fact that $\theta\leq1$, we
conclude that
\begin{align*}
\|\hat{v}\| & \leq\frac{1}{\lam}\|v_{r}\|+\frac{1}{\lam}\|z_{0}-\hat{z}\|+\|\nabla f_{1}(z_{0})-\nabla f_{1}(\hat{z})\|\\
 & \leq\frac{1}{\lam}\left(1+\theta+\lam L_{1}(z_{0},z)+\theta\left[\frac{2+\lam M_{2}^{+}+\lam L_{1}(z,\hat{z})+\lam L_{2}(z,\hat{z})}{\sqrt{1+\lam M_{2}^{+}}}\right]\right)\|z-z_{0}\|\\
 & \leq\left[L_{1}(z_{0},z)+\frac{2+\theta C_{\lam}(z,\hat{z})}{\lam}\right]\|z-z_{0}\|.
\end{align*}
\end{proof}
We make a few remarks about \prettyref{prop:spec_refine}. First,
it follows from (a) that $(\hat{z},\hat{v})$ satisfies the inclusion
in \eqref{eq:rho_approx_nco2}. Second, it follows from (a) and (c)
that if $\theta=0$, then $(\hat{z},\hat{v})=(0,0)$, and hence $\hat{z}$
is an exact stationary point of \ref{prb:eq:nco2}. In general, \eqref{eq:vhat_acg_bd}
implies that the residual $\|\hat{v}\|$ is directly proportional
to $\|z-z_{0}\|$, and hence, becomes smaller as this quantity approaches
\textit{zero}. 

For the sake of future referencing, we state the specialized refinement
procedure (SRP) for generating $(\hat{z},\hat{v})$ in \prettyref{alg:sref}.

\begin{mdframed}
\mdalgcaption{SR Procedure}{alg:sref}
\begin{smalgorithmic}
	\Require{$(m_1, M_1, m_2,M_2) \in \r^4, \quad h \in \cConv({\cal Z}), \enskip f_1 \in {\cal C}_{m_1, M_1}(Z), \enskip f_2 \in {\cal C}_{m_2, M_2}(Z), \enskip (z,z_0,v) \in Z\times{\cal Z}\times{\cal Z}, \enskip  \lam > 0$;}
	\Initialize{$\psi_s \Lleftarrow \lam\left[\ell_{f_{1}}(\cdot;z_{0})+f_{2}\right]+\frac{1}{2}\|\cdot-z_{0}\|^{2}, \enskip\psi_n \Lleftarrow \lam h, \enskip L \gets \lam M_2^{+} + 1$ (see \eqref{eq:mMi_def});}
	\vspace*{.5em}
	\Procedure{SREF}{$f_1, f_2, h, z, z_0, v, M_2, \lam$}
		\StateEq{$\hat{z} \gets \argmin_{u\in{\cal Z}}\left\{ \ell_{\psi_{s}}(u;z)-\left\langle v,u\right\rangle +\frac{L}{2}\|u-z\|^{2}+\psi_{n}(u)\right\}$}
		\StateEq{$v_{r} \gets v + L(z-\hat{z}) + \nabla\psi_{s}(\hat{z}) - \nabla\psi_{s}(z)$}
		\StateEq{$\hat{v} \gets \frac{1}{\lam}\left(v_{r}+z_{0}-\hat{z}\right)+\nabla f_{1}(\hat{z})-\nabla f_{1}(\acgX 0)$}
		\StateEq{\Return{$(\hat{z}, \hat{v})$}}
	\EndProcedure
\end{smalgorithmic}
\end{mdframed}

Inequalities \eqref{eq:vr_acg_bd} and \eqref{eq:vhat_acg_bd} play
an important technical role in the complexity analysis of the two
prox-type methods of the next two sections. Sufficient conditions
for their validity are provided in \prettyref{lem:gen_refine}(c)--(d),
with (c) being the weaker one, in view of (d). When $\psi_{s}\in{\cal F}_{\mu}(Z)$,
it is shown that every iterate of our proposed ACG variant always
satisfies the inclusion in \eqref{eq:cvx_inexact}, and hence, verifying
the validity of the sufficient condition in (c) amounts to simply
checking whether the inequality in \eqref{eq:cvx_inexact} holds.
When $\psi_{s}\notin{\cal F}_{\mu}(Z)$, verification of the inclusion
in \eqref{eq:cvx_inexact}, and hence the sufficient condition in
(d), is generally not possible, while the one in (c) is. This is a
major advantage of the sufficient condition in (c), which is exploited
in this chapter towards the development of adaptive prox-type methods
which attempt to approximately solve \eqref{eq:acg_motivating_prb}
when $\psi_{s}\notin{\cal F}_{\mu}(Z)$.

To ease future referencing, we state below the problem for finding
a triple $(z,v,\varepsilon)$ satisfying the sufficient condition
in \prettyref{lem:gen_refine}(c).

\begin{mdframed}\textbf{Problem} ${\cal B}:$ Given the same inputs
as in Problem~${\cal A}$, find $(z,v,\varepsilon)\in Z\times\mathcal{Z}\times\mathbb{R}_{+}$
satisfying the inequality in \eqref{eq:cvx_inexact} and 
\begin{equation}
\Delta_{\mu}(\hat{z};z,v)\leq\varepsilon,\label{eq:prb_B_Delta_ineq}
\end{equation}
where $\Delta_{\mu}(\cdot;\cdot,\cdot)$ is as in \eqref{eq:Delta_def}
and the point $\hat{z}$ is given by \eqref{eq:refine_def}.\end{mdframed} 

We now present the specialized ACG (S.ACG) method in \prettyref{alg:s_acgm},
which solves Problem~${\cal A}$ when $\psi_{s}\in{\cal F}_{\mu}(Z)$
and solves Problem~${\cal B}$ whenever two key inequalities are
always satisfied, one at every iteration and one at the end of its
execution. The termination status of the method is stored in the variable
$\pi_{S}$ which is \texttt{true} if the method solves Problem~${\cal B}$
and \texttt{false} otherwise.

\begin{mdframed}
\mdalgcaption{S.ACG Method}{alg:s_acgm}
\begin{smalgorithmic}
	\Require{$(\mu,L) \in \r_{++}^2, \enskip \psi_n \in \cConv({\cal Z}), \enskip \psi_s \in {\cal C}_{L}(Z), \enskip y_0 \in Z, \enskip \theta \in (0,1)$;}
	\Initialize{$\pi_S \gets {\tt true}, \enskip \psi \Lleftarrow \psi_s + \psi_n,$}
	\vspace*{.5em}
	\Procedure{S.ACG}{$\psi_s, \psi_n, y_0, \theta, \mu, L$}
	\For{$k=1,...$}
		\StateEq{$\lam_k \gets 1/L$}
		\StateEq{Generate $(A_k, y_k, \tilde{x}_{k-1}, \tilde{r}_k, \tilde{\eta}_k)$ according to \prettyref{alg:acgm}.}
		\StateStep{\algpart{1}\textbf{Check} the first failure point.}
		\If{$\frac{1}{1+\mu A_k}\|A_k \tilde{r}_k + y_k - y_0\|^2 + 2 A_k \tilde{\eta}_k \leq \|y_k - y_0\|^2$} \label{ln:s_acg_invar}
			\StateEq{$\pi_S \gets {\tt false}$}
			\StateEq{\Return{$(y_0, \infty, \infty, \pi_S)$}}
		\EndIf
		\If{$\|\tilde{r}_k\|^2 + 2\tilde{\eta}_k \leq \theta^2 \|y_k - y_0\|^2$} \label{ln:s_acgm_hpe}
			\StateEq{$\hat{y}_k \gets \argmin_{u\in{\cal Z}}\left\{ \ell_{\psi_{s}}(u;z)-\left\langle v,u\right\rangle +\frac{M}{2}\|u-z\|^{2}+\psi_{n}(u)\right\}$}
			\StateStep{\algpart{2}\textbf{Check} the second failure point.}
			\If{$\Delta_\mu(\hat{y}_k; y_k, \tilde{r}_k) > \tilde{\eta}_k$} \Comment{See \eqref{eq:Delta_def}}
				\StateEq{$\pi_S \gets {\tt false}$}
				\StateEq{\Return{$(y_0, \infty, \infty, \pi_S)$}}
			\Else \label{ln:s_acgm_Delta}
				\StateEq{\Return{$(y_k, r_k, \tilde{\eta}_k, \pi_S)$}}
			\EndIf
		\EndIf
	\EndFor
	\EndProcedure
\end{smalgorithmic}
\end{mdframed}

The next result presents the key properties of the S.ACG method (S.ACGM). 
\begin{prop}
\label{prop:s_acg_properties}The following properties hold about
the S.ACGM:
\begin{itemize}
\item[(a)] it stops in
\begin{equation}
{\cal O}\left(\left[1+\sqrt{\frac{L}{\mu}}\right]\log_{1}^{+}\left[LK_{\theta}(1+\mu K_{\theta})\right]\right)\label{eq:s_acg_total_compl}
\end{equation}
iterations, where $K_{\theta}=1+\sqrt{2}/\theta$;
\item[(b)] if its stops with a quadruple $(z,v,\varepsilon,\pi_{S})=(y_{k},\tilde{r}_{k},\tilde{\eta}_{k},\pi_{S})$
where $\pi_{S}=$ \texttt{\emph{true}}, then the triple $(z,v,\varepsilon)$
solves Problem~${\cal B}$;
\item[(c)] if $\psi_{S}\in{\cal F}_{\mu}(Z)$, then it always stops with a quadruple
$(z,v,\varepsilon,\pi_{S})=(y_{k},\tilde{r}_{k},\tilde{\eta}_{k},\pi_{S})$
where $\pi_{S}=$ \texttt{\emph{true}}, and the triple $(z,v,\varepsilon)$
solves Problem~${\cal A}$.
\end{itemize}
\end{prop}

\begin{proof}
(a) See \prettyref{app:oth_acgm_props}.

(b) Using the successful checks in \prettyref{ln:s_acgm_hpe} and
\prettyref{ln:s_acgm_Delta} of the method, it follows that the triple
$(z,v,\varepsilon)$ solves Problem~${\cal B}$.

(c) Using \prettyref{prop:acgm_vartn}(a)--(b) and the definition
of the approximate subdifferential, it follows that the method always
stops with $\pi_{S}=$ \texttt{true} when $\psi_{s}\in{\cal F}_{\mu}(Z)$.
On the other hand, \prettyref{prop:acgm_vartn}(a), the definition
of the approximate subdifferential, and the successful check in \prettyref{ln:s_acgm_hpe}
of the method imply that the triple $(z,v,\varepsilon)$ solves Problem~${\cal A}$.
\end{proof}
It is worth recalling that in the applications we consider, the cost
of the ACG call is small compared to SVD computation that is performed
before solving each subproblem as in \eqref{eq:icg_subprb}. Hence,
in the analysis that follows, we present complexity results related
to the number of subproblems solved rather than the total number of
ACG iterations. We do note, however, that the number of ACG iterations
per subproblem is finite in view of \prettyref{prop:s_acg_properties}(a).

\section{Accelerated Inexact Composite Gradient (AICG) Method}

\label{sec:aicg}

This section presents the static AICGM and its dynamic variant.

We first state the static AICGM in \prettyref{alg:aicg}, which uses
\prettyref{alg:sref} and \prettyref{alg:s_acgm} as subroutines.
Given $z_{0}\in Z$ and a special choice of $\lam>0$, its main idea
is to attempt to generate its $k^{{\rm th}}$ iterate by using the
S.ACGM to obtain the inexact update
\[
z_{k}\approx\min_{u\in{\cal Z}}\left\{ \lam\left[\ell_{f_{1}}(u;z_{k-1})+f_{2}(u)+h(u)\right]+\frac{1}{2}\|u-z_{k-1}\|^{2}\right\} .
\]
The iterate is then refined using the SRP in \prettyref{alg:sref}
and termination of the method occurs when either: (i) a refined iterate
solving \prettyref{prb:approx_nco2} is found; or (ii) a failure condition
has been triggered. The termination status of the method is store
in a variable $\pi_{S}$ which is \texttt{true} if the former scenario
occurs and \texttt{false} the latter scenario occurs.

\begin{mdframed}
\mdalgcaption{Static AICG Method}{alg:aicg}
\begin{smalgorithmic}
	\Require{$\hat{\rho} > 0, \enskip (m_1, M_1, m_2, M_2)\in\r^4, \enskip h \in \cConv(Z), \enskip f_1 \in {\cal C}_{m_1,M_1}(Z), \enskip f_2 \in {\cal C}_{m_2,M_2}(Z), \enskip (\lam,\theta) \in \r_{++}^2 \text{ s.t. } \lam M_1 + \theta^2 < 1/2, \enskip z_0 \in Z$;}
	\Initialize{$\mu \gets 1, \enskip L \gets \lam M_2^{+}+1 (\text{see } \eqref{eq:mMi_def}), \enskip \pi_S \gets {\tt true}$}
	\vspace*{.5em}
	\Procedure{St.AICG}{$f_1, f_2, h, z_0, \lam, \theta, M_2, \hat{\rho}$}
	\For{$k=1,...$}
		\StateStep{\algpart{1}\textbf{Attack} the $k^{\rm th}$ prox-linear subproblem.}
		\StateEq{$\psi_s^k \Lleftarrow \lam\left[\ell_{f_{1}}(\cdot;z_{k-1})+f_{2}\right]+\frac{1}{2}\|\cdot-z_{k-1}\|^{2}$}
			\StateEq{$(z_k, v_k, \varepsilon_k, \pi_k^{\rm acg}) \gets \text{S.ACG}(\psi_s^k, \lam h, z_{k-1}, \theta, \mu, L)$} \label{ln:aicg_s_acgm_call}
		\StateStep{\algpart{2}\textbf{Check} a special convexity condition.}
		\If{$\lnot (\pi_k^{\rm acg})$ \textbf{ or } $\Delta_\mu(z_{k-1};z_k,v_k) > \varepsilon_k$}  \Comment{See \eqref{eq:Delta_def}}
			\StateEq{$\pi_S \gets {\tt false}$}
			\StateEq{\Return{$(z_0, \infty, \pi_S)$}}
		\EndIf
		\StateStep{\algpart{3}\textbf{Check} the termination condition.}
		\StateEq{$(\hat z_k, \hat v_k) \gets \text{SREF}(f_1, f_2, h, z_k, z_{k-1}, v_k, M_2, \lam)$}
		\If{$\|\hat{v}_k\| \leq \hat{\rho}$} \label{ln:aicg_stop_cond}
			\StateEq{\Return{$(\hat{z}_k, \hat{v}_k, \pi_S)$}}
		\EndIf
	\EndFor
	\EndProcedure
\end{smalgorithmic}
\end{mdframed}

Some remarks about this method are in order. To ease the discussion,
let us refer to the ACG iterations performed in \prettyref{ln:aicg_s_acgm_call}
of the method as \textbf{inner iterations} and the iterations over
the indices $k$ as\emph{ }\textbf{outer iterations}. First, in view
of the requirement on $(\lam,\theta)$, if $M_{1}>0$ then $0<\lam<(1-2\theta^{2})/(2M_{1})$
whereas if $M_{1}\leq0$ then $0<\lam<\infty$. Second, it may fail
to obtain a pair satisfying \eqref{eq:rho_approx_nco2}, i.e. when
$\pi_{S}$ = \texttt{false}. In \prettyref{thm:aicg_compl}(c) below,
we state that a sufficient condition for the method to stop successfully
is that $f_{2}$ be convex. This property will be important when we
present the dynamic AICGM, which: (i) repeatedly calls the static
method; and (ii) incrementally transfers convexity from $f_{1}$ to
$f_{2}$ between each call until a termination where $\pi_{S}=$ \texttt{true}
is achieved.

The next result, whose proof is deferred to \prettyref{subsec:aicg_props},
summarizes some facts about the static AICGM. Before proceeding, we
first define some useful quantities. For and $\lam>0$ and $u,w\in{\cal Z}$,
define 
\begin{gather}
\widetilde{\ell}_{\phi}(u;w):=\ell_{f_{1}}(u;w)+f_{2}(u)+h(u),\quad\overline{C}_{\lam}:=\frac{1+\lam(M_{2}^{+}+L_{1}+L_{2})}{\sqrt{1+\lam M_{2}^{+}}}.\label{eq:ell_phi_C_bar_lam_def}
\end{gather}

\begin{thm}
\label{thm:aicg_compl}The following statements hold about the static
AICGM: 
\begin{itemize}
\item[(a)] it stops in 
\begin{equation}
{\cal O}\left(\left[\sqrt{\lam}L_{1}+\frac{1+\theta\overline{C}_{\lam}}{\sqrt{\lam}}\right]^{2}\left[\frac{\phi(z_{0})-\phi_{*}}{\hat{\rho}^{2}}\right]\right)\label{eq:aicg_outer_compl}
\end{equation}
outer iterations, where $\phi_{*}$ is as in \ref{asmp:snco3}; 
\item[(b)] if it stops with $\pi_{S}=$ \texttt{\emph{true}}, then the first
two arguments of its output triple $(\hat{z},\hat{v},\pi_{S})$ solve
\prettyref{prb:approx_nco2}; 
\item[(c)] if $f_{2}$ is convex, then it always stops with $\pi_{S}=$ \texttt{\emph{true}}. 
\end{itemize}
\end{thm}

We now make three remarks about the above results. First, if $\theta={\cal O}(1/\overline{C}_{\lam})$
then \eqref{eq:aicg_outer_compl} reduces to 
\begin{equation}
{\cal O}\left(\left[\sqrt{\lam}L_{1}+\frac{1}{\sqrt{\lam}}\right]^{2}\left[\frac{\phi(z_{0})-\phi_{*}}{\hat{\rho}^{2}}\right]\right).\label{eq:compl_aicg}
\end{equation}
Moreover, comparing the above complexity to the iteration complexity
of the CGM (see \prettyref{alg:cgm}), which is known \citep{Nesterov2013}
to solve \prettyref{prb:approx_nco2} in 
\begin{equation}
{\cal O}\left(\left[\sqrt{\lam}(L_{1}+L_{2})+\frac{1}{\sqrt{\lam}}\right]^{2}\left[\frac{\phi(z_{0})-\phi_{*}}{\hat{\rho}^{2}}\right]\right)\label{eq:compl_ecg}
\end{equation}
iterations, we see that \eqref{eq:compl_aicg} is smaller than \eqref{eq:compl_ecg}
in magnitude when $L_{2}$ is large. Second, \prettyref{thm:aicg_compl}(b)
shows that if the method stops with $\pi_{S}=$ \texttt{true}, regardless
of the convexity of $f_{2}$, then its output pair $(\hat{z},\hat{v})$
is always a solution of \prettyref{prb:approx_nco2}. Third, it is
shown in \prettyref{prop:aicg_v_hat_rate_alt}, that the quantities
$L_{1}$ and $\overline{C}_{\lam}$ in all the previous complexity
results can be replaced by their averaged counterparts in \eqref{eq:avg_def}.
As these averaged quantities only depend on $\{(\icgY i,\hat{z}_{i})\}_{i=1}^{k}$,
we can infer that the static AICG method adapts to the local geometry
of its input functions.

We now state the (dynamic) AICG variant in \prettyref{alg:dynamic_aicg},
which address the possibility of failure by repeatedly calling the
static AICGM.

\begin{mdframed}
\mdalgcaption{AICG Method}{alg:dynamic_aicg}
\begin{smalgorithmic}
	\Require{$\hat{\rho} > 0, \enskip (m_1, M_1, m_2, M_2)\in\r^4, \enskip h \in \cConv(Z), \enskip f_1 \in {\cal C}_{m_1,M_1}(Z), \enskip f_2 \in {\cal C}_{m_2,M_2}(Z), \enskip (\lam, \theta) \in \r_{++}^2 \text{ s.t. } \lam M_1 + \theta^2 < 1/2, \enskip z_0 \in Z, \enskip \xi_0 > 0$;}
	\Initialize{$\mu \gets 1, \enskip L \gets \lam M_2^{+}+1 (\text{see } \eqref{eq:mMi_def})$}
	\vspace*{.5em}
	\Procedure{AICG}{$f_1, f_2, h, z_0, \lam, \theta, M_2, \xi_1, \hat{\rho}$}
	\For{$k=1,...$}
		\StateStep{\algpart{1}\textbf{Call} the static AICGM with perturbed inputs.}
		\StateEq{$f_1^k \Lleftarrow f_1 - \frac{\xi_k}{2}\|\cdot\|^2$} 
		\StateEq{$f_2^k \Lleftarrow f_2 + \frac{\xi_k}{2}\|\cdot\|^2$}
		\StateEq{$(\hat{z}, \hat{v}, \pi_S) \gets \text{St.AICG}(f_1^k, f_2^k, h, z_0, \lam, \theta, M_2 + \xi_k, \hat{\rho})$}
		\StateStep{\algpart{2} Either \textbf{stop} with a solution or \textbf{increase} $\xi_k$ for the next AICG call.}
		\If{$\pi_S$}
			\StateEq{\Return{$(\hat z, \hat v)$}}
		\Else
			\StateEq{$\xi_{k+1} \gets 2 \xi_k$}
		\EndIf		
	\EndFor
	\EndProcedure
\end{smalgorithmic}
\end{mdframed}

Some remarks about the above method are in order. First, in view of
the requirement on $(\lam,\theta)$ and the fact that the upper curvature
of $f_{1}^{k}$ is monotonically decreasing in $k$, the parameter
$\lam$ does not need to be changed for each static AICG call. Second,
in view \prettyref{thm:aicg_compl}(c), every static AICG call always
terminates with success whenever $f_{2}^{k}$ is convex. As a consequence,
assumption \ref{asmp:snco2} implies that the total number of static
AICG calls is at most $\left\lceil \log(2m_{2}^{+}/\xi_{1})\right\rceil $.
Third, in view of the second remark and \prettyref{thm:aicg_compl}(b),
the methods always obtains a solution of \prettyref{prb:approx_nco2}
in a finite number of static AICG outer iterations. Finally, in view
of second remark again, the total number of static AICG outer iterations
is as in \prettyref{thm:aicg_compl}(a) but with: (i) an additional
multiplicative factor of $\left\lceil \log(2m_{2}^{+}/\xi_{0})\right\rceil $;
and (ii) the constants $m_{1}$ and $M_{2}$ replaced with $(m_{1}+2m_{2}^{+})$
and $(M_{2}+2m_{2}^{+})$, respectively. It is worth mentioning that
a more refined analysis, such as the one in \prettyref{sec:r_aipp},
can be applied in order to remove the factor of $\left\lceil \log(2m_{2}^{+}/\xi_{0})\right\rceil $
from the previously mentioned complexity.

\subsection{AICG Properties and Iteration Complexity}

\label{subsec:aicg_props}

This subsection establishes the key properties of the static AICGM
and gives the proof of \prettyref{thm:aicg_compl}. 

We first start with a technical lemma that describes the progress,
in terms of function value, between consecutive iterations. Its statement,
and the statement of subsequent results, will make use of the key
constants in \eqref{eq:mMi_def}.
\begin{lem}
\label{lem:aicg_resid_rate}Let $\{(\icgY i,\hat{z}_{i},\hat{v}_{i})\}_{i=1}^{k}$
be the collection of iterates generated by the static AICGM. For every
$i\geq1$, we have 
\begin{align}
\frac{1}{4\lam}\|\icgY{i-1}-\icgY i\|^{2} & \leq\phi(\icgY{i-1})-\widetilde{\ell}_{\phi}(\icgY i;\icgY{i-1})-\frac{M_{1}}{2}\|\icgY i-\icgY{i-1}\|^{2}\leq\phi(\icgY{i-1})-\phi(\icgY i),\label{eq:aicg_descent}
\end{align}
where $\widetilde{\ell}_{\phi}$ is as in \eqref{eq:ell_phi_C_bar_lam_def}. 
\end{lem}

\begin{proof}
Let $i\geq1$ be fixed, define 
\[
\mu:=1,\quad\psi_{s}:=\lam\left[\ell_{f_{1}}(\cdot;z_{i-1})+f_{2}\right]+\frac{1}{2}\|\cdot-z_{i-1}\|^{2},\quad\psi_{n}:=\lam h,
\]
and let $(\icgY i,v_{i},\varepsilon_{i},\pi_{i})$ be the output of
the $i^{{\rm th}}$ call to the S.ACG algorithm. Moreover, let $\Delta_{\mu}(\cdot;\cdot,\cdot)$
be as in \eqref{eq:Delta_def} with $(\psi_{s},\psi_{n})$ as above.
Using the definition of $\widetilde{\ell}_{\phi}$ and fact that $(z,v,\varepsilon)=(\icgY i,v_{i},\varepsilon_{i})$
solves Problem~${\cal B}$ in \prettyref{sec:icg_prelim}, we have
that 
\begin{align*}
\varepsilon_{i} & \geq\Delta_{1}(\icgY{i-1};\icgY i,v_{i})\\
 & =\lam\widetilde{\ell}_{\phi}(\icgY i;\icgY{i-1})-\lam\phi(\icgY{i-1})-\inner{v_{i}}{\icgY i-\icgY{i-1}}+\|\icgY i-\icgY{i-1}\|^{2}.
\end{align*}
Rearranging the above inequality and using assumption \ref{asmp:snco2},
the requirement on $(\lam,\theta)$ (in the AICGM), and the fact that
$\left\langle a,b\right\rangle \geq-\|a\|^{2}/2-\|b\|^{2}/2$ for
every $a,b\in{\cal Z}$ yields 
\begin{align*}
 & \lam\phi(\icgY{i-1})-\lam\widetilde{\ell}_{\phi}(\icgY i;\icgY{i-1})\geq\left\langle v_{i},\icgY{i-1}-\icgY i\right\rangle -\varepsilon_{i}+\|\icgY i-\icgY{i-1}\|^{2}\\
 & =\frac{1}{2}\|\icgY i-\icgY{i-1}\|^{2}-\frac{1}{2}\left(\|v_{i}\|^{2}+2\varepsilon_{i}\right)\geq\left(\frac{1-\sigma^{2}}{2}\right)\|\icgY i-\icgY{i-1}\|^{2}\\
 & =\frac{\lam M_{1}}{2}\|\icgY i-\icgY{i-1}\|^{2}+\left(\frac{1-\lam M_{1}-\sigma^{2}}{2}\right)\|\icgY i-\icgY{i-1}\|^{2}\\
 & =\frac{\lam M_{1}}{2}\|\icgY i-\icgY{i-1}\|^{2}+\frac{1}{4}\|\icgY i-\icgY{i-1}\|^{2}.
\end{align*}
Rearranging terms yields the first inequality of \eqref{eq:aicg_descent}.
The second inequality of \eqref{eq:aicg_descent} follows from the
first inequality, the fact that $\widetilde{\ell}_{\phi}(\icgY i;\icgY{i-1})+M_{1}\|\icgY i-\icgY{i-1}\|^{2}/2\geq\phi(\icgY i)$
from assumption \ref{asmp:snco2}, and the definition of $\widetilde{\ell}_{\phi}$. 
\end{proof}
The next results establish the rate at which the residual $\|\hat{v}_{i}\|$
tends to 0. 
\begin{lem}
\label{lem:p_norm_tech}Let $p>1$ be given. Then, for every $a,b\in\r^{k}$,
we have 
\[
\min_{1\leq i\leq k}\left\{ |a_{i}b_{i}|\right\} \leq k^{-p}\left\Vert a\right\Vert _{1}\left\Vert b\right\Vert _{1/(p-1)}.
\]
\end{lem}

\begin{proof}
Let $p>1$ and $a,b\in\r^{k}$ be fixed and let $q\geq1$ be such
that $p^{-1}+q^{-1}=1$. Using the fact that $\left\langle x,y\right\rangle \leq\|x\|_{p}\|y\|_{q}$
for every $x,y\in\r^{k}$, and denoting $\tilde{a}$ and $\tilde{b}$
to be vectors with entries $|a_{i}|^{1/p}$ and $|b_{i}|^{1/p}$,
respectively, we have that 
\begin{align*}
 & k\min_{1\leq i\leq k}\left\{ |a_{i}b_{i}|\right\} ^{1/p}\leq\sum_{i=1}^{k}|a_{i}b_{i}|^{1/p}\\
 & \leq\|\tilde{a}\|_{p}\|\tilde{b}\|_{q}=\|a\|_{1}^{1/p}\left(\sum_{i=1}^{k}|b_{i}|^{q/p}\right)^{1/q}=\left(\|a\|_{1}\|b\|_{q/p}\right)^{1/p}.
\end{align*}
Dividing by $k$, taking the $p^{{\rm th}}$ power on both sides,
and using the fact that $p/q=p-1$, yields 
\[
\min_{1\leq i\leq k}\left\{ |a_{i}b_{i}|\right\} \leq k^{-p}\|a\|_{1}\|b\|_{q/p}=k^{-p}\|a\|_{1}\|b\|_{1/(p-1)}.
\]
\end{proof}
\begin{prop}
\label{prop:aicg_v_hat_rate_alt}Let $\{(\icgY i,\hat{z}_{i},\hat{v}_{i})\}_{i=1}^{k}$
be as in \prettyref{lem:aicg_resid_rate} and define the quantities
\begin{gather}
\begin{gathered}L_{1,k}^{{\rm avg}}:=\frac{1}{k}\sum_{i=1}^{k}L_{1}(\icgY i,\icgY{i-1}),\quad C_{\lam,k}^{{\rm avg}}:=\frac{1}{k}\sum_{i=1}^{k}C_{\lam}(\hat{z}_{i},\icgY i),\\
\end{gathered}
\label{eq:avg_def}
\end{gather}
where $C_{\lam}(\cdot,\cdot)$ and $\overline{C}_{\lam}$ are as in
\eqref{eq:C_lam_fn_def} and \eqref{eq:ell_phi_C_bar_lam_def}, respectively.
Then, we have 
\[
\min_{i\leq k}\|\hat{v}_{i}\|={\cal O}\left(\left[\sqrt{\lam}L_{1,k}^{{\rm avg}}+\frac{1+\theta C_{\lam,k}^{{\rm avg}}}{\sqrt{\lam}}\right]\left[\frac{\phi(z_{0})-\phi_{*}}{k}\right]^{1/2}\right)+\frac{\hat{\rho}}{2}.
\]
\end{prop}

\begin{proof}
Using \prettyref{prop:spec_refine} with $(z,z_{0})=(\icgY i,\icgY{i-1})$
and the fact that $C_{\lam}(\cdot,\cdot)\leq\overline{C}_{\lam}$
and $L_{1}(\cdot,\cdot)\leq L_{1}$, we have $\|\hat{v}_{i}\|\le{\cal E}_{i}\|\icgY i-\icgY{i-1}\|$,
for every $i\leq k$, where 
\[
{\cal E}_{i}:=\frac{2+\lam L_{1}(\icgY i,\icgY{i-1})+\theta C_{\lam}(\hat{z}_{i},\icgY i)}{\lam}\quad\forall i\geq1.
\]
As a consequence, using the sum of the second bound in \prettyref{lem:aicg_resid_rate}
from $i=1$ to $k$, the definitions in \eqref{eq:avg_def}, and \prettyref{lem:p_norm_tech}
with $p=3/2$, $a_{i}={\cal E}_{i}$, and $b_{i}=\|\icgY i-\icgY{i-1}\|$
for $i=1$ to $k$, yields 
\begin{align*}
\min_{i\leq k}\|\hat{v}_{i}\| & \leq\min_{i\leq k}{\cal E}_{i}\|\icgY i-\icgY{i-1}\|\leq\frac{1}{k^{3/2}}\left(\sum_{i=1}^{k}{\cal E}_{i}\right)\left(\sum_{i=1}^{k}\|\icgY i-\icgY{i-1}\|^{2}\right)^{1/2}\\
 & ={\cal O}\left(\left[\sqrt{\lam}L_{1,k}^{{\rm avg}}+\frac{1+\theta C_{\lam,k}^{{\rm avg}}}{\sqrt{\lam}}\right]\left[\frac{\phi(z_{0})-\phi_{*}}{k}\right]^{1/2}\right).
\end{align*}
\end{proof}
We are now ready to give the proof of \prettyref{thm:aicg_compl}. 
\begin{proof}[\textit{Proof of \prettyref{thm:aicg_compl}}]
(a) This follows from \prettyref{prop:aicg_v_hat_rate_alt}, the
fact that $C_{\lam}(\cdot,\cdot)\leq\overline{C}_{\lam}$ and $L_{f_{1}}(\cdot,\cdot)\leq L_{1}$,
and the termination condition in \prettyref{ln:aicg_stop_cond} of
the AICGM.

(b) The fact that $(\hat{z},\hat{v})=(\hat{z}_{k},\hat{v}_{k})$ satisfies
the inclusion of \eqref{eq:rho_approx_nco2} follows from \prettyref{prop:spec_refine}
with $(z,v,z_{0})=(\icgY k,v_{k},\icgY{k-1})$. The fact that $\|\hat{v}\|\leq\hat{\rho}$
follows from the termination condition in \prettyref{ln:aicg_stop_cond}
of the AICGM.

(c) This follows from \prettyref{prop:s_acg_properties}(c) and the
fact that method stops in finite number of iterations from part (a). 
\end{proof}

\section[Doubly-Accelerated Inexact Composite Gradient (D.AICG) Method]{Doubly-Accelerated Inexact Composite \protect \\
Gradient (D.AICG) Method}

This subsection presents the static D.AICGM, but omits its dynamic
variant for the sake of brevity. We do argue, however, that the dynamic
variant can be stated in the same way as the dynamic AICG variant
in \prettyref{sec:aicg} but with the call to the static AICGM replaced
with a call to the static D.AICGM of this subsection.

We start by stating some additional assumptions. It is assumed that:
\begin{itemize}
\item[(i)] the set $Z$ is closed; 
\item[(ii)] there exists a bounded set $\Omega\supseteq Z$ for which a projection
oracle exists. 
\end{itemize}
We first state the static D.AICGM in \prettyref{alg:d_aicg}, which
uses \prettyref{alg:sref} and \prettyref{alg:s_acgm} as subroutines.
Given $z_{0}\in Z$ and a special choice of $\lam>0$, its main idea
is to attempt to generate its $k^{{\rm th}}$ iterate using the S.ACGM
and project oracle of $\Omega$ to obtain accelerated updates
\begin{align*}
a_{k} & =\frac{1+\sqrt{1+4A_{k-1}}}{2},\quad A_{k}=A_{k-1}+a_{k},\\
\tilde{y}_{k} & =\frac{A_{k-1}}{A_{k}}z_{k-1}+\frac{a_{k-1}}{A_{k}}y_{k-1},\\
z_{k}^{a} & \approx\min_{u\in{\cal Z}}\left\{ \lam\left[\ell_{f_{1}}(u;z_{k-1})+f_{2}(u)+h(u)\right]+\frac{1}{2}\|u-z_{k-1}\|^{2}\right\} ,\\
y_{k} & =\argmin_{u\in\Omega}\frac{1}{2}\left\Vert u-\left[y_{k-1}-a_{k-1}\left(v_{k}+\tilde{y}_{k-1}-z_{k}^{a}\right)\right]\right\Vert ^{2},\\
z_{k} & =\argmin_{u\in\left\{ z_{k-1},z_{k}^{a}\right\} }\phi(u),
\end{align*}
where $y_{0}=z_{0}$, $A_{0}=0$, and $v_{k}$ is a residual that
is obtained from computing $z_{k}^{a}$. In particular, the S.ACGM
is used in the inexact update of $z_{k}^{a}$. The iterate is then
refined using the SRP in \prettyref{alg:sref} and termination of
the method occurs when either: (i) a refined iterate solving \prettyref{prb:approx_nco2}
is found; or (ii) a failure condition has been triggered. The termination
status of the method is store in a variable $\pi_{S}$ which is \texttt{true}
if the former scenario occurs and \texttt{false} the latter scenario
occurs.

\begin{mdframed}
\mdalgcaption{Static D.AICG Method}{alg:d_aicg}
\begin{smalgorithmic}
	\Require{$\hat{\rho} > 0, \enskip (m_1, M_1, m_2, M_2)\in\r^4, \enskip h \in \cConv(Z), \enskip f_1 \in {\cal C}_{m_1,M_1}(Z), \enskip f_2 \in {\cal C}_{m_2,M_2}(Z), \enskip (\lam,\theta) \in \r_{++}^2 \text{ s.t. } \lam M_1 + \theta^2 < 1/2, \enskip z_0 \in Z$;}
	\Initialize{$\mu \gets 1, \enskip L \gets \lam M_2^{+}+1 (\text{see } \eqref{eq:mMi_def}), \enskip \pi_S \gets {\tt true}, \enskip y_0 \gets z_0, \enskip A_0 \gets 0, \enskip \phi \Lleftarrow f_1 + f_2 + h;$}
	\vspace*{.5em}
	\Procedure{St.D.AICG}{$f_1, f_2, h, z_0, \lam, \theta, M_2, \hat{\rho}$}
	\For{$k=1,...$}
		\StateStep{\algpart{1}\textbf{Attack} the $k^{\rm th}$ prox-linear subproblem.}
		\StateEq{$a_{k-1} \gets \frac{1+\sqrt{1+4A_{k-1}}}{2}$}
		\StateEq{$A_{k} \gets A_{k-1}+a_{k-1}$}
		\StateEq{$\tilde{y}_{k-1} \gets \frac{A_{k-1}z_{k-1}+a_{k-1}y_{k-1}}{A_{k}}$};
		\StateEq{$\psi_s^k \Lleftarrow \lam\left[\ell_{f_{1}}(\cdot;z_{k-1})+f_{2}\right]+\frac{1}{2}\|\cdot-z_{k-1}\|^{2}$}
		\StateEq{$(z_k^{a}, v_k, \varepsilon_k, \pi_k^{\rm acg}) \gets \text{S.ACG}(\psi_s^k, \lam h, \tilde{y}_{k-1}, \theta, \mu, L)$} \label{ln:d_aicg_s_acgm_call}
		\StateStep{\algpart{2}\textbf{Check} a special convexity condition.}
		\If{$\lnot (\pi_k^{\rm acg})$ \textbf{ or } $\Delta_\mu(z_{k-1};z_k^{a},v_k) > \varepsilon_k$}  \Comment{See \eqref{eq:Delta_def}} \label{ln:d_aicg_check}
			\StateEq{$\pi_S \gets {\tt false}$}
			\StateEq{\Return{$(z_0, \infty, \pi_S)$}}
		\EndIf
		\StateStep{\algpart{3}\textbf{Check} the termination condition.}
		\StateEq{$(\hat z_k, \hat v_k) \gets \text{SREF}(f_1, f_2, h, z_k, \tilde{y}_{k-1}, v_k, M_2, \lam)$}
		\If{$\|\hat{v}_k\| \leq \hat{\rho}$} \label{ln:d_aicg_stop_cond}
			\StateEq{\Return{$(\hat{z}_k, \hat{v}_k, \pi_S)$}}
		\EndIf
		\StateStep{\algpart{4}\textbf{Compute} an accelerated prox step.}
		\StateEq{$y_k \gets \argmin_{u\in\Omega}\frac{1}{2}\left\Vert u-[y_{k-1}-a_{k-1}\left(v_{k}+\tilde{y}_{k-1}-z_k^{a}\right)]\right\Vert ^{2}$}  \label{ln:d_aicg_yk_update}
		\StateEq{$z_k \gets \argmin_{u\in\left\{z_{k-1}, z_k^{a}\right\} } \phi(u)$} \label{ln:d_aicg_zk_update}
	\EndFor
	\EndProcedure
\end{smalgorithmic}
\end{mdframed}

Some remarks about this method are in order. To ease the discussion,
let us refer to the ACG iterations performed in \prettyref{ln:d_aicg_s_acgm_call}
of the method as \textbf{inner iterations} and the iterations over
the indices $k$ as\emph{ }\textbf{outer iterations}. First, similar
to the static AICGM, the static D.AICGM may fail without obtaining
a pair that solves \prettyref{prb:approx_nco2}. \prettyref{thm:d_aicg_compl}(c)
shows that a sufficient condition for the method to stop successfully
is that $f_{2}$ be convex. Using arguments similar to the ones employed
to derive the dynamic AICG variant, a dynamic D.AICG variant can also
be developed that repeatedly invokes the static D.AICGM in place of
the static AICGM. Second, in view of the update for $\aicgYMin k$
in \prettyref{ln:d_aicg_zk_update}, the collection of function values
$\{\phi(\aicgYMin i)\}_{i=0}^{k}$ is non-increasing. Third, in view
of the requirement on $(\lam,\theta)$, if $M_{1}>0$ then $0<\lam<(1-2\theta^{2})/(2M_{1})$
whereas if $M_{1}\leq0$ then $0<\lam<\infty$.

The next result summarizes some facts about the D.AICGM. Before proceeding,
we introduce the useful constants 
\begin{gather}
\begin{gathered}D_{z}:=\sup_{u,z\in Z}\|u-z\|,\quad D_{\Omega}:=\sup_{u,z\in\Omega}\|u-z\|,\quad\Delta_{\phi}^{0}:=\phi(\aicgYMin 0)-\phi_{*},\\
d_{0}:=\inf_{u^{*}\in{\cal Z}}\left\{ \|\aicgYMin 0-u^{*}\|:\phi(u^{*})=\phi_{*}\right\} ,\quad E_{\lam,\theta}:=\sqrt{\lam}L_{1}+\frac{1+\theta\overline{C}_{\lam}}{\sqrt{\lam}}.
\end{gathered}
\label{eq:d_aicg_diam}
\end{gather}

\begin{thm}
\label{thm:d_aicg_compl} The following statements hold about the
static D.AICGM: 
\begin{itemize}
\item[(a)] it stops in 
\begin{equation}
{\cal O}\left(\frac{E_{\lam,\theta}^{2}[m_{1}^{+}D_{z}^{2}+\Delta_{\phi}^{0}]}{\hat{\rho}^{2}}+\frac{E_{\lam,\theta}[m_{1}^{+}+1/\lam]^{1/2}D_{\Omega}}{\hat{\rho}}\right)\label{eq:d_aicg_outer_compl}
\end{equation}
outer iterations; 
\item[(b)] if it stops with $\pi_{S}=$ \texttt{\emph{true}}, then the first
two arguments of its output triple $(\hat{z},\hat{v},\pi_{S})$ solve
\prettyref{prb:approx_nco2}; 
\item[(c)] if $f_{2}$ is convex, then it always stops with $\pi_{S}=$ \texttt{\emph{true}}
in
\begin{equation}
{\cal O}\left(\frac{E_{\lam,\theta}^{2}m_{1}^{+}D_{z}^{2}}{\hat{\rho}^{2}}+\frac{E_{\lam,\theta}[m_{1}^{+}]^{1/2}D_{\Omega}}{\hat{\rho}}+\frac{E_{\lam,\theta}^{2/3}d_{0}^{2/3}\lam^{-1/3}}{\hat{\rho}^{2/3}}\right)\label{eq:d_aicg_cvx_outer_compl}
\end{equation}
outer iterations. 
\end{itemize}
\end{thm}

We now make three remarks about the above results. First, in the ``best''
scenario of $\max\{m_{1},m_{2}\}\leq0$, we have that \eqref{eq:d_aicg_cvx_outer_compl}
reduces to 
\[
{\cal O}\left(\left[L_{1}+\frac{1}{\lam}\right]^{2/3}\left[\frac{d_{0}^{2/3}}{\hat{\rho}^{2/3}}\right]\right),
\]
which has a smaller dependence on $\hat{\rho}$ when compared to \eqref{eq:compl_aicg}.
In the ``worst'' scenario of $\min\{m_{1},m_{2}\}>0$, if we take
$\theta={\cal O}(1/\overline{C}_{\lam})$, then \eqref{eq:d_aicg_outer_compl}
reduces to 
\[
{\cal O}\left(\left[\sqrt{\lam}L_{1}+\frac{1}{\sqrt{\lam}}\right]^{2}\left[\frac{m_{1}^{+}D_{z}^{2}+\phi(\aicgYMin 0)-\phi_{*}}{\hat{\rho}^{2}}\right]\right),
\]
which has the same dependence on $\hat{\rho}$ as in \eqref{eq:compl_aicg}.
Second, part (c) shows that if the method stops with an output pair
$(\hat{z},\hat{v})$, regardless of the convexity of $f_{2}$, then
that pair is always an approximate solution of \ref{prb:eq:nco2}.
Third, \prettyref{prop:gen_v_hat_rate_d_aicg} shows that the quantities
$L_{1}$ and $\overline{C}_{\lam}$ in all the previous complexity
results can be replaced by their averaged counterparts in \eqref{eq:d_avg_def}.
As these averaged quantities only depend on $\{(\aicgY i,\hat{z}_{i},\aicgXTilde{i-1})\}_{i=1}^{k}$,
we can infer that the static D.AICGM, like the static AICGM of the
previous subsection, also adapts to the local geometry of its input
functions.

\subsection{D.AICG Properties and Iteration Complexity}

This subsection establishes several key properties of static D.AICGM
and gives the proof of \prettyref{thm:d_aicg_compl}.

To avoid repetition, we assume throughout this subsection that $k\geq1$
denotes an arbitrary successful outer iteration of the D.AICGM and
let 
\[
\{(a_{i},A_{i},\aicgYMin i,\aicgY i,\aicgX i,\aicgXTilde{i-1},\hat{z}_{i},\hat{v}_{i},v_{i},\varepsilon_{i})\}_{i=1}^{k}
\]
denote the sequence of all iterates generated by it up to and including
the $k^{{\rm th}}$ iteration. Observe that this implies that the
$i^{{\rm th}}$ D.AICG outer iteration for any $1\leq i\leq k$ has
$\pi_{S}=$ \texttt{true}, i.e. the (only) S.ACG call in this iteration
does not stop with $\pi_{i}^{{\rm acg}}=$ \texttt{false} and $\Delta_{1}(\icgY{i-1};\aicgY i,v_{i})\leq\varepsilon_{i}$.
Moreover, throughout this subsection we let 
\begin{equation}
\gammaBTFn_{i}(u)=\ell_{f_{1}}(u;\aicgXTilde{i-1})+f_{2}(u)+h(u),\quad\gammaBFn_{i}(u)=\gammaBTFn_{i}(\aicgY i)+\frac{1}{\lam}\inner{v_{i}+\aicgXTilde{i-1}-\aicgY i}{u-\aicgY i}.\label{eq:theory_gamma}
\end{equation}

The first set of results present some basic properties about the functions
$\gammaBTFn_{i}$ and $\gammaBFn_{i}$ as well as the iterates generated
by the method.
\begin{lem}
\label{lem:gamma_props}The following statements hold for any $s\in Z$
and $1\leq i\leq k$: 
\begin{itemize}
\item[(a)] $\gammaBFn_{i}(\aicgY i)=\gammaBTFn_{i}(\aicgY i)$; 
\item[(b)] $\aicgX i=\argmin_{u\in\Omega}\left\{ \lam a_{i-1}\gammaBFn_{i}(u)+\|u-\aicgX{i-1}\|^{2}/2\right\} ;$ 
\item[(c)] $\aicgY i-v_{i}=\argmin_{u\in{\cal Z}}\left\{ \lam\gammaBFn_{i}(u)+\|u-\aicgXTilde{i-1}\|^{2}/2\right\} ;$ 
\item[(d)] $-M_{1}\|u-\aicgXTilde{i-1}\|^{2}/2\leq\gammaBTFn_{i}(u)-\phi(u)\leq m_{1}\|u-\aicgXTilde{i-1}\|^{2}/2$; 
\item[(e)] $\phi(\aicgYMin{i-1})\geq\phi(\aicgYMin i)$ and $\phi(\aicgY i)\geq\phi(\aicgYMin i)$. 
\end{itemize}
\end{lem}

\begin{proof}
To keep the notation simple, denote 
\begin{gather}
\begin{gathered}(\YP,\YMinP,\YM,\XtM)=(\aicgY i,\aicgYMin i,\aicgYMin{i-1},\aicgXTilde{i-1}),\quad(\XP,\XM)=(\aicgX i,\aicgX{i-1}),\\
(\AP,\AM,\aM)=(A_{i},A_{i-1},a_{i-1}),\quad(v,\varepsilon)=(v_{i},\varepsilon_{i}).
\end{gathered}
\label{eq:d_aicg_proof_notation}
\end{gather}

(a) This is immediate from the definitions of $\gammaBFn$ and $\gammaBTFn$
in \eqref{eq:theory_gamma}.

(b) Define $\aicgXhat i:=\aicgX{k-1}-a_{k-1}\left(v_{k}+\aicgXTilde{k-1}-\aicgY k\right)$.
Using the definition of $\gammaBFn$ in \eqref{eq:theory_gamma},
we have that 
\begin{align*}
\argmin_{u\in\Omega}\left\{ \lam\aM\gammaFn u+\frac{1}{2}\|u-\XM\|^{2}\right\}  & =\argmin_{u\in\Omega}\left\{ a\left\langle v+\XtM-\YP,u-x\right\rangle +\frac{1}{2}\|u-\XM\|^{2}\right\} \\
 & =\argmin_{u\in\Omega}\frac{1}{2}\left\Vert u-\left(\XM-a\left[v+\XtM-\YP\right]\right)\right\Vert ^{2}\\
 & =\argmin_{u\in\Omega}\frac{1}{2}\left\Vert u-\XhP\right\Vert ^{2}=\XP.
\end{align*}

(c) Using the definition of $\gammaBFn$ in \eqref{eq:theory_gamma},
we have that 
\[
\lam\nabla\gammaFn{\YP-v}+(\YP-v)-\XtM=(v+\XtM-\YP)+(\YP-v)-\XtM=0,
\]
and hence, the point $\YP-v$ is the global minimum of $\lam\gammaBFn+\|\cdot-\XtM\|^{2}/2$.

(d) This follows from the fact that $f_{1}\in{\cal C}_{m_{1},M_{1}}(Z)$
and the definition of $\gammaBTFn$ in \eqref{eq:theory_gamma}.

(e) This follows immediately from the update rule of $\aicgYMin i$
in \prettyref{ln:d_aicg_zk_update} of the D.AICGM. 
\end{proof}
\begin{lem}
\label{lem:d_aicg_Delta_props}Let $w=\aicgXTilde{i-1}$ and $\Delta_{1}(\cdot;\cdot,\cdot)$
be as in \eqref{eq:Delta_def} with 
\begin{equation}
\psi_{s}=\lam\left[\ell_{f_{1}}(\cdot;z_{k-1})+f_{2}\right]+\frac{1}{2}\|\cdot-z_{k-1}\|^{2},\quad\psi_{n}=\lam h.\label{eq:d_aicg_tech_psi}
\end{equation}
Then, following statements hold: 
\begin{itemize}
\item[(a)] the triple $(\aicgY i,v_{i},\varepsilon_{i})$ solves Problem~${\cal B}$
and satisfies $\Delta_{1}(\icgY{i-1};\aicgY i,v_{i})\leq\varepsilon$,
and hence 
\begin{gather}
\begin{gathered}\|v_{i}\|+2\varepsilon_{i}\leq\sigma^{2}\|\aicgY i-\aicgXTilde{i-1}\|^{2},\quad\Delta_{1}(u;\aicgY i,v_{i})\leq\varepsilon_{i}\quad\forall u\in\{\hat{z}_{i},\aicgYMin{i-1}\},\end{gathered}
\label{eq:d_aicg_main_ineq}
\end{gather}
\item[(b)] if $f_{2}$ is convex, then $(\aicgY i,v_{i},\varepsilon_{i})$ solves
Problem~${\cal A}$; 
\item[(c)] $\Delta_{1}(s;\aicgY i,v_{i})=\lam[\gammaBFn_{i}(s)-\gammaBTFn_{i}(s)];$ 
\item[(d)] $\Delta_{1}(\aicgYMin i;\aicgY i,v_{i})\leq\varepsilon$. 
\end{itemize}
\end{lem}

\begin{proof}
(a) This follows from \prettyref{ln:d_aicg_check} of the D.AICGM
and \prettyref{prop:s_acg_properties}(b).

(b) This follows from the S.ACG call in \prettyref{ln:d_aicg_s_acgm_call}
of the D.AICGM, the fact that $h$ is convex, and \prettyref{prop:s_acg_properties}(c)
with $\psi_{s}=\gammaBTFn_{i}+\|\cdot-\aicgXTilde{i-1}\|^{2}/2$.

(c) Using the definitions of $(\psi_{s},\psi_{n})$ and $(\gammaBFn,\gammaBTFn)$,
we have that 
\begin{align*}
\Delta_{1}(s;\YP,v) & =(\psi_{s}+\psi_{n})(\YP)-(\psi_{s}+\psi_{n})(s)-\left\langle v,\YP-s\right\rangle +\frac{1}{2}\|s-\YP\|^{2}\\
 & =\left[\lam\widetilde{\gamma}(\YP)+\frac{1}{2}\|\YP-\tilde{x}\|^{2}\right]-\left[\lam\widetilde{\gamma}(s)+\frac{1}{2}\|s-\tilde{x}\|^{2}\right]-\left\langle v,\YP-s\right\rangle +\frac{1}{2}\|s-\YP\|^{2}\\
 & =\left[\lam\gamma(s)+\frac{1}{2}\|s-\tilde{x}\|^{2}\right]-\left[\lam\widetilde{\gamma}(s)+\frac{1}{2}\|s-\tilde{x}\|^{2}\right]\\
 & =\lam\gamma(s)-\lam\widetilde{\gamma}(s).
\end{align*}

(d) If $\aicgYMin i=\aicgYMin{i-1}$, then this follows from \prettyref{ln:d_aicg_check}
of the method. On the other hand, if $\aicgYMin i=\aicgY i$, then
this follows from part (c). 
\end{proof}
We now state some well-known (see, for example, \prettyref{lem:aK_AK}
with $\lam_{k}=\tau_{k}=1$) properties of $A_{i}$ and $a_{i-1}$. 
\begin{lem}
\label{lem:A_k_props} For every $1\leq i\leq k$, we have that: 
\begin{itemize}
\item[(a)] $a_{i-1}^{2}=A_{i}$; 
\item[(b)] $i^{2}/4\leq A_{i}\leq i^{2}$. 
\end{itemize}
\end{lem}

The next two lemmas are technical results that are needed to establish
the key inequality in \prettyref{prop:descent_d_aicg}. 
\begin{lem}
\label{lem:main_resid_bd} For every $u\in Z$ and $1\leq i\leq k$,
we have that 
\begin{gather*}
\frac{1}{2}\left(A_{i-1}\|\aicgYMin{i-1}-\aicgXTilde{i-1}\|^{2}+a_{i-1}\|u-\aicgXTilde{i-1}\|^{2}\right)\leq2D_{\Omega}^{2}+a_{i-1}D_{z}^{2}.
\end{gather*}
\end{lem}

\begin{proof}
Throughout the proof, we use the notation in \eqref{eq:d_aicg_proof_notation}.
Using the relation $(p+q)^{2}\leq2p^{2}+2q^{2}$ for every $p,q\in\r$,
\prettyref{lem:A_k_props}(a), the fact that $A\leq A^{+}$, $x\in\Omega$,
and $y\in Z$, and the definition of $\XtM$, and the definitions
of $D_{\Omega}$ and $D_{z}$ in \eqref{eq:d_aicg_diam}, we conclude
that 
\begin{align*}
\AM\|\YM-\XtM\|^{2}+\aM\|u-\XtM\|^{2} & =\AM\left\Vert \frac{\aM}{\AP}(\YM-\XM)\right\Vert ^{2}+\aM\left\Vert \frac{\AM}{\AP}(u-\YM)+\frac{\aM}{\AP}(u-\XM)\right\Vert ^{2}\\
 & \leq\frac{A}{\AP}\left(\|(\YM-u)+(u-\XM)\|^{2}+2a\left[\frac{\AM^{2}}{\AP^{2}}\|u-\YM\|^{2}+\frac{a^{2}}{\AP^{2}}\|u-\XM\|^{2}\right]\right)\\
 & \leq\frac{2A}{A^{+}}\left(\|u-\YM\|^{2}+\|u-\XM\|^{2}\right)+2\aM\|u-\YM\|^{2}+\frac{2\aM}{\AP}\|u-\XM\|^{2}\\
 & \leq2\left[\|u-x\|^{2}+(1+\aM)\|u-y\|^{2}\right]\\
 & \leq2\left[D_{\Omega}^{2}+(1+\aM)D_{z}^{2}\right].
\end{align*}
The conclusion now follows from dividing both sides of the above inequalities
by 2 and using the fact that $D_{z}\leq D_{\Omega}$. 
\end{proof}
\begin{lem}
For every $u\in Z$ and $1\leq i\leq k$, we have that 
\begin{align}
 & A_{i}\left[\phi(\aicgYMin i)+\left(\frac{1-\lambda M_{1}}{2\lambda}\right)\|\aicgY i-\aicgXTilde{i-1}\|^{2}-\frac{\|v_{i}\|^{2}}{2\lam}\right]+\frac{1}{2\lambda}\|u-\aicgX i\|^{2}\nonumber \\
 & \leq A_{i-1}\gammaBFn_{i}(\aicgYMin{i-1})+a_{i-1}\gammaBFn_{i}(u)+\frac{1}{2\lambda}\|u-\aicgX{i-1}\|^{2}.\label{eq:d_aicg_subdescent}
\end{align}
\end{lem}

\begin{proof}
Throughout the proof, we use the notation in \eqref{eq:d_aicg_proof_notation}.
We first present two key expressions. First, using the definition
of $\gammaBFn$ in \eqref{eq:theory_gamma} and \prettyref{lem:gamma_props}(c),
it follows that 
\begin{align}
\min_{u\in{\cal Z}}\left\{ \lam\gammaFn u+\frac{1}{2}\|u-\XtM\|^{2}\right\}  & =\lam\widetilde{\gamma}(\YP)-\left\langle v+\XtM-\YP,v\right\rangle +\frac{1}{2}\left\Vert v+\XtM-\YP\right\Vert ^{2}\nonumber \\
 & =\lam\widetilde{\gamma}(\YP)-\|v\|^{2}-\left\langle v,\XtM-\YP\right\rangle +\frac{1}{2}\left\Vert v+\XtM-\YP\right\Vert ^{2}\nonumber \\
 & =\lam\widetilde{\gamma}(\YP)-\frac{1}{2}\|v\|^{2}+\frac{1}{2}\|\XtM-\YP\|^{2}.\label{eq:gamma_reg_t_bd}
\end{align}
Second, \prettyref{lem:gamma_props}(b) and the fact that the function
$\aM\gammaBFn+\|\cdot-\XM\|^{2}/(2\lam)$ is $(1/\lam)$-strongly
convex imply that 
\begin{equation}
\aM\gammaFn{\XP}+\frac{1}{2\lam}\|\XP-\XM\|^{2}\leq\aM\gammaFn u+\frac{1}{2\lam}\|u-\XM\|^{2}-\frac{1}{2\lam}\|u-\XP\|^{2}.\label{eq:XP_strong_cvx}
\end{equation}
Using \eqref{eq:gamma_reg_t_bd}, \prettyref{lem:gamma_props}(d)--(e),
\prettyref{lem:A_k_props}(a), and the fact that $\gammaBFn$ is affine,
we have that 
\begin{align}
 & \AP\left[\phi(\YMinP)+\left(\frac{1-\lambda M_{1}}{2\lambda}\right)\|\YP-\XtM\|^{2}\right]\nonumber \\
 & \leq\AP\left[\gammaTFn{\YP}+\frac{1}{2\lam}\|\YP-\XtM\|^{2}\right]\\
 & =\AP\left[\min_{u\in{\cal Z}}\left\{ \gammaFn u+\frac{1}{2\lam}\|u-\XtM\|^{2}\right\} +\frac{\|v\|^{2}}{2\lam}\right]\nonumber \\
 & \leq\AP\left[\gammaFn{\frac{\AM\YM+\aM\XP}{\AP}}+\frac{1}{2\lambda}\left\Vert \frac{\AM\YM+\aM\XP}{\AP}-\frac{\AM\YM+\aM\XM}{\AP}\right\Vert ^{2}+\frac{\|v\|^{2}}{2\lam}\right]\nonumber \\
 & =\AM\gammaFn{\YM}+\aM\gammaFn{\XP}+\frac{\aM^{2}}{2\lambda\AP}\|\XM-\XP\|^{2}+\frac{\AP}{2\lam}\|v\|^{2}\nonumber \\
 & =\AM\gammaFn{\YM}+\aM\gammaFn{\XP}+\frac{1}{2\lambda}\|\XM-\XP\|^{2}+\frac{\AP}{2\lam}\|v\|^{2}\label{eq:sub_descent_ineq}
\end{align}
The conclusion now follows from combining \eqref{eq:XP_strong_cvx}
with \eqref{eq:sub_descent_ineq}. 
\end{proof}
We now present an inequality that plays an important role in the analysis
of the D.AICGM. 
\begin{prop}
\label{prop:descent_d_aicg} Let $\Delta_{1}(\cdot;\cdot,\cdot)$
be as in \eqref{eq:Delta_def} with $(\psi_{s},\psi_{n})$ as in \eqref{eq:d_aicg_tech_psi},
and define 
\begin{equation}
\theta_{i}(u):=A_{i}\left[\phi(\aicgYMin i)-\phi(u)\right]+\frac{1}{2\lam}\|u-\aicgX i\|^{2}\quad\forall i\geq0.\label{eq:theta_def}
\end{equation}
For every $u\in Z$ satisfying $\Delta_{1}(u;\aicgY i,v_{i})\leq\varepsilon$
and $1\leq i\leq k$, we have that 
\begin{equation}
\frac{A_{i}}{4\lam}\|\aicgY i-\aicgXTilde{i-1}\|^{2}\leq m_{1}^{+}\left(a_{i-1}D_{z}^{2}+2D_{\Omega}^{2}\right)+\theta_{i-1}(u)-\theta_{i}(u).\label{eq:descent_d_aicg}
\end{equation}
\end{prop}

\begin{proof}
Throughout the proof, we use the notation in \eqref{eq:d_aicg_proof_notation}
together with the notation $\thetaM=\theta_{i-1}$ and $\thetaP=\theta_{i}$.
Let $u\in\dom h$ be such that $\Delta_{1}(u;\YP,v)\leq\varepsilon$.
Subtracting $A\phi(u)$ from both sides of the inequality in \eqref{eq:d_aicg_subdescent}
and using the definition of $\thetaP$ we have 
\begin{align}
 & \frac{\AP}{2\lam}\left[(1-\lam M_{1})\|\YP-\XtM\|^{2}-\|v\|^{2}\right]+\thetaP(u)\nonumber \\
 & =\frac{\AP}{2\lam}\left[(1-\lam M_{1})\|\YP-\XtM\|^{2}-\|v\|^{2}\right]+\AP\left[\phi(\YMinP)-\phi(u)\right]+\frac{1}{2\lam}\|u-\YP\|^{2}\nonumber \\
 & \leq\AM\gammaFn{\YM}+\aM\gammaFn u-\AM\phi(u)+\frac{1}{2\lambda}\|u-\XM\|^{2}\nonumber \\
 & =\aM\left[\gammaFn u-\phi(u)\right]+\AM\left[\gammaFn{\YM}-\phi(\YM)\right]+\thetaM(u).\label{eq:descent2_ineq1}
\end{align}
Moreover, using \prettyref{lem:d_aicg_Delta_props}(a) and (c), and
with our assumption that $\Delta_{1}(u;\YP,v)\leq\varepsilon$, we
have that 
\begin{align}
\gammaFn s-\phi(s) & =\gammaTFn s-\phi(s)+\frac{\Delta_{1}(s;\YP,v)}{\lam}\leq\frac{m_{1}^{+}}{2}\|s-\XtM\|^{2}+\frac{\varepsilon}{\lam}\quad\forall s\in\{u,\YM\}.\label{eq:descent2_ineq2}
\end{align}
Combining \eqref{eq:descent2_ineq1}, \eqref{eq:descent2_ineq2},
and \prettyref{lem:main_resid_bd} then yields 
\begin{align*}
 & \frac{\AP}{2\lam}\left[(1-\lam M_{1})\|\YP-\XtM\|^{2}-\|v\|^{2}\right]+\thetaP(u)\\
 & \leq\frac{m_{1}^{+}}{2}\left[\aM\|u-\XtM\|^{2}+\AM\|\YM-\XtM\|^{2}\right]+\frac{\varepsilon A_{+}}{\lam}+\thetaM(u)\\
 & \leq m_{1}^{+}\left(\aM D_{z}^{2}+2D_{\Omega}^{2}\right)+\frac{\varepsilon A_{+}}{\lam}+\thetaM(u).
\end{align*}
Re-arranging the above terms and using the restriction on $(\lam,\theta)$
(in the D.AICGM) together with the first inequality in \eqref{eq:d_aicg_main_ineq},
we conclude that 
\begin{align*}
 & m_{1}^{+}\left(\aM D_{h}^{2}+2D_{\Omega}^{2}\right)+\thetaM(u)-\thetaP(u)\\
 & \geq\frac{\AP}{2\lam}\left[(1-\lam M_{1})\|\YP-\XtM\|^{2}-\|v\|^{2}-2\varepsilon\right]\\
 & \geq\frac{\AP(1-\lam M_{1}-\sigma^{2})}{2\lam}\|\YP-\XtM\|^{2}\\
 & \geq\frac{\AP}{4\lam}\|\YP-\XtM\|^{2}.
\end{align*}
\end{proof}
The following result describes some important technical bounds obtained
by summing \eqref{eq:descent_d_aicg} for two different choices of
$u$ (possibly changing with $i$) from $i=1$ to $k$.
\begin{prop}
\label{prop:sum_d_aicg_descent} Let $\Delta_{\phi}^{0}$ and $d_{0}$
be as in \eqref{eq:d_aicg_diam} and define 
\begin{gather}
S_{k}:=\frac{1}{4\lam}\sum_{i=1}^{k}A_{i}\|\aicgY i-\aicgXTilde{i-1}\|^{2}.\label{eq:S_k_def}
\end{gather}
Then, the following statements hold: 
\begin{itemize}
\item[(a)] $S_{k}={\cal O}_{1}(k^{2}[m_{1}^{+}D_{z}^{2}+\Delta_{\phi}^{0}]+k[m_{1}^{+}+1/\lam]D_{\Omega}^{2});$ 
\item[(b)] if $f_{2}$ is convex, then $S_{k}={\cal O}_{1}(k^{2}m_{1}^{+}D_{z}^{2}+km_{1}^{+}D_{\Omega}^{2}+d_{0}^{2}/\lam).$ 
\end{itemize}
\end{prop}

\begin{proof}
(a) Let $\Delta_{1}(\cdot;\cdot,\cdot)$ be defined as in \eqref{eq:Delta_def}
with $(\psi_{s},\psi_{n})$ given by \eqref{eq:d_aicg_tech_psi}.
Using \eqref{eq:theta_def}, the fact that $\aicgX i,\aicgY i\in\Omega$,
the fact that $A_{i}$ is nonnegative and increasing, and the definitions
of $\theta_{i}$ and $D_{\Omega}$ in \eqref{eq:theta_def} and \eqref{eq:d_aicg_diam},
respectively, we have that 
\begin{align}
\sum_{i=1}^{k}\left[\theta_{i-1}(\aicgYMin i)-\theta_{i}(\aicgYMin i)\right] & \leq\sum_{i=1}^{k}A_{i-1}\left[\phi(\aicgYMin{i-1})-\phi(\aicgYMin i)\right]+\frac{1}{2\lam}\sum_{i=1}^{k}\|\aicgYMin i-\aicgX{i-1}\|^{2}\nonumber \\
 & \leq A_{k}\sum_{i=1}^{k}\left[\phi(\aicgYMin{i-1})-\phi(\aicgYMin i)\right]+\frac{k}{2\lam}D_{\Omega}^{2}\nonumber \\
 & \leq A_{k}\left[\phi(\aicgYMin 0)-\phi_{*}\right]+\frac{k}{2\lam}D_{\Omega}^{2}.\label{eq:ncvx_theta_bd}
\end{align}
Moreover, noting \prettyref{lem:d_aicg_Delta_props}(d) and using
\prettyref{prop:descent_d_aicg} with $u=y_{i}$, we conclude that
\eqref{eq:descent_d_aicg} holds with $u=y_{i}$ for every $1\leq i\leq k$.
Summing these $k$ inequalities and using \eqref{eq:ncvx_theta_bd},
the definition of $S_{k}$ in \eqref{eq:S_k_def}, and \prettyref{lem:A_k_props}(b)
yields the desired conclusion.

(b) Assume now that $f_{2}$ is convex and let $\aicgYMin *$ be a
point such that $\phi(\aicgYMin *)=\phi_{*}$ and $\|\aicgYMin 0-\aicgYMin *\|=d_{0}$.
It then follows from \prettyref{lem:d_aicg_Delta_props}(b) and \prettyref{lem:gen_refine}(d)
with $(z,v)=(\aicgY i,v_{i})$ that $\Delta_{1}(\aicgYMin *;\aicgY i,v_{i})\leq\varepsilon$
for every $1\leq i\leq k$. The conclusion now follows by using an
argument similar to the one in (a) but which instead sums \eqref{eq:descent_d_aicg}
with $u=\aicgYMin *$ from $i=1$ to $k$, and uses the fact that
\begin{align*}
\sum_{i=1}^{k}\left[\theta_{i-1}(\aicgYMin *)-\theta_{i}(\aicgYMin *)\right]=\theta_{0}(\aicgYMin *)-\theta_{k}(\aicgYMin *)\leq\frac{1}{2\lam}\|\aicgYMin 0-\aicgYMin *\|^{2}=\frac{d_{0}}{2\lam},
\end{align*}
where the inequality is due to the fact that $\theta_{k}(\aicgYMin *)\geq0$
(see \eqref{eq:theta_def}) and $A_{0}=0$. 
\end{proof}
We now establish the rate at which the residual $\|\hat{v}_{i}\|$
tends to 0.
\begin{prop}
\label{prop:gen_v_hat_rate_d_aicg}Let $S_{k}$ be as in \eqref{eq:S_k_def}.
Moreover, define the quantities
\begin{gather}
\begin{gathered}L_{1,k}^{{\rm avg}}:=\frac{1}{k}\sum_{i=1}^{k}L_{1}(\aicgY i,\aicgXTilde{i-1}),\quad C_{\lam,k}^{{\rm avg}}:=\frac{1}{k}\sum_{i=1}^{k}C_{\lam}(\hat{z}_{i},\aicgY i),\\
\end{gathered}
\label{eq:d_avg_def}
\end{gather}
where $C_{\lam}(\cdot,\cdot)$ and $\overline{C}_{\lam}$ are as in
\eqref{eq:C_lam_fn_def} and \eqref{eq:ell_phi_C_bar_lam_def}, respectively.
Then, we have
\begin{align*}
\min_{i\leq k}\|\hat{v}_{i}\| & ={\cal O}_{1}\left(\left[\sqrt{\lam}L_{1,k}^{{\rm avg}}+\frac{1+\theta C_{\lam,k}^{{\rm avg}}}{\sqrt{\lam}}\right]\left[\frac{S_{k}}{k^{3}}\right]^{1/2}\right)+\frac{\hat{\rho}}{2}.
\end{align*}
\end{prop}

\begin{proof}
Let $\ell=\left\lceil k/2\right\rceil $. Using \prettyref{prop:spec_refine}
with $(z,w)=(\aicgY i,\XtM_{i-1})$ and the bounds $C_{\lam}(\cdot,\cdot)\leq\overline{C}_{\lam}$
and $L_{1}(\cdot,\cdot)\leq L_{1}$ we have that $\|\hat{v}_{i}\|\leq{\cal E}_{i}\|\aicgY i-\aicgXTilde{i-1}\|$,
for every $\ell\leq i\leq k$, where 
\[
{\cal E}_{i}=\frac{2+\lam L_{1}(\aicgY i,\aicgXTilde{i-1})+\theta C_{\lam}(\hat{z}_{i},\aicgY i)}{\lam}\quad\forall i\geq1.
\]
As a consequence, using the definition of $S_{k}$ in \eqref{eq:S_k_def},
the definitions in \eqref{eq:d_avg_def}, \prettyref{lem:p_norm_tech}
with $p=3/2$, $a_{i}={\cal E}_{i}/\sqrt{A_{i}}$, and $b_{i}=\sqrt{A_{i}}\|\aicgY i-\aicgXTilde{i-1}\|$
for $i\in\{\ell,...,k\}$, \prettyref{lem:A_k_props}(b), and the
fact that $(k-\ell+1)\geq k/2$, yields 
\begin{align*}
\min_{\ell\leq i\leq k}\|\hat{v}_{i}\| & \leq\min_{\ell\leq i\leq k}{\cal E}_{i}\|\aicgY i-\aicgXTilde{i-1}\|\\
 & \leq\frac{1}{(k-\ell+1)^{3/2}}\left(\sum_{i=\ell}^{k}\frac{{\cal E}_{i}}{\sqrt{A_{i}}}\right)\left(\sum_{i=\ell}^{k}A_{i}\|\aicgY i-\aicgXTilde{i-1}\|^{2}\right)^{1/2}\\
 & \leq\frac{2^{3/2}}{k^{3/2}}\left(\frac{2}{k}\sum_{i=1}^{k}{\cal E}_{i}\right)\left(4\lam S_{k}\right)^{1/2}\\
 & ={\cal O}_{1}\left(\left[\sqrt{\lam}L_{1,k}^{{\rm avg}}+\frac{1+\theta C_{\lam,k}^{{\rm avg}}}{\sqrt{\lam}}\right]\left[\frac{S_{k}}{k^{3}}\right]^{1/2}\right).
\end{align*}
\end{proof}
We are now ready to give the proof of \prettyref{thm:d_aicg_compl}. 
\begin{proof}[\textit{Proof of \prettyref{thm:d_aicg_compl}}]
(a) This follows from \prettyref{prop:gen_v_hat_rate_d_aicg}, \prettyref{prop:sum_d_aicg_descent}(a),
the fact that $C_{\lam}(\cdot,\cdot)\leq\overline{C}_{\lam}$ and
$L_{f_{1}}(\cdot,\cdot)\leq L_{1}$, and the termination condition
in \prettyref{ln:d_aicg_stop_cond} of the D.AICGM.

\noindent (b) The fact that $(\hat{z},\hat{v})=(\hat{z}_{k},\hat{v}_{k})$
satisfies the inclusion of \eqref{eq:rho_approx_nco2} follows from
\prettyref{prop:spec_refine} with $(z,v,\acgX 0)=(\aicgY k,v_{k},\aicgXTilde{k-1})$.
The fact that $\|\hat{v}\|\leq\hat{\rho}$ follows the termination
condition in \prettyref{ln:d_aicg_stop_cond} of the D.AICGM.

(c) The fact that the method does not stop with $\pi_{S}=$ \texttt{false}
follows from \prettyref{prop:s_acg_properties}(c). The bound in \eqref{eq:d_aicg_cvx_outer_compl}
follows from a similar argument as in part (a) except that \prettyref{prop:sum_d_aicg_descent}(a)
is replaced with \prettyref{prop:sum_d_aicg_descent}(b). 
\end{proof}

\section{Exploiting the Spectral Decomposition}

\label{sec:spectral_details} 

Recall that at every outer iteration of the ICG methods in the previous
sections, a call to the S.ACG algorithm is made to tentatively solve
the Problem~${\cal B}$ (see \prettyref{sec:icg_prelim}) associated
with \eqref{eq:icg_subprb}. Our goal in this section is to present
a significantly more efficient ACG variant (based on the idea outlined
at the beginning of this chapter) for solving the same problem when
the underlying problem of interest is \ref{prb:eq:snco}.

Throughout our presentation, we make use of the functions ${\rm dg}:\r^{r}\mapsto\r^{r\times r}$
and ${\rm Dg}:\r^{m\times n}\mapsto\r^{r}$ given pointwise by 
\begin{equation}
\left[\dg z\right]_{ij}=\begin{cases}
z_{i}, & \text{if }i=j,\\
0, & \text{otherwise},
\end{cases}\quad\left[\Dg Z\right]_{i}=Z_{ii},\label{eq:dg_Dg_def}
\end{equation}
for every $z\in\r^{r},Z\in\r^{m\times n},$ and $(i,j)\in\{1,...,r\}^{2}$.

The content of this section is divided into two subsections. The first
one presents the aforementioned algorithm, whereas the second one
proves its key properties.

\subsection{Spectral ACG Method}

\label{subsec:spectral_exploit}

This subsection presents an efficient spectral ACG method ($\sigma$.ACGM),
which utilizes the S.ACGM of \prettyref{sec:icg_prelim}, for solving
the Problem~${\cal B}$ associated with \eqref{eq:icg_subprb}. 

Throughout our presentation, we let $\acgMatX 0$ represent the starting
point given to the S.ACGM by the two ICG methods. Moreover, we assume
that we have a method \texttt{SVD(...)} that returns a triple $(P,\sigma(Z),Q)$
representing the SVD of its input $Z$. More specifically, if \texttt{(P,
s, Q) $\gets$ SVD(Z)} then it holds that $Z=P[\dg s]Q^{*}$.

We now state the $\sigma$.ACGM in \prettyref{alg:sp_acgm}, which
uses the S.ACGM of \prettyref{sec:icg_prelim} and the aforementioned
SVD method as subroutines.

\begin{mdframed}
\mdalgcaption{$\sigma$.ACG Method}{alg:sp_acgm}
\begin{smalgorithmic}
	\Require{$M_2 \in \r_{++}, \enskip \enskip h^{\cal V} \in {\cal C}(\r^{r}), \enskip f_1 \in {\cal C}(\dom[h^{\cal V} \circ \sigma]), \enskip f_2^{\cal V} \in {\cal C}_{M_2}(\dom h),  \enskip Z_0 \in \r^{m \times n}, \enskip \theta \in (0,1)$;}
	\Initialize{$\mu \gets 1, \enskip L \gets \lam M_2 + 1, \enskip \pi_S \gets {\tt true}, \enskip \psi_n^{\cal V} \gets \lam h^{\cal V}$}
	\vspace*{.5em}
	\Procedure{$\sigma$.ACG}{$f_1, f_2^{\cal V}, h^{\cal V}, Z_0, \theta, \mu, L$}
		\StateStep{\algpart{1}\textbf{Attack} a vectorized prox-linear subproblem using the S.ACGM.}
		\StateEq{$Z_0^{\lam} \gets Z_0 - \lam \nabla f_1(Z_0)$}
		\StateEq{$(P, s, Q) \gets \text{SVD}(Z_0^\lam)$} \label{ln:svd}
		\StateEq{$\psi_s^{\cal V} \Lleftarrow \lam f_{2}^{\cal V}-\inner{s}{\cdot}+\frac{1}{2}\|\cdot\|^2$}
		\StateEq{$(z, v, \varepsilon, \pi_S) \gets \text{S.ACG}(\psi_s^{\cal V}, \psi_n^{\cal V}, \Dg(P^* Z_0 Q), \theta, \mu, L$} \label{ln:sp_s_acg_call}
		\StateStep{\algpart{1}\textbf{Terminate} based on the status of the S.ACGM call}
		\If{$\pi_S$}
			\StateEq{$Z \gets P(\dg z)Q^*$}
			\StateEq{$V \gets P(\dg v)Q^*$}
			\StateEq{\Return{$(Z, V, \varepsilon, \pi_S)$}}
		\Else
			\StateEq{\Return{$(Z_0, \infty, \infty, \pi_S)$}}
		\EndIf
	\EndProcedure
\end{smalgorithmic}
\end{mdframed}

We now make two remarks about the method. First, since it calls the
S.ACGM in \prettyref{ln:sp_s_acg_call}, its iteration complexity
is the same as the one given for the S.ACGM, i.e. as in \prettyref{prop:s_acg_properties}(a).
Second, because the functions $\psi_{s}^{{\cal V}}$ and $\psi_{n}^{{\cal V}}$
used in its S.ACG call have vector inputs over $\r^{r}$, the steps
in the $\sigma$.ACGM are significantly less costly than the ones
in an analogous S.ACGM call, which use functions with matrix inputs
over $\r^{m\times n}$.

The following result, whose proof is deferred to the next subsection,
presents the key properties of the $\sigma$.ACGM.
\begin{prop}
\label{prop:acg_implementation}Let $(Z,V,\varepsilon,\pi_{S})$ be
the output of a call to the $\sigma$.ACGM. Then, the following properties
hold: 
\begin{itemize}
\item[(a)] if $\pi_{S}=$\texttt{\emph{ true}}, then the triple $(Z,V,\varepsilon)$
solves the Problem~${\cal B}$ associated with \eqref{eq:icg_subprb}; 
\item[(b)] if $f_{2}$ is convex, then $\pi_{S}=$\texttt{\emph{ true}} and
the triple $(Z,V,\varepsilon)$ solves the Problem~${\cal A}$ associated
with \eqref{eq:icg_subprb}. 
\end{itemize}
\end{prop}

\subsection{Proof of \texorpdfstring{\prettyref{prop:acg_implementation}}{Key Proposition}}

This subsection gives the proof of \prettyref{prop:acg_implementation}.

Let the quantities $(P,Q)$ and $(\psi_{s}^{{\cal V}},\psi_{n}^{{\cal V}})$
be generated by a call of the $\sigma$.ACGM. Moreover, for every
$(u,U)\in\r^{r}\times\r^{m\times n}$, define the functions
\begin{gather}
\begin{gathered}f_{2}(U):=f_{2}^{{\cal V}}\circ\sigma(U),\quad h:=h^{{\cal V}}\circ\sigma,\quad\psi^{{\cal V}}(u):=\psi_{s}^{{\cal V}}(u)+\psi_{n}^{{\cal V}}(u)\\
{\cal M}(u):=P(\dg u)Q^{*},\quad{\cal V}(U):=\Dg(P^{*}UQ).
\end{gathered}
\label{eq:vec_mat_fns}
\end{gather}
The result below relates the function triple $(\psi_{s}^{{\cal V}},\psi_{n}^{{\cal V}},\psi^{{\cal V}})$
to the function triple $(\psi_{s},\psi_{n},\psi)$ given by 
\[
\psi_{s}:=\lam\left[\ell_{f_{1}}(\cdot,Z_{0})+f_{2}\circ\sigma\right]+\frac{1}{2}\|\cdot-Z_{0}\|^{2},\quad\psi_{n}:=\lam(h\circ\sigma),\quad\psi=\psi_{s}+\psi_{n}.
\]

\begin{lem}
\label{lem:psi_spec_props} Let $(z,v,\varepsilon,\pi_{S})$ and $(Z,V)$
be generated by a call to the $\sigma$.ACGM in which $\pi_{S}=$\texttt{\emph{
true}}. Then, the following properties hold: 
\begin{itemize}
\item[(a)] we have 
\[
\psi_{n}^{{\cal V}}(z)=\psi_{n}(Z),\quad\psi_{s}^{{\cal V}}(z)+B_{0}^{\lam}=\psi_{s}(Z),
\]
where $B_{0}^{\lam}:=\lam f_{1}(\acgMatX 0)-\lam\inner{\nabla f_{1}(\acgMatX 0)}{\acgMatX 0}+\|\acgMatX 0\|_{F}^{2}/2$; 
\item[(b)] we have 
\begin{equation}
V\in\pt_{\varepsilon}\left(\psi-\frac{1}{2}\|\cdot-Z\|_{F}^{2}\right)(Z)\iff v\in\pt_{\varepsilon}\left(\psi^{{\cal V}}-\frac{1}{2}\|\cdot-z\|^{2}\right)(z).\label{eq:mat_vec_spec_incl}
\end{equation}
\end{itemize}
\end{lem}

\begin{proof}
(a) The relationship between $\psi_{n}^{{\cal V}}$,and $\psi_{n}$
is immediate. On the other hand, using the definitions of $Z,f_{2}$,
and $B_{0}^{\lam}$, we have 
\begin{align*}
\psi_{s}^{{\cal V}}(z)+B_{0}^{\lam} & =\lam f_{2}(Z)-\inner{\acgMatX 0^{\lam}}Z+\frac{1}{2}\|Z\|_{F}^{2}+B_{0}^{\lam}\\
 & =\lam\left[f_{2}(Z)+f_{1}(\acgMatX 0)+\inner{\nabla f_{1}(\acgMatX 0)}{Z-\acgMatX 0}\right]+\frac{1}{2}\|Z-\acgMatX 0\|_{F}^{2}\\
 & =\psi_{s}(Z).
\end{align*}

(b) Let $S_{0}=V+\acgMatX 0^{\lam}-Z$ and $s_{0}=v+\sigma(\acgMatX 0^{\lam})-z$,
and note that $S_{0}={\cal M}(s_{0})$. Moreover, in view of part
(a) and the definition of $\psi$, observe that the left inclusion
in \eqref{eq:mat_vec_spec_incl} is equivalent to $S_{0}\in\pt_{\varepsilon}(\lam[f_{2}+h])(Z)$.
Using this observation, the fact that $S_{0}$ and $Z$ have a simultaneous
SVD, and \prettyref{thm:spectral_approx_subdiff} with $(S,s)=(S_{0},s_{0})$,
$\Psi=\lam[f_{2}+h]$, and $\Psi^{{\cal V}}=\lam[f_{2}^{{\cal V}}+h^{{\cal V}}]$,
we have that the left inclusion in \eqref{eq:mat_vec_spec_incl} is
also equivalent to $s_{0}\in\pt_{\varepsilon}(\lam[f_{2}^{{\cal V}}+h^{{\cal V}}])(z)$.
The conclusion now follows from the observing that the latter inclusion
is equivalent to the right inclusion in \eqref{eq:mat_vec_spec_incl}. 
\end{proof}
We are now ready to give the proof of \prettyref{prop:acg_implementation}.
\begin{proof}[Proof of \prettyref{prop:acg_implementation}]
(a) Let $(z,v)=({\cal V}(Z),{\cal V}(V))$ and remark that the successful
termination of the algorithm implies that the inequality in \eqref{eq:cvx_inexact}
and \eqref{eq:prb_B_Delta_ineq} hold. Using this remark, the fact
that $\|V\|_{F}^{2}=\|v\|^{2}$, and the bound 
\begin{align}
 & \sigma^{2}\|\acgX j-\acgX 0\|^{2}=\sigma^{2}\left(\|\acgX j\|^{2}-2\inner{\acgX j}{{\cal V}(\acgX 0)}+\|\acgMatX 0\|_{F}^{2}\right)+\sigma^{2}(\|{\cal V}(\acgX 0)\|^{2}-\|\acgMatX 0\|_{F}^{2})\nonumber \\
 & \leq\sigma^{2}\left(\|\acgMatX j\|^{2}-2\inner{\acgMatX j}{\acgMatX 0}+\|\acgMatX 0\|_{F}^{2}\right)=\sigma^{2}\|\acgMatX j-\acgMatX 0\|_{F}^{2},\label{eq:hpe_equiv}
\end{align}
we then have that the inequality in \eqref{eq:cvx_inexact} also holds
with $(z,v)=(Z,V)$.

To show the corresponding inequality for \eqref{eq:prb_B_Delta_ineq},
let $L=\lam M_{2}+1$ and consider the refined quantities
\begin{align*}
\hat{Z} & =\argmin_{U\in\r^{n\times m}}\left\{ \ell_{\psi_{s}}(U;Z)-\left\langle V,U\right\rangle +\frac{L}{2}\|U-Z\|^{2}+\psi_{n}(U)\right\} \\
\hat{z} & =\argmin_{u\in\r^{r}}\left\{ \ell_{\psi_{s}^{{\cal V}}}(u;z)-\left\langle v,u\right\rangle +\frac{L}{2}\|u-z\|^{2}+\psi_{n}^{{\cal V}}(u)\right\} 
\end{align*}
as well as the corresponding residuals
\begin{align*}
V_{r} & =V+L(Z-\hat{Z})+\nabla\psi_{s}(\hat{Z})-\nabla\psi_{s}(Z),\\
v_{r} & =v+L(z-\hat{z})+\nabla\psi_{s}^{{\cal V}}(\hat{z})-\nabla\psi_{s}^{{\cal V}}(z).
\end{align*}
Moreover, let $\Delta_{1}^{{\cal V}}(\cdot;\cdot,\cdot)$ be as in
\eqref{eq:Delta_def} with $(\psi_{s},\psi_{n})=(\psi_{s}^{{\cal V}},\psi_{n}^{{\cal V}})$
and $\Delta_{1}(\cdot;\cdot,\cdot)$ as in \eqref{eq:Delta_def}.
Using \prettyref{lem:spec_prox} with $\Psi=\psi_{n}$ and $S=V+MZ-\nabla\psi_{s}(Z)$
and \prettyref{lem:spec_prop}(b) we have that $\icgMatY r$, $V_{r}$,
$Z$, and $V$ have a simultaneous SVD. As a consequence, it follows
from \prettyref{lem:psi_spec_props}(a) that 
\begin{align*}
\varepsilon\geq & \Delta_{1}^{{\cal V}}(\hat{z};z,v)=\psi^{{\cal V}}(z)-\psi^{{\cal V}}(\hat{z})-\inner v{\hat{z}-z}+\frac{1}{2}\|\hat{z}-z\|^{2}\\
 & =\psi(Z)-\psi(\hat{Z})-\inner V{\hat{Z}-Z}+\frac{1}{2}\|\hat{Z}-Z\|^{2}\\
 & =\Delta_{1}(\hat{Z};Z,V).
\end{align*}
The conclusion now follows from the above and the definition of the
specialized refinement procedure in \prettyref{sec:icg_prelim}.

(b) This follows from part (a), \prettyref{prop:s_acg_properties}(c),
and \prettyref{lem:psi_spec_props}(b). 
\end{proof}

\section{Numerical Experiments}

\label{sec:num_spectral}

This section examines the performance of several solvers for finding
approximate stationary points of \ref{prb:eq:snco} where $(f_{1},f_{2}^{{\cal V}},h^{{\cal V}})$
satisfy assumptions \ref{asmp:snco1}--\ref{asmp:snco3} of \prettyref{chap:spectral}
with $(f_{2},h)=(f_{2}^{{\cal V}}\circ\sigma,h^{{\cal V}}\circ\sigma)$.
All experiments are run on Linux 64-bit machines each containing Xeon
E5520 processors and at least 8 GB of memory using MATLAB 2020a. It
is worth mentioning that the complete code for reproducing the experiments
is freely available online\footnote{See the code in \texttt{./tests/thesis/} from the GitHub repository
\href{https://github.com/wwkong/nc_opt/}{https://github.com/wwkong/nc\_opt/}}.

The algorithms benchmarked in this section are as follows.
\begin{itemize}
\item \textbf{AICG}: an instance of \prettyref{alg:dynamic_aicg} in which
$\xi=M_{1}$, $\lam=5/M_{1}$, $\sigma=(9/10-\max\{\lam(M_{1}-\xi,0\})$,
the ACG call is replaced by an R.ACG call with $L_{0}=\lam(M/100)+1$.
\item \textbf{CG}: an instance of \prettyref{alg:cgm} in which $\lam_{k}=1/(M_{1}+M_{2})$
for every $k\geq1$.
\item \textbf{D.AICG}: an instance of the dynamic version of \prettyref{alg:d_aicg}
in which $\xi=M_{1}$, $\lam=5/M_{1}$, $\sigma=(1/2-\max\{\lam(M_{1}-\xi,0\})$,
the ACG call is replaced by an R.ACG call with $L_{0}=\lam(M/100)+1$.
\item \textbf{AG}: a variant of the AG method described in \prettyref{subsec:num_unconstr}
in which $\{(\alpha_{k},\beta_{k},\lam_{k})\}_{k\geq1}$ are as in
\citep[Corollary 1]{Ghadimi2016} with $L_{\Psi}=M_{1}+M_{2}$.
\end{itemize}
Given a tolerance $\hat{\rho}>0$ and an initial point $\icgMatY 0\in Z$,
each algorithm in this section seeks a pair $(\hat{Z},\hat{V})\in Z\times\r^{m\times n}$
satisfying 
\begin{gather*}
\hat{V}\in\nabla f_{1}(\hat{Z})+\nabla(f_{2}^{{\cal V}}\circ\sigma)(\hat{Z})+\pt(h^{{\cal V}}\circ\sigma)(\hat{Z}),\\
\frac{\|\hat{V}\|}{\|\nabla f_{1}(\icgMatY 0)+(f_{2}^{{\cal V}}\circ\sigma)(\icgMatY 0)\|+1}\leq\hat{\rho}.
\end{gather*}
Moreover, each algorithm is given a time limit of either 10800 or
7200 seconds. The bold numbers in each of the tables in this section
highlight the algorithm that performed the most efficiently in terms
of function value.

\subsection{Ball-Constrained Matrix Completion}

\label{subsec:ball_mc}

This subsection presents computational results for the ball-constrained
matrix (BC-MC) problem in \citep{Kong2020}. More specifically, given
a quadruple $(\alpha,\beta,\mu,\theta)\in\r_{++}^{4}$, a data matrix
$A\in\r^{m\times n}$, and indices $\Omega$, this subsection considers
the BC-MC problem 
\begin{align*}
\begin{aligned}\min_{U\in\r^{m\times n}}\  & \frac{1}{2}\|P_{\Omega}(U-A)\|_{F}^{2}+\kappa_{\mu}\circ\sigma(U)+\tau_{\alpha}\circ\sigma(U)\\
\text{s.t.}\  & \|U\|_{F}^{2}\leq\sqrt{mn}\cdot\max_{i,j}|A_{ij}|,
\end{aligned}
\end{align*}
where $P_{\Omega}$ is the linear operator that zeros out any entry
that is not in $\Omega$ and 
\begin{align*}
\kappa_{\mu}(z)=\frac{\mu\beta}{\theta}\sum_{i=1}^{n}\log\left(1+\frac{|z_{i}|}{\theta}\right),\quad\tau_{\alpha}(z)=\alpha\beta\left[1-\exp\left(-\frac{\|z\|_{2}^{2}}{2\theta}\right)\right]
\end{align*}
for every $z\in\rn$. Here, the function $\kappa_{\mu}+\tau_{\alpha}$
is a nonconvex generalization of the convex elastic net regularizer
\citep{Sun2012}, and it is well-known \citep{Yao2017} that the function
$\kappa_{\mu}-\mu\|\cdot\|_{*}$ is concave, differentiable, and has
a $(2\beta\mu/\theta^{2})$-Lipschitz continuous gradient.

We now describe the different data matrices that are considered. Each
matrix $A\in\r^{m\times n}$ is obtained from a different collaborative
filtering system where each row represents a unique user, each column
represents a unique item, and each entry represents a particular rating.
\prettyref{tab:mc_data_mat} lists the names of each data set, where
the data originates from (in the footnotes), and some basic statistics
about the matrices.

\begin{table}[th]
\centering{}%
\begin{tabular}{|>{\raggedright}m{2.7cm}|>{\centering}m{1cm}|>{\centering}m{1cm}|>{\centering}m{1.6cm}|>{\centering}m{1.6cm}|>{\centering}m{1.6cm}|}
\hline 
{\footnotesize{}Name} & {\footnotesize{}$m$} & {\footnotesize{}$n$} & {\footnotesize{}\% nonzero} & {\footnotesize{}$\min_{i,j}A_{ij}$} & {\footnotesize{}$\max_{i,j}A_{ij}$}\tabularnewline
\hline 
{\footnotesize{}Jester}\tablefootnote{The ratings in the file ``jester\_dataset\_1\_1.zip'' from \url{http://eigentaste.berkeley.edu/dataset/}.} & {\footnotesize{}24938} & {\footnotesize{}100} & {\footnotesize{}24.66\%} & {\footnotesize{}-9.95} & {\footnotesize{}10}\tabularnewline
{\footnotesize{}Anime}\tablefootnote{A subset of the ratings from \url{https://www.kaggle.com/CooperUnion/anime-recommendations-database}
where each user has rated at least 720 items.} & {\footnotesize{}506} & {\footnotesize{}9437} & {\footnotesize{}10.50\%} & {\footnotesize{}1} & {\footnotesize{}10}\tabularnewline
{\footnotesize{}MovieLens 100K}\tablefootnote{The ratings in the file ``ml-latest-small.zip'' from \url{https://grouplens.org/datasets/movielens/}.} & {\footnotesize{}610} & {\footnotesize{}9724} & {\footnotesize{}1.70\%} & {\footnotesize{}0.5} & {\footnotesize{}5}\tabularnewline
{\footnotesize{}FilmTrust}\tablefootnote{See the ratings in the file ``ratings.txt'' under the FilmTrust
section in \url{https://www.librec.net/datasets.html}.} & {\footnotesize{}1508} & {\footnotesize{}2071} & {\footnotesize{}1.14\%} & {\footnotesize{}0.5} & {\footnotesize{}8}\tabularnewline
{\footnotesize{}MovieLens 1M}\tablefootnote{See the ratings in the file ``ml-1m.zip'' from \url{https://grouplens.org/datasets/movielens/}.} & {\footnotesize{}6040} & {\footnotesize{}3952} & {\footnotesize{}4.19\%} & {\footnotesize{}1} & {\footnotesize{}5}\tabularnewline
\hline 
\end{tabular}\caption{Description of the BC-MC data matrices.\label{tab:mc_data_mat}}
\end{table}

We now describe the experiment parameters considered. First the starting
point $Z_{0}$ is randomly generated from a shifted binomial distribution
that closely follows the data matrix $A$. More specifically, the
entries of $Z_{0}$ are distributed according to a $\textsc{Binomial}(n,\mu/n)-\underline{A}$
distribution, where $\mu$ is the sample average of the nonzero entries
in $A$, the integer $n$ is the ceiling of the range of ratings in
$A$, and $\underline{A}$ is the minimum rating in $A$. Second,
the decomposition of the objective function is as follows 
\begin{gather*}
f_{1}=\frac{1}{2}\|P_{\Omega}(\cdot-A)\|_{F}^{2},\quad f_{2}^{{\cal V}}=\mu\left[\kappa_{\mu}(\cdot)-\frac{\beta}{\theta}\|\cdot\|_{1}\right]+\tau_{\alpha}(\cdot),\quad h^{{\cal V}}=\frac{\mu\beta}{\theta}\|\cdot\|_{1}+\delta_{{\cal F}}(\cdot),
\end{gather*}
where ${\cal F}=\{U\in\r^{m\times n}:\|U\|_{F}\leq\sqrt{mn}\cdot\max_{i,j}|A_{ij}|\}$
is the set of feasible solutions. Third, in view of the previous decomposition,
the curvature parameters are set to be 
\[
m_{1}=0,\quad M_{1}=1,\quad m_{2}=\frac{2\beta\mu}{\theta^{2}}+\frac{2\alpha\beta}{\theta}\exp\left(\frac{-3\theta}{2}\right),\quad M_{2}=\frac{\alpha\beta}{\theta},
\]
where it can be shown that the smallest and largest eigenvalues of
$\nabla^{2}\tau_{\alpha}(z)$ are bounded below and above by $-2\alpha\beta\exp(-3\theta/2)/\theta$
and $\alpha\beta/\theta$, respectively, for every $z\in\rn$. Fourth,
each problem instance uses a specific data matrix $A$ from \prettyref{tab:mc_data_mat},
the hyperparameters $(\alpha,\beta,\mu,\theta)=(10,20,2,1)$ and $\hat{\rho}=10^{-6}$,
and $\Omega$ to be the index set of nonzero entries in the chosen
matrix $A$. Finally, a cutoff time of 10800 seconds is used for the
MovieLens 1M dataset and a cutoff time of 7200 seconds is used for
the other datasets.

Figure~\prettyref{fig:mc_graphs} contains the plots of the log objective
function value against the runtime, listed in increasing order of
the smallest dimension in the data matrix.

\begin{figure}[th]
\begin{centering}
\includegraphics[scale=0.55]{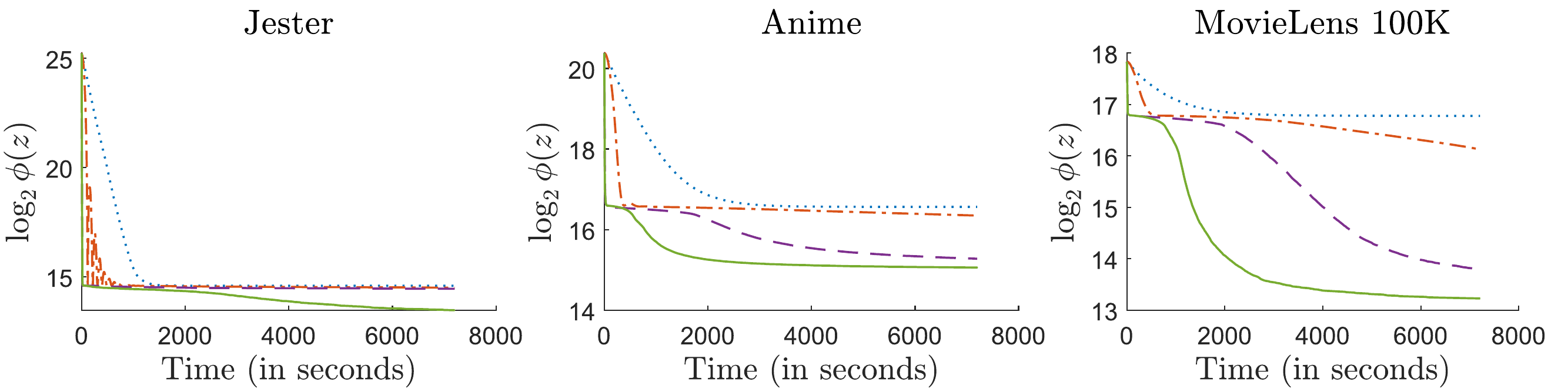}
\par\end{centering}
\begin{centering}
\includegraphics[scale=0.55]{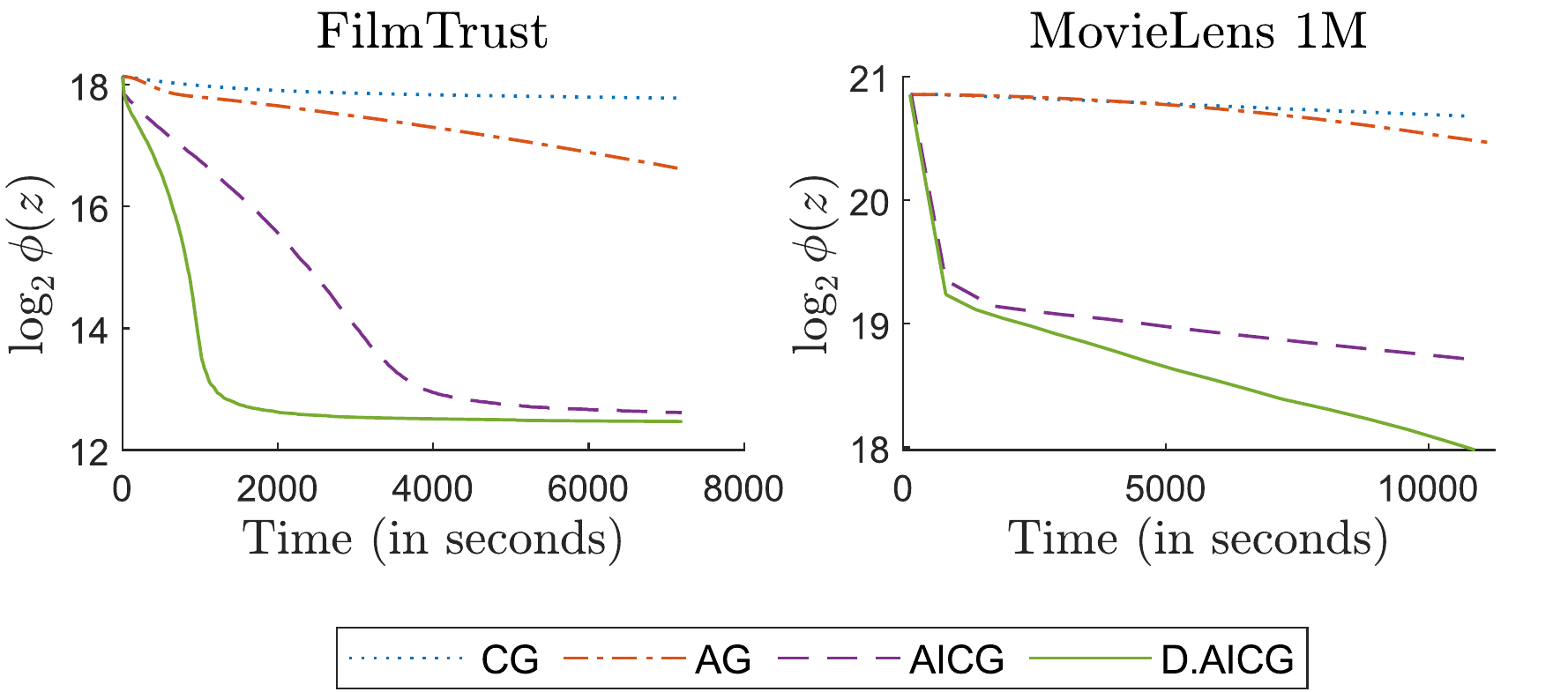}
\par\end{centering}
\centering{}\caption{Function value vs. runtime for the BC-MC problems.}
\label{fig:mc_graphs}
\end{figure}

\subsection{Multiblock Ball-Constrained Matrix Completion}

This subsection presents computational results for the multiblock
ball-constrained matrix (MBC-MC) problem in \citep{Kong2020}. Given
a quadruple $(\alpha,\beta,\mu,\theta)\in\r_{++}^{4}$, a block decomposable
data matrix $A\in\r^{m\times n}$ with blocks $\{A_{i}\}_{i=1}^{k}\subseteq\r^{p\times q}$,
and indices $\Omega$, this subsection considers the MBC-MC problem:
\begin{align*}
\begin{aligned}\min_{U\in\r^{m\times n}}\  & \frac{1}{2}\|P_{\Omega}(U-A)\|_{F}^{2}+\sum_{i=1}^{k}\left[\kappa_{\mu}\circ\sigma(U_{i})+\tau_{\alpha}\circ\sigma(U_{i})\right]\\
\text{s.t.}\  & \|U\|_{F}^{2}\leq\sqrt{mn}\cdot\max_{i,j}|A_{ij}|,
\end{aligned}
\end{align*}
where $P_{\Omega}$, $\kappa_{\mu}$, and $\tau_{\alpha}$ are as
in \prettyref{subsec:ball_mc} and $U_{i}\in\r^{p\times q}$ is the
$i^{{\rm th}}$ block of $U$ with the same indices as $A_{i}$ with
respect to $A$.

We now describe the two classes of data matrices that are considered.
Every data matrix is a $5$-by-$5$ block matrix consisting of $50$-by-$100$
sized submatrices. Every submatrix contains only 25\% nonzero entries
and each data matrix generates its submatrix entries from different
probability distributions. More specifically, for a sampled probability
$p\sim\textsc{Uniform}[0,1]$ specific to a fixed submatrix, one class
uses a $\textsc{Binomial}(n,p)$ distribution with $n=10$, while
the other uses a $\textsc{TruncatedNormal}(\mu,\sigma)$ distribution
with $\mu=10p$, $\sigma^{2}=10p(1-p)$, and upper and lower bounds
$0$ and $10$, respectively.

We now describe the experiment parameters considered. First, the decomposition
of the objective function and the quantities $Z_{0}$, $(m_{1},M_{1})$,
$(m_{2},M_{2})$, $\hat{\rho}$, and $\Omega$ are the same as in
\prettyref{subsec:ball_mc}. Second, we fix $(\beta,\theta)=(20,1)$
and vary $(\alpha,\mu,A)$ across the different problem instances.
Finally, a cutoff time of 7200 seconds is used for all problem instances
tested.

\prettyref{fig:bmc_binom} contains the plots of the log objective
function value against the runtime for the binomial data set, listed
in increasing order of $M_{2}$. The corresponding plots for the truncated
normal data set are similar to the binomial plots, so we omit them
for the sake of brevity. \prettyref{tab:bmc_binom} and \prettyref{tab:bmc_tnorm}
respectively contain the last function values of each algorithm for
the binomial and truncated normal data sets, listed in increasing
order of $M_{2}$. Moreover, each row of these tables corresponds
to a different choice of $(\mu,\alpha)$ and the bolded numbers highlight
which algorithm performed the best in terms of the last function value.

\begin{figure}[th]
\begin{centering}
\includegraphics[scale=0.55]{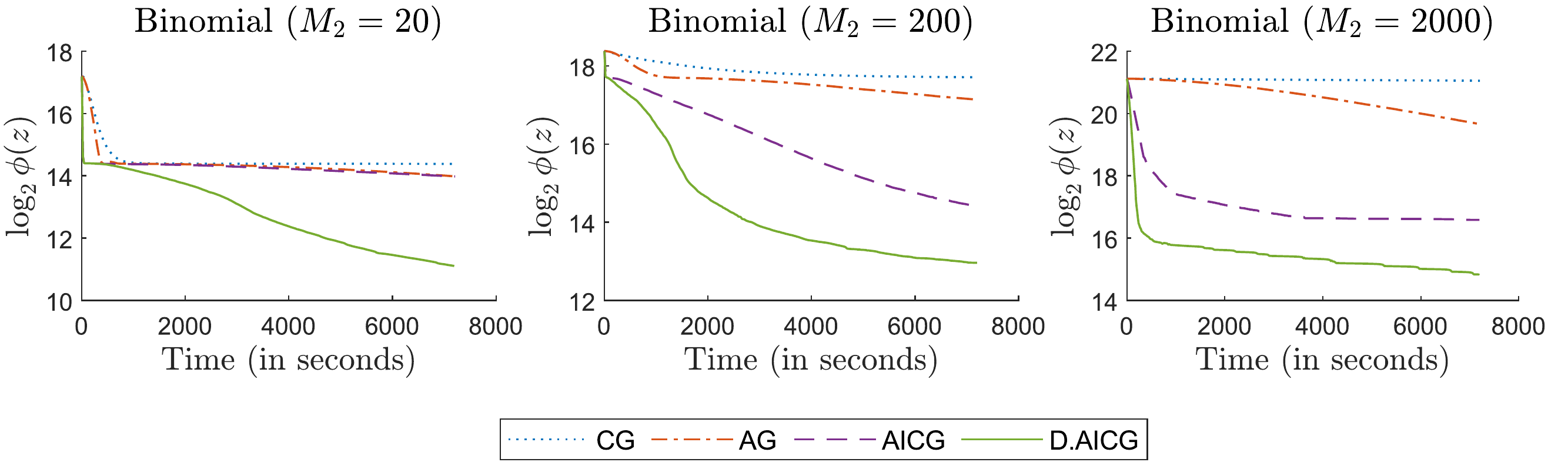}
\par\end{centering}
\caption{Function value vs. runtime for the binomial MBC-MC problems.}
\label{fig:bmc_binom}
\end{figure}

\begin{table}[th]
\begin{centering}
\begin{tabular}{|>{\centering}m{1.6cm}|>{\centering}m{1.2cm}|>{\centering}p{1.7cm}|>{\centering}p{1.7cm}|>{\centering}p{1.7cm}|>{\centering}p{1.7cm}|}
\hline 
\multicolumn{2}{|c|}{Parameters} & \multicolumn{4}{c|}{\textbf{Last Function Value}}\tabularnewline
\hline 
{\footnotesize{}$(\mu,\alpha)$} & {\footnotesize{}$M_{2}$} & {\footnotesize{}CG} & {\footnotesize{}AG} & {\footnotesize{}AICG} & {\footnotesize{}D.AICG}\tabularnewline
\hline 
{\footnotesize{}$(1,0.2)$} & {\footnotesize{}$20$} & {\footnotesize{}2.13E+04} & {\footnotesize{}1.62E+04} & {\footnotesize{}1.61E+04} & \textbf{\footnotesize{}2.20E+03}\tabularnewline
{\footnotesize{}$(10,2)$} & {\footnotesize{}$200$} & {\footnotesize{}2.15E+05} & {\footnotesize{}1.44E+05} & {\footnotesize{}2.19E+04} & \textbf{\footnotesize{}7.98E+03}\tabularnewline
{\footnotesize{}$(100,20)$} & {\footnotesize{}$2000$} & {\footnotesize{}2.17E+06} & {\footnotesize{}8.24E+05} & {\footnotesize{}9.82E+04} & \textbf{\footnotesize{}2.92E+04}\tabularnewline
\hline 
\end{tabular}
\par\end{centering}
\caption{Last function values for the binomial MBC-MC problems.\label{tab:bmc_binom}}
\end{table}

\begin{table}[th]
\begin{centering}
\begin{tabular}{|>{\centering}m{1.6cm}|>{\centering}m{1.2cm}|>{\centering}p{1.7cm}|>{\centering}p{1.7cm}|>{\centering}p{1.7cm}|>{\centering}p{1.7cm}|}
\hline 
\multicolumn{2}{|c|}{Parameters} & \multicolumn{4}{c|}{\textbf{Last Function Value}}\tabularnewline
\hline 
{\footnotesize{}$(\mu,\alpha)$} & {\footnotesize{}$M_{2}$} & {\footnotesize{}CG} & {\footnotesize{}AG} & {\footnotesize{}AICG} & {\footnotesize{}D.AICG}\tabularnewline
\hline 
{\footnotesize{}$(1,0.2)$} & {\footnotesize{}$20$} & {\footnotesize{}2.14E+04} & {\footnotesize{}8.92E+03} & {\footnotesize{}1.26E+04} & \textbf{\footnotesize{}1.25E+03}\tabularnewline
{\footnotesize{}$(10,2)$} & {\footnotesize{}$200$} & {\footnotesize{}2.21E+05} & {\footnotesize{}1.75E+05} & {\footnotesize{}3.29E+04} & \textbf{\footnotesize{}1.16E+04}\tabularnewline
{\footnotesize{}$(100,20)$} & {\footnotesize{}$2000$} & {\footnotesize{}2.27E+06} & {\footnotesize{}1.71E+06} & {\footnotesize{}1.06E+05} & \textbf{\footnotesize{}4.50E+04}\tabularnewline
\hline 
\end{tabular}
\par\end{centering}
\caption{Last function values for the truncated normal MBC-MC problems.\label{tab:bmc_tnorm}}
\end{table}

\subsection{Discussion of the Results}

We see that the D.AICGM and AICGM are generally more efficient than
the AG and CG methods, respectively. The D.AICGM method, in particular,
appears to escape local minima more quickly than the other methods.
Moreover, the larger the constant $M_{2}$ is, the more efficient
the ICG methods are compared to the benchmark methods. Curiously,
the larger the smallest dimension of the matrix space is, the more
efficient the inexact methods are compared to the exact ones. 

We conjecture that the efficiency of the spectral methods is attributed
to the fact that the main iterations of the methods are performed
within the space of singular values rather than in the space of matrices.

\section{Conclusion and Additional Comments}

In this chapter, we presented two methods for finding approximate
stationary points of a class of spectral NCO problems. More specifically,
the methods are inexact variants of the CGM (see \prettyref{alg:cgm})
and an accelerated monotonic CGM. We established an ${\cal O}(\hat{\rho}^{-2})$
iteration complexity bound for finding $\hat{\rho}$-approximate stationary
points for both methods and an ${\cal O}(\hat{\rho}^{-2/3})$ bound
for the accelerated method when the objective function is convex.
Through several new results about spectral functions, we also developed
a variant of the ACGM in \prettyref{alg:acgm} which is especially
efficient for spectral NCO problems.

The next chapter presents some practical improvements on the methods
and procedures developed in this and previous chapters.

\subsection*{Additional Comments}

It is worth mentioning that the outer iteration scheme of the D.AICGM
is a monotonic and inexact generalization of the AG method in \citep{Ghadimi2016}.
More specifically, the AG method can be viewed as a version of the
D.AICGM where: (i) $\theta=0$; (ii) the S.ACG call in \prettyref{ln:d_aicg_s_acgm_call}
is replaced by an exact solver of \eqref{eq:icg_subprb}; and (iii)
the update of $\aicgX k$ in \prettyref{ln:d_aicg_yk_update} is replaced
by an update involving a prox evaluation of the function $a_{k-1}(f_{2}+h)$.
Hence, the D.AICGM can be significantly more efficient when its S.ACG
call is more efficient than an exact solver of \eqref{eq:icg_subprb}
and/or when the projection onto $\Omega$ is more efficient than the
proximal evaluation of $a_{k-1}(f_{2}+h)$.

\subsection*{Future Work}

It would be worth investigating if the developments in \prettyref{sec:spectral_details}
are applicable to other first-order iterative optimization algorithms
and/or other classes of NCO problems. 

\newpage{}

\appendix
\titleformat{\chapter}[display]{\centering\normalsize\bfseries}{\vspace*{-3em}APPENDIX \thechapter}{-0.5em}{\normalsize\bfseries\uppercase}

\chapter{Properties of the PPM and CGM}

\label{app:prox_vartn_props}

This appendix presents the proofs of propositions related to the PPM
and CGM.
\begin{proof}[Proof of \prettyref{prop:ppm_vartn}]
(a) The optimality of $z_{k}$ and the definition of $v_{k}$ immediately
yield $v_{k}=(z_{k-1}-z_{k})/\lam_{k}\in\pt\psi(z_{k})$

(b) Using the inclusion in (a) and the fact that $\lam_{k}\geq0$,
we immediately have 
\[
\psi(z_{k-1})\geq\psi(z_{k})+\left\langle v_{k},z_{k-1}-z_{k}\right\rangle =\psi(z_{k})+\frac{1}{\lam_{k}}\|z_{k-1}-z_{k}\|^{2}>\psi(z_{k}).
\]
(c) Summing the inequality in (b) from indices $1$ to $k$ yields
\begin{align*}
\left(\sum_{i=1}^{k}\lam_{i}\right)\cdot\min_{i\leq k}\|v_{i}\|^{2} & \leq\sum_{i=1}^{k}\lam_{i}\|v_{i}\|^{2}\leq\sum_{i=1}^{k}\frac{1}{\lam_{i}}\|z_{i-1}-z_{i}\|^{2}\leq\sum_{i=1}^{k}\left[\psi(z_{i-1})-\psi(z_{i})\right]\\
 & =\psi(z_{0})-\psi(z_{k}),
\end{align*}
which implies the desired inequality.
\end{proof}
\begin{proof}[Proof of \prettyref{prop:cgm_basic_vartn}]
(a) The optimality of $z_{k}$ and the definitions of $v_{k}$ and
$\ell_{\psi_{s}}$ imply that
\[
\nabla\psi_{s}(z_{k})+\pt\psi_{n}(z_{k})\ni\nabla\psi_{s}(z_{k})-\nabla\psi_{s}(z_{k-1})+\frac{1}{\lam_{k}}(z_{k-1}-z_{k})=v_{k}.
\]

(b) The above inclusion in part (a) and \eqref{eq:pgm_descent} imply
that 
\begin{align*}
\psi_{n}(z_{k-1}) & \geq\psi_{n}(z_{k})+\left\langle v_{k}-\psi_{s}(z_{k}),z_{k-1}-z_{k}\right\rangle \\
 & =\psi_{n}(z_{k})+\left\langle \nabla\psi_{s}(z_{k-1}),z_{k}-z_{k-1}\right\rangle +\frac{1}{\lam_{i}}\|z_{k-1}-z_{k}\|^{2}\\
 & \geq\psi_{n}(z_{k})+\left[\psi_{s}(z_{k})-\psi_{s}(z_{k-1})\right]+\left(\frac{1}{\lam_{k}}-\frac{L_{k}}{2}\right)\|z_{k-1}-z_{k}\|^{2},
\end{align*}
which implies the rightmost inequality. The leftmost inequality follows
from the assumption that $L_{k}<2/\lam_{k}$. 

(c) Using the inequality in (b) from indices $1$ to $k$, the definition
of $v_{k}$, \eqref{eq:pgm_lipschitz}, and the inequality $\|a+b\|^{2}\leq2\|a\|^{2}+2\|b\|^{2}$,
it holds that
\begin{align*}
 & \left(\frac{1}{4}\sum_{i=1}^{k}\lam_{i}\xi_{i}\right)\min_{i\leq k}\|v_{i}\|^{2}\\
 & \leq\sum_{i=1}^{k}\frac{1}{2}\left(\frac{2-\lam_{i}L}{2\lam_{i}}\right)\left(\frac{1}{\lam_{i}^{2}}+L_{i}^{2}\right)^{-1}\|v_{i}\|^{2}\\
 & =\sum_{i=1}^{k}\frac{1}{2}\left(\frac{2-\lam_{i}L}{2\lam_{i}}\right)\left(\frac{1}{\lam_{i}^{2}}+L_{i}^{2}\right)^{-1}\left\Vert \frac{1}{\lam_{i}}(z_{i-1}-z_{i})+\nabla\psi_{s}(z_{i})-\nabla\psi_{s}(z_{i-1})\right\Vert ^{2}\\
 & \leq\sum_{i=1}^{k}\left(\frac{2-\lam_{i}L}{2\lam_{i}}\right)\left(\frac{1}{\lam_{i}^{2}}+L_{i}^{2}\right)^{-1}\left[\frac{1}{\lam_{i}^{2}}\|z_{i-1}-z_{i}\|^{2}+\|\nabla\psi_{s}(z_{i})-\nabla\psi_{s}(z_{i-1})\|^{2}\right]\\
 & \leq\sum_{i=1}^{k}\left(\frac{2-\lam_{i}L}{2\lam_{i}}\right)\|z_{i-1}-z_{i}\|^{2}\\
 & \leq\sum_{i=1}^{k}\left[\psi(z_{i-1})-\psi(z_{i})\right]=\psi(z_{0})-\psi(z_{k}),
\end{align*}
which clearly implies the inequality in \eqref{eq:pgm_props}.
\end{proof}
\begin{proof}[Proof of \prettyref{prop:cgm_ext_vartn}]

(a) It follows from \prettyref{prop:cgm_basic_vartn} and the definitions
of $q$ and $v$ that $q\in\nabla\psi_{s}(z^{-})+\pt\psi_{n}(z)$.
The desired inclusion and inequality now follow from \prettyref{prop:transportation}
with $(s,\varepsilon,\bar{z})=(q-\nabla\psi_{s}(z^{-}),\varepsilon,z^{-})$
and $\psi=\psi_{n}$.

(b) Clearly, part (a) shows that $(q,\varepsilon)$ is feasible to
\eqref{eq:minimal_cgs}. Assume now that $(r,\delta)$ satisfies $r\in\nabla\psi_{s}(z^{-})+\pt_{\delta}\psi_{n}(z^{-})$,
or equivalently
\[
\psi_{n}(u)\geq\psi_{n}(z^{-})+\langle r-\nabla\psi_{s}(z^{-}),u-z^{-}\rangle-\delta\quad\forall u\in{\cal Z}.
\]
Using the above inequality with $u=z$ and the definitions of $q$
and $\varepsilon$, we then conclude that
\begin{align*}
\lam\|q\|^{2}+2\varepsilon & =\frac{1}{\lam}\|z-z^{-}\|^{2}+2\left[\psi_{n}(z^{-})-\psi_{n}(z)+\inner{q-\nabla\psi_{s}(z^{-})}{z-z^{-}}\right]\\
 & =2\left[\psi_{n}(z^{-})-\psi_{n}(z)-\inner{\nabla\psi_{s}(z^{-})}{z-z^{-}}\right]-\frac{1}{\lam}\|z-z^{-}\|^{2}\\
 & \leq2\delta-2\left\langle r,z-z^{-}\right\rangle -\frac{1}{\lam}\|z-z^{-}\|^{2}\\
 & =2\delta-2\lam\left\langle r,q\right\rangle -\lam\|q\|^{2}\\
 & \leq2\delta+\lam\|r\|^{2}+\lam\|q\|^{2}-\lam\|q\|^{2}=\lam\|r\|^{2}+2\delta,
\end{align*}
where the last inequality follows from the inequality $2\left\langle a,b\right\rangle \leq\|a\|^{2}+\|b\|^{2}$
for every $a,b\in{\cal Z}$. Since $(r,\delta)$ are feasible to \eqref{eq:minimal_cgs},
the result follows.

(c) Using \eqref{eq:pgm_descent} and the definitions of $q$ and
$\varepsilon$ yield
\begin{align*}
\lam\|q\|^{2}+2\varepsilon & =2\left[\psi_{n}(z^{-})-\psi_{n}(z)-\inner{\nabla\psi_{s}(z^{-})}{z-z^{-}}\right]-\frac{1}{\lam}\|z^{-}-z\|^{2}\\
 & =2\left[\psi(z^{-})-\psi(z)\right]+2\left[\psi_{s}(z)-\ell_{\psi_{s}}(z;z^{-})\right]-\frac{1}{\lam}\|z^{-}-z\|^{2}\\
 & \leq2\left[\psi(z^{-})-\psi(z)\right]+\left(L-\frac{1}{\lam}\right)\|z^{-}-z\|^{2}.
\end{align*}
\end{proof}
\begin{proof}[Proof of \prettyref{prop:gradient method}.]
Define the quantities 
\begin{equation}
\Psi_{\lambda}=\Psi_{\lambda,k}:=g+\frac{1}{2\lambda}\|\cdot-z_{k-1}\|^{2},\quad r_{k}:=\frac{z_{k-1}-z_{k}}{\lambda},\label{eq:psilambda}
\end{equation}
and note that $\nabla\Psi_{\lambda}(z_{k-1})=\nabla g(z_{k-1})$,
and that $\Psi_{\lambda}$ is convex due to \ref{asmp:nco2} and the
assumption $\lambda<1/m$. Hence \prettyref{prop:transportation}
and $\nabla g(z_{k-1})=\nabla\Psi_{\lambda}(z_{k-1})\in\partial_{\varepsilon_{k}}\Psi_{\lambda}(z_{k})$,
where $\varepsilon_{k}=\Psi_{\lambda}(z_{k})-\Psi_{\lambda}(z_{k-1})-\langle\nabla\Psi_{\lam}(z_{k-1}),z_{k}-z_{k-1}\rangle\geq0$.
The previous inclusion combined with the optimality of $z_{k}$ and
definition of $r_{k}$ imply that $r_{k}\in\partial h(z_{k})+\partial_{\varepsilon_{k}}\Psi_{\lambda}(z_{k})\subset\partial_{\varepsilon_{k}}(h+\Psi_{\lambda})(z_{k})$
where the last inclusion follows immediately from the definition of
the operator $\partial_{\varepsilon_{k}}$ and convexity of $h$.
Hence, since $(\tilde{\varepsilon}_{k},\tilde{v}_{k})=\lambda(\varepsilon_{k},r_{k})$
(see \eqref{eq:statCGM} and \eqref{eq:psilambda}), it follows from
the above inclusion and the definition of $\Psi_{\lambda}$ that the
triple $(z_{k},\tilde{v}_{k},\tilde{\varepsilon}_{k})$ satisfies
the inclusion in \eqref{eq:GIPPF} with $\phi=g+h$ and $\lambda_{k}=\lambda$. 

Now, to prove that the inequality in \eqref{eq:err_crit_GIPP} holds,
first note that the definitions of $\varepsilon_{k}$ and $\Psi_{\lambda}$
together with property \ref{asmp:nco2}, imply that $\varepsilon_{k}\leq(\lambda M+1)\|z_{k-1}-z_{k}\|^{2}/(2\lambda).$
Combining the previous inequality with the relations $\tilde{v}_{k}=z_{k-1}-z_{k}$
and $\tilde{\varepsilon}_{k}=\lambda\varepsilon_{k}$, we obtain 
\begin{align*}
\|\tilde{v}_{k}\|^{2}+2\tilde{\varepsilon}_{k} & =\|z_{k-1}-z_{k}\|^{2}+2\lambda\varepsilon_{k}\leq\|z_{k-1}-z_{k}\|^{2}+(\lambda M+1)\|z_{k-1}-z_{k}\|^{2}\\[2mm]
 & =(\lambda M+2)\|z_{k-1}-z_{k}\|^{2}=\frac{\lambda M+2}{4}\|z_{k-1}-z_{k}+\tilde{v}_{k}\|^{2}.
\end{align*}
Hence, since $\lambda M<2$, we conclude that $\sigma=(\lambda M+2)/4<1$
and that \eqref{eq:err_crit_GIPP} holds. 
\end{proof}
\newpage{}

\chapter{Properties of the ACGM}

\label{app:acgm_vartn_props}

This appendix presents important properties and proofs related to
the ACGM in \prettyref{chap:background}. Throughout this appendix,
we assume that the iterates 
\[
\{(x_{k},y_{k},r_{k},\eta_{k})\}_{k\geq1},\quad\{(\tau_{k},a_{k},A_{k},\gamma_{k},q_{k},\Gamma_{k})\}_{k\geq1},
\]
are generated by the ACGM and the quantities $\mu$, $\{\lam_{k}\}_{k\geq1}$,
and $\psi$ are from its input and initialization, respectively. 

We first present some basic properties involving the function pairs
$\{(\gamma_{k},q_{k})\}_{k\geq1}$. 
\begin{lem}
\label{lem:qK_gammaK}The following statements hold for every $k\geq1$:
\begin{itemize}
\item[(a)] $\gamma_{k-1}(y_{k})=q_{k}(y_{k})$ and $\gamma_{k-1}\leq q_{k}\leq\psi$;
\item[(b)] it holds that 
\[
\min_{u\in{\cal Z}}\left\{ q_{k}(u)+\frac{1}{2\lam_{k}}\|u-\tilde{x}_{k-1}\|^{2}\right\} =\min_{u\in{\cal Z}}\left\{ \gamma_{k-1}(u)+\frac{1}{2\lam_{k}}\|u-\tilde{x}_{k-1}\|^{2}\right\} .
\]
 
\end{itemize}
\end{lem}

\begin{proof}
(a) The fact that $\gamma_{k-1}(y_{k})=q_{k}(y_{k})$ is immediate
from the definitions of $\gamma_{k}$ and $q_{k}$. The fact that
$\psi\geq q_{k}$ follows from the assumption that $\psi_{s}\in{\cal F}_{\mu}(Z)$.
To show that $\gamma_{k-1}\leq q_{k}$, observe that the optimality
of $y_{k}$ and the fact that $q_{k}\in{\cal F}_{\mu}(Z)$ imply that
\begin{align*}
 & \lam_{k}q_{k}(y_{k})+\frac{1}{2}\|y_{k}-\tilde{x}_{k-1}\|^{2}+\left(\frac{\lam_{k}\mu+1}{2}\right)\|u-y_{k}\|^{2}\\
 & \leq\lam_{k}q_{k}(u)+\frac{1}{2}\|u-\tilde{x}_{k-1}\|^{2}\quad\forall u\in{\cal Z}.
\end{align*}
Rearranging terms and using the definition $\gamma_{k-1}$, we conclude
that
\begin{align*}
q_{k}(u) & \geq q_{k}(y_{k})+\frac{1}{2\lam_{k}}\|y_{k}-\tilde{x}_{k-1}\|^{2}+\left(\frac{\lam_{k}\mu+1}{2\lam_{k}}\right)\|u-y_{k}\|^{2}-\frac{1}{2\lam_{k}}\|u-\tilde{x}_{k-1}\|^{2}\\
 & =q_{k}(y_{k})+\frac{1}{2\lam_{k}}\left[\|y_{k}-\tilde{x}_{k-1}\|^{2}+\|u-y_{k}\|^{2}-\|u-\tilde{x}_{k-1}\|^{2}\right]+\frac{\mu}{2}\|u-y_{k}\|^{2}\\
 & =q_{k}(y_{k})+\frac{1}{2\lam_{k}}\left(2\|y_{k}-\tilde{x}_{k-1}\|^{2}+2\left\langle \tilde{x}_{k-1}-y_{k},u-\tilde{x}_{k-1}\right\rangle \right)+\frac{\mu}{2}\|u-y_{k}\|^{2}\\
 & =\gamma_{k-1}(u)+\frac{1}{\lam_{k}}\|y_{k}-\tilde{x}_{k-1}\|^{2}\geq\gamma_{k-1}(u)\quad\forall u\in{\cal Z}.
\end{align*}

(b) Recall that $y_{k}$ is an optimal solution of the left problem.
Suppose that $\bar{y}$ is an optimal solution of the right problem.
Since $\gamma_{k}$ is a smooth convex function, the optimality of
$\bar{y}$ and the definition of $\gamma_{k}$ imply that 
\begin{align*}
0 & =\nabla\gamma_{k-1}(\bar{y})+\frac{1}{\lam_{k}}(\bar{y}-\tilde{x}_{k-1})=\frac{1}{\lam_{k}}(\tilde{x}_{k-1}-y_{k})+\mu(\bar{y}-y_{k})+\frac{1}{\lam_{k}}(\bar{y}-\tilde{x}_{k-1})\\
 & =\left(\mu+\frac{1}{\lam_{k}}\right)(\bar{y}-y_{k}),
\end{align*}
which, since $\mu,\lam_{k}>0$, implies that $\bar{y}=y_{k}$.
\end{proof}
We next present properties involving the scalars $\{(\lam_{k},a_{k},A_{k})\}_{k\geq1}$.
\begin{lem}
\label{lem:aK_AK}The following statements hold for every $k\geq1$:
\begin{itemize}
\item[(a)] $a_{k}^{2}=\tau_{k}A_{k+1}$;
\item[(b)] it holds that
\[
A_{k}\geq\max\left\{ \frac{1}{4}\left(\sum_{i=1}^{k}\sqrt{\lam_{i-1}}\right)^{2},\lam_{1}\prod_{i=2}^{k}\left(1+\sqrt{\frac{\lam_{i-1}\mu}{2}}\right)^{2}\right\} .
\]
\end{itemize}
\end{lem}

\begin{proof}
(a) Let $k\geq1$ be fixed. It is easy to see that $a_{k}$ is a root
of the quadratic function $x\mapsto x^{2}-\tau_{k}x+\tau_{k}A_{k}$
and hence, using the definitions of $\tau_{k}$ and the update rule
of $A_{k+1}$, it holds that 
\begin{align*}
0 & =a_{k}^{2}-\tau_{k}(a_{k}+A_{k})=a_{k}^{2}-\tau_{k}A_{k+1},
\end{align*}
which implies the desired identity.

(b) We first make the observation that
\begin{equation}
a_{k-1}=\frac{\tau_{k-1}+\sqrt{\tau_{k-1}^{2}+4\tau_{k-1}A_{k-1}}}{2}\geq\frac{\tau_{k-1}}{2}+\sqrt{\tau_{k-1}A_{k-1}}\label{eq:ak_lower_bd}
\end{equation}

We now show that $A_{k}$ is bounded below by the first term in the
max. Using \eqref{eq:ak_lower_bd} and the update rule for $A_{k}$,
it holds that 
\begin{align*}
A_{k} & =A_{k-1}+a_{k-1}\geq\frac{\tau_{k-1}}{2}+\sqrt{\tau_{k-1}A_{k-1}}+A_{k-1}\\
 & \geq\left(\sqrt{A_{k-1}}+\frac{\sqrt{\tau_{k-1}}}{2}\right)^{2}
\end{align*}
which, by taking square roots on both sides, yields
\[
\sqrt{A_{k}}\geq\sqrt{A_{k-1}}+\frac{\sqrt{\tau_{k-1}}}{2}\geq\sqrt{A_{k-1}}+\frac{\sqrt{\lam_{k-1}}}{2}.
\]
 Applying this relationship recursively, and squaring the resulting
relation yields the desired bound on $A_{k}$.

We next show that $A_{k}$ is bounded below by the second term in
the max on the right-hand-side. If $k=1$, the inequality follows
immediately from the fact that $A_{1}=\lam_{1}$. Instead, suppose
that $k\geq2$. Using \eqref{eq:ak_lower_bd} and the update rule
for $A_{k}$, it holds that 
\begin{align*}
A_{k}= & A_{k-1}+a_{k-1}\geq\frac{\tau_{k-1}}{2}+\sqrt{\tau_{k-1}A_{k-1}}+A_{k-1}\\
= & \left(\sqrt{A_{k-1}}+\sqrt{\frac{\tau_{k-1}}{2}}\right)^{2}+\frac{\tau_{k-1}}{4}\\
\geq & \left(\sqrt{A_{k-1}}+\sqrt{\frac{\mu\lam_{k-1}A_{k-1}}{2}}\right)^{2}+\frac{\mu\lam_{k-1}A_{k-1}}{4}\\
= & A_{k-1}\left[\left(1+\sqrt{\frac{\mu\lam_{k-1}}{2}}\right)^{2}+\frac{\mu\lambda_{k-1}A_{k-1}}{4}\right]\geq A_{k-1}\left(1+\sqrt{\frac{\mu\lam_{k-1}}{2}}\right)^{2}.
\end{align*}
Applying the above relationship recursively yields the desired relation
\[
A_{k}\geq A_{1}\prod_{i=2}^{k}\left(1+\sqrt{\frac{\lam_{i-1}\mu}{2}}\right)^{2}=\lam_{1}\prod_{i=2}^{k}\left(1+\sqrt{\frac{\lam_{i-1}\mu}{2}}\right)^{2}.
\]
\end{proof}
We now present properties involving the iterates $\{(x_{k},y_{k})\}_{k\geq1}$
generated by the method.
\begin{lem}
\label{lem:xK_yK}The following statements hold for every $k\geq1$:
\begin{itemize}
\item[(a)] $\Gamma_{k}\in{\cal F}_{\mu}({\cal Z})$ is a quadratic function
and $\Gamma_{k}\leq\psi$;
\item[(b)] $x_{k}=\argmin_{u\in{\cal Z}}\left\{ A_{k}\Gamma_{k}(u)+\|u-x_{0}\|^{2}/2\right\} $;
\item[(c)] if there exists $L_{k}>0$ satisfying \eqref{eq:acgm_descent}, then
it holds that
\[
A_{k}\psi(y_{k})\leq\min_{u\in{\cal Z}}\left\{ A_{k}\Gamma_{k}(u)+\frac{1}{2}\|u-x_{0}\|^{2}\right\} .
\]
\end{itemize}
\end{lem}

\begin{proof}
(a) Observe that recursively applying the definition of $\Gamma_{k}$
yields the identity $\Gamma_{k}=\sum_{i=0}^{k-1}a_{i}\gamma_{i}/(\sum_{i=0}^{k-1}a_{i})$,
which shows that $\Gamma_{k}$ is a convex combination of the functions
$\{\gamma_{i}\}_{i=0}^{k-1}$. The desired conclusion now follows
from the definition of $\gamma_{k}$ and \prettyref{lem:qK_gammaK}(a).

(b) We proceed by induction on $k$. The case of $k=0$ is obvious.
Suppose instead that $x_{k-1}=\argmin_{u\in{\cal Z}}\left\{ A_{k-1}\Gamma_{k-1}(u)+\|u-x_{0}\|^{2}/2\right\} $
for some $k\geq2$. The optimality of $x_{k-1}$ implies that
\begin{equation}
x_{k-1}-x_{0}+A_{k-1}\nabla\Gamma_{k-1}(x_{k-1})=0.\label{eq:opt_xk}
\end{equation}
Moreover, since $\Gamma_{k}\in{\cal F}_{\mu}({\cal Z})$ is a quadratic
function (see part (a)), it holds that
\begin{equation}
\nabla\Gamma_{k}(x_{k})=\nabla\Gamma_{k}(x_{k-1})+\mu(x_{k}-x_{k-1}).\label{eq:GammaK_grad}
\end{equation}
Let us now verify the optimality condition on $x_{k}$. Using \eqref{eq:GammaK_grad},
\eqref{eq:opt_xk}, the update rule for $x_{k}$, the definition of
$\gamma_{k-1}$, and our hypothesis, we have that
\begin{align*}
 & x_{k}-x_{0}+A_{k}\nabla\Gamma_{k}(x_{k})\\
 & =x_{k}-x_{0}+A_{k}\left[\nabla\Gamma_{k}(x_{k-1})+\mu(x_{k}-x_{k-1})\right]\\
 & =x_{k}-x_{0}+A_{k-1}\nabla\Gamma_{k-1}(x_{k-1})+a_{k-1}\nabla\gamma_{k-1}(x_{k-1})+\mu A_{k}(x_{k}-x_{k-1})\\
 & =\left[x_{k}-x_{k-1}\right]+\left[x_{k-1}-x_{0}+A_{k-1}\nabla\Gamma_{k-1}(x_{k-1})\right]+\\
 & \qquad\left[a_{k-1}\nabla\gamma_{k-1}(x_{k-1})+\mu A_{k}(x_{k}-x_{k-1})\right]\\
 & =(1+\mu A_{k})(x_{k}-x_{k-1})+a_{k-1}\nabla\gamma_{k-1}(x_{k-1})\\
 & =-a_{k-1}\left[\frac{1}{\lam_{k}}(\tilde{x}_{k-1}-y_{k})+\mu(x_{k-1}-y_{k})\right]+a_{k-1}\nabla\gamma_{k-1}(x_{k-1})\\
 & =-a_{k-1}\nabla\gamma_{k-1}(z_{k-1})+a_{k-1}\nabla\gamma_{k-1}(x_{k-1})=0,
\end{align*}
which implies that $x_{k}=\argmin_{u\in{\cal Z}}\left\{ A_{k}\Gamma_{k}(u)+\|u-x_{0}\|^{2}/2\right\} $.

(c) We proceed by induction by $k$. The case of $k=0$ is obvious.
Suppose instead that 
\[
A_{k-1}\psi(y_{k-1})\leq\min_{u\in{\cal Z}}\left\{ A_{k-1}\Gamma_{k-1}(u)+\frac{1}{2}\|u-x_{0}\|^{2}\right\} 
\]
 for some $k\geq2$. Using the fact that $\Gamma$ is $\mu$-strongly
convex, the optimality of $x_{k-1}$ in part (b), and our hypothesis,
we have that 
\begin{align}
 & \psi(y_{k-1})+\frac{A_{k-1}\mu+1}{2}\|u-x_{k-1}\|^{2}\nonumber \\
 & \leq A_{k-1}\Gamma_{k-1}(x_{k-1})+\frac{1}{2}\|x_{k-1}-x_{0}\|^{2}+\left(\frac{A_{k-1}\mu+1}{2}\right)\|u-x_{k-1}\|^{2}\nonumber \\
 & \leq A_{k-1}\Gamma_{k-1}(u)+\frac{1}{2}\|u-x_{0}\|^{2},\label{eq:acgm_subdescent}
\end{align}
for every $u\in{\cal Z}$. Combining \eqref{eq:acgm_subdescent},
\prettyref{lem:aK_AK}(a), \prettyref{lem:qK_gammaK}(a), and the
fact that $\gamma_{k-1}$ is convex yields
\begin{align}
 & \min_{u\in{\cal Z}}\left\{ A_{k}\Gamma_{k}(u)+\frac{1}{2}\|u-x_{0}\|^{2}\right\} \nonumber \\
= & \min_{u\in{\cal Z}}\left\{ A_{k-1}\Gamma_{k-1}(u)+a_{k-1}\gamma_{k-1}(u)+\frac{1}{2}\|u-x_{0}\|^{2}\right\} \nonumber \\
\geq & \min_{u\in{\cal Z}}\left\{ A_{k-1}\psi(y_{k-1})+\left(\frac{A_{k-1}\mu+1}{2}\right)\|u-x_{k-1}\|^{2}+a_{k-1}\gamma_{k-1}(x)\right\} \nonumber \\
\geq & \min_{u\in{\cal Z}}\left\{ A_{k-1}\gamma_{k-1}(y_{k-1})+\left(\frac{A_{k-1}\mu+1}{2}\right)\|u-x_{k-1}\|^{2}+a_{k-1}\gamma_{k-1}(x)\right\} \nonumber \\
= & \min_{u\in{\cal Z}}\left\{ (A_{k-1}+a_{k-1})\gamma_{k-1}\left(\underbrace{\frac{A_{k-1}y_{k-1}+a_{k-1}u}{A_{k-1}+a_{k-1}}}_{:=\tilde{u}}\right)+\left(\frac{A_{k-1}\mu+1}{2}\right)\|u-x_{k-1}\|^{2}\right\} \nonumber \\
= & \min_{\tilde{u}\in{\cal Z}}\left\{ A_{k}\gamma_{k-1}(\tilde{u})+\left(\frac{A_{k-1}\mu+1}{2}\right)\left(\frac{A_{k}^{2}}{a_{k-1}^{2}}\right)\|\tilde{u}-\tilde{x}_{k-1}\|^{2}\right\} \nonumber \\
= & A_{k}\min_{\tilde{u}\in{\cal Z}}\left\{ \gamma_{k-1}(\tilde{u})+\left(\frac{A_{k-1}\mu+1}{2}\right)\left(\frac{A_{k}}{a_{k-1}^{2}}\right)\|\tilde{u}-\tilde{x}_{k-1}\|^{2}\right\} \nonumber \\
= & A_{k}\min_{\tilde{u}\in{\cal Z}}\left\{ \gamma_{k-1}(\tilde{u})+\frac{1}{2\lam_{k}}\|\tilde{u}-\tilde{x}_{k-1}\|^{2}\right\} .\label{eq:acgm_bd1}
\end{align}
On the other hand, using \eqref{eq:acgm_descent}, the definition
of $y_{k}$, and \prettyref{lem:qK_gammaK}(b), it holds that
\begin{align}
\psi(y_{k}) & \leq l_{\psi_{s}}(y_{k};\tilde{x}_{k-1})+\psi_{n}(y_{k})+\frac{L_{k}}{2}\|y_{k}-\tilde{x}_{k-1}\|^{2}\nonumber \\
 & =q_{k}(y_{k})+\frac{1}{2}\|y_{k}-\tilde{x}_{k-1}\|^{2}+\frac{1}{2}\left(\left[L_{k}-\mu\right]-\frac{1}{\lam_{k}}\right)\|y_{k}-\tilde{x}_{k-1}\|^{2}\nonumber \\
 & \leq q_{k}(y_{k})+\frac{1}{2\lam_{k}}\|y_{k}-\tilde{x}_{k-1}\|^{2}\nonumber \\
 & =\min_{u\in{\cal Z}}\left\{ q_{k}(u)+\frac{1}{2\lam_{k}}\|u-\tilde{x}_{k-1}\|^{2}\right\} \nonumber \\
 & =\min_{u\in{\cal Z}}\left\{ \gamma_{k-1}(u)+\frac{1}{2\lam_{k}}\|u-\tilde{x}_{k-1}\|^{2}\right\} .\label{eq:acgm_bd2}
\end{align}
Combining \eqref{eq:acgm_bd1} and \eqref{eq:acgm_bd2}, we conclude
that
\[
A_{k}\psi(y_{k})\leq\min_{u\in{\cal Z}}\left\{ A_{k}\Gamma_{k}(u)+\frac{1}{2}\|u-x_{0}\|^{2}\right\} .
\]
\end{proof}

We are now ready to give the proofs of \prettyref{prop:acgm_vartn}
and \prettyref{prop:acgm_vartn_ext}.
\begin{proof}[Proof of \prettyref{prop:acgm_vartn}.]

(a) The optimality of $x_{k}$ and the definition of $r_{k}$ imply
that $r_{k}\in\pt\Gamma_{k}(x_{k})$. Using the previous inclusion,
the definition of $\eta_{k}$, and the fact that $\psi\geq\Gamma_{k}$
yields
\begin{align*}
\psi(u) & \geq\Gamma_{k}(u)\geq\Gamma(x_{k})+\left\langle r_{k},u-x_{k}\right\rangle =\psi(y_{k})+\left\langle r_{k},u-y_{k}\right\rangle -\eta_{k}
\end{align*}
for every $u\in{\cal Z}$, which is exactly the desired inclusion.
The fact that $\eta\geq0$ follows from the above relationship evaluated
at $u=y_{k}$.

(b) Using \prettyref{lem:xK_yK}(b) and (c) and the definitions of
$\eta_{k}$ and $r_{k}$ yields
\begin{align*}
0 & \leq\Gamma_{k}(x_{k})+\frac{1}{2A_{k}}\|x_{k}-x_{0}\|^{2}-\psi(y_{k})\\
 & =-\eta_{k}-\left\langle r_{k},y_{k}-x_{k}\right\rangle +\frac{1}{2A_{k}}\|x_{k}-x_{0}\|^{2}\\
 & =-\eta_{k}+\left[-\frac{1}{A_{k}}\left\langle x_{0}-x_{k},y_{k}-x_{k}\right\rangle +\frac{1}{2A_{k}}\|x_{k}-x_{0}\|^{2}\right]\\
 & =-\eta_{k}+\left[\frac{1}{2A_{k}}\|y_{k}-x_{0}\|^{2}-\frac{1}{2A_{k}}\|y_{k}-x_{k}\|^{2}\right]\\
 & =-\eta_{k}+\frac{1}{2A_{k}}\|y_{k}-x_{0}\|^{2}-\frac{1}{2A_{k}}\|A_{k}r_{k}+y_{k}-x_{0}\|^{2}
\end{align*}
which, together with the identity $x_{0}=y_{0}$, yields the desired
inequality.

(c) Let $y^{*}$ be an optimal solution of \ref{prb:eq:co}. Using
\prettyref{lem:xK_yK}(a) and (c) it holds that 
\[
A_{k}\psi(y_{k})\leq A_{k}\Gamma_{k}(y^{*})+\frac{1}{2}\|y^{*}-y_{0}\|^{2}\leq A_{k}\psi(y^{*})+\frac{1}{2}\|y^{*}-y_{0}\|^{2},
\]
which implies the desired inequality.
\end{proof}
\begin{proof}[Proof of \prettyref{prop:acgm_vartn_ext}.]

(a) Using \prettyref{lem:xK_yK}(b) we first observe that $(x_{k}-x_{0})/A_{k}\in\Gamma_{k}(x_{k})$
and hence, by \prettyref{lem:xK_yK}(a) and the definition of $r_{k}$,
it holds that
\[
\tilde{r}_{k}=\frac{x_{k}-x_{0}}{A_{k}}+\mu(y_{k}-x_{k})\in\pt\left(\Gamma_{k}-\frac{\mu}{2}\|\cdot-y_{k}\|^{2}\right)(x_{k})
\]
 Using the above inclusion, the fact that $\Gamma_{k}-\mu\|\cdot\|^{2}/2$
is affine, and the definition of $\eta_{k}$, we conclude that for
every $u\in{\cal Z}$ it holds that
\begin{align}
\psi(u)-\frac{\mu}{2}\|u-y_{k}\|^{2} & \geq\Gamma_{k}(u)-\frac{\mu}{2}\|u-y_{k}\|^{2}\nonumber \\
 & =\Gamma_{k}(y_{k})-\frac{\mu}{2}\|y_{k}-x_{k}\|^{2}+\left\langle \tilde{r}_{k},u-x_{k}\right\rangle \nonumber \\
 & =\psi(y_{k})+\left\langle \tilde{r}_{k},u-y_{k}\right\rangle -\tilde{\eta}_{k},\label{eq:gen_Gamma_bd}
\end{align}
which is equivalent to \eqref{eq:acgm_incl}. The fact that $\tilde{\eta}_{k}\geq0$
follows from the above inequality with $u=y_{k}$.

(b) Using \prettyref{lem:xK_yK}(b) and (c) and the definitions of
$\tilde{\eta}_{k}$ and $\tilde{r}_{k}$ yields
\begin{align*}
0 & \leq\Gamma_{k}(x_{k})+\frac{1}{2A_{k}}\|x_{k}-x_{0}\|^{2}-\psi(y_{k})\\
 & =-\tilde{\eta}_{k}-\frac{\mu}{2}\|y_{k}-x_{k}\|^{2}-\left\langle \tilde{r}_{k},y_{k}-x_{k}\right\rangle +\frac{1}{2A_{k}}\|x_{k}-x_{0}\|^{2}\\
 & =-\tilde{\eta}_{k}-\frac{\mu}{2}\|y_{k}-x_{k}\|^{2}+\left[-\frac{1}{A_{k}}\left\langle x_{0}-x_{k},y_{k}-x_{k}\right\rangle +\frac{1}{2A_{k}}\|x_{k}-x_{0}\|^{2}\right]\\
 & =-\tilde{\eta}_{k}-\frac{\mu}{2}\|y_{k}-x_{k}\|^{2}+\left[\frac{1}{2A_{k}}\|y_{k}-x_{0}\|^{2}-\frac{1}{2A_{k}}\|y_{k}-x_{k}\|^{2}\right]\\
 & =-\tilde{\eta}_{k}+\frac{1}{2A_{k}}\|y_{k}-x_{0}\|^{2}-\frac{1}{2A_{k}}\left(1+\mu A_{k}\right)\|y_{k}-x_{k}\|^{2}\\
 & =-\tilde{\eta}_{k}+\frac{1}{2A_{k}}\|y_{k}-x_{0}\|^{2}-\frac{1}{2A_{k}(1+\mu A_{k})}\|A_{k}\tilde{r}_{k}+y_{k}-x_{0}\|^{2}
\end{align*}
which, together with the identity $x_{0}=y_{0}$, yields the desired
inequality.
\end{proof}
We next give the proof of \prettyref{lem:nest_complex}.
\begin{proof}[Proof of \prettyref{lem:nest_complex}.]
Let $j$ be such that 
\[
A_{j}\geq\frac{2(1+\sqrt{\sigma})^{2}}{\sigma}.
\]
Using the triangle inequality, the previous bound on $A_{j}$, the
relation $(a+b)^{2}\leq2a{}^{2}+2b{}^{2}$ for every $a,b\in\r$,
and \prettyref{prop:acgm_vartn}(b), we obtain
\begin{align*}
\|r_{j}\|^{2}+2\eta_{j} & \leq\max\left\{ 1/A_{j}^{2},1/(2A_{j})\right\} \left(\|A_{j}r_{j}\|^{2}+4A_{j}\eta_{j}\right)\\
 & \leq\max\left\{ 1/A_{j}^{2},1/(2A_{j})\right\} \left(2\|A_{j}r_{j}+y_{j}-y_{0}\|^{2}+2\|y_{j}-y_{0}\|^{2}+4A_{j}\eta_{j}\right)\\
 & \leq\max\left\{ (2/A_{\ell})^{2},2/A_{\ell}\right\} \|y_{j}-y_{0}\|^{2}\leq\frac{\sigma}{(1+\sqrt{\sigma})^{2}}\|y_{j}-y_{0}\|^{2}.
\end{align*}
On the other hand, the triangle inequality and simple calculations
yield 
\[
\|y_{j}-y_{0}\|^{2}\leq(1+\sqrt{\sigma})\|y_{0}-y_{j}+r_{j}\|^{2}+\left(1+\frac{1}{\sqrt{\sigma}}\right)\|r_{j}\|^{2}.
\]
Combining the previous bounds, we obtain 
\[
\|r_{j}\|^{2}+2\eta_{j}\leq\frac{\sigma}{1+\sqrt{\sigma}}\|y_{0}-y_{j}+r_{j}\|^{2}+\frac{\sqrt{\sigma}}{1+\sqrt{\sigma}}\|u_{j}\|^{2}
\]
which easily implies \eqref{ineq:Nest_vksigma}. 

Let us now show what conditions on $j$ yield $A_{j}\geq2(1+\sqrt{\sigma})^{2}/\sigma$.
Using the first bound in \prettyref{lem:aK_AK} with $\lam_{k}=1/L$,
it is straightforward to show that the condition 
\[
j\geq\left\lceil \frac{2\sqrt{2L}(1+\sqrt{\sigma})}{\sqrt{\sigma}}\right\rceil 
\]
 suffices. 

On the other hand, using the second bound in \prettyref{lem:aK_AK}
with $\lam_{k}=1/L$ and the bound $\log(1+t)\geq t/2$ for $t\in[0,1]$,
it is straightforward to show that the condition 
\[
j\geq\left\lceil 1+\sqrt{\frac{2L}{\mu}}\log_{1}^{+}\left(\frac{2L\left[1+\sqrt{\sigma}\right]^{2}}{\sigma}\right)\right\rceil 
\]
suffices. The conclusion now follows from combining the previous bounds
on $j$.
\end{proof}
\newpage{}

\chapter{Properties of the S.ACGM and R.ACGM}

\label{app:oth_acgm_props}

This appendix contains proofs related to the S.ACGM and R.ACGM in
\prettyref{chap:spectral,chap:practical}, respectively.

We first give the proof of \prettyref{prop:s_acg_properties}(a).
\begin{proof}[Proof of \prettyref{prop:s_acg_properties}(a)]
Let $\ell$ be the first iteration where 
\begin{equation}
\min\left\{ \frac{A_{\ell}^{2}}{4(1+\mu A_{\ell})},\frac{A_{\ell}}{2}\right\} \geq K_{\theta}^{2}\label{eq:spectral_A_bd}
\end{equation}
and suppose that the method has not stopped with $\pi_{S}=$ \texttt{false}
before iteration $\ell$. We show that it must stop with $\pi_{S}=$
\texttt{true} at the end of the $\ell^{{\rm th}}$ iteration. Combining
the triangle inequality, the successful checks in \prettyref{ln:s_acg_invar}
of the method, \eqref{eq:spectral_A_bd}, and the relation $(a+b)^{2}\leq2a^{2}+2b^{2}$
for all $a,b\in\r,$ we first have that 
\begin{align*}
 & \|\acgU{\ell}\|^{2}+2\eta_{\ell}\\
 & \leq\max\left\{ \frac{1+\mu A_{\ell}}{A_{\ell}^{2}},\frac{1}{2A_{\ell}}\right\} \left(\frac{1}{1+\mu A_{\ell}}\|A_{\ell}\tilde{r}_{\ell}\|^{2}+4A_{\ell}\eta_{\ell}\right)\\
 & \leq\max\left\{ \frac{1+\mu A_{\ell}}{A_{\ell}^{2}},\frac{1}{2A_{\ell}}\right\} \left(\frac{2}{1+\mu A_{\ell}}\|A_{\ell}\tilde{r}_{\ell}+\acgX{\ell}-\acgX 0\|^{2}+2\|\acgX{\ell}-\acgX 0\|^{2}+4A_{\ell}\tilde{\eta}_{\ell}\right)\\
 & \leq\max\left\{ \frac{4(1+\mu A_{\ell})}{A_{\ell}^{2}},\frac{2}{A_{\ell}}\right\} \|\acgX{\ell}-\acgX 0\|^{2}\leq\frac{1}{K_{\theta}^{2}}\|\acgX{\ell}-\acgX 0\|^{2}\leq\theta^{2}\|\acgX{\ell}-\acgX 0\|^{2},
\end{align*}
and hence the method must terminate at the $\ell^{{\rm th}}$ iteration.
We now bound $\ell$ based on the requirement in \eqref{eq:spectral_A_bd}.
Solving for the quadratic in $A_{\ell}$ in the first bound of \eqref{eq:spectral_A_bd},
it is easy to see that $A_{\ell}\geq4\mu K_{\theta}^{2}+2K_{\theta}$
implies \eqref{eq:spectral_A_bd}. On the other hand, for the second
condition in \eqref{eq:spectral_A_bd}, it is immediate that $A_{\ell}\geq2K_{\theta}^{2}$
implies \eqref{eq:spectral_A_bd}. In view of \prettyref{prop:acgm_vartn}(c)
with $\lam_{i}=1/L$ for every $i\geq1$, and the previous two bounds,
it follows that
\[
A_{\ell}\geq\frac{1}{L}\left(1+\sqrt{\frac{\mu}{2L}}\right)^{2(\ell-1)}\geq2K_{\theta}(1+2\mu K_{\theta}^{2})
\]
implies \eqref{eq:spectral_A_bd}. Using the bound $\log(1+t)\geq t/(1+t)$
for $t\geq0$ and the above bound on $\ell$, it is straightforward
to see that $\ell$ is on the same order of magnitude as in \eqref{eq:s_acg_total_compl}.
\end{proof}
We next give the proof of \prettyref{lem:basic_r_acgm_props}.
\begin{proof}[Proof of \prettyref{lem:basic_r_acgm_props}.]
Using our assumption that $f\in{\cal C}_{M}(Z)$, and hence $\psi_{s}\in{\cal C}_{L_{\lam}}(Z)$,
together with the check in \prettyref{ln:r_acgm_check} of \prettyref{alg:r_acgm},
it holds that the stepsizes $\{\lam_{i}\}_{i\geq1}$ in the R.ACGM
are constant with a value of $1/L_{\lam}$. Hence, using \prettyref{prop:acgm_vartn}
and \eqref{eq:r_aipp_acg_inputs}, it holds that 
\begin{equation}
A_{i}\geq\frac{1}{L_{\lam}}\left(1+\sqrt{\frac{1}{2L_{\lam}}}\right)^{2(i-1)}\quad\forall i\geq1.\label{eq:r_acgm_Abd}
\end{equation}
Now, let $\ell$ denote the quantity in \eqref{eq:r_acgm_compl},
and suppose the R.ACGM has performed $\ell$ iterations in which \eqref{eq:r_acg_check1}
and \eqref{eq:r_acg_check2} hold for every $i\leq\ell$. Using \eqref{eq:r_acgm_Abd},
the definition of $C_{\theta,\tau}$ in \eqref{eq:C_theta_tau_def},
and the fact that $\log(1+t)\geq t/2$ for all $t\in[0,1]$, it holds
that
\[
A_{\ell}\geq\frac{1}{L_{\lam}}\left(1+\sqrt{\frac{1}{2L_{\lam}}}\right)^{2(\ell-1)}\geq2C_{\theta,\tau}>2.
\]
Combining the triangle inequality, \eqref{eq:r_acg_check1}, the bounds
$2/A_{\ell}\leq1/C$ and $(2/A_{\ell})^{2}<2/A_{\ell}<1$ from above,
and the relation $(a+b)^{2}\leq2(a^{2}+b^{2})$ for all $a,b\in\R$,
we obtain 
\begin{align*}
\|r_{\ell}\|^{2}+2\eta_{\ell} & \leq\max\{1/A_{\ell}^{2},1/(2A_{\ell})\}(\|A_{\ell}r_{\ell}\|^{2}+4A_{\ell}\eta_{\ell})\\
 & \leq\max\{1/A_{\ell}^{2},1/(2A_{\ell})\}(2\|A_{\ell}r_{\ell}+y_{\ell}-y_{0}\|^{2}+2\|y_{\ell}-x_{0}\|^{2}+4A_{\ell}\eta_{\ell})\\
 & \leq\max\{(2/A_{\ell})^{2},2/A_{\ell}\}\|y_{\ell}-y_{0}\|^{2}\leq\frac{1}{C_{\theta,\tau}}\|y_{\ell}-y_{0}\|^{2}.
\end{align*}
On the other hand, using the triangle inequality and the fact that
$(a+b)^{2}\leq(1+s)a^{2}+(1+1/s)b^{2}$ for every $(a,b,s)\in\R\times\R\times R_{++}$
(under the choice of $s=1/(\sqrt{C}-1)$), we obtain 
\[
\|y_{\ell}-y_{0}\|^{2}\leq\frac{\sqrt{C_{\theta,\tau}}}{\sqrt{C_{\theta,\tau}}-1}\|y_{0}-y_{\ell}+r_{\ell}\|^{2}+\sqrt{C_{\theta,\tau}}\|r_{\ell}\|^{2}.
\]
Combining the previous estimates, we then conclude that 
\[
\|u_{\ell}\|^{2}+2\eta_{\ell}\leq\frac{1}{C_{\theta,\tau}-\sqrt{C_{\theta,\tau}}}\|x_{0}-x_{\ell}+u_{\ell}\|^{2}+\frac{1}{\sqrt{C_{\theta,\tau}}}\|u_{\ell}\|^{2},
\]
which, after a simple algebraic manipulation, easily implies that
\begin{align}
\frac{1}{\sqrt{C_{\theta,\tau}}-1}\|x_{0}-x_{\ell}+u_{\ell}\|^{2} & \geq2\sqrt{C_{\theta,\tau}}\eta_{\ell}+\left(\sqrt{C_{\theta,\tau}}-1\right)\|u_{\ell}\|^{2}\nonumber \\
 & \geq\left(\sqrt{C_{\theta,\tau}}-1\right)\left(\|u_{\ell}\|^{2}+2\eta_{\ell}\right).\label{eq:gipp_sigma_bd}
\end{align}
Using the first term in the maximum of \eqref{eq:C_theta_tau_def}
together with the second inequality of \eqref{eq:gipp_sigma_bd} immediately
implies that \eqref{eq:r_acg_stop1} holds with $j=\ell$. To show
that \eqref{eq:r_acg_stop2} holds at $j=\ell$, observe that the
definition of $\psi$ in \eqref{eq:r_aipp_acg_inputs}, \eqref{eq:r_acg_check2}
with $j=\ell$, the second inequality of \eqref{eq:gipp_sigma_bd},
and the second term in the maximum of \eqref{eq:C_theta_tau_def}
imply that 
\begin{align*}
\lam\left[\phi(x_{0})-\phi(x_{\ell})\right] & \geq\left\langle r_{\ell},y_{0}-y_{\ell}\right\rangle +\eta_{\ell}+\frac{1}{2}\|y_{\ell}-y_{0}\|^{2}\\
 & =\frac{1}{2}\left[\|y_{0}-y_{\ell}+r_{\ell}\|^{2}-\left(\|r_{\ell}\|^{2}+2\eta_{\ell}\right)\right]\\
 & \geq\frac{1}{2}\left[1+\left(\sqrt{C_{\theta,\tau}}-1\right)^{-2}\right]\|y_{0}-y_{\ell}+r_{\ell}\|^{2}\geq\frac{1}{\theta}\|y_{0}-y_{\ell}+r_{\ell}\|^{2}.
\end{align*}
\end{proof}
\newpage{}

\chapter{Properties of the CRP}

\label{app:ref_props}

This appendix contains proofs related to the CRP in \prettyref{chap:unconstr_nco}.

We first give the proof of \prettyref{prop:crp_props}.
\begin{proof}[Proof of \prettyref{prop:crp_props}]

(a) This follows immediately from \prettyref{prop:cgm_ext_vartn}(a)
with $(\psi_{s},\psi_{n},z_{k-1})=(f,h,z)$ and $(q_{k},\varepsilon_{k})=(q_{r},\varepsilon_{r})$.

(b) Using the definition of $\varepsilon_{r}$, it follows that $q_{r}\in\nabla f(z)+\pt_{\varepsilon_{r}}h(z)$
if and only if 
\begin{align*}
h(u) & \geq h(z)+\left\langle q_{r}-\nabla f(z),u-z\right\rangle -\varepsilon_{r}\\
 & =h(z_{r})+\left\langle q_{r}-\nabla f(z),u-z_{r}\right\rangle \quad\forall u\in{\cal Z},
\end{align*}
or equivalently, $q_{r}\in\nabla f(z)+\pt h(z_{r})$. The desired
inclusion now follows from the previous inclusion and the definition
of $v_{r}$. The desired inequality follows from \ref{asmp:nco2}
and \prettyref{prop:cgm_ext_vartn}(b)--(c) with
\[
(\psi_{s},\psi_{n},z_{k-1})=(f,h,z),\quad(q_{k},\varepsilon_{k})=(q_{r},\varepsilon_{r}),\quad L=L_{\lam},\quad\lam=\frac{1}{L_{\lam}}.
\]

(c) Let $(\bar{\rho},\bar{\varepsilon})$ and $(z^{-},\tilde{v},\tilde{\varepsilon})$
satisfying \eqref{eq:rho_eps_approx} be given, and define the function
\begin{equation}
\psi_{s}(u):=f(u)+\frac{1}{2\lam}\|u-z^{-}\|^{2}-\frac{1}{\lam}\left\langle \tilde{v},u\right\rangle \quad\forall u\in Z.\label{eq:refine_aux_defs}
\end{equation}
Clearly, the inclusion in \eqref{eq:rho_eps_approx} holds if and
only if $0\in\pt_{\tilde{\varepsilon}}(\psi_{s}+h)(z)$, or equivalently,
$(\psi_{s}+h)(u)\geq(\psi_{s}+h)(z)-\tilde{\varepsilon}$ for every
$u\in{\cal Z}$. In particular, for $u=z_{r}$, we have $(\psi_{s}+h)(z)-(\psi_{s}+h)(z_{r})\leq\tilde{\varepsilon}.$
Using the previous bound, the second inequality in \eqref{eq:rho_eps_approx},
and \prettyref{prop:cgm_ext_vartn}(c) with 
\[
(z^{-},\psi_{n})=(z,h),\quad L=L_{\lam},\quad\lam=\frac{1}{L_{\lam}},
\]
it holds that there exists $(q_{\psi},\varepsilon_{\psi})\in{\cal Z}\times\r_{+}$
satisfying $q_{\psi}\in\nabla\psi_{s}(z)+\pt_{\varepsilon_{\psi}}h(z)$
and
\[
\|q_{\psi}\|^{2}+L_{\lam}\varepsilon_{\psi}\leq L_{\lam}\left[(\psi_{s}+h)(z)-(\psi_{s}+h)(z_{r})\right]\leq L_{\lam}\tilde{\varepsilon}\leq L_{\lam}\bar{\varepsilon}.
\]
Since the previous inclusion implies that $q_{\psi}+(z^{-}-z+\tilde{v})/\lam\in\nabla f(z)+\pt_{\varepsilon_{\psi}}h(z)$,
it follows from \prettyref{prop:cgm_ext_vartn}(b) with 
\[
(z^{-},q,\varepsilon)=(z,q_{r},\varepsilon_{r}),\quad(\psi_{s},\psi_{n})=(f,h),\quad L=L_{\lam},\quad\lam=\frac{1}{L_{\lam}},
\]
the first inequality in \eqref{eq:rho_eps_approx}, and the triangle
inequality, that
\begin{align}
\|q_{r}\|^{2} & \leq\|q_{r}\|^{2}+2L_{\lam}\varepsilon_{r}\leq\left\Vert q_{\psi}+\frac{1}{\lam}(z^{-}-z+\tilde{v})\right\Vert ^{2}+2L_{\lam}\varepsilon_{\psi}\nonumber \\
 & \leq\left(\|q_{\psi}\|+\frac{1}{\lam}\|z^{-}-z+\tilde{v}\|\right)^{2}+2L_{\lam}\varepsilon_{\psi}\nonumber \\
 & \leq\left(\|q_{\psi}\|^{2}+2L_{\lam}\varepsilon_{\psi}\right)+2\rho\|q_{\psi}\|+\rho^{2}\nonumber \\
 & \leq2L_{\lam}\bar{\varepsilon}+2\rho\sqrt{2L_{\lam}}+\rho^{2}=\left(\bar{\rho}^{2}+L_{\lam}\bar{\varepsilon}\right)^{2},\label{eq:qr_bd}
\end{align}
which implies the second inequality in \eqref{eq:cref_resid_bds}.
On the other hand, using assumption \ref{asmp:nco2}, i.e,. $\nabla f$
is $\max\{m,M\}$-Lipschitz continuous, and the definitions of $v_{r}$
and $q_{r}$ yields
\begin{align*}
\|v_{r}\|-\|q_{r}\| & \leq\|v_{r}-q_{r}\|=\|\nabla f(z_{r})-\nabla f(z)\|\leq\max\{m,M\}\|z_{r}-z\|\\
 & =\frac{\max\{m,M\}}{L_{\lam}}\|q_{r}\|,
\end{align*}
which, together with \eqref{eq:qr_bd}, implies the second inequality
in \eqref{eq:cref_resid_bds}.
\end{proof}
\begin{proof}[Proof of \prettyref{prop:eff_refine}]
(a) This follows from assumptions \ref{asmp:nco1}--\ref{asmp:nco2},
the definition of $\varepsilon_{r}$, and \prettyref{prop:cgm_ext_vartn}(b)
with
\[
(\psi_{s},\psi_{n},z_{k-1})=(f_{\lam},h_{\lam},z),\quad L=L_{\lam},\quad\lam=\frac{1}{L_{\lam}}.
\]

(b) The optimality of $z_{r}$ implies that 
\begin{align*}
\pt h(z_{r}) & \ni-\frac{1}{\lam}\nabla f_{\lam}(z)+\frac{L_{\lam}}{\lam}(z_{r}-z)\\
 & =-\nabla f(z)+\frac{1}{\lam}(z^{-}-z+v)+\frac{L_{\lam}}{\lam}(z-z_{r})\\
 & =v_{r}-\nabla f(z_{r})
\end{align*}
which immediately implies the desired inclusion. To show the desired
inequality, we use part (a), the triangle inequality, assumption \ref{asmp:nco2},
and the definition of $v_{r}$ to conclude that
\begin{align*}
\|v_{r}\| & \leq\frac{1}{\lambda}\|z^{-}-z+v\|+\left\Vert \frac{L_{\lam}}{\lam}(z-z_{r})+\nabla f(z_{r})-\nabla f(z)\right\Vert \\
 & \leq\frac{1}{\lambda}\|z^{-}-z+v\|+\left(\frac{L_{\lam}}{\lam}+\max\left\{ m,M\right\} \right)\|z-z_{r}\|\\
 & \leq\frac{1}{\lambda}\|z^{-}-z+v\|+\left(\frac{L_{\lam}}{\lam}+\max\left\{ m,M\right\} \right)\sqrt{\frac{2\varepsilon_{r}}{L_{\lam}}}\\
 & =\frac{1}{\lambda}\|z^{-}-z+v\|+\left(\frac{1}{\lam}+\frac{\max\left\{ m,M\right\} }{L_{\lam}}\right)\sqrt{2\varepsilon_{r}L_{\lam}}.
\end{align*}

(c) Using the inclusion in \eqref{eq:prp_incl} and the definition
of $\varepsilon_{r}$, it holds that 
\begin{align*}
\varepsilon_{r} & =(f_{\lambda}+h_{\lambda})(z)-(f_{\lambda}+h_{\lambda})(z_{r})\\
 & =\lambda(f+h)(z)-\left[\lambda(f+h)(z_{r})+\frac{1}{2}\|z-z_{r}\|^{2}\right]+\left\langle v,z_{r}-z\right\rangle \\
 & \leq\varepsilon.
\end{align*}
Combining the above bound with part (b) and the inequalities in \eqref{eq:prp_incl}
yields
\begin{align*}
\|v_{r}\| & \leq\frac{1}{\lambda}\|z^{-}-z+v\|+\left(\frac{1}{\lam}+\frac{\max\left\{ m,M\right\} }{L_{\lam}}\right)\sqrt{2\varepsilon_{r}L_{\lam}}\\
 & \leq\bar{\rho}+\left(\frac{1}{\lam}+\frac{\max\left\{ m,M\right\} }{L_{\lam}}\right)\sqrt{2\lam\bar{\varepsilon}L_{\lam}}.
\end{align*}
\end{proof}
\newpage{}

\chapter{Convex Functions and Convex Sets}

\label{app:cvx}

This appendix consists of several appendices that contain results
related to convex functions and convex sets.

\section{Properties of Subdifferentials}

The below technical result presents a fact about approximate subdifferentials,
and its proof can be found, for example, in \citep[Lemma A.2]{Melo2020}.
\begin{lem}
\label{lem:auxNewNest2}Let proper function $\tilde{\phi}:\Re^{n}\to(-\infty,\infty]$,
scalar $\tilde{\sigma}\in(0,1)$ and $(z_{0},z_{1})\in Z\times\dom\tilde{\phi}$
be given, and assume that there exists $(v_{1},\varepsilon_{1})$
such that 
\begin{gather}
v_{1}\in\partial_{\varepsilon_{1}}\left(\tilde{\phi}+\frac{1}{2}\|\cdot-z_{0}\|^{2}\right)(z_{1}),\quad\|v_{1}\|^{2}+2\varepsilon_{1}\leq\tilde{\sigma}^{2}\|v+z_{0}-z_{1}\|^{2}.\label{eq:aux_prox_incl}
\end{gather}
Then, for every $z\in Z$ and $s>0$, we have 
\[
\tilde{\phi}(z_{1})+\frac{1}{2}\left[1-\tilde{\sigma}^{2}(1+s^{-1})\right]\|v_{1}+z_{0}-z_{1}\|^{2}\le\tilde{\phi}(z)+\frac{s+1}{2}\|z-z_{0}\|^{2}.
\]
\end{lem}

\section{Properties of Convex Cones}

The first result presents some well-known (see, for example, \citep[Chapter 6]{Beck2017}
and \citep[Example 11.4]{Rockafellar2009}) properties about the projection
and distance functions over a closed convex set. 
\begin{lem}
\label{lem:dist_props} Let ${\cal K}\subseteq{\cal Z}$ be a closed
convex set. Then the following properties hold: 
\begin{itemize}
\item[(a)] for every $u,z\in{\cal Z}$, we have $\|\Pi_{{\cal K}}(u)-\Pi_{{\cal K}}(u)\|\leq\|u-z\|$; 
\item[(b)] the function $d(\cdot):={\rm dist}^{2}(\cdot,{\cal K})/2$ is differentiable,
and its gradient, given by 
\begin{equation}
\nabla d(u)=u-\Pi_{{\cal K}}(u)\in N_{{\cal K}}(\Pi_{{\cal K}}(u))\quad\forall u\in\rn,\label{eq:proj_incl}
\end{equation}
is 1-Lipschitz continuous;
\item[(c)] if ${\cal K}$ is a cone, then holds that $u\in N_{{\cal K}^{+}}(p)$
if and only if $\inner up=0$, $u\in-{\cal K}$, and $p\in{\cal K}^{+}$. 
\end{itemize}
\end{lem}

The next result presents a well-known fact (see, for example, \citep[Sub-subsection 2.13.2]{Dattorro2005})
about closed convex cones.
\begin{lem}
\label{lem:cone_generator} For any closed convex cone ${\cal K}$,
we have that $x\in\intr{\cal K}$ if and only if 
\begin{align*}
\inf_{p\in{\cal K}^{+}} & \left\{ \inner px:\|p\|=1\right\} >0.
\end{align*}
\end{lem}

\section{Properties of Max Functions}

\label{app:smoothing}

This appendix contains results about functions that can be described
be as the maximum of a family of differentiable functions.

The technical lemma below, which is a special case of \citep[Theorem 10.2.1]{Facchinei2008},
presents a key property about max functions.
\begin{lem}
\label{lem:diff_danskin} Assume that the triple $(\Psi,X,Y)$ satisfies
\ref{asmp:mco_a1}--\ref{asmp:mco_a2} in \prettyref{sec:minmax_prelim_asmp}
with $\Phi=\Psi$. Moreover, define 
\begin{equation}
q(x):=\sup_{y\in Y}\Psi(x,y),\quad Y(x):=\{y\in Y:\Psi(x,y)=q(x)\},\quad\forall x\in X.\label{eq:qYbar_def}
\end{equation}
Then, for every $(x,d)\in X\times{\cal X}$, it holds that 
\[
q'(x;d)=\max_{y\in Y(x)}\inner{\nabla_{x}\Psi(x;y)}d.
\]
Moreover, if $Y(x)$ reduces to a singleton, say $Y(x)=\{y(x)\}$,
then $q$ is differentiable at $x$ and $\nabla q(x)=\nabla_{x}\Psi(x,y(x))$. 
\end{lem}

Under assumptions \ref{asmp:mco_a1}--\ref{asmp:mco_a4} in \prettyref{sec:minmax_prelim_asmp},
the next result establishes Lipschitz continuity of the gradient of
$q$. It is worth mentioning that it generalizes related results in
\citep[Theorem 5.26]{Beck2017} (which covers the case where $\Psi$
is bilinear) and \citep[Proposition 4.1]{Monteiro2013} (which makes
the stronger assumption that $\Psi(\cdot,y)$ is convex for every
$y\in Y$).
\begin{prop}
\label{prop:Psi_global_ext} If the triple $(\Psi,X,Y)$ satisfies
\ref{asmp:mco_a1}--\ref{asmp:mco_a4} in \prettyref{sec:minmax_prelim_asmp}
with $\Phi=\Psi$ and it holds that $-\Psi(x,\cdot)\in{\cal F}_{\mu}(Y)$
for some $\mu>0$ and every $x\in X$, then the following properties
hold: 
\begin{itemize}
\item[(a)] the function $y(\cdot)$ given by 
\[
y(x):=\argmax_{y\in Y}\Psi(x,y)\quad\forall x\in X
\]
is $Q_{\mu}$-Lipschitz continuous on $X$, where 
\begin{equation}
Q_{\mu}:=\frac{L_{y}}{\mu}+\sqrt{\frac{L_{x}+m}{\mu}};\label{eq:Q_mu}
\end{equation}
\item[(b)] $\nabla q(\cdot)$ is $L_{\mu}$-Lipschitz continuous on $X$, where
$q$ is as in \eqref{eq:qYbar_def} and 
\begin{equation}
L_{\mu}:=L_{y}Q_{\mu}+L_{x}.\label{eq:L_mu}
\end{equation}
\end{itemize}
\end{prop}

\begin{proof}
(a) Let $x,\tilde{x}\in X$ be given and denote $(y,\tilde{y})=(y(x),y(\tilde{x}))$.
Define 
\begin{equation}
\alpha(u):=\Psi(u,y)-\Psi(u,\tilde{y})\quad\forall u\in X.\label{eq:alpha_def}
\end{equation}
and observe that the optimality conditions of $y$ and $\tilde{y}$
imply that 
\begin{equation}
\alpha(x)\geq\frac{\mu}{2}\|y-\tilde{y}\|^{2},\quad-\alpha(\tilde{x})\geq\frac{\mu}{2}\|y-\tilde{y}\|^{2}.\label{eq:alpha_lower_bd}
\end{equation}
Using \eqref{eq:alpha_lower_bd}, \eqref{eq:sp_lower_curv}, \eqref{eq:M_L_xy},
\eqref{eq:sp_upper_curv}, and the Cauchy-Schwarz inequality, we conclude
that 
\begin{align*}
\mu\|y-\tilde{y}\|^{2}\leq\alpha(x)-\alpha(\tilde{x}) & \leq\left\langle \nabla_{x}\Psi(x,y)-\nabla_{x}\Psi(x,\tilde{y}),x-\tilde{x}\right\rangle +\frac{L_{x}+m}{2}\|x-\tilde{x}\|^{2}\\
 & \leq\|\nabla_{x}\Psi(x,y)-\nabla_{x}\Psi(x,\tilde{y})\|\cdot\|x-\tilde{x}\|+\frac{L_{x}+m}{2}\|x-\tilde{x}\|^{2}\\
 & \leq L_{y}\|y-\tilde{y}\|\cdot\|x-\tilde{x}\|+\frac{L_{x}+m}{2}\|x-\tilde{x}\|^{2}.
\end{align*}
Considering the above as a quadratic inequality in $\|\tilde{y}-y\|$
yields the bound 
\begin{align*}
\|y-\tilde{y}\| & \leq\frac{1}{2\mu}\left[L_{y}\|x-\tilde{x}\|+\sqrt{L_{y}^{2}\|x-\tilde{x}\|^{2}+4\mu(L_{x}+m)\|x-\tilde{x}\|^{2}}\right]\\
 & \leq\left[\frac{L_{y}}{\mu}+\sqrt{\frac{L_{x}+m}{\mu}}\right]\|x-\tilde{x}\|=Q_{\mu}\|x-\tilde{x}\|
\end{align*}
which is the conclusion of (a).

(b) Let $x,\tilde{x}\in X$ be given and denote $(y,\tilde{y})=(y(x),y(\tilde{x}))$.
Using part (a), \prettyref{lem:diff_danskin}, and \eqref{eq:M_L_xy}
we have that 
\begin{align*}
\|\nabla q(x)-\nabla q(\tilde{x})\| & =\|\nabla_{x}\Psi(x,y)-\nabla_{x}\Psi(\tilde{x},\tilde{y})\|\\
 & \leq\|\nabla_{x}\Psi(x,y)-\nabla_{x}\Psi(x,\tilde{y})\|+\|\nabla_{x}\Psi(x,\tilde{y})-\nabla_{x}\Psi(\tilde{x},\tilde{y})\|\\
 & \leq L_{y}\|y-\tilde{y}\|+L_{x}\|x-\tilde{x}\|\leq(L_{y}Q_{\mu}+L_{x})\|x-\tilde{x}\|=L_{\mu}\|x-\tilde{x}\|,
\end{align*}
which is the conclusion of (b). 
\end{proof}
\newpage{}

\chapter{Notions of Stationary Points}

\label{app:statn}

This appendix contains technical results about different notions of
stationary points in an optimization problem.

\section{Directional and Primal-Dual Stationarity}

The main goal of this appendix is to prove \prettyref{prop:dd_sp_cvx,prop:nco_refine},
which are used in the proofs of \prettyref{prop:impl1_statn,prop:prox_statn,prop:ne_statn_pt}
given in \prettyref{app:statn_notions}. Several technical lemmas
are stated and proved to accomplish the above goal. Some of these
technical results (e.g. \prettyref{lem:gen_conv_tech}(a) and \prettyref{lem:sion_minimax})
are stated without proof as they are broadly available in the convex
analysis literature. Others (e.g. \prettyref{lem:gen_conv_tech}(b)
and \prettyref{lem:compl_approx2}) are given proofs because we could
not find a suitable reference for them. 

The first technical lemma presents some general results about proper
convex functions and nonempty closed convex sets.
\begin{lem}
\label{lem:gen_conv_tech}Let $\psi$ be a convex function and let
$C\subseteq{\cal X}$ be a nonempty closed convex set. Then, the following
statements hold: 
\begin{itemize}
\item[(a)] $\inf_{\|d\|\le1}\sigma_{C}(d)=\left[-\min_{u\in C}\|u\|\right]$; 
\item[(b)] if $C\cap\ri(\dom\psi)\neq\emptyset$, then $\inf_{x\in C}\cl\psi(x)=\inf_{x\in C}\psi(x)<\infty$. 
\end{itemize}
\end{lem}

\begin{proof}
(a) See, for example, the proof of \citep[Lemma 5.1]{Burke1991} with
$g=0$.

(b) Define $\psi_{*}:=\inf_{x\in C}\psi(x),\quad\psi_{*}^{\cl}:=\inf_{x\in C}\cl\psi(x)$.
Then, note that the assumption of (b) implies that $\psi_{*}<\infty$.
Now, assume for contradiction that the conclusion of (b) does not
hold. Since $\cl\psi\le\psi$, and hence $\psi_{*}^{\cl}\leq\psi_{*}$,
we must have $\psi_{*}^{\cl}<\psi_{*}$. Hence, due to a well-known
infimum property, there exists $\bar{x}\in C$ such that $\cl\psi(\bar{x})<\psi_{*}<\infty$.
In particular, it follows that $\psi_{*}\in\r$, and hence that $\psi(x)>-\infty$
for every $x\in C$, in view of the definition of $\psi_{*}$. Now,
by assumption, there exists $x_{0}\in C\cap\ri(\dom\psi)$ which,
in view of the previous conclusion, satisfies $\psi(x_{0})>-\infty$.
As $x_{0}\in\ri(\dom\psi)$, this implies that $\psi$ is proper due
to \citep[Theorem 7.2]{Rockafellar1997}. Hence, in view of \citep[Theorem 7.5]{Rockafellar1997}
with $f=\psi$, we have 
\[
\cl\psi(\bar{x})=\lim_{y\in(\bar{x},x_{0}],\,y\to\bar{x}}\psi(y)
\]
where $(\bar{x},x_{0}]:=\left\{ tx_{0}+(1-t)\bar{x}:t\in(0,1]\right\} $.
On the other hand, as $x_{0},\bar{x}\in C$ and $C$ is convex, we
have $(\bar{x},x_{0}]\subseteq C$. This inclusion and the definition
of $\psi_{*}$ then imply that the above limit, and hence $\cl\psi(\bar{x})$,
is greater than or equal to $\psi_{*}$, which contradicts the previously
obtained inequality $\psi_{*}>\cl\psi(\bar{x})$. 
\end{proof}
The following technical lemma presents an important property about
the directional derivative of a composite function $(f+h)$.
\begin{lem}
\label{lem:compl_approx2}Let $h:{\cal X}\mapsto(-\infty,\infty]$
be a proper convex function and let $f$ be a differentiable function
on $\dom h$. Then, for any $x\in\dom h$, it holds that 
\begin{equation}
\inf_{\|d\|\leq1}(f+h)'(x;d)=\inf_{\|d\|\leq1}\left[\inner{\nabla f(x)}d+\sigma_{\partial h(x)}(d)\right]=-\inf_{u\in\nabla f(x)+\pt h(x)}\|u\|.\label{eq:term_lee_gen}
\end{equation}
\end{lem}

\begin{proof}
Let $x\in\dom h$ be fixed and define $\tilde{h}(\cdot):=\inner{\nabla f(x)}{\cdot}+h(\cdot)$.
We first claim that $\inf_{\|d\|\leq1}\tilde{h}'(x;d)=\inf_{\|d\|\leq1}[\cl\tilde{h}'(x;\cdot)](d)$.
Before showing this claim, let us show how it proves the desired conclusion.
Since the definition of $\tilde{h}$ implies that $(f+h)'(x;\cdot)=\tilde{h}'(x;\cdot)$
and $\partial\tilde{h}(x)=\nabla f(x)+\partial h(x)$, it follows
from our previous claim and \citep[Theorem 23.2]{Rockafellar1997}
with $f=\tilde{h}$ that 
\begin{align}
\inf_{\|d\|\leq1}(f+h)'(x;d) & =\inf_{\|d\|\leq1}\tilde{h}'(x;d)=\inf_{\|d\|\leq1}[\cl\tilde{h}'(x;\cdot)](d)\nonumber \\
 & =\inf_{\|d\|\leq1}\sigma_{\partial\tilde{h}(x)}(d)=\inf_{\|d\|\leq1}\left[\inner{\nabla f(x)}d+\sigma_{\partial h(x)}(d)\right],\label{eq:D3}
\end{align}
which gives the first identity in \eqref{eq:term_lee_gen}. The second
identity in \eqref{eq:term_lee_gen} follows from \prettyref{lem:gen_conv_tech}
with $C=\pt{\tilde{h}}(x)$ and the last identity in \eqref{eq:D3}.

To complete the proof, we now justify the claim made the in the previous
paragraph. Define ${\cal B}:=\{d\in{\cal X}:\|d\|\leq1\}$ and $\psi(\cdot):=\tilde{h}'(x;\cdot)$.
In view of \prettyref{lem:gen_conv_tech} with $C={\cal B}$, it suffices
to show that ${\cal B}\cap\ri(\dom\psi)\neq\emptyset$. To show this,
note that the convexity of $\tilde{h}$ and the discussion following
\citep[Theorem 23.1]{Rockafellar1997} imply that $\dom\psi=\bigcup_{t>0}(\dom h-x)/t$,
which is a nonempty convex cone. Hence, it follows from \citep[Theorem 6.2]{Rockafellar1997}
and the discussion in the second paragraph following \citep[Corollary 6.8.1]{Rockafellar1997}
that $\ri(\dom\psi)$ is also a nonempty convex cone. This conclusion
clearly implies that ${\cal B}\cap\ri(\dom\psi)\neq\emptyset$. 
\end{proof}
It is worth mentioning that the result above is a generalization of
the one given in \citep[Lemma 5.1]{Burke1988}, which only considers
the case where $(f+h)$ is real-valued and locally Lipschitz.

The next technical lemma, which can be found in \citep[Corollary 3.3]{Sion1958},
presents a well-known min-max identity.
\begin{lem}
\label{lem:sion_minimax}Let a convex set $D\subseteq{\cal X}$ and
compact convex set $Y\subseteq{\cal Y}$ be given. Moreover, let $\psi:D\times Y\mapsto\r$
be a function in which $\psi(\cdot,y)$ is convex lower semicontinuous
for every $y\in Y$ and $\psi(d,\cdot)$ is concave upper semicontinuous
for every $d\in D$. Then, 
\[
\inf_{d\in{\cal X}}\sup_{y\in{\cal Y}}\psi(d,y)=\sup_{y\in{\cal Y}}\inf_{d\in{\cal X}}\psi(d,y).
\]
\end{lem}

The next result establishes an identity similar to \prettyref{lem:compl_approx2}
but for the case where $f$ is a max function.
\begin{prop}
\label{prop:dd_sp_cvx}Assume the quadruple $(\Psi,h,X,Y)$ satisfies
assumptions \ref{asmp:mco_a1}--\ref{asmp:mco_a4} of \prettyref{sec:minmax_prelim_asmp}
with $\Phi=\Psi$. Moreover, suppose that $\Psi(\cdot,y)$ is convex
for every $y\in Y$, and let $q$ and $Y(\cdot)$ be as in \prettyref{lem:diff_danskin}.
Then, for every $\bar{x}\in X$, it holds that 
\begin{equation}
\inf_{\|d\|\leq1}(q+h)'(\bar{x};d)=\ -\inf_{u\in Q(\bar{x})}\|u\|\label{eq:spec_dd_min}
\end{equation}
where
\begin{equation}
Q(\bar{x}):=\pt h(\bar{x})+\bigcup_{y\in Y(\bar{x})}\left\{ \nabla_{x}\Psi(\bar{x},y)\right\} .\label{eq:Q_def}
\end{equation}
Moreover, if $\pt h(\bar{x})$ is nonempty, then the infimum on the
right-hand side of \eqref{eq:spec_dd_min} is achieved. 
\end{prop}

\begin{proof}
Let $\bar{x}\in X$ and define
\begin{equation}
\psi(d,y):=(\Psi_{y}+h)'(\bar{x};d),\quad\forall(d,x,y)\in{\cal X}\times\Omega\times Y.\label{eq:pPsi_def}
\end{equation}
We claim that $\psi$ in \eqref{eq:pPsi_def} satisfies the assumptions
on $\psi$ in \prettyref{lem:sion_minimax} with $Y=Y(\bar{x})$ and
$D$ given by 
\[
D:=\left\{ d\in{\cal Z}:\|d\|\le1,d\in F_{X}(\bar{x})\right\} ,
\]
where $F_{X}(\bar{x}):=\{t(x-\bar{x}):x\in X,t\geq0\}$ is the set
of feasible directions at $\bar{x}$. Before showing this claim, we
use it to show that \eqref{eq:spec_dd_min} holds. First observe that
\ref{asmp:mco_a2} and \prettyref{lem:diff_danskin} imply that $q'(\bar{x};d)=\sup_{y\in Y}\Psi_{y}'(\bar{x};d)$
for every $d\in{\cal X}$. Using then \prettyref{lem:sion_minimax}
with $Y=Y(\bar{x})$, \prettyref{lem:compl_approx2} with $(f,x)=(\Psi_{\bar{y}},\bar{x})$
for every $\bar{y}\in{Y}(\bar{x})$, and the previous observation,
we have that 
\begin{align}
 & \inf_{\|d\|\leq1}(q+h)'(\bar{x};d)=\inf_{d\in D}(q+h)'(\bar{x};d)=\inf_{d\in D}\sup_{y\in Y(\bar{x})}(\Psi_{y}+h)'(\bar{x};d)\nonumber \\
 & =\inf_{d\in D}\sup_{y\in Y(\bar{x})}\psi(d,y)=\sup_{y\in Y(\bar{x})}\inf_{d\in D}\psi(d,y)=\sup_{y\in Y(\bar{x})}\inf_{\|d\|\leq1}(\Psi_{y}+h)'(\bar{x};d)\nonumber \\
 & =\sup_{y\in Y(\bar{x})}\left[-\inf_{u\in\nabla_{x}\Phi(\bar{x},y)+\pt h(\bar{x})}\|u\|\right]=\left[-\inf_{u\in Q(\bar{x})}\|u\|\right].\label{eq:partial_equiv_approx}
\end{align}
Let us now assume that $\pt h(\bar{x})$ is nonempty, and hence, $Q(\bar{x})$
is nonempty as well. Note that continuity of the function $\nabla_{x}\Psi(\bar{x},\cdot)$
from assumption \ref{asmp:mco_a2} and the compactness of $Y(\bar{x})$
imply that $Q$ is closed. Moreover, since $\|u\|\geq0$, it holds
that any sequence $\{u_{k}\}_{k\geq1}$ where $\lim_{k\to\infty}\|u_{k}\|=\inf_{u\in Q(\bar{x})}\|u\|$
is bounded. Combining the previous two remarks with the Bolzano-Weierstrass
Theorem, we conclude that $\inf_{u\in Q(\bar{x})}\|u\|=\min_{u\in Q(\bar{x})}\|u\|$,
and hence \eqref{eq:spec_dd_min} holds.

To complete the proof, we now justify the above claim on $\psi$.
First, for any given $y\in Y(\bar{x})$, it follows from \citep[Theorem 23.1]{Rockafellar1997}
with $f(\cdot)=\Psi_{y}(\cdot)$ and the definitions of $q$ and $Y(\bar{x})$
that 
\begin{equation}
\psi(d,\bar{y})=\Psi_{\bar{y}}'(\bar{x};d)=\inf_{t>0}\frac{\Psi_{y}(\bar{x}+td)-q(\bar{x})}{t}\quad\forall d\in{\cal X}.\label{eq:Psi_ddir}
\end{equation}
Since assumption \ref{asmp:mco_a3} implies that $\Psi(\bar{x},\cdot)$
is upper semicontinuous and concave on $Y$, it follows from \eqref{eq:Psi_ddir},
\citep[Theorem 5.5]{Rockafellar1997}, and \citep[Theorem 9.4]{Rockafellar1997}
that $\psi(d,\cdot)$ is upper semicontinuous and concave on $Y$
for every $d\in{\cal X}$. On the other hand, since $\Psi(\cdot,y)$
is assumed to be lower semicontinuous and convex on $X$ for every
$y\in Y$, it follows from \eqref{eq:Psi_ddir}, the fact that $\bar{x}\in\intr\Omega$,
and \citep[Theorem 23.4]{Rockafellar1997}, that $\psi(\cdot,y)$
is lower semicontinuous and convex on ${\cal X}$, and hence $D\subseteq{\cal X}$,
for every $y\in Y(\bar{x})$. 
\end{proof}
The last technical result is a specialization of the one given in
\citep[Theorem 4.2.1]{Hiriart-Urruty1993}.
\begin{prop}
\label{prop:nco_refine}Let a proper closed function $\phi:{\cal X}\mapsto(-\infty,\infty]$
and assume that $[\phi+\|\cdot\|^{2}/2\lambda]\in{\cal F}_{\mu}({\cal X})$
for some scalars $\mu,\lam>0$. If a quadruple $(x^{-},x,u,\varepsilon)\in{\cal X}\times\dom\phi\times{\cal X}\times\r_{+}$
together with $\lam$ satisfy 
\begin{equation}
u\in\pt_{\varepsilon}\left(\phi+\frac{1}{2\lambda}\|\cdot-x^{-}\|^{2}\right)(x),
\label{eq:prox_approx_app}
\end{equation}
then the point $\hat{x}\in\dom\phi$ given by 
\begin{equation}
\hat{x}:=\argmin_{x'}\left\{ \phi_{\lam}(x'):=\phi(x')+\frac{1}{2\lam}\|x'-x^{-}\|^{2}-\left\langle u,x'\right\rangle \right\} \label{eq:x_hat_def}
\end{equation}
satisfies 
\begin{equation}
\inf_{\|d\|\leq1}\phi'(\hat{x};d)\geq-\frac{1}{\lambda}\|x^{-}-x+\lambda u\|-\sqrt{\frac{2\varepsilon}{\lam^{2}\mu}},\quad\|\hat{x}-x\|\leq\sqrt{\frac{2\varepsilon}{\mu}}.\label{eq:phi_refine_nonsmooth}
\end{equation}
\end{prop}

\begin{proof}
We first observe that \eqref{eq:prox_approx_app} implies that 
\begin{equation}
\phi_{\lam}(x')\geq\phi_{\lam}(x)-\varepsilon\quad\forall x'\in{\cal X}.\label{eq:prox_gipp}
\end{equation}
Remark that \eqref{eq:prox_gipp} at $x'=\hat{x}$, the optimality
of $\hat{x}$, and the $\mu$--strong convexity of $\phi_{\lam}$
imply that 
\[
\frac{\mu}{2}\|\hat{x}-x\|^{2}\leq\phi_{\lam}(x)-\phi_{\lam}(\hat{x})\leq\varepsilon
\]
from which we conclude that $\|\hat{x}-x\|\leq\sqrt{2\varepsilon/\mu}$,
i.e. the second inequality in \eqref{eq:phi_refine_nonsmooth}. On
the other hand, using the definition of $\phi_{\lam}$, the triangle
inequality, and the previous bound on $\|\hat{x}-x\|$, we obtain
\begin{align}
0 & \leq\inf_{\|d\|\leq1}\phi_{\lam}'(\hat{x};d)=\inf_{\|d\|\leq1}\phi'(\hat{x};d)-\frac{1}{\lambda}\left\langle d,\lambda u+x^{-}-\hat{x}\right\rangle \nonumber \\
 & \leq\inf_{\|d\|\leq1}\phi'(\hat{x};d)+\frac{\|x^{-}-x+\lambda u\|}{\lambda}+\frac{\|x-\hat{x}\|}{\lam}\nonumber \\
 & \leq\inf_{\|d\|\leq1}\phi'(\hat{x};d)+\frac{\|x^{-}-x+\lambda u\|}{\lambda}+\sqrt{\frac{2\varepsilon}{\lam^{2}\mu}},\label{eq:tech_nco_dd1}
\end{align}
which clearly implies the first inequality in \eqref{eq:phi_refine_nonsmooth}. 
\end{proof}

\section{Equivalent Notions of Stationarity}

\label{app:statn_notions}

This appendix presents the proofs of \prettyref{prop:impl1_statn,prop:prox_statn,prop:ne_statn_pt}.

The first technical result shows that an approximate primal-dual stationary
point is equivalent to an approximate directional stationary point
of a perturbed version of problem \ref{prb:eq:min_max_co}.
\begin{lem}
\label{lem:approx_unconstr} Suppose the quadruple $(\Phi,h,X,Y)$
satisfies assumptions \ref{asmp:mco_a1}--\ref{asmp:mco_a4} of \prettyref{sec:minmax_prelim_asmp}
and let $(\bar{x},\bar{u},\bar{v})\in X\times{\cal X}\times{\cal Y}$
be given. Then, there exists $\bar{y}\in Y$ such that the quadruple
$(\bar{u},\bar{v},\bar{x},\bar{y})$ satisfies the inclusion in \eqref{eq:sp_approx_sol}
if and only if 
\begin{equation}
\inf_{\|d\|\le1}(p_{\bar{u},\bar{v}}+h)'(\bar{x};d)\ge0,\label{eq:perturb_dd}
\end{equation}
where the function $p_{_{\bar{u},\bar{v}}}$ is given by 
\begin{equation}
p_{\bar{u},\bar{v}}(x):=\max_{y\in Y}\left[\Phi(x,y)+\inner{\bar{v}}y-\inner{\bar{u}}x\right]\quad\forall x\in\Omega.
\end{equation}
\end{lem}

\begin{proof}
Let $(\bar{x},\bar{u},\bar{v})\in X\times{\cal X}\times{\cal Y}$
be given, define 
\begin{align}
\Psi(x,y):=\Phi(x,y)+\inner{\bar{v}}y-\inner{\bar{u}}x+m\|x-\bar{x}\|^{2}\quad\forall(x,y)\in\Omega\times Y,\label{eq:Psi_app_def}
\end{align}
and let $q$ and $Y(\cdot)$ be as in \prettyref{lem:diff_danskin}.
It is easy to see that $q=p_{\bar{u},\bar{v}}$, the function $\Psi$
satisfies the assumptions on $\Psi$ in \prettyref{prop:dd_sp_cvx},
and $\bar{x}$ satisfies \eqref{eq:perturb_dd} if and only if $\inf_{\|d\|\leq1}(q+h)'(\bar{x};d)\geq0$.
The desired conclusion follows from \prettyref{prop:dd_sp_cvx}, the
previous observation, and the fact that $\bar{y}\in Y(\bar{x})$ if
and only if $\bar{v}\in\pt[-\Phi(\bar{x},\cdot)](\bar{y})$. 
\end{proof}
We are now ready to give the proof of \prettyref{prop:impl1_statn}.
\begin{proof}[Proof of \prettyref{prop:impl1_statn}]
Suppose $([\bar{x},\bar{y}],[\bar{u},\bar{v}])$ is a $(\rho_{x},\rho_{y})$-primal-dual
stationary point of \ref{prb:eq:min_max_co}. Moreover, let $\Psi$,
$q$, and $D_{y}$ be as in \eqref{eq:Psi_app_def}, \eqref{eq:qYbar_def}
and \eqref{eq:Dy_def}, respectively, and define 
\[
\hat{q}(x):=q(x)+h(x)\quad\forall x\in X.
\]
Using \prettyref{lem:approx_unconstr}, we first observe that $\inf_{\|d\|\leq1}\hat{q}(\bar{x};d)\geq0$.
Since $\hat{q}$ is convex from assumption \ref{asmp:mco_a4}, it
follows from the previous bound and \prettyref{lem:compl_approx2}
with $(f,h)=(0,\hat{q})$, that $\min_{u\in\pt\hat{q}(\bar{x})}\|u\|\leq0$,
and hence, $0\in\pt\hat{q}(\bar{x})$. Moreover, using the Cauchy-Schwarz
inequality, the second inequality in \eqref{eq:sp_approx_sol}, the
previous inclusion, and the definition of $q$ and $\Psi$, it follows
that for every $x\in{\cal X}$, 
\begin{align*}
\hat{p}(x)+D_{y}\rho_{y}-\inner{\bar{u}}x+m\|x-\bar{x}\|^{2} & \geq\hat{q}(x)\geq\hat{q}(\bar{x})\geq\hat{p}(\bar{x})-D_{y}\rho_{y}-\inner{\bar{u}}{\bar{x}},
\end{align*}
and hence that $\bar{u}\in\pt_{\varepsilon}(\hat{p}+m\|\cdot-\bar{x}\|^{2})(\bar{x})$
where $\varepsilon=2D_{y}\rho_{y}$. Using now the first inequality
in \eqref{eq:sp_approx_sol}, \prettyref{prop:nco_refine} with $(\phi,x,x^{-},u)=(\hat{p},\bar{x},\bar{x},\bar{u})$
and also $(\varepsilon,\lam,\mu)=(D_{y}\rho_{y},1/(2m),m)$, we conclude
that there exists $\hat{x}$ such that $\|\hat{x}-\bar{x}\|\leq\sqrt{2D_{y}\rho_{y}/m}$
and 
\[
\inf_{\|d\|\leq1}\hat{p}'(\hat{x};d)\geq-\|\bar{u}\|-2\sqrt{2mD_{y}\rho_{y}}\geq-\rho_{x}-2\sqrt{2mD_{y}\rho_{y}}.
\]
\end{proof}
We next give the proof of \prettyref{prop:prox_statn}.
\begin{proof}[Proof of \prettyref{prop:prox_statn}]

(a) We first claim that $\hat{P}_{\lam}\in{\cal F}_{\alpha}(X)$,
where $\alpha=1/\lambda-m$. To see this, note that $\Phi(\cdot,y)+m\|\cdot\|^{2}/2$
is convex for every $y\in Y$ from assumption \ref{asmp:mco_a4}.
The claim now follows from assumption \ref{asmp:mco_a3}, the fact
that the supremum of a collection of convex functions is also convex,
and the definition of $\hat{p}$ in \ref{prb:eq:min_max_co}.

Suppose the pair $(x,\delta)$ satisfies \eqref{eq:dd_approx_sol}
and \eqref{eq:spec_delta_bd}. If $\hat{x}=x_{\lambda}$ in \eqref{eq:dd_approx_sol},
then clearly the second inequality in \eqref{eq:dd_approx_sol}, the
fact that $\lam<1/m$, and \eqref{eq:spec_delta_bd} imply the inequality
in \eqref{eq:prox_stn_point}, and hence, that $x$ is a $(\lam,\varepsilon)$-prox
stationary point. Suppose now that $\hat{x}\neq x_{\lambda}$. Using
the convexity of $\hat{P}_{\lambda}$, we first have that $\hat{P}'_{\lambda}(\hat{x};d)=\inf_{t>0}\left[\hat{P}_{\lambda}(\hat{x}+td)-\hat{P}_{\lambda}(\hat{x})\right]/t$
for every $d\in{\cal X}$. Hence, using both inequalities in \eqref{eq:dd_approx_sol}
and the previous identity, it holds that 
\begin{align*}
\frac{\hat{P}_{\lambda}(x_{\lam})-\hat{P}_{\lambda}(\hat{x})}{\|x_{\lam}-\hat{x}\|} & \geq\hat{P}_{\lambda}'\left(\hat{x};\frac{x_{\lam}-\hat{x}}{\|x_{\lam}-\hat{x}\|}\right)=\hat{p}'\left(\hat{x};\frac{x_{\lam}-\hat{x}}{\|x_{\lam}-\hat{x}\|}\right)+\frac{1}{\lambda}\inner{\frac{x_{\lam}-\hat{x}}{\|x_{\lam}-\hat{x}\|}}{\hat{x}-x}\\
 & \geq-\delta-\frac{1}{\lambda}\|\hat{x}-x\|\geq-\delta\left(\frac{1+\lambda}{\lambda}\right).
\end{align*}
Using the optimality of $x_{\lambda}$, the $\alpha$-strong convexity
of $\hat{P}_{\lambda}$ (see our claim on $\hat{p}$ in the first
paragraph), and the above bound, we conclude that 
\[
\frac{1}{2\alpha}\|\hat{x}-x_{\lambda}\|^{2}\leq\hat{P}_{\lambda}(\hat{x})-\hat{P}_{\lambda}(x_{\lambda})\leq\delta\left(\frac{1+\lambda}{\lambda}\right)\|\hat{x}-x_{\lambda}\|.
\]
Thus, $\|\hat{x}-x_{\lambda}\|\leq2\alpha\delta(1+\lambda)/\lambda$.
Using the previous bound, the second inequality in \eqref{eq:dd_approx_sol},
and \eqref{eq:spec_delta_bd} yields 
\[
\|x-x_{\lambda}\|\leq\|x-\hat{x}\|+\|\hat{x}-x_{\lambda}\|\leq\left(1+2\alpha\left[\frac{1+\lambda}{\lambda}\right]\right)\delta\leq\lambda\varepsilon,
\]
which implies \eqref{eq:prox_stn_point}, and hence, that $x$ is
a $(\lam,\varepsilon)$-prox stationary point.

(b) Suppose that the point $x$ is a $(\lambda,\varepsilon)$-prox
stationary point with $\varepsilon\leq\delta\cdot\min\{1,1/\lambda\}$.
Then the optimality of $x_{\lambda}$, the fact that $\hat{P}_{\lambda}$
is convex (see the beginning of part (a)), the inequality in \eqref{eq:prox_stn_point},
and the Cauchy-Schwarz inequality imply that 
\[
0\leq\inf_{\|d\|\leq1}\left[\hat{p}'(x_{\lambda};d)+\frac{1}{\lambda}\left\langle d,x_{\lambda}-x\right\rangle \right]\leq\inf_{\|d\|\leq1}\hat{p}'(x_{\lambda};d)+\varepsilon\leq\inf_{\|d\|\leq1}\hat{p}'(x_{\lambda};d)+\delta,
\]
which, together with the fact that $\lambda\varepsilon\leq\delta$,
imply that $x$ satisfies \eqref{eq:dd_approx_sol} with $\hat{x}=x_{\lambda}$. 
\end{proof}
Finally, we give the proof of \prettyref{prop:ne_statn_pt}.
\begin{proof}[Proof of \prettyref{prop:ne_statn_pt}]
This follows by using \prettyref{lem:compl_approx2} with $(f,h)=(\Phi(\cdot,\bar{y}),h)$
and $(f,h)=(0,-\Phi(\bar{x},\cdot))$. 
\end{proof}
\newpage{}

\chapter{Spectral Functions}

\label{app:spectral} 

This section presents some results about spectral functions as well
as the proof of \prettyref{prop:acg_implementation}. It is assumed
that the reader is familiar with the key quantities given in \prettyref{subsec:spectral_exploit},
e.g. \eqref{eq:vec_mat_fns}, and the functions in \eqref{eq:dg_Dg_def}.

We first state two well-known results \citep{Beck2017,Lewis1995}
about spectral functions.
\begin{lem}
\label{lem:spec_prop}Let $\Psi=\Psi^{{\cal V}}\circ\sigma$ for some
absolutely symmetric function $\widetilde{\Psi}:\r^{r}\mapsto\r$.
Then, the following properties hold: 
\begin{itemize}
\item[(a)] $\Psi^{*}=(\Psi^{{\cal V}}\circ\sigma)^{*}=(\Psi^{{\cal V}})^{*}\circ\sigma$; 
\item[(b)] $\nabla\Psi=(\nabla\Psi^{{\cal V}})\circ\sigma$; 
\end{itemize}
\end{lem}

\begin{lem}
\label{lem:spec_prox} Let $(\Psi,\Psi^{{\cal V}})$ be as in \prettyref{lem:spec_prop},
the pair $(S,Z)\in{\cal Z}\times\dom\Psi$ be fixed, and the decomposition
$S=P[\dg\sigma(S)]Q^{*}$ be an SVD of $S$, for some $(P,Q)\in{\cal U}^{m}\times{\cal U}^{n}$.
If $\Psi\in\cConv\r^{m\times n}$ and $\Psi^{{\cal V}}\in\cConv\r^{r}$,
then for every $M>0$, we have 
\[
S\in\pt\left(\Psi+\frac{M}{2}\|\cdot\|_{F}^{2}\right)(Z)\iff\begin{cases}
\sigma(S)\in\pt\left(\Psi^{{\cal V}}+\frac{M}{2}\|\cdot\|^{2}\right)(\sigma(Z)),\\
Z=P[\dg\sigma(Z)]Q^{*}.
\end{cases}
\]
\end{lem}

We now present a new result about spectral functions. 
\begin{thm}
\label{thm:spectral_approx_subdiff}Let $(\Psi,\Psi^{{\cal V}})$
be as in \prettyref{lem:spec_prop} and the point $Z\in\r^{m\times n}$
be such that $\sigma(Z)\in\dom\Psi^{{\cal V}}$. Then for every $\varepsilon\geq0$,
we have $S\in\pt_{\varepsilon}\Psi(\acgMatX{})$ if and only if $\sigma(S)\in\pt_{\varepsilon(S)}\Psi^{{\cal V}}(\sigma(Z))$,
where 
\begin{equation}
\varepsilon(S):=\varepsilon-\left[\left\langle \sigma(Z),\sigma(S)\right\rangle -\left\langle Z,S\right\rangle \right]\geq0.\label{eq:eps_s_spectral}
\end{equation}
Moreover, if $S$ and $Z$ have a simultaneous SVD, then $\varepsilon(S)=\varepsilon$. 
\end{thm}

\begin{proof}
Using \prettyref{lem:spec_prop}(a), \eqref{eq:eps_s_spectral}, and
the well-known fact that $S\in\pt_{\varepsilon}\Psi(Z)$ if and only
if $\varepsilon\geq\Psi(Z)+\Psi^{*}(S)-\left\langle Z,S\right\rangle $,
we have that $S\in\pt_{\varepsilon}\Psi(Z)$ if and only if 
\begin{align*}
\varepsilon(S) & =\varepsilon-\left[\left\langle \sigma(Z),\sigma(S)\right\rangle -\left\langle Z,S\right\rangle \right]\\
 & \geq\Psi(Z)+\Psi^{*}(S)-\left\langle Z,S\right\rangle -\left[\left\langle \sigma(Z),\sigma(S)\right\rangle -\left\langle Z,S\right\rangle \right]\\
 & =\Psi^{{\cal V}}(\sigma(Z))+(\Psi^{{\cal V}})^{*}(\sigma(S))-\left\langle \sigma(Z),\sigma(S)\right\rangle ,
\end{align*}
or, equivalently, $\sigma(S)\in\pt_{\varepsilon(S)}\Psi^{{\cal V}}(\sigma(Z))$
and $\varepsilon(S)\geq0$. 

To show that the existence of a simultaneous SVD of $S$ and $Z$
implies $\varepsilon(S)=\varepsilon$ it suffices to show that $\inner{\sigma(S)}{\sigma(Z)}=\inner SZ$.
Indeed, if $S=P[\dg\sigma(S)]Q^{*}$ and $Z=P[\dg\sigma(Z)]Q^{*}$,
for some $(P,Q)\in{\cal U}^{m}\times{\cal U}^{n}$, then we have 
\[
\inner SZ=\inner{\dg\sigma(S)}{P^{*}P[\dg\sigma(Z)]Q^{*}Q}=\inner{\dg\sigma(S)}{\dg\sigma(Z)}=\inner{\sigma(S)}{\sigma(Z)}.
\]
\end{proof}
\newpage{}

\chapter{Computational Details}

\label{app:comp_details}

This appendix presents technical details about the numerical experiments
considered in \prettyref{chap:numerical}.

\subsection*{Generating Parameters for the Quadratic Matrix Problem}

In the unconstrained QM problem of \prettyref{chap:numerical}, recall
that the objective function is of the form
\begin{equation}
f(Z):=\frac{\alpha_{1}}{2}\|{\cal C}Z-d\|^{2}-\frac{\alpha_{2}}{2}\|D{\cal B}Z\|^{2}\label{eq:comp_prb1}
\end{equation}
where ${\cal B}$ and ${\cal C}$ are linear operators, $D$ is a
diagonal matrix, and $d$ is a vector. This appendix describes how,
for a given $(m,M)\in\r_{++}^{2}$, the parameters $\alpha_{1},\alpha_{2}$
are chosen so that $M=\lambda_{\max}(\nabla^{2}f(x))$ and $-m=\lambda_{\min}(\nabla^{2}f(x))$.

Suppose ${\cal B}$ and ${\cal C}$ are full rank. Define the Hessian
matrix
\[
H_{\xi,\tau}:=\alpha_{1}{\cal C}^{*}{\cal C}-\alpha_{2}{\cal B}^{*}D^{2}{\cal B}=\nabla^{2}f(x)
\]
and note that the operators ${\cal B}^{*}D^{2}{\cal B}$ and ${\cal C}^{*}{\cal C}$
are symmetric positive semidefinite. By Weyl's inequality, it holds
that for any $\gamma>0$ we have
\begin{align*}
\lambda_{k}(H_{\xi,\tau}-\gamma{\cal B}^{*}D^{2}{\cal B}) & \leq\lambda_{k}(H_{\xi,\tau})\\
\lambda_{k}(H_{\xi,\tau}) & \leq\lambda_{k}(H_{\xi,\tau}+\gamma{\cal C}^{*}{\cal C})
\end{align*}
for $k=1,...,n$. The above two inequalities imply that $H_{\xi,\tau}$
is monotonically decreasing in $\xi$ and monotonically increasing
in $\tau$. In addition, if ${\cal B}$, ${\cal C}$, and $D$ are
nonzero, then
\begin{align*}
\lim_{\gamma\to\infty}\lambda_{1}(H_{\xi,\tau}+\gamma{\cal C}^{*}{\cal C}) & =+\infty\\
\lim_{\gamma\to\infty}\lambda_{n}(H_{\xi,\tau}-\gamma{\cal C}^{*}{\cal C}) & =-\infty\\
\lim_{\gamma\to\infty}\lambda_{1}(H_{\xi,\tau}+\gamma{\cal B}^{*}D^{2}{\cal B}) & =+\infty\\
\lim_{\gamma\to\infty}\lambda_{n}(H_{\xi,\tau}-\gamma{\cal B}^{*}D^{2}{\cal B}) & =-\infty
\end{align*}
Thus, for a fixed $\xi_{0}>0$, we can find a $\tau_{0}>0$ such that
$\lambda_{\max}(H_{\xi_{0},\tau_{0}})/\lambda_{\min}(H_{\xi_{0},\tau_{0}})=-M/m$
by bisection search and set $(\xi,\tau)=(\xi_{0},\tau_{0})\cdot(M/\tau_{0})$
to obtain the desired conditions $M=\lambda_{\max}(H_{\xi,\tau})$
and $-m=\lambda_{\min}(H_{\xi,\tau})$. 

In \prettyref{chap:numerical}, we implement the above approach, we
with $\xi_{0}=10^{-6}$ and $\tau_{0}=1$ as an initial candidate
solution.

\newpage{}

\chapter{Curvature Constants}

\label{app:num_Lipschitz}

This appendix presents the description of $y_{\xi}$ and justification
for the constants $m,L_{x},$ and $L_{y}$ for each of the optimization
problems in \prettyref{chap:numerical}.

\subsection*{Maximum of a finite number of nonconvex functions}

Recall that
\begin{equation}
M=\lambda_{\text{\ensuremath{\max}}}(\nabla^{2}f_{i}),\quad-m=\lambda_{\min}(\nabla^{2}f_{i})\quad\forall i\in\{1,...,k\}.\label{eq:curv_cond_mm}
\end{equation}
Since $Y=\Delta^{k}$, it is easy to verify that 
\[
y_{\xi}(x)=\argmax_{y'}\left\{ \|y'-\xi g(x)\|:y'\in\Delta^{k}\right\} \quad\forall x\in\rn.
\]
For the validity of the constants $m,L_{x},$ and $L_{y}$, we first
define, for every $1\leq i\leq k$, the quantities 
\[
P_{i}=\alpha_{i}C_{i}^{T}d_{i},\quad Q_{i}^{x}:=\alpha_{i}C_{i}^{T}C_{i}x-\beta_{i}B_{i}^{T}D_{i}^{T}D_{i}B_{i}x\quad\forall x\in\rn,
\]
and observe that $\nabla_{x}\Phi(x,y)=\sum_{i=1}^{k}(Q_{i}^{x}+P_{i})y_{i}.$
Now, using the fact that $y\in\Delta^{k}$, \eqref{eq:curv_cond_mm},
and defining $N_{i}:=\alpha_{i}C_{i}^{T}C_{i}-\beta_{i}B_{i}^{T}D_{i}^{T}D_{i}B_{i}$,
we then have that 
\begin{align*}
\lambda_{\max}(\nabla_{xx}^{2}\Phi) & \leq\sum_{i=1}^{k}y_{i}\lambda_{\max}(N_{i})=\sum_{i=1}^{k}y_{i}\lambda_{\max}(\nabla^{2}g_{i})\leq M=L_{x},\\
\lambda_{\min}(\nabla_{xx}^{2}\Phi) & \geq\sum_{i=1}^{k}y_{i}\lambda_{\min}(N_{i})=\sum_{i=1}^{k}y_{i}\lambda_{\min}(\nabla^{2}g_{i})\geq-m\geq-L_{x},
\end{align*}
and hence we conclude that the choice of $m$ and $L_{x}$ in \eqref{eq:L_mm_unconstr}
is valid. On the other hand, using the fact that $\|x\|\leq1$ for
every $x\in\Delta^{n}$ and \eqref{eq:curv_cond_mm}, we then have
that for every $y,y'\in Y$, 
\begin{align*}
 & \|\nabla_{x}\Phi(x,y)-\nabla_{x}\Phi(x,y')\|=\left\Vert \sum_{i=1}^{k}(Q_{i}^{x}+P_{i})(y_{i}-y_{i}')\right\Vert \\
 & \leq\left(\sqrt{\sum_{i=1}^{k}M^{2}\|x\|^{2}}+\|P\|\right)\|y-y'\|\leq L_{y}\|y-y'\|,
\end{align*}
where $P$ is a an $n$--by--$k$ matrix whose $i^{th}$ column
is $\alpha_{i}C_{i}^{T}d_{i}$, and hence we conclude that the choice
of $L_{y}$ in \eqref{eq:L_mm_unconstr} is valid.

\subsection*{Truncated robust regression}

Since $Y=\Delta^{n}$, it is easy to verify that 
\[
y_{\xi}(x)=\argmax_{y'}\left\{ \|y'-\xi g(x)\|:y'\in\Delta^{n}\right\} \quad\forall x\in\r^{k}.
\]
For the validity of the constants $m,L_{x},$ and $L_{y}$, we first
define for every $1\leq i\leq k$ the function 
\[
\tau_{j}(x):=\left[e^{-b_{j}\left\langle a_{j},x\right\rangle }\right]\left[1+e^{-b_{j}\left\langle a_{j},x\right\rangle }\right]^{-1}\left[\alpha+\ell_{j}(x)\right]^{-1}\quad\forall x\in\r^{k},
\]
and observe that $\nabla_{x}\Phi(x,y)=-\alpha\sum_{j=1}^{n}\left[y_{j}b_{j}\tau_{j}(x)\right]a_{j}$
and also that 
\begin{equation}
\sup_{x\in\r^{k}}|\tau_{j}(x)|\leq\alpha^{-1},\label{eq:tau_j_bd}
\end{equation}
for every $1\leq j\leq n$. Now, using the fact that $y\in\Delta^{n}$,
the bound \eqref{eq:tau_j_bd}, and the Mean Value Theorem applied
to $\tau_{j}$, we have that for every $x,x'\in\r^{k}$, 
\begin{align*}
 & \|\nabla_{x}\Phi(x,y)-\nabla_{x}\Phi(x',y)\|\leq\alpha\sum_{j=1}^{n}y_{j}\|a_{j}\left[\tau_{j}(x)-\tau_{j}(x')\right]\|\\
 & \leq\alpha\max_{j}\left(\|a_{j}\left[\tau_{j}(x)-\tau_{j}(x')\right]\|\right)=\alpha\max_{1\leq j\leq n}\left[\|a_{j}\|\cdot|\tau_{j}(x)-\tau_{j}(x')|\right]\\
 & \leq\alpha\max_{1\leq j\leq n}\left[\|a_{j}\|\sup_{x\in\r^{k}}\|\nabla\tau_{j}(x)\|\|x-x'\|\right]\\
 & =\alpha\max_{1\leq j\leq n}\left[\|a_{j}\|^{2}\sup_{x\in\r^{k}}\left|\frac{\tau_{j}(z)}{\alpha+\ell_{j}(z)}\right|\right]\|x-x'\|\\
 & \leq\frac{1}{\alpha}\max_{1\leq j\leq n}\|a_{j}\|^{2}\|x-x'\|=L_{x}\|x-x'\|,
\end{align*}
and hence we conclude that the choice of $m=L_{x}$ in \eqref{eq:L_trr}
is valid. On the other hand, using the bound \eqref{eq:tau_j_bd},
we have that for every $y,y'\in\rn$, 
\begin{align*}
 & \|\nabla_{x}\Phi(x,y)-\nabla_{x}\Phi(x,y')\|=\alpha\left\Vert \sum_{j=1}^{n}b_{j}\tau_{j}(x)a_{j}[y_{j}-y_{j}']\right\Vert \leq L_{y}\|y-y'\|,
\end{align*}
and hence we conclude that the choice of $L_{y}$ in \eqref{eq:L_trr}
is valid.

\subsection*{Power control in the presence of a jammer}

For every $1\leq k\leq K$ and $1\leq n\leq N$, we first define the
quantities 
\[
S_{k,n}^{-}(X,y):=\sigma^{2}+B_{k,n}y_{n}+\sum_{j=1,j\neq k}^{K}{\cal A}_{j,k,n}X_{j,n},\quad S_{k,n}(X,y):={\cal A}_{k,k,n}X_{k,n}+S_{k,n}^{-},
\]
as well as 
\[
T_{j,n}(X,y):=\left[S_{j,n}^{-}(X,y)+S_{j,n}(X,y)\right]/\left[S_{j,n}(X,y)S_{j,n}^{-}(X,y)\right]^{2},
\]
for every $(X,y)\in\r^{K\times N}\times\r^{N}$. Observe now that
\begin{equation}
\frac{\pt\Phi}{\pt y_{n}}(X,y)=\frac{B_{k,n}}{S_{k,n}(X,y)S_{k,n}^{-}(X,y)}\quad\forall n\in\{1,...,N\}.\label{eq:Phi_pc_def}
\end{equation}
The form in \eqref{eq:Phi_pc_def} implies that $\nabla_{y}\Phi(X,y)$
is a separable function in $y$ where each component is a monotonically
decreasing function in its argument. Hence, since $Y=Q_{N/2}^{N\times1}$,
the computation of $y_{\xi}$ reduces to an $N$--dimensional bisection
search on the functions 
\[
F_{n}(y;\xi)=\left[\sum_{k=1}^{K}\frac{B_{k,n}}{S_{k,n}(X,y)S_{k,n}^{-}(X,y)}\right]-\frac{y_{n}}{\xi}\quad\forall n\in\{1,...,N\}.
\]
For the validity of the constants $m,L_{x},$ and $L_{y}$, we first
observe that, for every $1\leq k\leq K$ and $1\leq n\leq N$ and
also $(X,y)\in\r^{K\times N}\times\r^{N}$, we have 
\[
\frac{\pt\Phi}{\pt X_{k,n}}(X,y)=-\frac{A_{k,k,n}}{S_{k,n}(X,y)}+\sum_{j=1,j\neq k}^{K}\frac{A_{k,j,n}}{S_{j,n}(X,y)S_{j,n}^{-}(X,y)}\quad\forall(X,y)\in\r^{K\times N}\times\r^{N}.
\]
Using the Mean Value Theorem with respect to $X_{k,n}$ on $\pt\Phi/\pt X_{k,n}$,
we have that for every $X,X'\in\r^{K\times N}$, 
\begin{align*}
 & \left|\frac{\pt}{\pt X_{k,n}}f(X,y)-\frac{\pt}{\pt X_{k,n}}f(X',y)\right|\leq\sup_{(X,y)\in\r^{K\times N}\times\r^{N}}\left|\frac{\pt^{2}}{\pt X_{k,n}^{2}}f(X,y)\right||X_{k,n}-X_{k,n}'|\\
 & =\sup_{(X,y)\in\r^{K\times N}\times\r^{N}}\left|\frac{{\cal A}_{k,k,n}^{2}}{S_{k,n}(X,y)}-\sum_{j=1,j\neq k}^{K}{\cal A}_{k,j,n}^{2}T_{j,n}(X,y)\right||X_{k,n}-X_{k,n}'|\\
 & \leq\frac{2\sum_{j=1}^{K}{\cal A}_{k,j,n}^{2}}{\min\{\sigma^{4},\sigma^{6}\}}|X_{k,n}-X_{k,n}'|\leq L_{x}|X_{k,n}-X_{k,n}'|,
\end{align*}
and hence we conclude that the choice of $L_{x}$ in \eqref{eq:L_pc}
is valid. On the other hand, using the Mean Value Theorem with respect
to $y_{n}$ on $\pt\Phi/\pt X_{k,n}$, we have that for every $y,y'\in\r^{K\times N}$,
\begin{align*}
 & \left|\frac{\pt}{\pt X_{k,n}}f(X,y)-\frac{\pt}{\pt X_{k,n}}f(X',y)\right|\leq\sup_{(X,y)\in\r^{K\times N}\times\r^{N}}\left|\frac{\pt^{2}}{\pt y_{n}X_{k,n}}f(X,y)\right||y_{n}-y_{n}'|\\
 & =\sup_{(X,y)\in\r^{K\times N}\times\r^{N}}\left|\frac{B_{k,n}{\cal A}_{k,k,n}}{S_{k,n}(X,y)}-\sum_{j=1,j\neq k}^{K}B_{k,n}{\cal A}_{k,j,n}T_{j,n}(X,y)\right||y_{n}-y_{n}'|\\
 & \leq\frac{2\sum_{j=1}^{K}B_{k,n}{\cal A}_{k,j,n}}{\min\{\sigma^{4},\sigma^{6}\}}|y_{n}-y_{n}'|\leq L_{y}|y_{n}-y_{n}'|,
\end{align*}
and hence we conclude that the choice of $L_{y}$ in \eqref{eq:L_pc}
is valid.

\newpage{}

\singlespacing
\phantomsection
\setlength{\bibitemsep}{\baselineskip}\printbibliography[title={\MakeUppercase{\large\textbf{References}}\vspace*{1em}}]

\addcontentsline{toc}{chapter}{References}

\newpage{}

\doublespacing
\phantomsection\begin{center}
\MakeUppercase{\textbf{\textsc{\large{}Vita}}}\vspace*{2em}
\par\end{center}

Weiwei ``William'' Kong is a Canadian national born on July 16,
1992, in Guangzhou, China. His family immigrated to Canada in the
Summer of 1999 to the suburban district of Scarborough in Toronto,
Ontario. In December 2014, he earned a B. Math. degree from the University
of Waterloo with distinction. After obtaining his undergraduate degree,
he then spent two years as a statistical modeler in Canada. In August
2016, he began his Ph.D. studies in Operations Research at Georgia
Tech, where he received the Thomas Johnson Fellowship. His research
has been funded in part by the Natural Sciences and Engineering Research
Council of Canada (NSERC), the Institute for Data Engineering and
Science (IDEaS), and Transdisciplinary Research Institute for Advancing
Data Science (TRIAD). During his doctoral studies, he obtained an
M.Sc. in Computational Science and Engineering at Georgia Tech in
May 2019 and two internships at Google Research as a Research and
Software Engineering Intern.

\addcontentsline{toc}{chapter}{Vita}
\end{document}